\documentclass[11pt,twoside=semi,a4paper]{article}

\usepackage{layout_thesis}

\RequirePackage{bm}
\RequirePackage{bbm}
\RequirePackage{tabto}
\RequirePackage{authblk}
\RequirePackage{mathtools}
\RequirePackage{enumerate}
\RequirePackage{enumitem}

\newcommand{\N}{\mathbb{N}}
\newcommand{\Z}{\mathbb{Z}}

\newcommand{\R}{\mathbb{R}}
\newcommand{\C}{\mathbb{C}}
\newcommand{\E}{\mathbb{E}}

\newcommand{\bD}{\mathbb D}

\newcommand{\bP}{\mathbb P}

\newcommand{\cB}{\mathcal B}

\newcommand{\cE}{\mathcal E}

\newcommand{\ol}{\overline}
\newcommand{\ul}{\underline}
\newcommand{\tr}{\operatorname{tr}}

\renewcommand{\Re}{\operatorname{Re}}
\renewcommand{\Im}{\operatorname{Im}}

\newcommand{\cs}{\bm{s}}

\newcommand{\tvarepsilon}{\tilde{\varepsilon}}

\newcommand{\diag}{\operatorname{diag}}
\newcommand{\dist}{\operatorname{dist}}
\newcommand{\supp}{\operatorname{supp}}

\newcommand{\Id}{\operatorname{Id}}

\makeatletter
\newcommand*\bigcdot{\mathpalette\bigcdot@{.5}}
\newcommand*\bigcdot@[2]{\mathbin{\vcenter{\hbox{\scalebox{#2}{$\m@th#1\bullet$}}}}}
\makeatother

\swapnumbers
\theoremstyle{plain}
\newtheorem{theorem}{Theorem}[section]

\newtheorem*{theorem*}{Theorem}
\newtheorem{lemma}[theorem]{Lemma}
\newtheorem{proposition}[theorem]{Proposition}
\newtheorem{corollary}[theorem]{Corollary}
\theoremstyle{definition}
\newtheorem{definition}[theorem]{Definition}

\newtheorem{assumption}[theorem]{Assumption}

\theoremstyle{remark}
\newtheorem{remark}[theorem]{Remark}

\swapnumbers
\theoremstyle{indented}

\title{Spectral approximation for the separable covariance mixture model}
\date{}
\author{Ben Deitmar}
\affil{\small \textit{Department of Mathematical Stochastics, University of Freiburg \protect\\ Ernst-Zermelo-Str. 1, 79104 Freiburg, Germany \protect\\ E-mail: ben.deitmar@stochastik.uni-freiburg.de}}

\begin{document}
\thispagestyle{empty}
\maketitle
	
\begin{abstract}
	\noindent
	This paper introduces the \textit{separable covariance mixture model}, which assumes a data-matrix $\bm{Y}$ to be of the form
	\begin{align*}
		& \sum\limits_{r=1}^R A_r \bm{X} B_r
	\end{align*}
	for one random $(d \times n)$-matrix $\bm{X}$ with independent centered variance-one entries, and for two families of deterministic matrices $A_1,\dots,A_R \in \C^{d \times d}$ and $B_1,\dots,B_R \in \C^{n \times n}$. Under certain assumptions (see Assumption \ref{Assumption_Basic}), it is shown that the resolvents $(\frac{1}{n} \bm{Y} \bm{Y}^* - z \Id_d)^{-1}$ and $(\frac{1}{n} \bm{Y}^* \bm{Y} - z \Id_n)^{-1}$ respectively approximate the deterministic matrices 
	\begin{align*}
		& -\frac{1}{z}\Big( \Id_d + \sum\limits_{r,s=1}^R \delta^{(B)}_{r,s}(z) A_{r} A_{s}^* \Big)^{-1} \ \ \text{ and } \ \ -\frac{1}{z}\Big( \Id_n + \sum\limits_{r,s=1}^R \delta^{(A)}_{r,s}(z) B_{s}^*B_{r} \Big)^{-1}
	\end{align*}
	(compare Theorem \ref{Thm_SepCovMP_NonAsymp}), where $\delta^{(A)}, \delta^{(B)} \in \C^{R \times R}$ are uniquely defined solutions to a certain dual system of equations (see Theorem \ref{Thm_DetEquiv}). The results are non-asymptotic and do not require simultaneous diagonalizability of the families $(A_r)_{r \leq R}$ or $(B_r)_{r \leq R}$, as was required in previous works such as \cite{SepCovMixture_Paul} or \cite{StrongMP}. An asymptotic application, which describes the limiting spectral distribution of the sample covariance matrix analogues $\frac{1}{n} \bm{Y} \bm{Y}^*$ or $\frac{1}{n} \bm{Y}^* \bm{Y}$, is included (see Corollary \ref{Cor_SepCovMP_Asymp}).
\end{abstract}

\section{Introduction}\label{Section_SepCov_Introduction}
In the analysis of iid high-dimensional (centered) data $Y_1,\dots,Y_n \in \C^d$, the assumption that the $(d \times n)$-data-matrix $\bm{Y} \coloneq [Y_1,\dots,Y_n]$ is of the form
\begin{align}\label{Eq_Intro_iidData}
	& \bm{Y} = \Sigma^{\frac{1}{2}} \bm{X} \ ,
\end{align}
for some deterministic $(d \times d)$-covariance-matrix $\Sigma$ and a random $(d \times n)$-matrix $\bm{X}$ with independent centered entries each with variance one, has been used to great effect for the analysis of the limiting eigenvalue distribution of $\bm{S}=\frac{1}{n} \bm{Y} \bm{Y}^*$ (cf. \cite{KnowlesAnisotropicLocalLaws, MP_original}), for CLTs of linear spectral statistics (cf. \cite{BaiCLT,NajimYao}) or for eigen inference (cf. \cite{MPIKaroui,LedoitWolf2}). Similar models are commonly employed to model data with more intricate dependence structure. For example, separable covariance models assume the data-matrix $\bm{Y}$ to be of the form
\begin{align}\label{Eq_Intro_SepCovData}
	& \bm{Y} = \Sigma_1^{\frac{1}{2}} \bm{X} \Sigma_2^{\frac{1}{2}}
\end{align}
for the same $(d \times n)$-matrix $\bm{X}$ and two separate covariance matrices $\Sigma_1 \in \C^{d \times d}$ and $\Sigma_2 \in \C^{n \times n}$ (cf. \cite{SepCovCLT,LixinSALDRM}). Data from stationary time series $(Y_t)_{t \in \Z}$ is often modeled by linear processes (cf. \cite{TimeMP4,Aue}), which are of the form
\begin{align}\label{Eq_Intro_LinProcData}
	& Y_t = \sum\limits_{k=0}^\infty \Psi_k \xi_{t-k}
\end{align}
for some deterministic $(d \times d)$-matrices $(\Psi_k)_{k \in \N_0}$ and iid isotropic $d$-dimensional innovation vectors $(\xi_t)_{t \in \Z}$. If $\bm{Y} = [Y_1,\dots,Y_n]$ consists of a finite sample from such a linear process, a finite-$K$ approximation might be
\begin{align}\label{Eq_Intro_LinProcDataApprox}
	& \bm{Y} \approx \sum\limits_{k=0}^K \Psi_k \bm{X} T_{k} \ ,
\end{align}
where $\bm{X} = [\xi_1,\dots,\xi_n]$ and $T_k$ is the $(n \times n)$-matrix with ones on the upper $k$-th diagonal and zeroes elsewhere. The paper \cite{SepCovMixture_Paul} studies the model
\begin{align}\label{Eq_Intro_PaulModel}
	& \bm{Y} = \sum\limits_{r=1}^R \Sigma_{1,r}^{\frac{1}{2}} \bm{X}_r \Sigma_{2,r}^{\frac{1}{2}}
\end{align}
for independent copies $\bm{X}_1,\dots,\bm{X}_R$ of $\bm{X}$ as in (\ref{Eq_Intro_SepCovData}) and families $(\Sigma_{1,r})_{r \leq R}$ and $(\Sigma_{2,r})_{r \leq R}$ of covariance matrices such that each family is simultaneously diagonalizable.
\\[0.5em]
While the limiting eigenvalue distributions of $\bm{S} = \frac{1}{n} \bm{Y} \bm{Y}^*$ and $\tilde{\bm{S}} = \frac{1}{n} \bm{Y}^* \bm{Y}$ are well-understood in the cases (\ref{Eq_Intro_iidData}) and (\ref{Eq_Intro_SepCovData}), available results for the cases (\ref{Eq_Intro_LinProcData}) and (\ref{Eq_Intro_PaulModel}) require restrictive assumptions such as the matrices $(\Psi_k)_{k \in \N_0}$ or $(\Sigma_{1,r})_{r \leq R}$ and $(\Sigma_{2,r})_{r \leq R}$ being simultaneously diagonalizable (cf. \cite{TimeMP4,Aue,StrongMP} for (\ref{Eq_Intro_LinProcData}) and \cite{SepCovMixture_Paul,StrongMP} for (\ref{Eq_Intro_PaulModel})).
\\[0.5em]
This paper introduces the \textit{separable covariance mixture model}
\begin{align}\label{Eq_Intro_SepCovMixture}
	& \bm{Y} = \sum\limits_{r=1}^R A_r \bm{X} B_r
\end{align}
for deterministic matrices $A_1,\dots,A_R \in \C^{d \times d}$ and $B_1,\dots,B_R \in \C^{n \times n}$, where $\bm{X}$ is again a random $(d \times n)$-matrix with independent centered entries with variance one.
\\[0.5em]
Theorem~\ref{Thm_SepCovMP_NonAsymp} will, under suitable conditions that do not include simultaneous diagonalizability (see Assumption~\ref{Assumption_Basic}), show that for all $z \in \C^+$ the resolvents
\begin{align*}
	& \bm{R}(z) \coloneq \big( \bm{S} - z\Id_d \big)^{-1} \ \ \text{ and } \ \ \tilde{\bm{R}}(z) \coloneq \big( \tilde{\bm{S}} - z\Id_n \big)^{-1}
\end{align*}
approximate the deterministic matrices
\begin{align*}
	& -\frac{1}{z}\Big( \Id_d + \sum\limits_{r',s'=1}^R \delta^{(B)}_{r',s'}(z) A_{r'} A_{s'}^* \Big)^{-1} \ \ \text{ and } \ \ -\frac{1}{z}\Big( \Id_n + \sum\limits_{r',s'=1}^R \delta^{(A)}_{r',s'}(z) B_{s'}^*B_{r'} \Big)^{-1}
\end{align*}
respectively, where $\delta^{(A)}, \delta^{(B)} \in \C^{R \times R}$ are uniquely defined solutions to a dual system of equations (see (\ref{Eq_DualSystem})). The approximation is in the sense that for any deterministic test-matrices $M \in \C^{d \times d}$ and $\tilde{M} \in \C^{n \times n}$, the differences
\begin{align*}
	& \Big| \frac{1}{n}\tr\big( M \bm{R}(z) \big) - \frac{1}{n}\tr\Big( M \, \frac{-1}{z} \Big( \Id_d + \sum\limits_{r,s=1}^R \delta^{(B)}_{r,s}(z) A_{r}A_{s}^* \Big)^{-1} \Big) \Big| \nonumber\\
	& \Big| \frac{1}{n}\tr\big( \tilde{M} \tilde{\bm{R}}(z) \big) - \frac{1}{n}\tr\Big( \tilde{M} \, \frac{-1}{z} \Big( \Id_n + \sum\limits_{r,s=1}^R \delta^{(A)}_{r,s}(z) B_{s}^*B_{r} \Big)^{-1} \Big) \Big|
\end{align*}
will be of order $\mathcal{O}(n^{\varepsilon-\frac{1}{2}})$ with high probability (cf. (\ref{Eq_MainRes_Result_M}) in Theorem~\ref{Thm_SepCovMP_NonAsymp}).
\\[0.5em]
The main theoretical tools are listed in Section~\ref{Section_SupportingTheorems} and consist of:
\begin{itemize}
	\item a quantitative uniqueness argument for near-solutions of the dual system of equations (see Theorem~\ref{Thm_DualApprox}),
	
	\item a specialized usage of Stein's lemma to prove approximations in expectations for the Gaussian case (see Theorem~\ref{Thm_EmpiricalDualSystem} and Section~\ref{Section_Stein}), after which the expectations may be removed with concentration bounds (see Section~\ref{Section_GaussianCase}),
	
	\item a standard universality argument to lift approximations in the Gaussian case to hold for more general models (see Theorem~\ref{Thm_Universality}).
\end{itemize}

\noindent
The consequences of Theorem~\ref{Thm_SepCovMP_NonAsymp} to the limiting eigenvalue distribution of $\bm{S}$ (or in this case $\tilde{\bm{S}}$) are presented in Corollary~\ref{Cor_SepCovMP_Asymp}.

\subsection{Model and notation}\label{Subsection_Model}
Let $\bm{X}$ be a random $(d \times n)$-matrix with independent centered entries, each with variance one. Let $A_1,\dots,A_R$ be deterministic $(d \times d)$-matrices and let $B_1,\dots,B_R$ be deterministic $(n \times n)$-matrices which satisfy the non-degeneracy condition that
\begin{align}\label{Eq_Assumption_NonDegeneracy}
	& \sum\limits_{r=1}^R A_rA_r^* \in \C^{d \times d} \ \ \text{ and } \ \ \sum\limits_{r=1}^R B_r^*B_r \in \C^{n \times n} \ \ \text{ are non-singular}
\end{align}
and also satisfy the identifiability condition that the $(R \times R)$-matrices
\begin{align}\label{Eq_Assumption_Identifiability}
	& \Big( \frac{1}{n} \tr\big(A_rA_s^*\big) \Big)_{r,s \leq R} \ \ \text{ and } \ \ \Big( \frac{1}{n} \tr\big(B_s^*B_r\big) \Big)_{r,s \leq R} \ \ \text{ are non-singular.}
\end{align}
For a discussion of these assumptions, see the paragraph to Assumption~\ref{ItemTempAssumption_NonDegeneracy} in Remark~\ref{Remark_AssumptionDiscussion}.
\\[0.5em]
Define
\begin{align}\label{Eq_DefY}
	& \bm{Y} \coloneq \sum\limits_{r=1}^R A_r \bm{X} B_r
\end{align}
and
\begin{align}\label{Eq_DefS}
	& \bm{S} \coloneq \frac{1}{n} \bm{Y}\bm{Y}^* \ \ ; \ \ \tilde{\bm{S}} \coloneq \frac{1}{n} \bm{Y}^*\bm{Y}
\end{align}
as well as the resolvents
\begin{align}\label{Eq_Def_RResolvent}
	& \bm{R}(z) \coloneq \big( \bm{S} - z\Id_d \big)^{-1} \ \ ; \ \ \tilde{\bm{R}}(z) \coloneq \big( \tilde{\bm{S}} - z\Id_n \big)^{-1}
\end{align}
for any $z \in \C^+$.
\\[0.5em]
Let $u_1,\dots,u_d$ denote the eigenvectors of $\bm{S}$ corresponding to the eigenvalues $\lambda_1(\bm{S}),\dots,\lambda_d(\bm{S})$ and let $\tilde{u}_1,\dots,\tilde{u}_n$ denote the eigenvectors of $\tilde{\bm{S}}$ corresponding to the eigenvalues $\lambda_1(\tilde{\bm{S}}),\dots,\lambda_n(\tilde{\bm{S}})$.
Define the empirical matrix-valued measures
\begin{align*}
	& \hat{\rho}^{(A)} , \hat{\rho}^{(B)} : \cB(\R) \rightarrow \{ M \in \C^{R \times R} \text{ positive semi-definite} \}
\end{align*}
by
\begin{align}\label{Eq_Def_hatRhoA}
	& \hat{\rho}^{(A)}(S) \coloneq \frac{1}{n} \sum\limits_{\substack{j=1 \\ \lambda_j(\bm{S}) \in S}}^d (u_j^* A_r A_s^* u_j)_{r,s \leq R}
\end{align}
and
\begin{align}\label{Eq_Def_hatRhoB}
	& \hat{\rho}^{(B)}(S) \coloneq \frac{1}{n} \sum\limits_{\substack{j=1 \\ \lambda_j(\tilde{\bm{S}}) \in S}}^n (\tilde{u}_j^* B_s^*B_r \tilde{u}_j)_{r,s \leq R}
\end{align}
for any set $S \in \cB(\R)$.
Also define the empirical matrix-valued Stieltjes transforms
\begin{align}\label{Eq_hatDeltaIntro}
	& \hat{\delta}^{(A)}, \hat{\delta}^{(B)} : \C^+ \coloneq \{z \in \C \mid \Im(z) > 0\} \rightarrow \C^{R \times R}
\end{align}
by
\begin{align}\label{Eq_Def_hatDelta_A}
	& \hat{\delta}^{(A)}_{r,s}(z) \coloneq \frac{1}{n} \tr\big( A_r A_s^* \bm{R}(z) \big) = \int_0^\infty \frac{1}{\lambda - z} \, d\hat{\rho}^{(A)}_{r,s}(\lambda)
\end{align}
and
\begin{align}\label{Eq_Def_hatDelta_B}
	& \hat{\delta}^{(B)}_{r,s}(z) \coloneq \frac{1}{n} \tr\big( B_s^* B_r \tilde{\bm{R}}(z) \big) = \int_0^\infty \frac{1}{\lambda - z} \, d\hat{\rho}^{(B)}_{r,s}(\lambda) \ .
\end{align}
To a square matrix $A$, define its real and imaginary parts as
\begin{align}\label{Eq_DefImaginaryMatrix}
	& \Re(A) \coloneq \frac{1}{2} \big( A + A^* \big) \ \ , \ \ \Im(A) \coloneq \frac{1}{2i} \big( A - A^* \big) \ .
\end{align}
For any two square matrices $A$ and $B$, write $A \prec B$ or $A \preceq B$, if the difference $B-A$ is positive definite or semi-definite respectively.
\\[0.5em]
The standard unit vectors in $\C^d$ and $\C^n$ will be denoted by
\begin{align}\label{Eq_Def_UnitVec}
	& e_{j,d} \coloneq (\mathbbm{1}_{i=j})_{i \leq d} \in \C^d \ \ \text{ and } \ \ e_{k,n} \coloneq (\mathbbm{1}_{l=k})_{l \leq n} \in \C^n
\end{align}
respectively.
Denote the eigenvalues of Hermitian matrices $M \in \C^{d \times d}$ in the ordered scheme
\begin{align}\label{Eq_EigenvalueOrdering}
	& \lambda_1(M) \geq \dots \geq \lambda_d(M) \ .
\end{align}
Let $||\cdot||$ denote the spectral norm, which is known to be sub-multiplicative, i.e.
\begin{align}\label{Eq_SubmultiplicativitySpectralNorm}
	& ||M_1 \, M_2|| \leq ||M_1|| \, ||M_2|| \ ,
\end{align}
and coincides with the largest eigenvalue for positive semi-definite matrices. Let $||\cdot||_F$ denote the Frobenius norm, which is known to satisfy
\begin{align}\label{Eq_FrobeniusBounds}
	& ||M|| \leq ||M||_F \leq \sqrt{\operatorname{rank}(M)} ||M|| \ \ \text{ and } \ \ ||M_1 \, M_2||_F \leq ||M_1|| \, ||M_2||_F \ .
\end{align}

\subsection{Deterministic equivalents}
The following theorem will be shown in Section~\ref{Proof_Thm_DetEquiv}. For $R=1$, it reduces to known properties of the deterministic equivalents for separable sample covariance matrices (compare Propositions 1.1 and 1.2 in \cite{CouilletHachem}).

\begin{theorem}[Deterministic equivalents]\label{Thm_DetEquiv}\
	\\
	Let $A_1,\dots,A_R$ be deterministic $(d \times d)$-matrices and let $B_1,\dots,B_R$ be deterministic $(n \times n)$-matrices such that conditions (\ref{Eq_Assumption_NonDegeneracy}) and (\ref{Eq_Assumption_Identifiability}) are true.\\
	The following properties hold.
	\begin{itemize}
		\item[a)]
		For each $z \in \C^+$, there exist $(R \times R)$-matrices $\delta^{(A)}(z)$ and $\delta^{(B)}(z)$ uniquely defined by the dual system of equations
		\begin{align}\label{Eq_DualSystem}
			& \forall r,s \in \{1,\dots,R\} : \nonumber\\
			& \delta^{(A)}_{r,s}(z) = -\frac{1}{z} \frac{1}{n} \tr\bigg( A_rA_s^* \Big( \Id_d + \sum\limits_{r',s'=1}^R \delta^{(B)}_{r',s'}(z) A_{r'} A_{s'}^* \Big)^{-1} \bigg) \nonumber\\
			& \delta^{(B)}_{r,s}(z) = -\frac{1}{z} \frac{1}{n} \tr\bigg( B_s^*B_r \Big( \Id_n + \sum\limits_{r',s'=1}^R \delta^{(A)}_{r',s'}(z) B_{s'}^*B_{r'} \Big)^{-1} \bigg)
		\end{align}
		and the properties
		\begin{align}\label{Eq_Uniqueness_posDefCondition}
			& 0 \prec \Im\big( \delta^{(A)}(z) \big) \ \ , \ \ 0 \prec \Im\big( \delta^{(B)}(z) \big) \ .
		\end{align}
		
		\item[b)]
		There exist matrix-valued measures
		\begin{align*}
			& \rho^{(A)} , \rho^{(B)} : \cB(\R) \rightarrow \{ M \in \C^{R \times R} \text{ positive semi-definite} \}
		\end{align*}
		which are uniquely defined by the fact that $\delta^{(A)}(z)$ and $\delta^{(B)}(z)$ from (a) are their matrix-valued Stieltjes transforms, i.e.
		\begin{align}\label{Eq_Uniqueness_deltaStil}
			& \forall z \in \C^+ : \ \int_\R \frac{1}{\lambda - z} \, d\rho^{(A)}(\lambda) = \delta^{(A)}(z) \ \ \text{ and } \ \ \int_\R \frac{1}{\lambda - z} \, d\rho^{(B)}(\lambda) = \delta^{(B)}(z) \ .
		\end{align}
		
		\item[c)]
		The matrix-valued measures satisfy
		\begin{align}\label{Eq_DetEquiv_RhoProp2}
			& \rho^{(A)}(\R) = \frac{1}{n} \tr\big( A_r A_s^* \big)_{r,s \leq R} \ \ \text{ and } \ \ \rho^{(B)}(\R) = \frac{1}{n} \tr\big( B_s^*B_r \big)_{r,s \leq R}
		\end{align}
		as well as
		\begin{align}\label{Eq_DetEquiv_RhoProp1}
			& \supp\big( \rho^{(A)} \big) \subset [0,8\sigma^4(1+\sqrt{c_*})^2] \ \ \text{ and } \ \ \supp\big( \rho^{(B)} \big) \subset [0,8\sigma^4(1+\sqrt{c_*})^2] \ ,
		\end{align}
		for any $c_* \geq \max(d/n, 1)$ and $\sigma^2 \geq \max\big( \sum\limits_{r=1}^R ||A_r||^2 , \sum\limits_{r=1}^R ||B_r||^2 \big)$.
		
		\item[d)]
		There exists a probability measure $\ul{\nu}$ on $\R$, which is uniquely defined by the property
		\begin{align}\label{Eq_DetEquiv_NuStilProp}
			& \forall z \in \C^+ : \ \cs_{\ul{\nu}}(z) \coloneq \int_\R \frac{1}{\lambda - z} \, d\ul{\nu}(\lambda) = -\frac{1}{z} - \sum\limits_{r,s=1}^R \delta^{(A)}_{r,s}(z) \delta^{(B)}_{r,s}(z) \ .
		\end{align}
		The support of $\ul{\nu}$ is contained in the interval $[0,8\sigma^4(1+\sqrt{c_*})^2]$ for $c_*$ and $\sigma^2$ as in (c).
	\end{itemize}
\end{theorem}


\subsection{The matrix Dyson equation}
The matrix Dyson equation generalizes well-known limiting laws such as the Marchenko--Pastur law for sample covariance matrices and the Semicircle law for Wigner matrices. For a random Hermitian $(N \times N)$-matrix $\bm{H}$, under certain conditions (see Section 2.2 of \cite{MDE1}) on the structure of $\bm{H}$, the resolvent
\begin{align*}
	& \bm{G} : \C^+ \rightarrow \C^{N \times N} \ \ ; \ \ \bm{G}(v) \coloneq \big( \bm{H} -v\Id_N \big)^{-1}
\end{align*}
is well-approximated by a deterministic equivalent matrix $G(v)$ defined as the unique solution to the matrix Dyson equation
\begin{align}\label{Eq_MatrxDysonEquation}
	& 0 = \Id_{N} + \big( v\Id_N - A + \mathcal{S}[G] \big) G
\end{align}
for which the matrix $\Im(G)$ is also positive definite. Here, the matrix $A$ and the operator $\mathcal{S}$ are given by
\begin{align*}
	& A \coloneq \E[\bm{H}] \ \ \text{ and } \ \ \mathcal{S}[M] = \E\big[ (\bm{H}-A) M (\bm{H}-A) \big] \ .
\end{align*}
To date, the most general results on the approximation of $\bm{G}(v)$ by $G(v)$ require strong assumptions on the covariance structure between entries of $\bm{H}$ (cf. \cite{MDE1, MDE_SlowCorDecay}), which makes them not applicable to the separable covariance mixture model analyzed in this paper.

\section{Main results}

\begin{assumption}[Non-asymptotic assumptions]\label{Assumption_Basic}\
	\begin{enumerate}[label=(A\arabic*)]
		\item\label{ItemAssumption_cBound}
		Let $d,n \in \N$ and let $c_* \geq 1$ denote a constant satisfying
		\begin{align}\label{Eq_Assumption_cBound}
			& \frac{d}{n} \leq c_* \ .
		\end{align}
		
		\item\label{ItemAssumption_sigmaBound}
		Let $(A_1,\dots,A_R)$ be a family of $(d \times d)$-matrices and let $(B_1,\dots,B_R)$ be a family of $(n \times n)$-matrices satisfying
		\begin{align}\label{Eq_sigmaAssumption_NonAsymp}
			& \sum\limits_{r=1}^R ||A_{r}||^2 \leq \sigma^2 \ \ \text{ and } \ \ \sum\limits_{r=1}^R ||B_{r}||^2 \leq \sigma^2
		\end{align}
		for some constant $\sigma^2 \geq 1$.
		
		\item\label{ItemAssumption_boundedSixthMoment}
		Suppose the random $(d \times n)$-matrix $\bm{X}$ has independent centered $\C$-valued entries, each with $\E[|\bm{X}_{i,j}|^2]=1$. Further suppose there exists a constant $C_6 \geq 1$ such that
		\begin{align*}
			& \forall i \leq d, \, j \leq n : \ \E\big[|\bm{X}_{i,j}|^6\big] \leq C_6 \ .
		\end{align*}
		
		\item\label{ItemTempAssumption_NonDegeneracy}
		Suppose there exists a constant $\tau \in (0,1)$ such that the \textit{non-degeneracy bounds}
		\begin{align}\label{Eq_AssumptionNonDegeneracy}
			& \lambda_{\min}\Big( \sum\limits_{r=1}^R A_rA_r^* \Big) \geq \tau \ \ \text{ and } \ \ \lambda_{\min}\Big( \sum\limits_{r=1}^R B_r^*B_r \Big) \geq \tau
		\end{align}
		hold and that the positive semi-definite $(R \times R)$-matrices
		\begin{align}\label{Eq_DefGram}
			& \mathbb{G}^{(A)} \coloneq \Big( \frac{1}{n} \tr\big(A_rA_s^*\big) \Big)_{r,s \leq R} \ \ \text{ and } \ \ \mathbb{G}^{(B)} \coloneq \Big( \frac{1}{n} \tr\big(B_s^*B_r\big) \Big)_{r,s \leq R}
		\end{align}
		satisfy the \textit{identifiability bounds}
		\begin{align}\label{Eq_AssumptionIdentifiability}
			& \lambda_{\min}\big(\mathbb{G}^{(A)}\big) \geq \tau \ \ \text{ and } \ \ \lambda_{\min}\big(\mathbb{G}^{(B)}\big) \geq \tau \ .
		\end{align}
		
	\end{enumerate}
\end{assumption}

\begin{remark}[Discussion of assumptions]\label{Remark_AssumptionDiscussion}\
	\\
	Assumption~\ref{ItemAssumption_cBound} may seem weaker than the usual assumption that $\frac{d}{n} \in [c_1,c_2]$ for some $c_1,c_2>0$ or even $\frac{d}{n} \xrightarrow{n \to \infty} c_\infty > 0$ as is often assumed in works on high-dimensional models (cf. \cite{BaiSALDRM,BaiCLT,MP_original}). While the presented results will also hold for $d \ll n$, the scaling of $\frac{1}{n}$ in (\ref{Eq_Def_hatRhoA})-(\ref{Eq_Def_hatDelta_B}) or (\ref{Eq_DualSystem}) implies that these results are of limited interest in the regime $\frac{d}{n} \xrightarrow{n \rightarrow \infty} 0$.
	\\[0.5em]
	Assumption~\ref{ItemAssumption_sigmaBound} is the model-specific version of a standard bound on the population eigenvalues, which is used in most works on eigen inference (cf. \cite{DingFan_SpikedShrinkage, MPIKaroui, MPIKong, LedoitWolf2}).
	\\[0.5em]
	Assumption~\ref{ItemAssumption_boundedSixthMoment} is a standard moment bound on the entries of $\bm{X}$. Such assumptions on the moments or tail behavior of the entries also arise in most works on random matrix theory (RMT) (cf. \cite{BaiSALDRM,VershyninHDP,MPIYaoZhengBai}). Proving results of RMT under minimal moment assumptions can be highly challenging (cf. \cite{Pastur_SemiCircleLaw, TaoVu_CircularLaw}) and some results, such as local laws, commonly assume the existence of all moments (cf. \cite{AltGram, KnowlesAnisotropicLocalLaws}).
	\\[0.5em]
	Assumption~\ref{ItemTempAssumption_NonDegeneracy} is specific to the model under consideration and thus cannot immediately be compared to assumptions in previous works. It is a quantitative strengthening of Assumptions (\ref{Eq_Assumption_NonDegeneracy}) and (\ref{Eq_Assumption_Identifiability}), and is employed in non-asymptotic bounds in the proof of Theorem~\ref{Thm_SepCovMP_NonAsymp}. Although this is not pursued here, Assumption (\ref{Eq_AssumptionNonDegeneracy}) (and thus indirectly (\ref{Eq_Assumption_NonDegeneracy})) may be removed entirely from the asymptotic result in Corollary~\ref{Cor_SepCovMP_Asymp} by an approximation argument. On the other hand, the non-quantitative Assumption (\ref{Eq_Assumption_Identifiability}) is necessary for the existence of the deterministic equivalents introduced in Theorem~\ref{Thm_DetEquiv}, which may be seen by comparing the condition (\ref{Eq_Uniqueness_posDefCondition}) with the property (\ref{Eq_DetEquiv_RhoProp2}).
\end{remark}

\begin{definition}[Spectral domain]\label{Def_SpectralDomain}\
	\\
	For any constants $\eta,\kappa>0$, let $\bD(\eta,\kappa) \subset \C^+$ denote the spectral domain
	\begin{align}\label{Eq_Def_bD}
		& \bD(\eta,\kappa) \coloneq \big\{ z \in \C^+ \ \big| \ \Im(z) \geq \eta , \, |z| \leq \kappa \big\} \ .
	\end{align}
\end{definition}

\begin{theorem}[Non-asymptotic MP law analogue]\label{Thm_SepCovMP_NonAsymp}\
	\\
	Suppose~\ref{ItemAssumption_cBound}-\ref{ItemTempAssumption_NonDegeneracy} hold.
	For any $\eta,\kappa>0$ and $\tvarepsilon \in (0,1)$, there exist constants
	\begin{align*}
		& \mathcal{C}=\mathcal{C}(c_*,\sigma^2,C_6,\tau,\eta,\kappa,\tvarepsilon)>0 \ \ \text{ and } \ \ N=N(R,c_*,\sigma^2,\tau,\eta,\kappa,\tvarepsilon)>0
	\end{align*}
	which neither depend on $d$ and $n$ nor on the model $\big( \bm{X}, (A_r)_{r \leq R}, (B_r)_{r \leq R} \big)$, such that for all $n \geq N$ the bound
	\begin{align}\label{Eq_MainRes_Result}
		& \bP\Big( \forall z \in \bD(\eta,\kappa) : \ \big|\big| \hat{\delta}^{(A)}(z) - \delta^{(A)}(z) \big|\big| \leq R \mathcal{C} n^{\frac{\tvarepsilon-1}{2}} \nonumber\\
		& \hspace{2.5cm} \text{ and } \big|\big| \hat{\delta}^{(B)}(z) - \delta^{(B)}(z) \big|\big| \leq R \mathcal{C} n^{\frac{\tvarepsilon-1}{2}} \Big) \nonumber\\
		& \geq 1 - \mathcal{C} R^2 n \exp\Big( -\frac{n^{\tvarepsilon}}{\mathcal{C} R^2} \Big)
	\end{align}
	holds.
	Additionally, for any matrices $M \in \C^{d \times d}$ and $\tilde{M} \in \C^{n \times n}$ with $||M||, ||\tilde{M}|| \leq \sigma^2$, one has
	\begin{align}\label{Eq_MainRes_Result_M}
		& \bP\Big( \forall z \in \bD(\eta,\kappa) : \nonumber\\
		& \hspace{0.5cm} \Big| \frac{1}{n}\tr\big( M \bm{R}(z) \big) - \frac{-1}{z n}\tr\Big( M \Big( \Id_d + \sum\limits_{r',s'=1}^R \delta^{(B)}_{r',s'}(z) A_{r'}A_{s'}^* \Big)^{-1} \Big) \Big| \leq R \mathcal{C} n^{\frac{\tvarepsilon-1}{2}} , \nonumber\\
		& \hspace{0.5cm} \Big| \frac{1}{n}\tr\big( \tilde{M} \tilde{\bm{R}}(z) \big) - \frac{-1}{z n}\tr\Big( \tilde{M} \Big( \Id_n + \sum\limits_{r',s'=1}^R \delta^{(A)}_{r',s'}(z) B_{s'}^*B_{r'} \Big)^{-1} \Big) \Big| \leq R \mathcal{C} n^{\frac{\tvarepsilon-1}{2}} \Big) \nonumber\\
		& \geq 1 - \mathcal{C} R^2 n \exp\Big( -\frac{n^{\tvarepsilon}}{\mathcal{C} R^2} \Big) \ .
	\end{align}
\end{theorem}

In order to state the asymptotic consequences of this theorem for the empirical matrix-valued measures $\hat{\rho}^{(A)}$ and $\hat{\rho}^{(B)}$ from (\ref{Eq_Def_hatRhoA}) and (\ref{Eq_Def_hatRhoB}), the following lemma characterizes the weak convergence of matrix-valued measures.

\begin{lemma}[Characterization of weak convergence for matrix-valued measures]\label{Lemma_MatrWeakConv}\
	\\
	For any matrix-valued measures
	\begin{align*}
		& \rho, \rho_1,\rho_2,\dots : \cB(\R) \rightarrow \{ M \in \C^{R \times R} \text{ positive semi-definite} \} \ ,
	\end{align*}
	the following statements are equivalent.
	\begin{itemize}
		\item[a)]
		For all $v \in \C^R$, the convergence of finite Radon measures $v^* \rho_n v \xRightarrow{n \to \infty} v^* \rho v$ holds.
		
		\item[b)]
		For every $f \in C_b(\R;\C)$, the convergence $\big|\big| \int_\R f \, d(\rho_n - \rho) \big|\big| \xrightarrow{n \to \infty} 0$ holds.
		
		\item[c)]
		The convergence
		\begin{align}\label{Eq_MatrWeakConv_SameMass}
			& ||\rho_n(\R) - \rho(\R)|| \xrightarrow{n \rightarrow \infty} 0
		\end{align}
		holds and the corresponding matrix-valued Stieltjes transforms
		\begin{align*}
			& \delta_n(z) \coloneq \int_\R \frac{1}{\lambda - z} \, d\rho_n(\lambda) \ \ \text{ and } \ \ \delta(z) \coloneq \int_\R \frac{1}{\lambda - z} \, d\rho(\lambda)
		\end{align*}
		satisfy
		\begin{align*}
			& \forall z \in \C^+ : \ ||\delta_n(z) - \delta(z)|| \xrightarrow{n \rightarrow \infty} 0 \ .
		\end{align*}
		
		\item[d)]
		The convergence (\ref{Eq_MatrWeakConv_SameMass}) holds and for all $v \in \C^R$ and $z \in \C^+$, one has
		\begin{align*}
			& v^* \delta_n(z) v \xrightarrow{n \rightarrow \infty} v^* \delta(z) v \ .
		\end{align*}
	\end{itemize}
\end{lemma}

If any of the equivalent statements in Lemma~\ref{Lemma_MatrWeakConv} hold, write $\rho_n \xRightarrow{n \to \infty} \rho$. Theorem~\ref{Thm_SepCovMP_NonAsymp} allows the following corollary.

\begin{corollary}[Asymptotic MP law analogue]\label{Cor_SepCovMP_Asymp}\
	\\
	Let $d=d_n$ be a sequence of integers satisfying
	\begin{align}\label{Eq_CorAsymp_cInfAssumption}
		& \frac{d_n}{n} \xrightarrow{n \rightarrow \infty} c_\infty \ \text{ for some constant $c_\infty > 0 \, $.}
	\end{align}
	For fixed $R \in \N$ and all $n \in \N$, let $A^{(n)}_1,\dots,A^{(n)}_R \in \C^{d_n \times d_n}$ and $B^{(n)}_1,\dots,B^{(n)}_R \in \C^{n \times n}$ be deterministic matrices satisfying~\ref{ItemAssumption_sigmaBound} and~\ref{ItemTempAssumption_NonDegeneracy} for the same constants $\sigma^2 \geq 1$ and $\tau \in (0,1)$.
	Suppose the deterministic equivalent matrix-valued measures $\rho^{(A,n)}$ and $\rho^{(B,n)}$ from Theorem~\ref{Thm_DetEquiv} satisfy the convergences
	\begin{align}\label{Eq_CorAsymp_RhoConvergence}
		& \rho^{(A,n)} \xRightarrow{n \rightarrow \infty} \rho^{(A,\infty)} \ \ \text{ and } \ \ \rho^{(B,n)} \xRightarrow{n \rightarrow \infty} \rho^{(B,\infty)}
	\end{align}
	for some matrix-valued measures $\rho^{(A,\infty)}$ and $\rho^{(B,\infty)}$. Then for any sequence $(\bm{X}_n)_{n \in \N}$ of random matrices satisfying~\ref{ItemAssumption_boundedSixthMoment} for the same constant $C_6 \geq 1$, the weak convergences
	\begin{align}\label{Eq_CorAsymp_hatRhoConvergence}
		& 1 = \bP\Big( \hat{\rho}^{(A,n)} \xRightarrow{n \rightarrow \infty} \rho^{(A,\infty)} \ \ \text{ and } \ \ \hat{\rho}^{(B,n)} \xRightarrow{n \rightarrow \infty} \rho^{(B,\infty)} \Big)
	\end{align}
	hold almost surely, where the matrix-valued measures $\hat{\rho}^{(A,n)}$ and $\hat{\rho}^{(B,n)}$ are as defined in (\ref{Eq_Def_hatRhoA}) and (\ref{Eq_Def_hatRhoB}) using the model $(\bm{X}_n, (A^{(n)}_r)_{r \leq R}, (B^{(n)}_r)_{r \leq R})$.
	\\[0.5em]
	Furthermore, for the random probability measure $\hat{\ul{\nu}}_n : \cB(\R) \rightarrow [0,1]$ given by
	\begin{align}\label{Eq_Def_hatNu}
		& \hat{\ul{\nu}}_n(S) \coloneq \frac{1}{n} \#\{ j \leq n \mid \lambda_j(\tilde{\bm{S}}^{(n)}) \in S \}
	\end{align}
	with $\tilde{\bm{S}}^{(n)} = \frac{1}{n} (\bm{Y}^{(n)})^* \bm{Y}^{(n)}$ and $\bm{Y}^{(n)} = \sum\limits_{r=1}^R A^{(n)}_r \bm{X}_n B^{(n)}_r$, it also holds that
	\begin{align}\label{Eq_CorAsymp_hatNuConvergence}
		& 1 = \bP\big( \hat{\ul{\nu}}_n \xRightarrow{n \to \infty} \ul{\nu}_\infty \big) \ ,
	\end{align}
	where $\ul{\nu}_\infty$ is the uniquely defined probability measure whose Stieltjes transform is given by
	\begin{align}\label{Eq_Def_NuStil}
		& \forall z \in \C^+ : \ \cs_{\ul{\nu}_\infty}(z) \coloneq \int_\R \frac{1}{\lambda - z} \, d\ul{\nu}_\infty(\lambda) = -\frac{1}{z} - \sum\limits_{r,s=1}^R \delta^{(A,\infty)}_{r,s}(z) \delta^{(B,\infty)}_{r,s}(z) \ .
	\end{align}
\end{corollary}

\begin{remark}[Examples used in the simulation study]\label{Remark_Examples}\
	\\
	The validity of Theorem~\ref{Thm_SepCovMP_NonAsymp} is illustrated by numerical simulations for the following three examples. These examples fall into the setting of Assumptions~\ref{ItemAssumption_cBound}-\ref{ItemTempAssumption_NonDegeneracy} and are not covered by previous generalizations of the Marchenko--Pastur law (cf. \cite{TimeMP4, MDE_SlowCorDecay, SepCovMixture_Paul, StrongMP}). The \texttt{Python} code used to create the plots may be found in the author's GitHub repository\footnote{\href{https://github.com/BenDeitmar/SepCovMixture_Examples}{{https://github.com/BenDeitmar/SepCovMixture\_Examples}}}.
	\begin{itemize}
		\item[1)] \textit{Mixture of data with two different covariances}:\\
		For $d = 5n$ and $R=2$, define
		\begin{align*}
			& B_1 \coloneq \diag\big( \underbrace{1,\dots,1}_{\times \lceil \frac{n}{2} \rceil},\underbrace{0,\dots,0}_{\times (n- \lceil \frac{n}{2} \rceil)} \big) \ \ \text{ and } \ \ B_2 \coloneq \diag\big( \underbrace{0,\dots0}_{\times \lceil \frac{n}{2} \rceil},\underbrace{1,\dots,1}_{\times (n-\lceil \frac{n}{2} \rceil)} \big)
		\end{align*}
		as well as
		\begin{align*}
			& A_1 \coloneq \diag\Big( \underbrace{\frac{1}{3},\dots,\frac{1}{3}}_{\times \lceil \frac{d}{2} \rceil},\underbrace{0,\dots,0}_{\times (d-\lceil \frac{d}{2} \rceil)} \Big) \ \ \text{ and } \ \ A_2 \coloneq V_d\diag\Big( \underbrace{\frac{1}{2},\dots,\frac{1}{2}}_{\times \lceil \frac{d}{2} \rceil},\underbrace{1,\dots,1}_{\times (d-\lceil \frac{d}{2} \rceil)} \Big) \ ,
		\end{align*}
		where $V_d = \frac{1}{\sqrt{d}} \big( e^{-2\pi \bm{i} \frac{(k-1)(l-1)}{d}} \big)_{k,l \in \{1,\dots,d\}}$ is the $(d \times d)$ fast Fourier transform matrix. Let the entries of the $(d \times n)$ random matrix $\bm{X}$ be iid standard complex normal, then the data-matrix $\bm{Y} = A_1 \bm{X} B_1 + A_2 \bm{X} B_2$ consists of independent Gaussian columns, the first $\lceil \frac{n}{2} \rceil$ of which have population covariance matrix $\Sigma_1 = A_1 A_1^*$ and the last $n-\lceil \frac{n}{2} \rceil$ of which have population covariance $\Sigma_2 = A_2 A_2^*$. The unitary matrix $V_d$ was added to the definition of $A_2$, to ensure that $\Sigma_1$ and $\Sigma_2$ are not simultaneously diagonalizable.
		\\[0.5em]
		Lastly, define the example test-matrices $M = \Id_d$ and $\tilde{M} = \frac{1}{\sqrt{n}} (1)_{k,l\leq n}$.
		\\[0.5em]
		Assumptions~\ref{ItemAssumption_cBound}-\ref{ItemAssumption_boundedSixthMoment} are satisfied with constants $c_* = 5$, $\sigma^2 = 1$ and $C_6=6$. Possible choices of $\tau$ such that Assumption~\ref{ItemTempAssumption_NonDegeneracy} holds, depend on $d$ and will be given for each realization of the example.
		
		\item[2)] \textit{Moving average process}:\\
		For $d=n$ and $R \leq n$, define $B_r$ as the matrix with ones on the $(r-1)$-th upper diagonal and zeroes else, i.e.
		\begin{align*}
			& \forall r \in \{1,\dots,R\} : \ B_r \coloneq \big( \mathbbm{1}_{k-(r-1)=l} \big)_{k,l \leq n} \ .
		\end{align*}
		Further, for independent Haar-unitary matrices $U_2,\dots,U_R$, define $A_1 = \Id_d$ and
		\begin{align*}
			& \forall r \in \{2,\dots,R\} : \ A_r \coloneq U_r \diag\Big( \frac{1}{d}, \frac{2}{d}, \dots, \frac{d}{d} \Big) \ .
		\end{align*}
		Let the entries of $\bm{X}$ be iid Rademacher random variables, i.e. $\bP(\bm{X}_{i,j} = 1) = \frac{1}{2} = \bP(\bm{X}_{i,j} = -1)$, which are also independent of the unitary matrices $U_2,\dots,U_R$. The columns $\bm{Y}_{\bullet,t}$ of the data-matrix $\bm{Y} = \sum\limits_{r=1}^R A_r \bm{X} B_r$ for $t \geq R$ form a moving average process by the observation
		\begin{align*}
			& \forall t \in \{R,\dots,n\} : \ \bm{Y}_{\bullet,t} = \sum\limits_{i=1}^{R-1} A_{i+1} \xi_{t-i} + \xi_t \ ,
		\end{align*}
		where the $d$-dimensional innovations $(\xi_t)_{t \in \{1,\dots,n\}}$ are given by $\xi_t = \bm{X}_{\bullet,t}$.
		\\[0.5em]
		Lastly, define $M = \diag\Big( \frac{1}{d}, \frac{2}{d}, \dots, \frac{d}{d} \Big)$ and $\tilde{M} = \diag(\underbrace{1,\dots,1}_{\times \lfloor \frac{n}{2} \rfloor},\underbrace{0,\dots,0}_{\times (n - \lfloor \frac{n}{2} \rfloor)})$.\\
		Assumptions~\ref{ItemAssumption_cBound}-\ref{ItemAssumption_boundedSixthMoment} are satisfied with constants $c_* = 1$, $\sigma^2 = 1$ and $C_6=1$. Possible choices of $\tau$ such that Assumption~\ref{ItemTempAssumption_NonDegeneracy} holds, depend on $d$, $n$, $R$ and the realizations of $U_2,\dots,U_R$ and will be given for each realization of the example.
		
		\item[3)] \textit{Mixture of permutations}:\\
		For $d = 2n$ and arbitrary $R$, let $p_1,\dots,p_R$ be iid random permutations of $\{1,\dots,d\}$ with uniform distribution on the symmetric group $\operatorname{Sym}_d$ and likewise let $q_1,\dots,q_R$ be iid random permutations of with uniform distribution on $\operatorname{Sym}_n$. Suppose $(p_r)_{r \leq R}$, $(q_r)_{r \leq R}$ and $\bm{X}$ are all independent, where the entries of $\sqrt{\frac{7}{5}} \bm{X}$ are iid with Student's $t$-distribution with $7$ degrees of freedom. Define
		\begin{align*}
			& A_r \coloneq \big( \mathbbm{1}_{p_r(i)=j} \big)_{i,j \leq d} \text{ as the permutation matrix to } p_r
		\end{align*}
		and
		\begin{align*}
			& B_r \coloneq \big( \mathbbm{1}_{q_r(k)=l} \big)_{k,l \leq n} \text{ as the permutation matrix to } q_r \ .
		\end{align*}
		Lastly, define the example test-matrices $M = \diag( 1, 0, \dots, 0)$ and $\tilde{M} = B_1$.
		\\[0.5em]
		Assumptions~\ref{ItemAssumption_cBound}-\ref{ItemAssumption_boundedSixthMoment} are satisfied with constants $c_* = 2$, $\sigma^2 = 1$ and $C_6=125$. Possible choices of $\tau$ such that Assumption~\ref{ItemTempAssumption_NonDegeneracy} holds, depend on $d$, $n$, $R$ and the realizations of $p_1,\dots,p_R,q_1,\dots,q_R$ and will be given for each realization of the example.
	\end{itemize}
\end{remark}

\vspace{-0.2cm}
\begin{figure}[h]
	{\centering
		\includegraphics[width=0.32\textwidth]{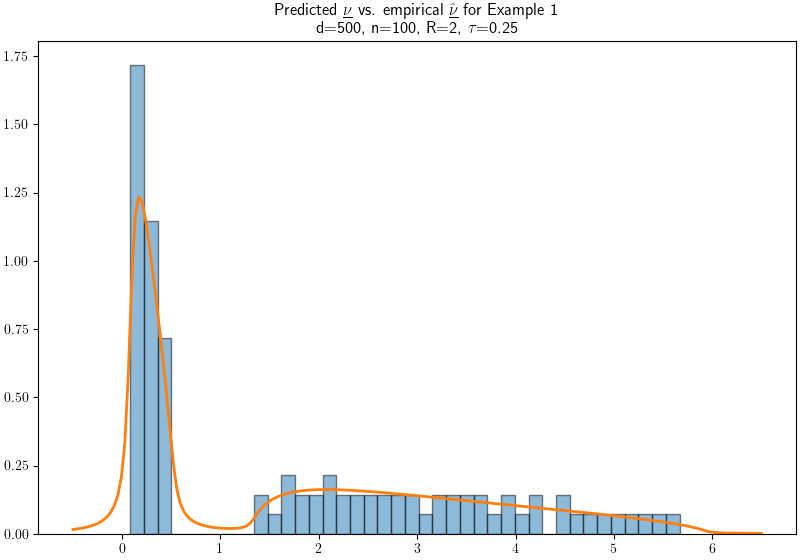} \includegraphics[width=0.32\textwidth]{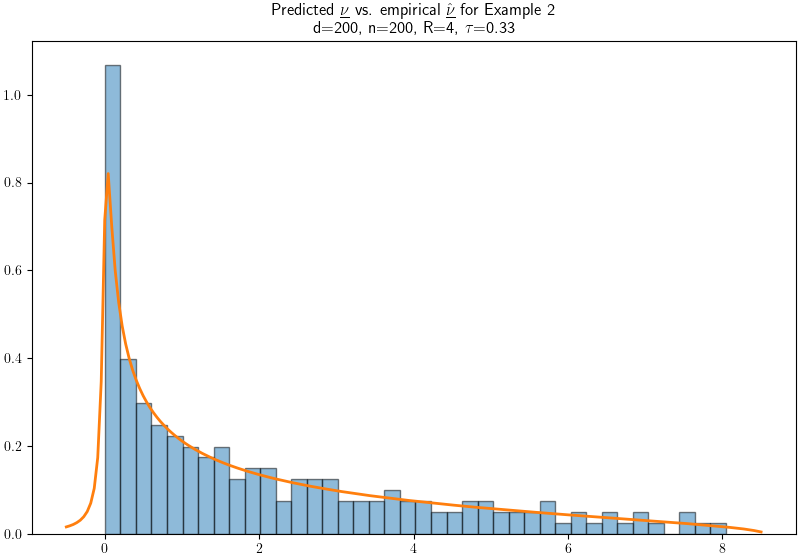} \includegraphics[width=0.32\textwidth]{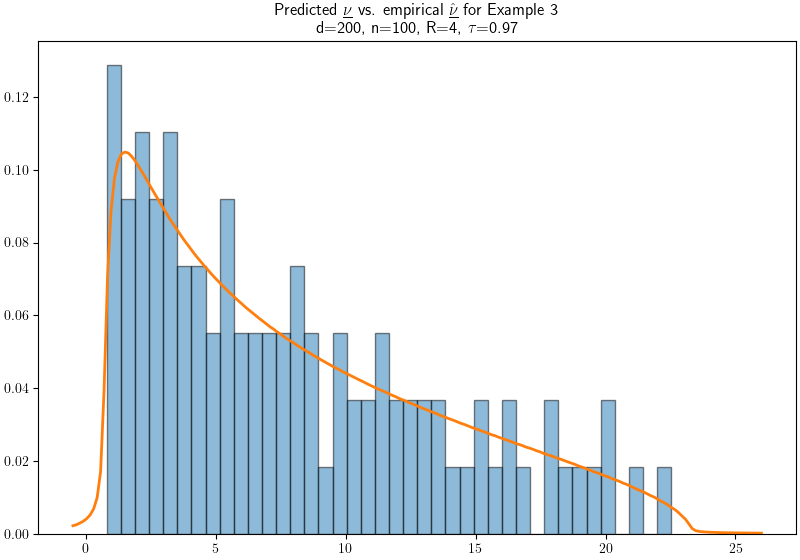}
	}
	\vspace{-0.2cm}
	\caption{Histograms (blue) of the eigenvalues of $\tilde{\bm{S}}$ for Example 1 with $n=100$ (left), Example 2 with $n=200$ and $R=4$ (middle) and Example 3 with $n=100$ and $R=4$ (right). Plots (orange) of the maps $x \mapsto \frac{1}{\pi}\Im(\cs_{\ul{\nu}}(x+i\eta))$ for $\eta=0.05$ are overlaid, where $\ul{\nu}$ is as defined in Theorem~\ref{Thm_DetEquiv}.}\label{Fig_SC1_Ex}
\end{figure}

\begin{figure}[h]
	{\centering
		\includegraphics[width=0.32\textwidth]{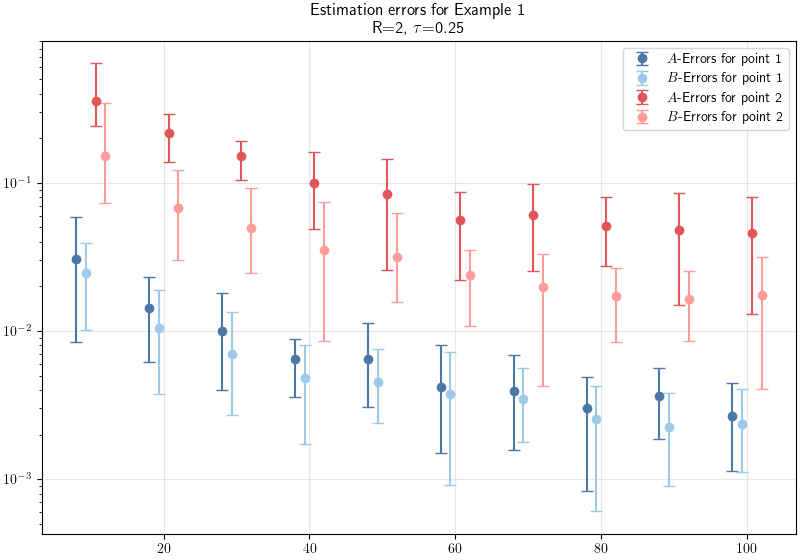} \includegraphics[width=0.32\textwidth]{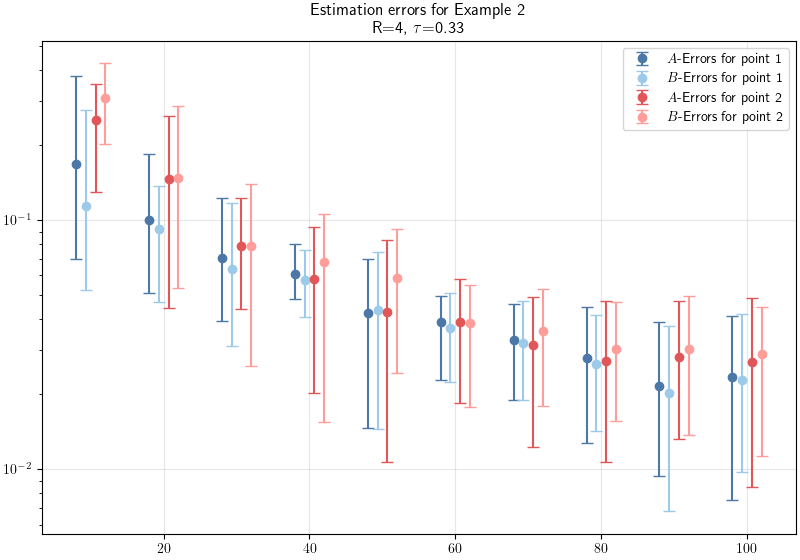} \includegraphics[width=0.32\textwidth]{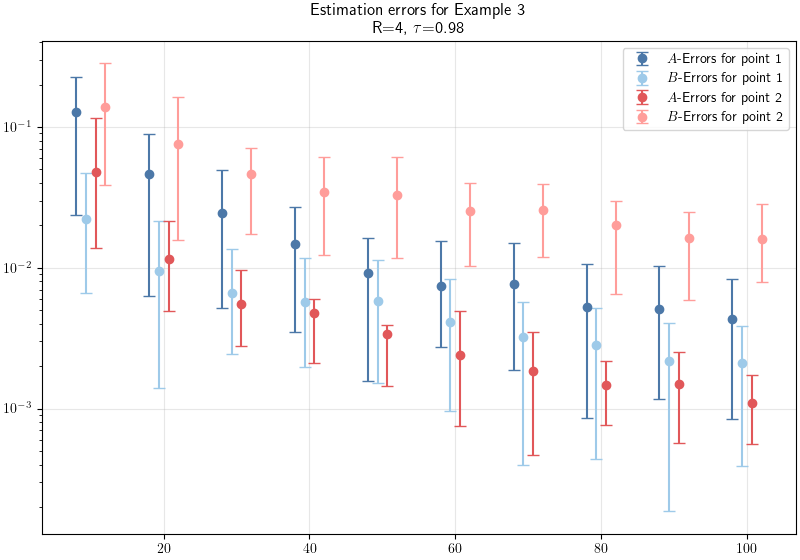}
	}
	\vspace{-0.2cm}
	\caption{Average approximation errors from (\ref{Eq_MainRes_Result_M}), where the upper absolute value in (\ref{Eq_MainRes_Result_M}) is called "A-Error" and the lower absolute value in (\ref{Eq_MainRes_Result_M}) is called "B-Error". The test-matrices $M, \tilde{M}$ are as defined in Remark~\ref{Remark_Examples}. The data is from 25 realizations of Examples 1 (left), 2 (middle) and 3 (right) with $n$ ranging from $10$ to $100$. The $z$-points are $z_1=1.5+\bm{i}$ and $z_2=1.5+0.1\bm{i}$ (left), $z_1=6+0.1\bm{i}$ and $z_2=0.5+0.1\bm{i}$ (middle) as well as $z_1=0+0.1\bm{i}$ and $z_2=2+0.1\bm{i}$ (right). The y-axes are logarithmic and the error bars present the range between the $10\%$ and $90\%$ quantiles.}\label{Fig_SC2_Ex}
	\vspace{-0.2cm}
\end{figure}

\newpage
\section{Supporting theorems}\label{Section_SupportingTheorems}
Begin by introducing a notation that allows for the comparison of resolvents $\bm{R}(z)$, when different inputs for the matrix $\bm{X}$ are inserted.

\begin{definition}[Robust notation]\label{Def_RobustNotation}\
	\\
	For any matrix $X$, let $d_X$ denote the number of rows and $n_X$ denote the number of columns. For given matrices $A_1,\dots,A_R \in \C^{d_X \times d_X}$ and $B_1,\dots,B_R \in \C^{n_X \times n_X}$, define
	\begin{align}\label{Eq_Def_Y_Robust}
		& \bm{Y}^{(X)} = \bm{Y}^{(X,A,B)} \coloneq \sum\limits_{r=1}^R A_r X B_r \ ,
	\end{align}
	where the dependence on $A_1,\dots,A_R,B_1,\dots,B_R$ is suppressed in notation. Further define
	\begin{align}
		& \bm{S}^{(X)} \coloneq \frac{1}{n_X} \bm{Y}^{(X)} (\bm{Y}^{(X)})^* \ \ \text{ and } \ \ \tilde{\bm{S}}^{(X)} \coloneq \frac{1}{n_X} (\bm{Y}^{(X)})^* \bm{Y}^{(X)} \label{Eq_Def_S_Robust}
	\end{align}
	as well as
	\begin{align}
		& \bm{R}^{(X)}(z) \coloneq \big( \bm{S}^{(X)} - z\Id_{d_X} \big)^{-1} \ \ \text{ and } \ \ \tilde{\bm{R}}^{(X)}(z) \coloneq \big( \tilde{\bm{S}}^{(X)} - z\Id_{n_X} \big)^{-1} \ . \label{Eq_Def_R_Robust}
	\end{align}
	Lastly, for ease of future notation, define the matrices
	\begin{align}
		& \bm{Q}^{(X)}(z) \coloneq \Id_{d_X} +  \sum\limits_{r,s=1}^R \frac{1}{n_X}\tr\big( B_{s}^* B_{r} \tilde{\bm{R}}^{(X)}(z) \big) A_{r} A_{s}^* \label{Eq_Def_Q}\\
		& \tilde{\bm{Q}}^{(X)}(z) \coloneq \Id_{n_X} +  \sum\limits_{r,s=1}^R \frac{1}{n_X}\tr\big( A_r A_s^* \bm{R}^{(X)}(z) \big) B_s^* B_r \ . \label{Eq_Def_tQ}
	\end{align}
\end{definition}

The following theorem is a generalization of the uniqueness statement for deterministic equivalents given in Theorem~\ref{Thm_DetEquiv} and it shows that even approximate solutions to the dual system of equations (\ref{Eq_DualSystem}) will be close to the true solutions.
It is a technical result, which will be used to show uniqueness in (a) of Theorem~\ref{Thm_DetEquiv} and which plays a crucial role in the proof of Theorem~\ref{Thm_SepCovMP_NonAsymp}. The theorem will be proved in Section~\ref{Proof_Thm_DualApprox} by an operator-valued generalization of the method employed in the proof of Lemma 4.3.3 in \cite{LixinSALDRM} (see pages 222--224 therein) to show uniqueness of deterministic equivalents for the separable covariance model.

\begin{theorem}[Quantitative uniqueness argument for solutions to (\ref{Eq_DualSystem})]\label{Thm_DualApprox}\
	\\
	Suppose~\ref{ItemAssumption_cBound},~\ref{ItemAssumption_sigmaBound} and~\ref{ItemTempAssumption_NonDegeneracy} hold. For any $\tilde{\tau},\tilde{\kappa},\eta,\kappa>0$, there exists a constant
	\begin{align*}
		& \mathcal{C} = \mathcal{C}(c_*, \sigma^2, \tau, \eta, \kappa, \tilde{\tau}, \tilde{\kappa})>0 \ ,
	\end{align*}
	which does not depend on $d$, $n$ or the matrices $A_1,\dots,A_R,B_1,\dots,B_R$, such that the following statement holds.
	\\[0.5em]
	For any $z \in \bD(\eta,\kappa)$, let $(\delta^{(A)}(z),\delta^{(B)}(z)) \in \C^{R \times R} \times \C^{R \times R}$ be a solution to (\ref{Eq_DualSystem}), which also satisfies
	\begin{align}
		& \tilde{\tau} \Id_R \preceq \Im\big( \delta^{(A)}(z) \big) \ \ , \ \ \tilde{\tau} \Id_R \preceq \Im\big( \delta^{(B)}(z) \big) \label{Eq_DualApprox_posDefCondition_delta}\\
		& ||\delta^{(A)}(z)|| \leq \tilde{\kappa} \ \ \text{ and } \ \ ||\delta^{(B)}(z)|| \leq \tilde{\kappa} \ . \label{Eq_DualApprox_BoundCondition_delta}
	\end{align}
	Further, let $(\tilde{\delta}^{(A)}(z),\tilde{\delta}^{(B)}(z)) \in \C^{R \times R} \times \C^{R \times R}$ be an approximate solution to (\ref{Eq_DualSystem}) in the sense
	\begin{align}\label{Eq_DualApprox_DualSystem}
		& \forall r,s \in \{1,\dots,R\} : \nonumber\\
		& \tilde{\delta}^{(A)}_{r,s}(z) = -\frac{1}{z} \frac{1}{n} \tr\bigg( A_rA_s^* \Big( \Id_d + \sum\limits_{r',s'=1}^R \tilde{\delta}^{(B)}_{r',s'}(z) A_{r'} A_{s'}^* \Big)^{-1} \bigg) + q^{(A)}_{r,s}(z) \nonumber\\
		& \tilde{\delta}^{(B)}_{r,s}(z) = -\frac{1}{z} \frac{1}{n} \tr\bigg( B_s^*B_r \Big( \Id_n + \sum\limits_{r',s'=1}^R \tilde{\delta}^{(A)}_{r',s'}(z) B_{s'}^*B_{r'} \Big)^{-1} \bigg) + q^{(B)}_{r,s}(z)
	\end{align}
	for $q^{(A)}(z),q^{(B)}(z) \in \C^{R \times R}$ whose spectral norms are bounded by
	\begin{align}
		& ||q^{(A)}(z)|| \leq \frac{\tau^2 \tilde{\tau}}{2|z|(1 + \tilde{\kappa} \sigma^2)^2} \label{Eq_DualApprox_qA_Assumption}\\
		& ||q^{(A)}(z)|| \leq \frac{\tau^7 \tilde{\tau} \Im(z)}{8|z|^4 \sigma^2 (1 + \tilde{\kappa} \sigma^2)^{7}} \label{Eq_DualApprox_qSmall}
	\end{align}
	and
	\begin{align}\label{Eq_DualApprox_tqSmall}
		& ||q^{(B)}(z)|| \leq \frac{\tau^9 \tilde{\tau} \Im(z)}{8|z|^3 c_* \sigma^6 (1 + \tilde{\kappa} \sigma^2)^{7}} \ .
	\end{align}
	Also assume that
	\begin{align}
		& \tilde{\tau} \Id_R \preceq \Im\big( \tilde{\delta}^{(A)}(z) \big) \ \ , \ \ \tilde{\tau} \Id_R \preceq \Im\big( \tilde{\delta}^{(B)}(z) \big) \label{Eq_DualApprox_posDefCondition_tdelta}\\
		& ||\tilde{\delta}^{(A)}(z)|| \leq \tilde{\kappa} \ \ \text{ and } \ \ ||\tilde{\delta}^{(B)}(z)|| \leq \tilde{\kappa} \ , \label{Eq_DualApprox_BoundCondition_tdelta}
	\end{align}
	then the difference between $(\tilde{\delta}^{(A)}(z),\tilde{\delta}^{(B)}(z))$ and $(\delta^{(A)}(z),\delta^{(B)}(z))$ may be bounded by
	\begin{align}\label{Eq_DualApprox_Result_deltaA}
		& \big|\big| \delta^{(A)}(z) - \tilde{\delta}^{(A)}(z) \big|\big| \leq \mathcal{C} \big(||q^{(A)}(z)||+||q^{(B)}(z)||\big)
	\end{align}
	and
	\begin{align}\label{Eq_DualApprox_Result_deltaB}
		& \big|\big| \delta^{(B)}(z) - \tilde{\delta}^{(B)}(z) \big|\big| \leq \mathcal{C} \big(||q^{(A)}(z)||+||q^{(B)}(z)||\big) \ .
	\end{align}
\end{theorem}

The next theorem shows that the empirical matrix-valued Stieltjes transforms $\hat{\delta}^{(A)}(z)$ and $\hat{\delta}^{(B)}(z)$ are, in the Gaussian case, approximate solutions to the dual system of equations (\ref{Eq_DualSystem}) with high probability. Together with the previous theorem, it will form the core of the proof of Theorem~\ref{Thm_SepCovMP_NonAsymp}. The theorem is proved in Sections~\ref{Section_Stein} and~\ref{Section_GaussianCase}, where Section~\ref{Section_Stein} employs Stein's Lemma to show that the expectations of the errors are small, and Section~\ref{Section_GaussianCase} uses concentration results to show that this translates to the errors being small with high probability.

\begin{theorem}[Approximate dual system of equations in the Gaussian case]\label{Thm_EmpiricalDualSystem}\
	\\
	Suppose~\ref{ItemAssumption_cBound},~\ref{ItemAssumption_sigmaBound} and~\ref{ItemTempAssumption_NonDegeneracy} hold and let $\bm{Z}$ denote a random $(d \times n)$-matrix whose entries are independent centered complex-valued Gaussian random variables with $\E[|\bm{Z}_{j,k}|^2]=1$. Note that the entries may differ in their second moments $\E[\bm{Z}_{j,k}^2]$.\\
	For any $\eta,\kappa>0$ there exists a constant $\mathcal{C} = \mathcal{C}(c_*,\sigma^2,\tau,\eta,\kappa)>0$, which does not depend on the exact choices of $A_1,\dots,A_R,B_1,\dots,B_R$, such that the empirical matrix-valued Stieltjes transforms $\big(\hat{\delta}^{(A,\bm{Z})}(z),\hat{\delta}^{(B,\bm{Z})}(z)\big)$, defined by
	\begin{align}\label{Eq_Def_hatDeltaZ}
		& \hat{\delta}^{(A,\bm{Z})}(z)_{r,s} \coloneq \frac{1}{n} \tr\big( A_rA_s^* \bm{R}^{(\bm{Z})}(z) \big) \ \ \text{ and } \ \ \hat{\delta}^{(B,\bm{Z})}(z)_{r,s} \coloneq \frac{1}{n} \tr\big( B_s^*B_r \tilde{\bm{R}}^{(\bm{Z})}(z) \big) \ ,
	\end{align}
	for any $\tvarepsilon \in (0,1)$ and deterministic matrices
	\begin{align*}
		& M \in \C^{d \times d} \ \ \text{ and } \ \ \tilde{M} \in \C^{n \times n} \ \ \text{ with } \ \ ||M||, ||\tilde{M}|| \leq \sigma^2
	\end{align*}
	satisfy the bounds
	\begin{align}\label{Eq_EmpDualGaussian_ResultA_M}
		& \bP\Big( \exists z \in \bD(\eta,\kappa) : \ \Big| \frac{1}{n} \tr\big( M \bm{R}^{(\bm{Z})}(z) \big) - \frac{-1}{z} \frac{1}{n} \tr\Big( M \Big( \Id_d + \sum\limits_{r',s'=1}^R \hat{\delta}^{(B,\bm{Z})}_{r',s'}(z) A_{r'} A_{s'}^* \Big)^{-1} \Big) \Big| \nonumber\\
		& \hspace{9.5cm} \geq \mathcal{C} n^{\frac{\tvarepsilon-1}{2}} + \mathcal{C} n^2 \exp(-n/\mathcal{C}) \Big) \nonumber\\
		& \leq \mathcal{C} n \exp\Big( -\frac{n^{\tvarepsilon}}{\mathcal{C} R^2} \Big)
	\end{align}
	and
	\begin{align}\label{Eq_EmpDualGaussian_ResultB_M}
		& \bP\Big( \exists z \in \bD(\eta,\kappa) : \ \Big| \frac{1}{n} \tr\big( \tilde{M} \tilde{\bm{R}}^{(\bm{Z})}(z) \big) - \frac{-1}{z} \frac{1}{n} \tr\Big( \tilde{M} \Big( \Id_n + \sum\limits_{r',s'=1}^R \hat{\delta}^{(A,\bm{Z})}_{r',s'}(z) B_{s'}^* B_{r'} \Big)^{-1} \Big) \Big| \nonumber\\
		& \hspace{9.5cm} \geq \mathcal{C}n^{\frac{\tvarepsilon-1}{2}} + \mathcal{C} n^2 \exp(-n/\mathcal{C}) \Big) \nonumber\\
		& \leq \mathcal{C} n \exp\Big( -\frac{n^{\tvarepsilon}}{\mathcal{C} R^2} \Big) \ .
	\end{align}
	Inserting $A_rA_s^*$ for $M$ in (\ref{Eq_EmpDualGaussian_ResultA_M}) yields
	\begin{align}\label{Eq_EmpDualGaussian_ResultA}
		& \bP\Big( \exists r,s \leq R , \, \exists z \in \bD(\eta,\kappa) : \nonumber\\
		& \hspace{1cm} \Big| \hat{\delta}^{(A,\bm{Z})}_{r,s}(z) - \frac{-1}{z} \frac{1}{n} \tr\Big( A_rA_s^* \Big( \Id_d + \sum\limits_{r',s'=1}^R \hat{\delta}^{(B,\bm{Z})}_{r',s'}(z) A_{r'} A_{s'}^* \Big)^{-1} \Big) \Big| \nonumber\\
		& \hspace{8cm} \geq \mathcal{C}n^{\frac{\tvarepsilon-1}{2}} + \mathcal{C} n^2 \exp(-n/\mathcal{C}) \Big) \nonumber\\
		& \leq \mathcal{C} R^2 n \exp\Big( -\frac{n^{\tvarepsilon}}{\mathcal{C} R^2} \Big)
	\end{align}
	and inserting $B_s^*B_r$ for $\tilde{M}$ in (\ref{Eq_EmpDualGaussian_ResultB_M}) yields
	\begin{align}\label{Eq_EmpDualGaussian_ResultB}
		& \bP\Big( \exists r,s \leq R , \, \exists z \in \bD(\eta,\kappa) : \nonumber\\
		& \hspace{1cm} \Big| \hat{\delta}^{(B,\bm{Z})}_{r,s}(z) - \frac{-1}{z} \frac{1}{n} \tr\Big( B_s^*B_r \Big( \Id_n + \sum\limits_{r',s'=1}^R \hat{\delta}^{(A,\bm{Z})}_{r',s'}(z) B_{s'}^* B_{r'} \Big)^{-1} \Big) \Big| \nonumber\\
		& \hspace{8cm} \geq \mathcal{C}n^{\frac{\tvarepsilon-1}{2}} + \mathcal{C} n^2 \exp(-n/\mathcal{C}) \Big) \nonumber\\
		& \leq \mathcal{C} R^2 n \exp\Big( -\frac{n^{\tvarepsilon}}{\mathcal{C} R^2} \Big) \ .
	\end{align}
\end{theorem}

While the previous two theorems suffice to show Theorem~\ref{Thm_SepCovMP_NonAsymp} for Gaussian data, the following universality argument will be required to extend the result to non-Gaussian models.

\begin{definition}[Similar Gaussian matrix]\label{Def_SimilarGaussianMatrix}\
	\\
	Let $\bm{U}$ and $\bm{V}$ be independent random $(d \times n)$-matrices each with iid entries such that
	\begin{align}\label{Eq_SimilarDef_UV}
		& \bm{U}_{i,j} \sim \mathcal{N}(0,1) \ \ \text{ and } \ \ \bm{V}_{i,j} \sim \mathcal{CN}(0,1) \ .
	\end{align}
	For the random $(d \times n)$-matrix $\bm{X}$ with independent centered entries, each with variance one, define the similar random $(d \times n)$-matrix $\bm{Z}$ to $\bm{X}$ by
	\begin{align}\label{Eq_SimilarDef_Z}
		& \bm{Z}_{i,j} \coloneq a_{i,j} \bm{U}_{i,j} + b_{i,j} \bm{V}_{i,j} \ ,
	\end{align}
	where $a_{i,j},b_{i,j} \in \C$ are chosen such that
	\begin{align}\label{Eq_SimilarDef_ab}
		& a_{i,j}^2 = \E[\bm{X}_{i,j}^2] \ \ \text{ and } \ \ |a_{i,j}|^2 + |b_{i,j}|^2 = 1 \ .
	\end{align}
\end{definition}

Straightforward calculations show that the first and second moments of all entries of $\bm{X}$ and $\bm{Z}$ coincide, i.e.
\begin{align*}
	& \E[\bm{Z}_{i,j}] = 0 = \E[\bm{X}_{i,j}] \ \ , \ \ \E[|\bm{Z}_{i,j}|^2] = 1 = \E[|\bm{X}_{i,j}|^2] \ \ \text{ and } \ \ \E[\bm{Z}_{i,j}^2] = \E[\bm{X}_{i,j}^2] \ .
\end{align*}

\begin{theorem}[Universality]\label{Thm_Universality}\
	\\
	Suppose~\ref{ItemAssumption_cBound}-\ref{ItemTempAssumption_NonDegeneracy} hold and let $\bm{Z}$ be the Gaussian matrix similar to $\bm{X}$ as constructed in Definition~\ref{Def_SimilarGaussianMatrix}. For any $\eta,\kappa>0$ there exists a constant $\mathcal{C} = \mathcal{C}(c_*, \sigma^2, C_6, \tau, \eta, \kappa)>0$, which does not depend on the model $\big( \bm{X}, (A_r)_{r \leq R}, (B_r)_{r \leq R} \big)$, such that for all $\tvarepsilon \in (0,1)$ the bound
	\begin{align}\label{Eq_Universality_Result1}
		& \bP\Big( \sup\limits_{z \in \bD(\eta,\kappa)}\Big| \frac{1}{n} \tr(M\bm{R}^{(\bm{X})}(z)) - \frac{1}{n} \tr(M\bm{R}^{(\bm{Z})}(z)) \Big| > \mathcal{C} n^{\frac{\tvarepsilon-1}{2}} ||M|| \Big) \leq \mathcal{C} n \exp\Big( -\frac{n^{\tvarepsilon}}{\mathcal{C} R^2} \Big)
	\end{align}
	holds for all $M \in \C^{d \times d}$
	and
	\begin{align}\label{Eq_Universality_Result2}
		& \bP\Big( \sup\limits_{z \in \bD(\eta,\kappa)}\Big| \frac{1}{n} \tr(\tilde{M}\tilde{\bm{R}}^{(\bm{X})}(z)) - \frac{1}{n} \tr(\tilde{M}\tilde{\bm{R}}^{(\bm{Z})}(z)) \Big| > \mathcal{C} n^{\frac{\tvarepsilon-1}{2}} ||\tilde{M}|| \Big) \leq \mathcal{C} n \exp\Big( -\frac{n^{\tvarepsilon}}{\mathcal{C} R^2} \Big)
	\end{align}
	holds for all $\tilde{M} \in \C^{n \times n}$.
\end{theorem}

\newpage
\section{Proof of Theorem~\ref{Thm_DualApprox}}\label{Proof_Thm_DualApprox}
Before starting the proof, four basic Lemmas are introduced.

\begin{lemma}[Preservation of positive definiteness]\label{Lemma_NonDegeneracyNew}\
	\\
	Suppose~\ref{ItemTempAssumption_NonDegeneracy} holds, then for any positive definite matrix $C \in \C^{R \times R}$, the bounds
	\begin{align}
		& \lambda_{\min}\Big(\sum\limits_{r,s=1}^R C_{r,s} A_rA_s^*\Big) \geq \tau \lambda_{\min}(C) \label{Eq_NonDegeneracyC_A}\\
		& \lambda_{\min}\Big(\sum\limits_{r,s=1}^R C_{r,s} B_s^*B_r\Big) \geq \tau \lambda_{\min}(C) \label{Eq_NonDegeneracyC_B}
	\end{align}
	hold. Furthermore, then the bounds
	\begin{align}
		& \lambda_{\min}\Big( \frac{1}{n} \tr\big( A_rA_s^* M \big)_{r,s \leq R} \Big) \geq \tau \lambda_{\min}(M) \label{Eq_NonDegeneracy_TrCalcA}\\
		& \lambda_{\min}\Big( \frac{1}{n} \tr\big( B_s^*B_r \tilde{M} \big)_{r,s \leq R} \Big) \geq \tau \lambda_{\min}(\tilde{M}) \ . \label{Eq_NonDegeneracy_TrCalcB}
	\end{align}
	are true for any positive definite matrices $M \in \C^{d \times d}$ and $\tilde{M} \in \C^{n \times n}$. As a trivial consequence of the previous bounds one has
	\begin{align}
		& \lambda_{\min}\Big( \sum\limits_{r,s=1}^R \frac{1}{n}\tr\big( A_r A_s^* M \big) B_s^* B_r \Big) \geq \tau^2 \lambda_{\min}(M) \label{Eq_NonDegeneracyM_A}\\
		& \lambda_{\min}\Big( \sum\limits_{r,s=1}^R \frac{1}{n}\tr\big( B_{s}^* B_{r} \tilde{M} \big) A_{r} A_{s}^* \Big) \geq \tau^2 \lambda_{\min}(\tilde{M}) \ . \label{Eq_NonDegeneracyM_B}
	\end{align}
\end{lemma}

\begin{lemma}[Simple operator bounds]\label{Lemma_SimpleOperatorBounds}\
	\\
	Suppose~\ref{ItemAssumption_sigmaBound} holds, then for any $(R \times R)$-matrix $C$, the bounds
	\begin{align}\label{Eq_SimpleBounds_ASum_BSum}
		& \Big|\Big|\sum\limits_{r,s=1}^R C_{r,s} A_rA_s^*\Big|\Big| \leq \sigma^2 ||C|| \ \ \text{ and } \ \ \Big|\Big|\sum\limits_{r,s=1}^R C_{r,s} B_s^*B_r\Big|\Big| \leq \sigma^2 ||C||
	\end{align}
	hold. For any $M \in \C^{d \times d}$ and $\tilde{M} \in \C^{n \times n}$ the bounds
	\begin{align}\label{Eq_SimpleBounds_TraceMatrixBounds}
		& \Big|\Big| \frac{1}{n}\tr\big( A_r A_s^* M \big)_{r,s \leq R} \Big|\Big| \leq \sigma^2 \frac{d}{n} ||M||\ \ \text{ and } \ \ \Big|\Big| \frac{1}{n}\tr\big( B_s^* B_r \tilde{M} \big)_{r,s \leq R} \Big|\Big| \leq \sigma^2 ||\tilde{M}||
	\end{align}
	also hold. A trivial consequence of the last two results is
	\begin{align}
		& \Big|\Big| \sum\limits_{r,s=1}^R \frac{1}{n}\tr\big( A_r A_s^* M \big) B_s^* B_r \Big|\Big| \leq \sigma^4 \frac{d}{n} ||M|| \label{Eq_SimpleBound_SB}\\
		& \Big|\Big| \sum\limits_{r,s=1}^R \frac{1}{n}\tr\big( B_{s}^* B_{r} \tilde{M} \big) A_{r} A_{s}^* \Big|\Big| \leq \sigma^4 ||\tilde{M}|| \ . \label{Eq_SimpleBound_SA}
	\end{align}
\end{lemma}

The following simple lemma was suggested by ChatGPT as the correct tool for bounding the operator norm a the custom operator $T_{\delta,\tilde{\delta}}$ in the upcoming proof of Theorem \ref{Thm_DualApprox}.

\begin{lemma}[Cauchy-Schwarz for operators]\label{Lemma_OperatorBound}\
	\\
	Let $C_1,\dots,C_N$ and $D_1,\dots,D_N$ be arbitrary $(n \times n)$-matrices. The operator-norm (with respect to the spectral norm) of the linear operator
	\begin{align*}
		& T : \C^{n \times n} \rightarrow \C^{n \times n} \ \ ; \ \ T(X) \coloneq \sum\limits_{j=1}^N C_j X D_j
	\end{align*}
	may be bounded by
	\begin{align}\label{Eq_Uniqueness_OpBound}
		& ||T||_{\textrm{Op}} \coloneq \sup\limits_{\substack{X \in \C^{n \times n} \\ X \neq 0}} \frac{||T(X)||}{||X||} \leq \Big|\Big| \sum\limits_{j=1}^N C_jC_j^* \Big|\Big|^{\frac{1}{2}} \, \Big|\Big| \sum\limits_{j=1}^N D_j^*D_j \Big|\Big|^{\frac{1}{2}} \ .
	\end{align}
\end{lemma}

\begin{lemma}[Near-eigenvectors of contractions are close to zero]\label{Lemma_ContractionEVs}\
	\\
	Let $T : B \rightarrow B$ be a linear contraction on a Banach space $(B,||\cdot||_B)$ satisfying
	\begin{align*}
		& ||T||_{\textrm{Op}} = \sup\limits_{\substack{x \in B \\ x \neq 0}} \frac{||T(x)||_B}{||x||_B} \leq 1 - \theta
	\end{align*}
	for some $\theta > 0$. For any $v \in B$, let $w=w(v)\coloneq T(v)-v$ such that
	\begin{align*}
		& T(v) = v + w \ ,
	\end{align*}
	then
	\begin{align*}
		& ||v||_B \leq \frac{||w||_B}{\theta} \ .
	\end{align*}
\end{lemma}

The proof of Theorem~\ref{Thm_DualApprox} commences.
\begin{itemize}
	\item[i)] \textit{Positive definiteness of $\Im\big( z \delta^{(A)}(z) \big)$ and $\Im\big( z \tilde{\delta}^{(A)}(z) \big)$}:\\
	Since $z \in \C^+$ is fixed, use the abbreviated notation
	\begin{align*}
		(\delta^{(A)},\delta^{(B)}) & = (\delta^{(A)}(z),\delta^{(B)}(z))\\
		(\tilde{\delta}^{(A)},\tilde{\delta}^{(B)}) & = (\tilde{\delta}^{(A)}(z),\tilde{\delta}^{(B)}(z))\\
		(q^{(A)}, q^{(B)}) & = (q^{(A)}(z), q^{(B)}(z)) \ .
	\end{align*}
	With the identity
	\begin{align}\label{Eq_DualApprox_ImInv_Identity}
		& -\Im(A^{-1}) = \frac{1}{2\bm{i}} \big( A^{-*} - A^{-1} \big) \overset{\text{(\ref{Eq_InverseDifferenceIdentity})}}{=} \frac{1}{2\bm{i}} A^{-*} \big( A - A^* \big) A^{-1} = A^{-*} \Im(A) A^{-1}
	\end{align}
	and the bound
	\begin{align}\label{Eq_DualApprox_VA_PrelimBound}
		& \Big|\Big| \sum\limits_{r,s=1}^R \delta^{(B)}_{r,s} A_rA_s^* \Big|\Big| \leq ||\delta^{(B)}|| \, \sum\limits_{r=1}^R ||A_r||^2 \leq ||\delta^{(B)}|| \sigma^2
	\end{align}
	one may calculate
	\begin{align}\label{Eq_DualApprox_zDelta_PosDef}
		& \lambda_{\min}\big( \Im\big( z \delta^{(A)} \big) \big) \nonumber\\
		& \overset{\text{(\ref{Eq_DualSystem})}}{=} \lambda_{\min}\bigg(-\frac{1}{n} \tr\bigg( A_rA_s^* \Im\Big(\Big( \Id_d + \sum\limits_{r',s'=1}^R \delta^{(B)}_{r',s'} A_{r'} A_{s'}^* \Big)^{-1}\Big) \bigg)_{r,s \leq R} \bigg) \nonumber\\
		& \overset{\text{(\ref{Eq_NonDegeneracy_TrCalcA})}}{\geq} \tau \lambda_{\min}\bigg(-\Im\Big(\Big( \Id_d + \sum\limits_{r',s'=1}^R \delta^{(B)}_{r',s'} A_{r'} A_{s'}^* \Big)^{-1}\Big) \bigg) \nonumber\\
		& \overset{\text{(\ref{Eq_DualApprox_ImInv_Identity})}}{=} \tau \lambda_{\min}\bigg(\Big( \overbrace{\Id_d + \sum\limits_{r',s'=1}^R \delta^{(B)}_{r',s'} A_{r'} A_{s'}^*}^{\sigma_{\max}(\cdot) \overset{\text{(\ref{Eq_DualApprox_VA_PrelimBound})}}{\leq} 1 + ||\delta^{(B)}|| \sigma^2} \Big)^{-*} \Big( \underbrace{\sum\limits_{r',s'=1}^R \big(\overbrace{\Im(\delta^{(B)})}^{\overset{\text{(\ref{Eq_DualApprox_posDefCondition_delta})}}{\succeq} \tilde{\tau} \Id_R}\big)_{r',s'} A_{r'} A_{s'}^*}_{\overset{\text{(\ref{Eq_NonDegeneracyC_A})}}{\succeq} \tau \tilde{\tau} \Id_d} \Big) \nonumber\\
		& \hspace{2cm} \times \Big( \underbrace{\Id_d + \sum\limits_{r',s'=1}^R \delta^{(B)}_{r',s'} A_{r'} A_{s'}^*}_{\sigma_{\max}(\cdot) \overset{\text{(\ref{Eq_DualApprox_VA_PrelimBound})}}{\leq} 1 + ||\delta^{(B)}|| \sigma^2} \Big)^{-1} \bigg) \nonumber\\
		& \geq \frac{\tau^2 \tilde{\tau}}{(1 + ||\delta^{(B)}|| \sigma^2)^2} \overset{\text{(\ref{Eq_DualApprox_BoundCondition_delta})}}{\geq} \frac{\tau^2 \tilde{\tau}}{(1 + \tilde{\kappa} \sigma^2)^2} \ .
	\end{align}
	Analogously it may be shown that
	\begin{align*}
		& \lambda_{\min}\bigg(-\frac{1}{n} \tr\bigg( A_rA_s^* \Im\Big(\Big( \Id_d + \sum\limits_{r',s'=1}^R \tilde{\delta}^{(B)}_{r',s'} A_{r'} A_{s'}^* \Big)^{-1}\Big) \bigg)_{r,s \leq R} \bigg) \geq \frac{\tau^2 \tilde{\tau}}{(1 + \tilde{\kappa} \sigma^2)^2} \ ,
	\end{align*}
	which with (\ref{Eq_DualApprox_DualSystem}) and assumption (\ref{Eq_DualApprox_qA_Assumption}) yields
	\begin{align}\label{Eq_DualApprox_ztDelta_PosDef}
		& \lambda_{\min}\big( \Im\big( z \tilde{\delta}^{(A)} \big) \big) \geq \frac{\tau^2 \tilde{\tau}}{2(1 + \tilde{\kappa} \sigma^2)^2} \ .
	\end{align}
	
	\item[ii)] \textit{Definition and properties of $V^{(A/B,\delta)}$}:\\
	Introduce the notation
	\begin{align}\label{Eq_DualApprox_DefV}
		& V^{(A,\delta)} \coloneq \sum\limits_{r,s=1}^R \delta^{(B)}_{r,s} A_r A_s^* \ \ ; \ \ V^{(B,\delta)} \coloneq \sum\limits_{r,s=1}^R \delta^{(A)}_{r,s} B_s^* B_r \ ,
	\end{align}
	then the observations
	\begin{align}\label{Eq_DualApprox_ImVA_Calc}
		& \Im\big(V^{(A,\delta)}\big) = \frac{1}{2\bm{i}} \big( V^{(A,\delta)} - (V^{(A,\delta)})^* \big) \nonumber\\
		& \overset{\text{(\ref{Eq_DualApprox_DefV})}}{=} \sum\limits_{r,s=1}^R \frac{1}{2\bm{i}} \big( \delta^{(B)}_{r,s} - \ol{\delta^{(B)}_{s,r}} \big) A_rA_s^* = \sum\limits_{r,s=1}^R \Im\big( \delta^{(B)}\big)_{r,s} A_rA_s^*
	\end{align}
	and analogously
	\begin{align}\label{Eq_DualApprox_ImVB_Calc}
		& \Im\big(V^{(B,\delta)}\big) = \frac{1}{2\bm{i}} \big( V^{(B,\delta)} - (V^{(B,\delta)})^* \big) \nonumber\\
		& \overset{\text{(\ref{Eq_DualApprox_DefV})}}{=} \sum\limits_{r,s=1}^R \frac{1}{2\bm{i}} \big( \delta^{(A)}_{r,s} - \ol{\delta^{(A)}_{s,r}} \big) B_s^*B_r = \sum\limits_{r,s=1}^R \Im\big( \delta^{(A)}\big)_{r,s} B_s^*B_r
	\end{align}
	by Lemma~\ref{Lemma_NonDegeneracyNew} prove
	\begin{align}
		& \lambda_{\min}\big( \Im\big(V^{(A,\delta)}\big) \big) \overset{\substack{\text{(\ref{Eq_DualApprox_ImVA_Calc})} \\ \text{(\ref{Eq_NonDegeneracyC_A})}}}{\geq} \tau \lambda_{\min}\big( \Im\big( \delta^{(B)}\big) \big) \overset{\text{(\ref{Eq_DualApprox_posDefCondition_delta})}}{\geq} \tau \tilde{\tau} \label{Eq_DualApprox_ImVA_posDef_delta}\\
		& \lambda_{\min}\big( \Im\big(V^{(B,\delta)}\big) \big) \overset{\substack{\text{(\ref{Eq_DualApprox_ImVB_Calc})} \\ \text{(\ref{Eq_NonDegeneracyC_B})}}}{\geq} \tau \lambda_{\min}\big( \Im\big( \delta^{(A)}\big) \big) \overset{\text{(\ref{Eq_DualApprox_posDefCondition_delta})}}{\geq} \tau \tilde{\tau} \ . \label{Eq_DualApprox_ImVB_posDef_delta}
	\end{align}
	The same constructions for $\tilde{\delta}$ instead of $\delta$ analogously yield
	\begin{align}
		& \lambda_{\min}\big( \Im\big(V^{(A,\tilde{\delta})}\big) \big) \overset{\substack{\text{(\ref{Eq_DualApprox_ImVA_Calc})} \\ \text{(\ref{Eq_NonDegeneracyC_A})}}}{\geq} \tau \lambda_{\min}\big( \Im\big( \tilde{\delta}^{(B)}\big) \big) \overset{\text{(\ref{Eq_DualApprox_posDefCondition_tdelta})}}{\geq} \tau \tilde{\tau} \label{Eq_DualApprox_ImVA_posDef_tdelta}\\
		& \lambda_{\min}\big( \Im\big(V^{(B,\tilde{\delta})}\big) \big) \overset{\substack{\text{(\ref{Eq_DualApprox_ImVB_Calc})} \\ \text{(\ref{Eq_NonDegeneracyC_B})}}}{\geq} \tau \lambda_{\min}\big( \Im\big( \tilde{\delta}^{(A)}\big) \big) \overset{\text{(\ref{Eq_DualApprox_posDefCondition_tdelta})}}{\geq} \tau \tilde{\tau} \ . \label{Eq_DualApprox_ImVB_posDef_tdelta}
	\end{align}
	
	\item[iii)] \textit{Definition and properties of $G^{(A/B,\delta)}$}:\\
	Further, introduce the matrices
	\begin{align}\label{Eq_DualApprox_DefG}
		& G^{(A,\delta)} \coloneq \big( \Id_d + V^{(A,\delta)} \big)^{-1} \ \ ; \ \ G^{(B,\delta)} \coloneq \big( \Id_n + V^{(B,\delta)} \big)^{-1} \ ,
	\end{align}
	which exist, and even satisfy
	\begin{align}\label{Eq_DualApprox_GBound}
		& ||G^{(A,\delta)}|| \leq \frac{1}{\tau \tilde{\tau}} \ \ \text{ and } \ \ ||G^{(B,\delta)}|| \leq \frac{1}{\tau \tilde{\tau}} \ ,
	\end{align}
	by (\ref{Eq_DualApprox_ImVA_posDef_delta}) and (\ref{Eq_DualApprox_ImVB_posDef_delta}).
	With the basic matrix identities
	\begin{align}\label{Eq_InverseDifferenceIdentity}
		& A^{-1} + B^{-1} = A^{-1}BB^{-1} + A^{-1}AB^{-1} = A^{-1} (A+B) B^{-1} \nonumber\\
		& A^{-1} - B^{-1} = A^{-1}BB^{-1} - A^{-1}AB^{-1} = A^{-1} (B-A) B^{-1} \ ,
	\end{align}
	one may calculate
	\begin{align}\label{Eq_DualApprox_ReGA_Calc}
		& \Re\big(G^{(A,\delta)}\big) = \frac{1}{2} \big( G^{(A,\delta)} + (G^{(A,\delta)})^* \big) \nonumber\\
		& \overset{\text{(\ref{Eq_InverseDifferenceIdentity})}}{=} \frac{1}{2} G^{(A,\delta)} \Big( \big( \Id_d + V^{(A,\delta)} \big) + \big( \Id_d + V^{(A,\delta)} \big)^* \Big) (G^{(A,\delta)})^* \nonumber\\
		& = G^{(A,\delta)} (G^{(A,\delta)})^* + G^{(A,\delta)} \Re\big(V^{(A,\delta)}\big) (G^{(A,\delta)})^*
	\end{align}
	as well as
	\begin{align}\label{Eq_DualApprox_ImGA_Calc}
		& \Im\big(G^{(A,\delta)}\big) = \frac{1}{2\bm{i}} \big( G^{(A,\delta)} - (G^{(A,\delta)})^* \big) \nonumber\\
		& \overset{\text{(\ref{Eq_InverseDifferenceIdentity})}}{=} \frac{1}{2\bm{i}} G^{(A,\delta)} \Big( \big( \Id_d + V^{(A,\delta)} \big)^* - \big( \Id_d + V^{(A,\delta)} \big) \Big) (G^{(A,\delta)})^* \nonumber\\
		& = - G^{(A,\delta)} \Im\big(V^{(A,\delta)}\big) (G^{(A,\delta)})^*
	\end{align}
	and analogously
	\begin{align}
		& \Re\big(G^{(B,\delta)}\big) = G^{(B,\delta)} (G^{(B,\delta)})^* + G^{(B,\delta)} \Re\big(V^{(B,\delta)}\big) (G^{(B,\delta)})^* \label{Eq_DualApprox_ReGB_Calc}\\
		& \Im\big(G^{(B,\delta)}\big) = - G^{(B,\delta)} \Im\big(V^{(B,\delta)}\big) (G^{(B,\delta)})^* \ . \label{Eq_DualApprox_ImGB_Calc}
	\end{align}
	This construction and the calculation (\ref{Eq_DualApprox_ReGA_Calc})-(\ref{Eq_DualApprox_ImGB_Calc}) hold analogously for $\tilde{\delta}$ instead of $\delta$.
	
	\item[iv)] \textit{Translation of the dual system of equations}:\\
	The dual system of equations (\ref{Eq_DualSystem}) by the construction (\ref{Eq_DualApprox_DefV}) and (\ref{Eq_DualApprox_DefG}) implies
	\begin{align}\label{Eq_DualApprox_DualSystem_VG_delta}
		& z V^{(B,\delta)} = - \sum\limits_{r,s=1}^R \frac{1}{n} \tr\big( A_rA_s^* G^{(A,\delta)} \big) B_s^* B_r \nonumber\\
		& z V^{(A,\delta)} = - \sum\limits_{r,s=1}^R \frac{1}{n} \tr\big( B_s^*B_r G^{(B,\delta)} \big) A_r A_s^*
	\end{align}
	and its approximate analogon (\ref{Eq_DualApprox_DualSystem}) implies
	\begin{align}\label{Eq_DualApprox_DualSystem_VG_tdelta}
		& z V^{(B,\tilde{\delta})} = - \sum\limits_{r,s=1}^R \frac{1}{n} \tr\big( A_rA_s^* G^{(A,\tilde{\delta})} \big) B_s^* B_r + z \sum\limits_{r,s=1}^R q^{(A)}_{r,s} B_s^* B_r \nonumber\\
		& z V^{(A,\tilde{\delta})} = - \sum\limits_{r,s=1}^R \frac{1}{n} \tr\big( B_s^*B_r G^{(B,\tilde{\delta})} \big) A_r A_s^* + z \sum\limits_{r,s=1}^R q^{(B)}_{r,s} A_r A_s^* \ .
	\end{align}
	Taking real and imaginary parts of (\ref{Eq_DualApprox_DualSystem_VG_tdelta}) then yields
	\begin{align}\label{Eq_DualApprox_RezVB_Calc_tdelta}
		& \Re(z) \Re\big( V^{(B,\tilde{\delta})} \big) - \Im(z) \Im\big( V^{(B,\tilde{\delta})} \big) = \Re\big( z V^{(B,\tilde{\delta})} \big) \nonumber\\
		& \overset{\text{(\ref{Eq_DualApprox_DualSystem_VG_tdelta})}}{=} -\frac{1}{2} \sum\limits_{r,s=1}^R \frac{1}{n} \Big( \tr\big( A_rA_s^* G^{(A,\tilde{\delta})} \big) + \ol{\tr\big( A_sA_r^* G^{(A,\tilde{\delta})} \big)} \Big) B_s^*B_r \nonumber\\
		& \hspace{0.5cm} + \frac{1}{2} \sum\limits_{r,s=1}^R \big( zq^{(A)}_{r,s} + \ol{zq^{(A)}_{s,r}} \big) B_s^*B_r \nonumber\\
		& = - \sum\limits_{r,s=1}^R \frac{1}{n} \tr\big( A_rA_s^* \Re\big(G^{(A,\tilde{\delta})}\big) \big) B_s^*B_r + \sum\limits_{r,s=1}^R \Re(zq^{(A)})_{r,s} B_s^*B_r \nonumber\\
		& \overset{\text{(\ref{Eq_DualApprox_ReGA_Calc})}}{=} - \sum\limits_{r,s=1}^R \frac{1}{n} \tr\Big( A_rA_s^* G^{(A,\tilde{\delta})} (G^{(A,\tilde{\delta})})^* \Big) B_s^*B_r \nonumber\\
		& \hspace{0.5cm} - \sum\limits_{r,s=1}^R \frac{1}{n} \tr\Big( A_rA_s^* G^{(A,\tilde{\delta})} \Re\big(V^{(A,\tilde{\delta})}\big) (G^{(A,\tilde{\delta})})^* \Big) B_s^*B_r \nonumber\\
		& \hspace{0.5cm} + \sum\limits_{r,s=1}^R \Re(zq^{(A)})_{r,s} B_s^*B_r
	\end{align}
	and
	\begin{align}\label{Eq_DualApprox_ImzVB_Calc_tdelta}
		& \Re(z) \Im\big( V^{(B,\tilde{\delta})} \big) + \Im(z) \Re\big( V^{(B,\tilde{\delta})} \big) = \Im\big( z V^{(B,\tilde{\delta})} \big) \nonumber\\
		& \overset{\text{(\ref{Eq_DualApprox_DualSystem_VG_tdelta})}}{=} -\frac{1}{2\bm{i}} \sum\limits_{r,s=1}^R \frac{1}{n} \Big( \tr\big( A_rA_s^* G^{(A,\tilde{\delta})} \big) - \ol{\tr\big( A_sA_r^* G^{(A,\tilde{\delta})} \big)} \Big) B_s^*B_r \nonumber\\
		& \hspace{0.5cm} + \frac{1}{2\bm{i}} \sum\limits_{r,s=1}^R \big( zq^{(A)}_{r,s} - \ol{zq^{(A)}_{s,r}} \big) B_s^*B_r \nonumber\\
		& = - \sum\limits_{r,s=1}^R \frac{1}{n} \tr\big( A_rA_s^* \Im\big(G^{(A,\tilde{\delta})}\big) \big) B_s^*B_r + \sum\limits_{r,s=1}^R \Im(zq^{(A)})_{r,s} B_s^*B_r \nonumber\\
		& \overset{\text{(\ref{Eq_DualApprox_ImGA_Calc})}}{=} \sum\limits_{r,s=1}^R \frac{1}{n} \tr\Big( A_rA_s^* G^{(A,\tilde{\delta})} \Im\big(V^{(A,\tilde{\delta})}\big) (G^{(A,\tilde{\delta})})^* \Big) B_s^*B_r + \sum\limits_{r,s=1}^R \Im(zq^{(A)})_{r,s} B_s^*B_r
	\end{align}
	as well as
	\begin{align}\label{Eq_DualApprox_RezVA_Calc_tdelta}
		& \Re(z) \Re\big( V^{(A,\tilde{\delta})} \big) - \Im(z) \Im\big( V^{(A,\tilde{\delta})} \big) = \Re\big( z V^{(A,\tilde{\delta})} \big) \nonumber\\
		& = - \sum\limits_{r,s=1}^R \frac{1}{n} \tr\Big( B_s^*B_r G^{(B,\tilde{\delta})} (G^{(B,\tilde{\delta})})^* \Big) A_rA_s^* \nonumber\\
		& \hspace{0.5cm} - \sum\limits_{r,s=1}^R \frac{1}{n} \tr\Big( B_s^*B_r G^{(B,\tilde{\delta})} \Re\big(V^{(B,\tilde{\delta})}\big) (G^{(B,\tilde{\delta})})^* \Big) A_rA_s^* \nonumber\\
		& \hspace{0.5cm} + \sum\limits_{r,s=1}^R \Re(zq^{(B)})_{r,s} A_r A_s^*
	\end{align}
	and
	\begin{align}\label{Eq_DualApprox_ImzVA_Calc_tdelta}
		& \Re(z) \Im\big( V^{(A,\tilde{\delta})} \big) + \Im(z) \Re\big( V^{(A,\tilde{\delta})} \big) = \Im\big( z V^{(A,\tilde{\delta})} \big) \nonumber\\
		& = \sum\limits_{r,s=1}^R \frac{1}{n} \tr\Big( B_s^*B_r G^{(B,\tilde{\delta})} \Im\big(V^{(B,\tilde{\delta})}\big) (G^{(B,\tilde{\delta})})^* \Big) A_rA_s^* + \sum\limits_{r,s=1}^R \Im(zq^{(B)})_{r,s} A_r A_s^* \ .
	\end{align}
	
	\item[v)] \textit{Strategic combination of equations}:\\
	Combining (\ref{Eq_DualApprox_RezVA_Calc_tdelta}) and (\ref{Eq_DualApprox_ImzVA_Calc_tdelta}) yields
	\begin{align}\label{Eq_DualApprox_Comb1}
		& |z|^2 \Im\big( V^{(A,\tilde{\delta})} \big) = \Im(z)^2 \Im\big( V^{(A,\tilde{\delta})} \big) + \Re(z)^2 \Im\big( V^{(A,\tilde{\delta})} \big) \nonumber\\
		& = -\Im(z) \text{(\ref{Eq_DualApprox_RezVA_Calc_tdelta})} + \Re(z) \text{(\ref{Eq_DualApprox_ImzVA_Calc_tdelta})} \nonumber\\
		& = \Im(z) \sum\limits_{r,s=1}^R \frac{1}{n} \tr\Big( B_s^*B_r G^{(B,\tilde{\delta})} (G^{(B,\tilde{\delta})})^* \Big) A_rA_s^* \nonumber\\
		& \hspace{0.5cm} + \Im(z) \sum\limits_{r,s=1}^R \frac{1}{n} \tr\Big( B_s^*B_r G^{(B,\tilde{\delta})} \Re\big(V^{(B,\tilde{\delta})}\big) (G^{(B,\tilde{\delta})})^* \Big) A_rA_s^* \nonumber\\
		& \hspace{0.5cm} + \Re(z) \sum\limits_{r,s=1}^R \frac{1}{n} \tr\Big( B_s^*B_r G^{(B,\tilde{\delta})} \Im\big(V^{(B,\tilde{\delta})}\big) (G^{(B,\tilde{\delta})})^* \Big) A_rA_s^* \nonumber\\
		& \hspace{0.5cm} - \Im(z) \sum\limits_{r,s=1}^R \Re(zq^{(B)})_{r,s} A_rA_s^* + \Re(z) \sum\limits_{r,s=1}^R \Im(zq^{(B)})_{r,s} A_rA_s^* \nonumber\\
		& = \Im(z) \sum\limits_{r,s=1}^R \frac{1}{n} \tr\Big( B_s^*B_r G^{(B,\tilde{\delta})} (G^{(B,\tilde{\delta})})^* \Big) A_rA_s^* \nonumber\\
		& \hspace{0.5cm} + \sum\limits_{r,s=1}^R \frac{1}{n} \tr\Big( B_s^*B_r G^{(B,\tilde{\delta})} \big( \Im(z) \Re\big(V^{(B,\tilde{\delta})}\big) + \Re(z) \Im\big(V^{(B,\tilde{\delta})}\big) \big) (G^{(B,\tilde{\delta})})^* \Big) A_rA_s^* \nonumber\\
		& \hspace{0.5cm} + \sum\limits_{r,s=1}^R \big( - \Im(z) \Re(zq^{(B)}) + \Re(z) \Im(zq^{(B)}) \big)_{r,s} A_rA_s^* \nonumber\\
		& = \Im(z) \sum\limits_{r,s=1}^R \frac{1}{n} \tr\Big( B_s^*B_r G^{(B,\tilde{\delta})} (G^{(B,\tilde{\delta})})^* \Big) A_rA_s^* \nonumber\\
		& \hspace{0.5cm} + \sum\limits_{r,s=1}^R \frac{1}{n} \tr\Big( B_s^*B_r G^{(B,\tilde{\delta})} \Im\big(z V^{(B,\tilde{\delta})}\big) (G^{(B,\tilde{\delta})})^* \Big) A_rA_s^* \nonumber\\
		& \hspace{0.5cm} + \sum\limits_{r,s=1}^R |z|^2 \Im(q^{(B)})_{r,s} A_rA_s^* \ ,
	\end{align}
	which in turn may be inserted into (\ref{Eq_DualApprox_ImzVB_Calc_tdelta}) for
	\begin{align}\label{Eq_DualApprox_Comb2}
		& |z|^2\Im\big( z V^{(B,\tilde{\delta})} \big) \nonumber\\
		& \overset{\text{(\ref{Eq_DualApprox_ImzVB_Calc_tdelta})}}{=} \sum\limits_{r,s=1}^R \frac{1}{n} \tr\Big( A_rA_s^* G^{(A,\tilde{\delta})} |z|^2 \Im\big(V^{(A,\tilde{\delta})}\big) (G^{(A,\tilde{\delta})})^* \Big) B_s^*B_r \nonumber\\
		& \hspace{0.5cm} + |z|^2 \sum\limits_{r,s=1}^R \Im(zq^{(A)})_{r,s} B_s^*B_r \nonumber\\
		& \overset{\text{(\ref{Eq_DualApprox_Comb1})}}{=} \Im(z) \sum\limits_{r,r',s,s'=1}^R \frac{1}{n} \tr\Big( A_rA_s^* G^{(A,\tilde{\delta})} A_{r'}A_{s'}^* (G^{(A,\tilde{\delta})})^* \Big) \nonumber\\
		& \hspace{3cm} \times \frac{1}{n} \tr\Big( B_{s'}^*B_{r'} G^{(B,\tilde{\delta})} (G^{(B,\tilde{\delta})})^* \Big) B_{s}^*B_{r} \nonumber\\
		& \hspace{0.5cm} + \sum\limits_{r,r',s,s'=1}^R \frac{1}{n} \tr\Big( A_rA_s^* G^{(A,\tilde{\delta})} A_{r'}A_{s'}^* (G^{(A,\tilde{\delta})})^* \Big) \nonumber\\
		& \hspace{2cm} \times \frac{1}{n} \tr\Big( B_{s'}^*B_{r'} G^{(B,\tilde{\delta})} \Im\big(z V^{(B,\tilde{\delta})}\big) (G^{(B,\tilde{\delta})})^* \Big) B_s^*B_r \nonumber\\
		& \hspace{0.5cm} + |z|^2 \sum\limits_{r,s=1}^R \Im(zq^{(A)})_{r,s} B_s^*B_r \nonumber\\
		& \hspace{0.5cm} + |z|^2 \sum\limits_{r,r',s,s'=1}^R \Im(q^{(B)})_{r',s'} \frac{1}{n} \tr\Big( A_rA_s^* G^{(A,\tilde{\delta})} A_{r'}A_{s'}^* (G^{(A,\tilde{\delta})})^* \Big) B_s^*B_r \ .
	\end{align}
	
	\item[vi)] \textit{Defining the operators $S^{(A)}$ and $S^{(B)}$ and translating equations}:\\
	Define the linear operators
	\begin{align*}
		& S^{(A)} : \C^{n \times n} \rightarrow \C^{d \times d} \ \ \text{ and } \ \ S^{(B)} : \C^{d \times d} \rightarrow \C^{n \times n}
	\end{align*}
	by
	\begin{align}\label{Eq_DualApprox_DefS}
		& S^{(A)}(\tilde{M}) \coloneq \sum\limits_{r,s=1}^R \frac{1}{n} \tr\big( B_s^*B_r \tilde{M} \big) A_rA_s^* \nonumber\\
		& S^{(B)}(M) \coloneq \sum\limits_{r,s=1}^R \frac{1}{n} \tr\big( A_rA_s^* M \big) B_s^*B_r \ .
	\end{align}
	The properties
	\begin{align}\label{Eq_DualApprox_S_passes_conjugation}
		& S^{(A)}(\tilde{M})^* = S^{(A)}(\tilde{M}^*) \ \ \text{ and } \ \ S^{(B)}(M)^* = S^{(B)}(M^*)
	\end{align}
	are straightforward to calculate. The equations (\ref{Eq_DualApprox_DualSystem_VG_delta}) and (\ref{Eq_DualApprox_DualSystem_VG_tdelta}) may then be written as
	\begin{align}\label{Eq_DualApprox_DualSystem_VSG_delta}
		& - z V^{(B,\delta)} = S^{(B)}\big( G^{(A,\delta)} \big) \nonumber\\
		& - z V^{(A,\delta)} = S^{(A)}\big( G^{(B,\delta)} \big)
	\end{align}
	and
	\begin{align}\label{Eq_DualApprox_DualSystem_VSG_tdelta}
		& - z V^{(B,\tilde{\delta})} + z \sum\limits_{r,s=1}^R q^{(A)}_{r,s} B_s^* B_r = S^{(B)}\big( G^{(A,\tilde{\delta})} \big) \nonumber\\
		& - z V^{(A,\tilde{\delta})} + z \sum\limits_{r,s=1}^R q^{(B)}_{r,s} A_r A_s^* = S^{(A)}\big( G^{(B,\tilde{\delta})} \big) \ .
	\end{align}
	Furthermore, the equality (\ref{Eq_DualApprox_Comb2}) may be written as
	\begin{align}\label{Eq_DualApprox_Comb3_tdelta}
		& |z|^2\Im\big( z V^{(B,\tilde{\delta})} \big) \nonumber\\
		& = \Im(z) S^{(B)}\Big( G^{(A,\tilde{\delta})} S^{(A)}\big( G^{(B,\tilde{\delta})} (G^{(B,\tilde{\delta})})^* \big) (G^{(A,\tilde{\delta})})^* \Big) \nonumber\\
		& \hspace{0.5cm} + S^{(B)}\Big( G^{(A,\tilde{\delta})} S^{(A)}\Big( G^{(B,\tilde{\delta})} \Im\big(z V^{(B,\tilde{\delta})}\big) (G^{(B,\tilde{\delta})})^* \Big) (G^{(A,\tilde{\delta})})^* \Big) \nonumber\\
		& \hspace{0.5cm} + |z|^2 \sum\limits_{r,s=1}^R \Im(zq^{(A)})_{r,s} B_s^*B_r\\
		& \hspace{0.5cm} + |z|^2 S^{(B)}\Big( G^{(A,\tilde{\delta})} \Big( \sum\limits_{r',s'=1}^R \Im(q^{(B)})_{r',s'} A_{r'}A_{s'}^* \Big) (G^{(A,\tilde{\delta})})^* \Big) \ .
	\end{align}
	The same calculations from (iv) and (v) hold with $(\delta,0,0)$ instead of $(\tilde{\delta},q^{(A)},q^{(B)})$, which yields
	\begin{align}\label{Eq_DualApprox_Comb3}
		& |z|^2\Im\big( z V^{(B,\delta)} \big) \nonumber\\
		& = \Im(z) S^{(B)}\Big( G^{(A,\delta)} S^{(A)}\big( G^{(B,\delta)} (G^{(B,\delta)})^* \big) (G^{(A,\delta)})^* \Big) \nonumber\\
		& \hspace{0.5cm} + S^{(B)}\Big( G^{(A,\delta)} S^{(A)}\Big( G^{(B,\delta)} \Im\big(z V^{(B,\delta)}\big) (G^{(B,\delta)})^* \Big) (G^{(A,\delta)})^* \Big) \ .
	\end{align}
	
	\item[vii)] \textit{Bounding error terms from (\ref{Eq_DualApprox_Comb3_tdelta})}:\\
	Lemma~\ref{Lemma_SimpleOperatorBounds} implies
	\begin{align}\label{Eq_DualApprox_BSumBound}
		& |z|^2 \Big|\Big| \sum\limits_{r,s=1}^R \Im(zq^{(A)})_{r,s} B_s^*B_r \Big|\Big| \leq |z|^2 ||\Im(zq^{(A)})|| \, \sigma^2
	\end{align}
	and
	\begin{align}\label{Eq_DualApprox_ASumBound}
		& \Big|\Big| \sum\limits_{r,s=1}^R \Im(q^{(B)})_{r,s} A_rA_s^* \Big|\Big| \leq ||\Im(q^{(B)})|| \, \sigma^2 \ .
	\end{align}
	By the construction of $S^{(B)}$, Lemma~\ref{Lemma_SimpleOperatorBounds} also proves
	\begin{align}\label{Eq_DualApprox_SBBound}
		& ||S^{(B)}||_{\textrm{Op}} \leq \frac{d}{n} \sigma^4 \overset{\text{\ref{ItemAssumption_cBound}}}{\leq} c_* \sigma^4 \ ,
	\end{align}
	which for the error term
	\begin{align*}
		& |z|^2 S^{(B)}\Big( G^{(A,\tilde{\delta})} \Big( \sum\limits_{r',s'=1}^R \Im(q^{(B)})_{r',s'} A_{r'}A_{s'}^* \Big) (G^{(A,\tilde{\delta})})^* \Big)
	\end{align*}
	from (\ref{Eq_DualApprox_Comb3_tdelta}), yields
	\begin{align}\label{Eq_DualApprox_SecondErrorBound}
		& |z|^2 \Big|\Big| S^{(B)}\Big( G^{(A,\tilde{\delta})} \Big( \sum\limits_{r',s'=1}^R \Im(q^{(B)})_{r',s'} A_{r'}A_{s'}^* \Big) (G^{(A,\tilde{\delta})})^* \Big) \Big|\Big| \nonumber\\
		& \overset{\text{(\ref{Eq_DualApprox_SBBound})}}{\leq} c_* \sigma^4 |z|^2 \Big|\Big| G^{(A,\tilde{\delta})} \Big( \sum\limits_{r',s'=1}^R \Im(q^{(B)})_{r',s'} A_{r'}A_{s'}^* \Big) (G^{(A,\tilde{\delta})})^* \Big|\Big| \nonumber\\
		& \overset{\text{(\ref{Eq_DualApprox_GBound})}}{\leq} \frac{c_* \sigma^4 |z|^2}{\tau^2 \tilde{\tau}^2} \Big|\Big| \sum\limits_{r',s'=1}^R \Im(q^{(B)})_{r',s'} A_{r'}A_{s'}^* \Big|\Big| \nonumber\\
		& \overset{\text{(\ref{Eq_DualApprox_ASumBound})}}{\leq} \frac{c_* \sigma^6 |z|^2}{\tau^2 \tilde{\tau}^2} ||\Im(q^{(B)})|| \leq \frac{c_* \sigma^6 |z|^2}{\tau^2 \tilde{\tau}^2} ||q^{(B)}|| \ .
	\end{align}
	
	\item[viii)] \textit{Lower bound for the first summand in (\ref{Eq_DualApprox_Comb3_tdelta})}:\\
	Analogously to (\ref{Eq_DualApprox_ASumBound}), one has the bounds
	\begin{align}\label{Eq_DualApprox_VBound}
		& ||V^{(A,\tilde{\delta})}|| \overset{\text{(\ref{Eq_DualApprox_DefV})}}{=} \Big|\Big| \sum\limits_{r,s=1}^R \tilde{\delta}^{(B)}_{r,s} A_rA_s^* \Big|\Big| \leq ||\tilde{\delta}^{(B)}|| \, \sum\limits_{r=1}^R ||A_r||^2 \overset{\text{\ref{ItemAssumption_sigmaBound}}}{\leq} ||\tilde{\delta}^{(B)}|| \, \sigma^2 \overset{\text{(\ref{Eq_DualApprox_BoundCondition_tdelta})}}{\leq} \tilde{\kappa} \sigma^2 \nonumber\\
		& ||V^{(B,\tilde{\delta})}|| \overset{\text{(\ref{Eq_DualApprox_DefV})}}{=} \Big|\Big| \sum\limits_{r,s=1}^R \tilde{\delta}^{(A)}_{r,s} B_s^* B_r \Big|\Big| \leq ||\tilde{\delta}^{(A)}|| \, \sum\limits_{r=1}^R ||B_r||^2 \overset{\text{\ref{ItemAssumption_sigmaBound}}}{\leq} ||\tilde{\delta}^{(A)}|| \, \sigma^2 \overset{\text{(\ref{Eq_DualApprox_BoundCondition_tdelta})}}{\leq} \tilde{\kappa} \sigma^2 \ ,
	\end{align}
	which by construction of $G^{(A,\tilde{\delta})}$ and $G^{(B,\tilde{\delta})}$ imply that
	\begin{align}\label{Eq_DualApprox_GLowerBound}
		& \sigma_{\min}\big( G^{(A,\tilde{\delta})} \big) \geq \big(1 + \tilde{\kappa} \, \sigma^2\big)^{-1} \nonumber\\
		& \sigma_{\min}\big( G^{(B,\tilde{\delta})} \big) \geq \big(1 + \tilde{\kappa} \, \sigma^2\big)^{-1} \ .
	\end{align}
	It directly follows that
	\begin{align*}
		& \big(1 + \tilde{\kappa} \sigma^2\big)^{-2} \Id_n \preceq G^{(B,\tilde{\delta})} (G^{(B,\tilde{\delta})})^*
	\end{align*}
	and from Lemma~\ref{Lemma_NonDegeneracyNew} one has
	\begin{align*}
		& \tau^2 \big(1 + \tilde{\kappa} \sigma^2\big)^{-2} \Id_d \preceq S^{(A)}\big( G^{(B,\tilde{\delta})} (G^{(B,\tilde{\delta})})^* \big) \ .
	\end{align*}
	Another application of (\ref{Eq_DualApprox_GLowerBound}) and Lemma~\ref{Lemma_NonDegeneracyNew} then yields
	\begin{align}\label{Eq_DualApprox_Sum1LowerBound_tdelta}
		& \tau^4 \big(1 + \tilde{\kappa} \sigma^2\big)^{-4} \Id_n \preceq S^{(B)}\Big( G^{(A,\tilde{\delta})} S^{(A)}\big( G^{(B,\tilde{\delta})} (G^{(B,\tilde{\delta})})^* \big) (G^{(A,\tilde{\delta})})^* \Big) \ .
	\end{align}
	In complete analogy one proves
	\begin{align}\label{Eq_DualApprox_Sum1LowerBound_delta}
		& \tau^4 \big(1 + \tilde{\kappa} \sigma^2\big)^{-4} \Id_n \preceq S^{(B)}\Big( G^{(A,\delta)} S^{(A)}\big( G^{(B,\delta)} (G^{(B,\delta)})^* \big) (G^{(A,\delta)})^* \Big) \ .
	\end{align}
	
	\item[ix)] \textit{Constructing the square root matrices $H_\delta$ and $H_{\tilde{\delta}}$}:\\
	In complete analogy to (\ref{Eq_DualApprox_ImVA_Calc}) one may show
	\begin{align*}
		& \Im\big(z V^{(B,\delta)}\big) = \sum\limits_{r,s=1}^R \Im\big( z \delta^{(A)}\big)_{r,s} B_s^*B_r \ \ ; \ \ \Im\big(z V^{(B,\tilde{\delta})}\big) = \sum\limits_{r,s=1}^R \Im\big( z \tilde{\delta}^{(A)}\big)_{r,s} B_s^*B_r \ .
	\end{align*}
	Lemma~\ref{Lemma_NonDegeneracyNew} together with (\ref{Eq_DualApprox_zDelta_PosDef}) and (\ref{Eq_DualApprox_ztDelta_PosDef}) then yields
	\begin{align*}
		& \tau \frac{\tau^2 \tilde{\tau}}{(1 + \tilde{\kappa} \sigma^2)^2} \Id_n \preceq \Im\big(z V^{(B,\delta)}\big) \ \ \text{ and } \ \ \tau \frac{\tau^2 \tilde{\tau}}{2(1 + \tilde{\kappa} \sigma^2)^2} \Id_n \preceq \Im\big(z V^{(B,\tilde{\delta})}\big) \ .
	\end{align*}
	Let $H_\delta$ and $H_{\tilde{\delta}}$ denote the positive definite square roots of $\Im\big(z V^{(B,\delta)}\big)$ and $\Im\big(z V^{(B,\tilde{\delta})}\big)$ respectively, then these satisfy
	\begin{align}\label{Eq_DualApprox_HProp1}
		& \frac{\sqrt{\tau^3 \tilde{\tau}}}{1 + \tilde{\kappa} \sigma^2} \Id_n \preceq H_\delta \ \ \text{ and } \ \ \frac{\sqrt{\tau^3 \tilde{\tau}/2}}{1 + \tilde{\kappa} \sigma^2} \Id_n \preceq H_{\tilde{\delta}}
	\end{align}
	as well as
	\begin{align}\label{Eq_DualApprox_HProp2}
		& \Im\big(z V^{(B,\delta)}\big) = H_\delta H_\delta \ \ \text{ and } \ \ \Im\big(z V^{(B,\tilde{\delta})}\big) = H_{\tilde{\delta}} H_{\tilde{\delta}} \ .
	\end{align}
	The bounds (\ref{Eq_DualApprox_VBound}), which hold analogously for $\delta$ instead of $\tilde{\delta}$, also yield
	\begin{align}\label{Eq_DualApprox_HBounds}
		& ||H_{\delta}|| \leq \sqrt{|z| \, \tilde{\kappa} \, \sigma^2} \ \ \text{ and } \ \ ||H_{\tilde{\delta}}|| \leq \sqrt{|z| \, \tilde{\kappa} \, \sigma^2} \ .
	\end{align}
	
	\item[x)] \textit{Constructing the operator $T_{\delta,\tilde{\delta}}$}:\\
	Define the linear operator
	\begin{align*}
		& T_{\delta,\tilde{\delta}} : \C^{n \times n} \rightarrow \C^{n \times n}
	\end{align*}
	by
	\begin{align}\label{Eq_DualApprox_DefT}
		& T_{\delta,\tilde{\delta}}(X) \coloneq \frac{1}{z^2} H_\delta^{-1} \, S^{(B)}\Big( G^{(A,\delta)} S^{(A)}\Big( G^{(B,\delta)} H_\delta X H_{\tilde{\delta}} G^{(B,\tilde{\delta})} \Big) G^{(A,\tilde{\delta})} \Big) \, H_{\tilde{\delta}}^{-1} \ .
	\end{align}
	The equalities (\ref{Eq_DualApprox_Comb3_tdelta}) and (\ref{Eq_DualApprox_Comb3}) will now be employed to show that $T_{\delta,\tilde{\delta}}$ is a contraction.
	\\[0.5em]
	To this end, introduce the notation
	\begin{align}\label{Eq_DualApprox_DefU}
		& U^{(a,k)} \coloneq \frac{1}{\sqrt{n}} \sum\limits_{r=1}^R A_r e_{a,d} e_{k,n}^\top B_r
	\end{align}
	for any $a \in \{1,\dots,d\}$ and $k \in \{1,\dots,n\}$, where $e_{\bullet,d}$ and $e_{\bullet,n}$ denote the unit vectors in $\C^d$ and $\C^n$ respectively. The calculation
	\begin{align*}
		& \sum\limits_{a=1}^d \sum\limits_{k=1}^n U^{(a,k)} \tilde{M} (U^{(a,k)})^* \overset{\text{(\ref{Eq_DualApprox_DefU})}}{=} \frac{1}{n} \sum\limits_{a=1}^d \sum\limits_{k=1}^n \Big( \sum\limits_{r=1}^R A_r e_{a,d} e_{k,n}^\top B_r \Big) \tilde{M} \Big( \sum\limits_{s=1}^R A_s e_{a,d} e_{k,n}^\top B_s \Big)^*\\
		& = \sum\limits_{r,s=1}^R \frac{1}{n} \sum\limits_{a=1}^d \sum\limits_{k=1}^n \big( e_{k,n}^\top B_r \tilde{M} B_s^* e_{k,n} \big) A_r e_{a,d}e_{a,d}^\top A_s^* = \sum\limits_{r,s=1}^R \frac{1}{n} \tr\big( B_r \tilde{M} B_s^* \big) A_rA_s^* 
	\end{align*}
	proves
	\begin{align}\label{Eq_DualApprox_SA_Decomp}
		& S^{(A)}(\tilde{M}) = \sum\limits_{a=1}^d \sum\limits_{k=1}^n U^{(a,k)} \tilde{M} (U^{(a,k)})^*
	\end{align}
	and analogously one shows
	\begin{align}\label{Eq_DualApprox_SB_Decomp}
		& S^{(B)}(M) = \sum\limits_{b=1}^d \sum\limits_{l=1}^n (U^{(b,l)})^* M U^{(b,l)} \ .
	\end{align}
	The operator (\ref{Eq_DualApprox_DefT}) may then be written as
	\begin{align*}
		& T_{\delta,\tilde{\delta}}(X)\\
		& \overset{\text{(\ref{Eq_DualApprox_SB_Decomp})}}{=} \frac{1}{z^2} H_\delta^{-1} \Big(\sum\limits_{b=1}^d \sum\limits_{l=1}^n (U^{(b,l)})^* G^{(A,\delta)} S^{(A)}\Big( G^{(B,\delta)} H_\delta X H_{\tilde{\delta}} G^{(B,\tilde{\delta})} \Big) G^{(A,\tilde{\delta})} U^{(b,l)} \Big) H_{\tilde{\delta}}^{-1}\\
		& \overset{\text{(\ref{Eq_DualApprox_SA_Decomp})}}{=} \frac{1}{z^2} \sum\limits_{a,b=1}^d \sum\limits_{k,l=1}^n \overbrace{H_\delta^{-1} (U^{(b,l)})^* G^{(A,\delta)} U^{(a,k)} G^{(B,\delta)} H_\delta}^{\eqcolon  C_{(a,b,k,l)}} X\\
		& \hspace{5cm} \times \underbrace{H_{\tilde{\delta}} G^{(B,\tilde{\delta})} (U^{(a,k)})^* G^{(A,\tilde{\delta})} U^{(b,l)} H_{\tilde{\delta}}^{-1}}_{\eqcolon  D_{(a,b,k,l)}} \ .
	\end{align*}
	Applying Lemma~\ref{Lemma_OperatorBound} before reinstating (\ref{Eq_DualApprox_SA_Decomp}) and (\ref{Eq_DualApprox_SB_Decomp}) then yields
	\begin{align}\label{Eq_DualApprox_TOpBound}
		& ||T_{\delta,\tilde{\delta}}||_{\textrm{Op}} \overset{\text{(\ref{Eq_Uniqueness_OpBound})}}{\leq} \frac{1}{|z|^2} \Big|\Big| \sum\limits_{a,b=1}^d \sum\limits_{k,l=1}^n C_{(a,b,k,l)} C_{(a,b,k,l)}^* \Big|\Big|^{\frac{1}{2}} \Big|\Big| \sum\limits_{a,b=1}^d \sum\limits_{k,l=1}^n D_{(a,b,k,l)}^* D_{(a,b,k,l)} \Big|\Big|^{\frac{1}{2}} \nonumber\\
		& = \frac{1}{|z|^2} \Big|\Big| \sum\limits_{a,b=1}^d \sum\limits_{k,l=1}^n H_\delta^{-1} (U^{(b,l)})^* G^{(A,\delta)} U^{(a,k)} G^{(B,\delta)} H_\delta \nonumber\\
		& \hspace{3cm} \times H_\delta (G^{(B,\delta)})^* (U^{(a,k)})^* (G^{(A,\delta)})^* U^{(b,l)} H_\delta^{-1} \Big|\Big|^{\frac{1}{2}} \nonumber\\
		& \hspace{1cm} \times \Big|\Big| \sum\limits_{a,b=1}^d \sum\limits_{k,l=1}^n H_{\tilde{\delta}}^{-1} (U^{(b,l)})^* (G^{(A,\tilde{\delta})})^* U^{(a,k)} (G^{(B,\tilde{\delta})})^* H_{\tilde{\delta}} \nonumber\\
		& \hspace{3cm} \times H_{\tilde{\delta}} G^{(B,\tilde{\delta})} (U^{(a,k)})^* G^{(A,\tilde{\delta})} U^{(b,l)} H_{\tilde{\delta}}^{-1} \Big|\Big|^{\frac{1}{2}} \nonumber\\
		& \overset{\text{(\ref{Eq_DualApprox_SA_Decomp})}}{=} \frac{1}{|z|^2} \Big|\Big| \sum\limits_{b=1}^d \sum\limits_{l=1}^n H_\delta^{-1} (U^{(b,l)})^* G^{(A,\delta)} S^{(A)} \Big( G^{(B,\delta)} H_\delta H_\delta (G^{(B,\delta)})^* \Big) \nonumber\\
		& \hspace{7cm} \times (G^{(A,\delta)})^* U^{(b,l)} H_\delta^{-1} \Big|\Big|^{\frac{1}{2}} \nonumber\\
		& \hspace{1cm} \times \Big|\Big| \sum\limits_{b=1}^d \sum\limits_{l=1}^n H_{\tilde{\delta}}^{-1} (U^{(b,l)})^* (G^{(A,\tilde{\delta})})^* S^{(A)} \Big( (G^{(B,\tilde{\delta})})^* H_{\tilde{\delta}} H_{\tilde{\delta}} G^{(B,\tilde{\delta})} \Big) \nonumber\\
		& \hspace{7.5cm} \times G^{(A,\tilde{\delta})} U^{(b,l)} H_{\tilde{\delta}}^{-1} \Big|\Big|^{\frac{1}{2}} \nonumber\\
		& \overset{\text{(\ref{Eq_DualApprox_SB_Decomp})}}{=} \frac{1}{|z|^2} \Big|\Big| H_\delta^{-1} S^{(B)} \Big( G^{(A,\delta)} S^{(A)} \Big( G^{(B,\delta)} H_\delta H_\delta (G^{(B,\delta)})^* \Big) (G^{(A,\delta)})^* \Big) H_\delta^{-1} \Big|\Big|^{\frac{1}{2}} \nonumber\\
		& \hspace{1cm} \times \Big|\Big| H_{\tilde{\delta}}^{-1} S^{(B)} \Big( (G^{(A,\tilde{\delta})})^* S^{(A)} \Big( (G^{(B,\tilde{\delta})})^* H_{\tilde{\delta}} H_{\tilde{\delta}} G^{(B,\tilde{\delta})} \Big) G^{(A,\tilde{\delta})} \Big) H_{\tilde{\delta}}^{-1} \Big|\Big|^{\frac{1}{2}} \nonumber\\
		& \overset{\substack{\text{(\ref{Eq_DualApprox_HProp2})} \\ \text{(\ref{Eq_DualApprox_S_passes_conjugation})}}}{=} \frac{1}{|z|^2} \Big|\Big| H_\delta^{-1} S^{(B)} \Big( G^{(A,\delta)} S^{(A)} \Big( G^{(B,\delta)} \Im\big(z V^{(B,\delta)}\big) (G^{(B,\delta)})^* \Big) (G^{(A,\delta)})^* \Big) H_\delta^{-1} \Big|\Big|^{\frac{1}{2}} \nonumber\\
		& \hspace{1cm} \times \Big|\Big| H_{\tilde{\delta}}^{-1} S^{(B)} \Big( G^{(A,\tilde{\delta})} S^{(A)} \Big( G^{(B,\tilde{\delta})} \Im\big(z V^{(B,\tilde{\delta})}\big) (G^{(B,\tilde{\delta})})^* \Big) (G^{(A,\tilde{\delta})})^* \Big) H_{\tilde{\delta}}^{-1} \Big|\Big|^{\frac{1}{2}} \ .
	\end{align}
	
	\item[xi)] \textit{Proving that $T_{\delta,\tilde{\delta}}$ is a contraction}:\\
	Plugging (\ref{Eq_DualApprox_HProp2}) into (\ref{Eq_DualApprox_Comb3}) leads to the equality
	\begin{align}\label{Eq_DualApprox_HSG_Identity_delta}
		& |z|^2 \Id_n \nonumber\\
		& = \Im(z) H_\delta^{-1} S^{(B)}\Big( G^{(A,\delta)} S^{(A)}\big( G^{(B,\delta)} (G^{(B,\delta)})^* \big) (G^{(A,\delta)})^* \Big) H_\delta^{-1} \nonumber\\
		& \hspace{0.5cm} + H_\delta^{-1} S^{(B)}\Big( G^{(A,\delta)} S^{(A)}\Big( G^{(B,\delta)} \Im\big(z V^{(B,\delta)}\big) (G^{(B,\delta)})^* \Big) (G^{(A,\delta)})^* \Big) H_\delta^{-1} \ ,
	\end{align} 
	where the first summand on the right hand side may by (\ref{Eq_DualApprox_Sum1LowerBound_delta}) and (\ref{Eq_DualApprox_HBounds}) be bounded from below by
	\begin{align}\label{Eq_DualApprox_Sum1Bound_delta}
		& \Im(z) H_\delta^{-1} S^{(B)}\Big( G^{(A,\delta)} S^{(A)}\big( G^{(B,\delta)} (G^{(B,\delta)})^* \big) (G^{(A,\delta)})^* \Big) H_\delta^{-1} \nonumber\\
		& \succeq \Im(z) \tau^4 (1 + \tilde{\kappa} \sigma^2)^{-4} \, \frac{1}{|z|} \big( ||\delta^{(A)}|| \, \sigma^2 \big)^{-1} \Id_n \nonumber\\
		& \succeq \frac{\tau^4 \Im(z)}{|z| (1 + \tilde{\kappa} \sigma^2)^{5}} \Id_n \ .
	\end{align}
	As all summands in (\ref{Eq_DualApprox_HSG_Identity_delta}) are positive definite, this implies
	\begin{align}\label{Eq_DualApprox_Factor1Bound}
		& \Big|\Big| H_\delta^{-1} S^{(B)}\Big( G^{(A,\delta)} S^{(A)}\Big( G^{(B,\delta)} \Im\big(z V^{(B,\delta)}\big) (G^{(B,\delta)})^* \Big) (G^{(A,\delta)})^* \Big) H_\delta^{-1} \Big|\Big| \nonumber\\
		& \leq |z|^2 - \frac{\tau^4 \Im(z)}{|z| (1 + \tilde{\kappa} \sigma^2)^{5}} \ ,
	\end{align}
	which gives a bound on the first factor on the right-hand side of (\ref{Eq_DualApprox_TOpBound}).
	\\[0.5em]
	Similarly, plugging (\ref{Eq_DualApprox_HProp2}) into (\ref{Eq_DualApprox_Comb3_tdelta}) yields
	\begin{align}\label{Eq_DualApprox_HSG_Identity_tdelta}
		& |z|^2 \Id_n \nonumber\\
		& = \Im(z) H_{\tilde{\delta}}^{-1} S^{(B)}\Big( G^{(A,\tilde{\delta})} S^{(A)}\big( G^{(B,\tilde{\delta})} (G^{(B,\tilde{\delta})})^* \big) (G^{(A,\tilde{\delta})})^* \Big) H_{\tilde{\delta}}^{-1} \nonumber\\
		& \hspace{0.5cm} + H_{\tilde{\delta}}^{-1} S^{(B)}\Big( G^{(A,\tilde{\delta})} S^{(A)}\Big( G^{(B,\tilde{\delta})} \Im\big(z V^{(B,\tilde{\delta})}\big) (G^{(B,\tilde{\delta})})^* \Big) (G^{(A,\tilde{\delta})})^* \Big) H_{\tilde{\delta}}^{-1} \nonumber\\
		& \hspace{0.5cm} + |z|^2 H_{\tilde{\delta}}^{-1} \sum\limits_{r,s=1}^R \Im(zq^{(A)})_{r,s} B_s^*B_r H_{\tilde{\delta}}^{-1} \nonumber\\
		& \hspace{0.5cm} + |z|^2 H_{\tilde{\delta}}^{-1} S^{(B)}\Big( G^{(A,\tilde{\delta})} \Big( \sum\limits_{r',s'=1}^R \Im(q^{(B)})_{r',s'} A_{r'}A_{s'}^* \Big) (G^{(A,\tilde{\delta})})^* \Big) H_{\tilde{\delta}}^{-1} \ .
	\end{align}
	The last two summands on the right-hand side of (\ref{Eq_DualApprox_HSG_Identity_tdelta}) may be bounded as follows
	\begin{align*}
		& |z|^2 \Big|\Big| H_{\tilde{\delta}}^{-1} \sum\limits_{r,s=1}^R \Im(zq^{(A)})_{r,s} B_s^*B_r H_{\tilde{\delta}}^{-1} \Big|\Big|\\
		& \overset{\text{(\ref{Eq_DualApprox_HProp1})}}{\leq} |z|^2 \frac{2(1 + \tilde{\kappa} \sigma^2)^2}{\tau^3 \tilde{\tau}} \Big|\Big| \sum\limits_{r,s=1}^R \Im(zq^{(A)})_{r,s} B_s^*B_r \Big|\Big|\\
		& \overset{\text{(\ref{Eq_DualApprox_BSumBound})}}{\leq} |z|^2 \frac{2(1 + \tilde{\kappa} \sigma^2)^2}{\tau^3 \tilde{\tau}} \sigma^2 ||\Im(zq^{(A)})|| \leq \frac{2 |z|^3\sigma^2(1 + \tilde{\kappa} \sigma^2)^2}{\tau^3 \tilde{\tau}} ||q^{(A)}||
	\end{align*}
	and
	\begin{align*}
		& |z|^2 \Big|\Big| H_{\tilde{\delta}}^{-1} S^{(B)}\Big( G^{(A,\tilde{\delta})} \Big( \sum\limits_{r',s'=1}^R \Im(q^{(B)})_{r',s'} A_{r'}A_{s'}^* \Big) (G^{(A,\tilde{\delta})})^* \Big) H_{\tilde{\delta}}^{-1} \Big|\Big| \nonumber\\
		& \overset{\text{(\ref{Eq_DualApprox_HProp1})}}{\leq} |z|^2 \frac{2(1 + \tilde{\kappa} \sigma^2)^2}{\tau^3 \tilde{\tau}} \Big|\Big| S^{(B)}\Big( G^{(A,\tilde{\delta})} \Big( \sum\limits_{r',s'=1}^R \Im(q^{(B)})_{r',s'} A_{r'}A_{s'}^* \Big) (G^{(A,\tilde{\delta})})^* \Big) \Big|\Big| \nonumber\\
		& \overset{\text{(\ref{Eq_DualApprox_SecondErrorBound})}}{\leq} \frac{2(1 + \tilde{\kappa} c_* \sigma^2)^2}{\tau^3 \tilde{\tau}} \, \frac{|z|^2 c_* \sigma^6}{\tau^2 \tilde{\tau}^2} ||q^{(B)}|| = \frac{2|z|^2 c_* \sigma^6(1 + \tilde{\kappa} \sigma^2)^2}{\tau^5 \tilde{\tau}^3} ||q^{(B)}|| \ .
	\end{align*}
	Assumptions (\ref{Eq_DualApprox_qSmall}) and (\ref{Eq_DualApprox_tqSmall}) thus imply
	\begin{align}\label{Eq_DualApprox_ErrorBound}
		& |z|^2 \Big|\Big| H_{\tilde{\delta}}^{-1} \sum\limits_{r,s=1}^R \Im(zq^{(A)})_{r,s} B_s^*B_r H_{\tilde{\delta}}^{-1} \Big|\Big| \nonumber\\
		& \hspace{0.5cm} + |z|^2 \Big|\Big| H_{\tilde{\delta}}^{-1} S^{(B)}\Big( G^{(A,\tilde{\delta})} \Big( \sum\limits_{r',s'=1}^R \Im(q^{(B)})_{r',s'} A_{r'}A_{s'}^* \Big) (G^{(A,\tilde{\delta})})^* \Big) H_{\tilde{\delta}}^{-1} \Big|\Big| \nonumber\\
		& \leq \frac{\tau^4 \Im(z)}{2|z| (1 + \tilde{\kappa} \sigma^2)^{5}} \ .
	\end{align}
	In complete analogy to (\ref{Eq_DualApprox_Sum1Bound_delta}), one observes
	\begin{align}\label{Eq_DualApprox_Sum1Bound_tdelta}
		& \Im(z) H_{\tilde{\delta}}^{-1} S^{(B)}\Big( G^{(A,\tilde{\delta})} S^{(A)}\big( G^{(B,\tilde{\delta})} (G^{(B,\tilde{\delta})})^* \big) (G^{(A,\tilde{\delta})})^* \Big) H_{\tilde{\delta}}^{-1} \succeq \frac{\tau^4 \Im(z)}{|z| (1 + \tilde{\kappa} \sigma^2)^{5}} \Id_n \ ,
	\end{align}
	and combining (\ref{Eq_DualApprox_HSG_Identity_tdelta}), (\ref{Eq_DualApprox_ErrorBound}) and (\ref{Eq_DualApprox_Sum1Bound_tdelta}) yields
	\begin{align}\label{Eq_DualApprox_Factor2Bound}
		& \Big|\Big| H_{\tilde{\delta}}^{-1} S^{(B)}\Big( G^{(A,\tilde{\delta})} S^{(A)}\Big( G^{(B,\tilde{\delta})} \Im\big(z V^{(B,\tilde{\delta})}\big) (G^{(B,\tilde{\delta})})^* \Big) (G^{(A,\tilde{\delta})})^* \Big) H_{\tilde{\delta}}^{-1} \Big|\Big| \nonumber\\
		& \leq |z|^2 - \frac{\tau^4 \Im(z)}{2|z| (1 + \tilde{\kappa} \sigma^2)^{5}} \ .
	\end{align}
	which gives a bound on the second factor on the right-hand side of (\ref{Eq_DualApprox_TOpBound}).
	\\[0.5em]
	The bounds (\ref{Eq_DualApprox_TOpBound}), (\ref{Eq_DualApprox_Factor1Bound}) and (\ref{Eq_DualApprox_Factor2Bound}) finally yield
	\begin{align}\label{Eq_DualApprox_TContraction}
		||T_{\delta,\tilde{\delta}}||_{\textrm{Op}} & \leq \Big( 1 - \frac{\tau^4 \Im(z)}{|z|^3 (1 + \tilde{\kappa} \sigma^2)^{5}} \Big)^{\frac{1}{2}} \Big( 1 - \frac{\tau^4 \Im(z)}{2|z|^3 (1 + \tilde{\kappa} \sigma^2)^{5}} \Big)^{\frac{1}{2}} \nonumber\\
		& \leq \Big( 1 - \frac{\tau^4 \Im(z)}{2|z|^3 (1 + \tilde{\kappa} \sigma^2)^{5}} \Big)^{\frac{1}{2}} \leq 1 - \frac{\tau^4 \Im(z)}{4|z|^3 (1 + \tilde{\kappa} \sigma^2)^{5}} \ ,
	\end{align}
	thus proving $T_{\delta,\tilde{\delta}}$ to be a contraction.
	
	\item[xii)] \textit{Proving that $V^{(B,\delta)} - V^{(B,\tilde{\delta})}$ is small by Lemma~\ref{Lemma_ContractionEVs}}:\\
	By linearity of the operators $S^{(A)}$ and $S^{(B)}$, calculate
	\begin{align*}
		& T_{\delta,\tilde{\delta}}\Big( H_\delta^{-1} \big( V^{(B,\delta)} - V^{(B,\tilde{\delta})} \big) H_{\tilde{\delta}}^{-1} \Big)\\
		& \overset{\text{(\ref{Eq_DualApprox_DefT})}}{=} \frac{1}{z^2} H_\delta^{-1} \, S^{(B)}\Big( G^{(A,\delta)} S^{(A)}\Big( \underbrace{G^{(B,\delta)} \big( V^{(B,\delta)} - V^{(B,\tilde{\delta})} \big) G^{(B,\tilde{\delta})}}_{\overset{\substack{\text{(\ref{Eq_DualApprox_DefG})} \\ \text{(\ref{Eq_InverseDifferenceIdentity})}}}{=} G^{(B,\tilde{\delta})} - G^{(B,\delta)}} \Big) G^{(A,\tilde{\delta})} \Big) \, H_{\tilde{\delta}}^{-1}\\
		& \overset{\substack{\text{(\ref{Eq_DualApprox_DualSystem_VSG_delta})} \\ \text{(\ref{Eq_DualApprox_DualSystem_VSG_tdelta})}}}{=} \frac{1}{z^2} H_\delta^{-1} \, S^{(B)}\Big( G^{(A,\delta)} \Big( zV^{(A,\delta)} - zV^{(A,\tilde{\delta})} + z\sum\limits_{r,s=1}^R q^{(B)}_{r,s} A_rA_s^* \Big) G^{(A,\tilde{\delta})} \Big) \, H_{\tilde{\delta}}^{-1}\\
		& = \frac{1}{z} H_\delta^{-1} \, S^{(B)}\Big( \overbrace{G^{(A,\delta)} \big( V^{(A,\delta)} - V^{(A,\tilde{\delta})} \big) G^{(A,\tilde{\delta})}}^{\overset{\substack{\text{(\ref{Eq_DualApprox_DefG})} \\ \text{(\ref{Eq_InverseDifferenceIdentity})}}}{=} G^{(A,\tilde{\delta})} - G^{(A,\delta)}} \Big) \, H_{\tilde{\delta}}^{-1}\\
		& \hspace{0.5cm} + \frac{1}{z} H_\delta^{-1} \, S^{(B)}\Big( G^{(A,\delta)} \Big( \sum\limits_{r,s=1}^R q^{(B)}_{r,s} A_rA_s^* \Big) G^{(A,\tilde{\delta})} \Big) \, H_{\tilde{\delta}}^{-1}\\
		& \overset{\substack{\text{(\ref{Eq_DualApprox_DualSystem_VSG_delta})} \\ \text{(\ref{Eq_DualApprox_DualSystem_VSG_tdelta})}}}{=} \frac{1}{z} H_\delta^{-1} \, \Big( zV^{(B,\delta)} - zV^{(B,\tilde{\delta})} + z\sum\limits_{r,s=1}^R q^{(A)}_{r,s} B_s^*B_r \Big) \, H_{\tilde{\delta}}^{-1}\\
		& \hspace{0.5cm} + \frac{1}{z} H_\delta^{-1} \, S^{(B)}\Big( G^{(A,\delta)} \Big( \sum\limits_{r,s=1}^R q^{(B)}_{r,s} A_rA_s^* \Big) G^{(A,\tilde{\delta})} \Big) \, H_{\tilde{\delta}}^{-1}\\
		& = H_\delta^{-1} \big( V^{(B,\delta)} - V^{(B,\tilde{\delta})} \big) H_{\tilde{\delta}}^{-1}\\
		& \hspace{0.5cm} + H_\delta^{-1} \, \Big( \sum\limits_{r,s=1}^R q^{(A)}_{r,s} B_s^*B_r \Big) \, H_{\tilde{\delta}}^{-1}\\
		& \hspace{0.5cm} + \frac{1}{z} H_\delta^{-1} \, S^{(B)}\Big( G^{(A,\delta)} \Big( \sum\limits_{r,s=1}^R q^{(B)}_{r,s} A_rA_s^* \Big) G^{(A,\tilde{\delta})} \Big) \, H_{\tilde{\delta}}^{-1} \ ,
	\end{align*}
	so $H_\delta^{-1} \big( V^{(B,\delta)} - V^{(B,\tilde{\delta})} \big) H_{\tilde{\delta}}^{-1}$ is a near-eigenvector of $T_{\delta,\tilde{\delta}}$ in the sense of Lemma~\ref{Lemma_ContractionEVs}. One by the same Lemma with
	\begin{align*}
		& \theta_{\text{Lemma~\ref{Lemma_ContractionEVs}}} \overset{\text{(\ref{Eq_DualApprox_TContraction})}}{=} \frac{\tau^4 \Im(z)}{4|z|^3 (1 + \tilde{\kappa} \sigma^2)^{5}}
	\end{align*}
	gets
	\begin{align}\label{Eq_DualApprox_HDiffBound}
		& \big|\big| H_\delta^{-1} \big( V^{(B,\delta)} - V^{(B,\tilde{\delta})} \big) H_{\tilde{\delta}}^{-1} \big|\big| \nonumber\\
		& \leq \frac{\big|\big| H_\delta^{-1} \, \big( \sum\limits_{r,s=1}^R q^{(A)}_{r,s} B_s^*B_r \big) \, H_{\tilde{\delta}}^{-1} + \frac{1}{z} H_\delta^{-1} \, S^{(B)}\big( G^{(A,\delta)} \big( \sum\limits_{r,s=1}^R q^{(B)}_{r,s} A_rA_s^* \big) G^{(A,\tilde{\delta})} \big) \, H_{\tilde{\delta}}^{-1} \big|\big|}{\frac{\tau^4 \Im(z)}{4|z|^3 (1 + \tilde{\kappa} \sigma^2)^{5}}} \nonumber\\
		& \overset{\text{(\ref{Eq_DualApprox_HProp1})}}{\leq} \frac{\frac{2(1+\tilde{\kappa}\sigma^2)^2}{\tau^3\tilde{\tau}} \big|\big| \big( \sum\limits_{r,s=1}^R q^{(A)}_{r,s} B_s^*B_r \big) \big|\big|}{\frac{\tau^4 \Im(z)}{4|z|^3 (1 + \tilde{\kappa} \sigma^2)^{5}}} \nonumber\\
		& \hspace{0.5cm} + \frac{\frac{1}{|z|} \frac{2(1+\tilde{\kappa}\sigma^2)^2}{\tau^3\tilde{\tau}} \overbrace{\big|\big| S^{(B)}\big( G^{(A,\delta)} \big( \sum\limits_{r,s=1}^R q^{(B)}_{r,s} A_rA_s^* \big) G^{(A,\tilde{\delta})} \big) \big|\big|}^{\overset{\text{(\ref{Eq_SimpleBound_SB})}}{\leq} c_* \sigma^4 ||G^{(A,\delta)} \big( \sum\limits_{r,s=1}^R q^{(B)}_{r,s} A_rA_s^* \big) G^{(A,\tilde{\delta})}||}}{\frac{\tau^4 \Im(z)}{4|z|^3 (1 + \tilde{\kappa} \sigma^2)^{5}}} \nonumber\\
		& \overset{\substack{\text{(\ref{Eq_SimpleBounds_ASum_BSum})} \\ \text{(\ref{Eq_DualApprox_GBound})}}}{\leq} \frac{\overbrace{\frac{2\sigma^2(1+\tilde{\kappa}\sigma^2)^2}{\tau^3\tilde{\tau}} ||q^{(A)}|| + \frac{2c_*\sigma^6(1+\tilde{\kappa}\sigma^2)^2}{|z|\tau^5\tilde{\tau}^3} ||q^{(B)}||}^{\overset{\substack{c_*,\sigma^2 \geq 1 \\ \tau,\tilde{\tau} \leq 1}}{\leq} \frac{2c_*\sigma^6 (1+\tilde{\kappa}\sigma^2)^2}{\tau^5 \tilde{\tau}^3} (1+1/|z|) (||q^{(A)}||+||q^{(B)}||)}}{\frac{\tau^4 \Im(z)}{4|z|^3 (1 + \tilde{\kappa} \sigma^2)^{5}}} \nonumber\\
		& \leq \frac{8 c_* \sigma^6 |z|^2 (|z|+1) (1+\tilde{\kappa} \sigma^2)^7}{\tau^9 \tilde{\tau}^3} \frac{||q^{(A)}||+||q^{(B)}||}{\Im(z)} \nonumber\\
		& \leq \frac{16 c_* \sigma^6 |z|^2 \kappa (1+\tilde{\kappa} \sigma^2)^7}{\tau^9 \tilde{\tau}^3 \eta} \big(||q^{(A)}||+||q^{(B)}||\big) \ ,
	\end{align}
	where in the last step it was without loss of generality assumed that $\kappa > 1$ to get $|z|+1 \leq \kappa+1 \leq 2\kappa$.
	Finally, the bounds
	\begin{align*}
		& \lambda_{\min}\big(H_\delta^{-1}\big) \overset{\text{(\ref{Eq_DualApprox_HBounds})}}{\geq} \frac{1}{\sqrt{|z| \tilde{\kappa} \sigma^2}} \ \ \text{ and } \ \ \lambda_{\min}\big(H_{\tilde{\delta}}^{-1}\big) \overset{\text{(\ref{Eq_DualApprox_HBounds})}}{\geq} \frac{1}{\sqrt{|z| \tilde{\kappa} \sigma^2}}
	\end{align*}
	may then be employed to get
	\begin{align}\label{Eq_DualApprox_VDiffBound}
		& \big|\big| V^{(B,\delta)} - V^{(B,\tilde{\delta})} \big|\big| \leq |z| \tilde{\kappa} \sigma^2 \, \big|\big| H_\delta^{-1} \big( V^{(B,\delta)} - V^{(B,\tilde{\delta})} \big) H_{\tilde{\delta}}^{-1} \big|\big| \nonumber\\
		& \overset{\text{(\ref{Eq_DualApprox_HDiffBound})}}{\leq} \frac{16 c_* \sigma^8 |z|^3 \kappa \tilde{\kappa} (1+\tilde{\kappa} \sigma^2)^7}{\tau^9 \tilde{\tau}^3 \eta} \big(||q^{(A)}||+||q^{(B)}||\big) \ .
	\end{align}
	
	\item[xiii)] \textit{Bounding the difference between $(\delta^{(A)},\delta^{(B)})$ and $(\tilde{\delta}^{(A)},\tilde{\delta}^{(B)})$}:\\
	Recalling the equations (\ref{Eq_DualSystem}) and (\ref{Eq_DualApprox_DualSystem}) as well as the definitions (\ref{Eq_DualApprox_DefV}) and (\ref{Eq_DualApprox_DefG}), one has
	\begin{align}\label{Eq_DualApprox_delta_Identities}
		& \delta^{(A)}_{r,s} - \tilde{\delta}^{(A)}_{r,s} = -\frac{1}{z} \frac{1}{n} \tr\big( A_rA_s^* \big(G^{(A,\delta)} - G^{(A,\tilde{\delta})}\big) \big) - q^{(A)}_{r,s} \nonumber\\
		& \delta^{(B)}_{r,s} - \tilde{\delta}^{(B)}_{r,s} = -\frac{1}{z} \frac{1}{n} \tr\big( B_s^*B_r \big(G^{(B,\delta)} - G^{(B,\tilde{\delta})}\big) \big) - q^{(B)}_{r,s}
	\end{align}
	and may then bound
	\begin{align}\label{Eq_DualApprox_deltaA_DiffBound}
		& \big|\big| \delta^{(B)} - \tilde{\delta}^{(B)} \big|\big| \overset{\text{(\ref{Eq_DualApprox_delta_Identities})}}{\leq} \frac{1}{|z|} \Big|\Big| \frac{1}{n} \tr\big( B_s^*B_r \big(G^{(B,\delta)} - G^{(B,\tilde{\delta})}\big) \big)_{r,s \leq R} \Big|\Big| + ||q^{(B)}|| \nonumber\\
		& \overset{\text{(\ref{Eq_SimpleBounds_TraceMatrixBounds})}}{\leq} \frac{\sigma^2 n}{|z| n} \big|\big| G^{(B,\delta)} - G^{(B,\tilde{\delta})} \big|\big| + ||q^{(B)}|| \nonumber\\
		& \overset{\text{(\ref{Eq_InverseDifferenceIdentity})}}{\leq} \frac{\sigma^2}{|z|} \big|\big| G^{(B,\delta)} \big( V^{(B,\tilde{\delta})} - V^{(B,\delta)} \big) G^{(B,\tilde{\delta})} \big|\big| + ||q^{(B)}|| \nonumber\\
		& \overset{\text{(\ref{Eq_DualApprox_GBound})}}{\leq} \frac{\sigma^2}{\tau^2 \tilde{\tau}^2 |z|} \big|\big| V^{(B,\tilde{\delta})} - V^{(B,\delta)} \big|\big| + ||q^{(B)}|| \nonumber\\
		& \overset{\text{(\ref{Eq_DualApprox_VDiffBound})}}{\leq} \frac{16 c_* \sigma^{10} \overbrace{|z|^2}^{\leq \kappa^2} \kappa \tilde{\kappa} (1+\tilde{\kappa} \sigma^2)^7}{\tau^{11} \tilde{\tau}^5 \eta} \big(||q^{(A)}||+||q^{(B)}||\big) + ||q^{(B)}|| \ .
	\end{align}
	Also observe
	\begin{align}\label{Eq_DualApprox_deltaB_DiffBound}
		& \big|\big| \delta^{(A)} - \tilde{\delta}^{(A)} \big|\big| \overset{\text{(\ref{Eq_DualApprox_delta_Identities})}}{\leq} \frac{1}{|z|} \Big|\Big| \frac{1}{n} \tr\big( A_rA_s^* \big(G^{(A,\delta)} - G^{(A,\tilde{\delta})}\big) \big)_{r,s \leq R} \Big|\Big| + ||q^{(A)}|| \nonumber\\
		& \overset{\text{(\ref{Eq_SimpleBounds_TraceMatrixBounds})}}{\leq} \frac{c_* \sigma^2}{|z|} \big|\big| G^{(A,\delta)} - G^{(A,\tilde{\delta})} \big|\big| + ||q^{(A)}|| \nonumber\\
		& \overset{\text{(\ref{Eq_InverseDifferenceIdentity})}}{=} \frac{c_* \sigma^2}{|z|} \Big|\Big| G^{(A,\delta)} \big( V^{(A,\tilde{\delta})} - V^{(A,\delta)} \big) G^{(A,\tilde{\delta})} \Big|\Big| + ||q^{(A)}|| \nonumber\\
		& \overset{\text{(\ref{Eq_DualApprox_GBound})}}{\leq} \frac{c_* \sigma^2}{|z| \tau^2 \tilde{\tau}^2} \big|\big| V^{(A,\tilde{\delta})} - V^{(A,\delta)} \big|\big| + ||q^{(A)}|| \nonumber\\
		& \overset{\substack{\text{(\ref{Eq_DualApprox_DualSystem_VSG_delta})} \\ \text{(\ref{Eq_DualApprox_DualSystem_VSG_tdelta})}}}{\leq} \frac{c_* \sigma^2}{|z|^2 \tau^2 \tilde{\tau}^2} \underbrace{\big|\big| S^{(A)}\big( G^{(B,\tilde{\delta})} - G^{(B,\delta)} \big) \big|\big|}_{\overset{\substack{\text{(\ref{Eq_DualApprox_DefS})} \\ \text{(\ref{Eq_SimpleBound_SA})}}}{\leq} \sigma^4 || G^{(B,\tilde{\delta})} - G^{(B,\delta)}||} + \frac{c_* \sigma^2}{|z| \tau^2 \tilde{\tau}^2} \underbrace{\Big|\Big| \sum\limits_{r,s=1}^R q^{(B)}_{r,s} A_rA_s^* \Big|\Big|}_{\overset{\text{(\ref{Eq_SimpleBounds_ASum_BSum})}}{\leq} \sigma^2 ||q^{(B)}||} + ||q^{(A)}|| \nonumber\\
		& \leq \frac{c_* \sigma^6}{|z|^2 \tau^2 \tilde{\tau}^2} \big|\big| G^{(B,\tilde{\delta})} - G^{(B,\delta)} \big|\big| + \frac{c_* \sigma^4}{|z| \tau^2 \tilde{\tau}^2} ||q^{(B)}|| + ||q^{(A)}|| \nonumber\\
		& \overset{\text{(\ref{Eq_InverseDifferenceIdentity})}}{\leq} \frac{c_* \sigma^6}{|z|^2 \tau^2 \tilde{\tau}^2} \big|\big| G^{(B,\tilde{\delta})} \big( V^{(B,\delta)} - V^{(B,\tilde{\delta})} \big) G^{(B,\delta)} \big|\big| + \frac{c_* \sigma^4}{|z| \tau^2 \tilde{\tau}^2} ||q^{(B)}|| + ||q^{(A)}|| \nonumber\\
		& \overset{\text{(\ref{Eq_DualApprox_GBound})}}{\leq} \frac{c_* \sigma^6}{|z|^2 \tau^4 \tilde{\tau}^4} \big|\big| V^{(B,\delta)} - V^{(B,\tilde{\delta})} \big|\big| + \frac{c_* \sigma^4}{|z| \tau^2 \tilde{\tau}^2} ||q^{(B)}|| + ||q^{(A)}|| \nonumber\\
		& \overset{\text{(\ref{Eq_DualApprox_VDiffBound})}}{\leq} \frac{c_* \sigma^6}{|z|^2 \tau^4 \tilde{\tau}^4} \frac{16 c_* \sigma^8 |z|^3 \kappa \tilde{\kappa} (1+\tilde{\kappa} \sigma^2)^7}{\tau^9 \tilde{\tau}^3 \eta} \big(||q^{(A)}||+||q^{(B)}||\big) \nonumber\\
		& \hspace{0.5cm} + \frac{c_* \sigma^4}{\underbrace{|z|}_{\geq \eta} \tau^2 \tilde{\tau}^2} ||q^{(B)}|| + ||q^{(A)}|| \nonumber\\
		& \leq \frac{16 c_*^2 \sigma^{14} \overbrace{|z|}^{\leq \kappa} \kappa \tilde{\kappa} (1+\tilde{\kappa} \sigma^2)^7}{\tau^{13} \tilde{\tau}^7 \eta} \big(||q^{(A)}||+||q^{(B)}||\big) + \frac{c_* \sigma^4}{\eta \tau^2 \tilde{\tau}^2} ||q^{(B)}|| + ||q^{(A)}|| \ .
	\end{align}
	The bounds (\ref{Eq_DualApprox_deltaA_DiffBound}) and (\ref{Eq_DualApprox_deltaB_DiffBound}) together prove (\ref{Eq_DualApprox_Result_deltaA}) and (\ref{Eq_DualApprox_Result_deltaB}) for
	\begin{align}
		& \mathcal{C} \coloneq \max\Big( \frac{16 c_* \sigma^{10} \kappa^3 \tilde{\kappa} (1+\tilde{\kappa} \sigma^2)^7}{\tau^{11} \tilde{\tau}^5 \eta} + 1, \frac{16 c_*^2 \sigma^{14} \kappa^2 \tilde{\kappa} (1+\tilde{\kappa} \sigma^2)^7}{\tau^{13} \tilde{\tau}^7 \eta} + \frac{c_* \sigma^4}{\eta \tau^2 \tilde{\tau}^2} + 1 \Big) \ .
	\end{align}
	This concludes the proof of Theorem~\ref{Thm_DualApprox}. \qed
\end{itemize}

\section{The dual system of equations in expectations for the Gaussian case}\label{Section_Stein}

\begin{lemma}[Mixed Stein's Lemma]\label{Lemma_SteinMixed}\
	\\
	Let $U$ be an open subset of $\C^2$ which contains $\{(z,\ol{z}) \mid z \in \C\}$, and let $g : U \rightarrow \C$ be a (separately) holomorphic function. For any centered complex Gaussian $Z$, the equality
	\begin{align}\label{Eq_Stein_Mixed}
		\E\big[ Z g(Z,\ol{Z}) \big] & = \E[Z^2] \E\big[ \partial_1 g(Z,\ol{Z}) \big] + \E[|Z|^2]\E\big[ \partial_2 g(Z,\ol{Z}) \big]
	\end{align}
	holds whenever the expectations on both sides exist. Here, $\partial_1 g$ and $\partial_2 g$ are the holomorphic derivatives of $g$ with regard to the first and second arguments respectively.
\end{lemma}

\begin{lemma}[Basic derivatives]\label{Lemma_BasicDerivatives}\
	\\
	Let $\frac{\partial}{\partial X_{i,j}}$ denote the Wirtinger derivative with respect to the complex variable $X_{i,j}$, which ignores occurrences of $\ol{X}_{i,j}$. Likewise, $\frac{\partial}{\partial \ol{X_{i,j}}}$ ignores occurrences of $X_{i,j}$. The elementary derivatives
	\begin{align}\label{Eq_DerivCalc_BaseX}
		& \frac{\partial}{\partial X_{i,j}} X = e_{i,d} e_{j,n}^\top \ \ ; \ \ \frac{\partial}{\partial X_{i,j}} X^* = 0 \ \ ; \ \ \frac{\partial}{\partial \ol{X_{i,j}}} X = 0 \ \ ; \ \ \frac{\partial}{\partial \ol{X_{i,j}}} X^* = e_{j,n} e_{i,d}^\top
	\end{align}
	by product rule allow for the calculation of
	\begin{align}\label{Eq_DerivCalc_XDeriv_S_Result}
		& \frac{\partial}{\partial X_{i,j}} \bm{S}^{(X)} = \frac{1}{n} \sum\limits_{r=1}^R A_r e_{i,d} e_{j,n}^\top B_r (\bm{Y}^{(X)})^* \ \text{ and } \ \frac{\partial}{\partial \ol{X_{i,j}}} \bm{S}^{(X)} = \frac{1}{n} \sum\limits_{s=1}^R \bm{Y}^{(X)} B_s^* e_{j,n} e_{i,d}^\top A_s^*
	\end{align}
	as well as
	\begin{align}\label{Eq_DerivCalc_XDeriv_tS_Result}
		& \frac{\partial}{\partial X_{i,j}} \tilde{\bm{S}}^{(X)} = \frac{1}{n} \sum\limits_{r=1}^R (\bm{Y}^{(X)})^* A_r e_{i,d} e_{j,n}^\top B_r \ \text{ and } \frac{\partial}{\partial \ol{X_{i,j}}} \tilde{\bm{S}}^{(X)} = \frac{1}{n} \sum\limits_{s=1}^R B_s^* e_{j,n} e_{i,d}^\top A_s^* \bm{Y}^{(X)} \ .
	\end{align}
	One may then further calculate
	\begin{align}
		& \frac{\partial}{\partial X_{i,j}} \big[ \bm{R}^{(X)}(z) \big] = - \frac{1}{n} \sum\limits_{r=1}^R \bm{R}^{(X)}(z) A_r e_{i,d} e_{j,n}^\top B_r (\bm{Y}^{(X)})^* \bm{R}^{(X)}(z) \label{Eq_DerivCalc_XDeriv_R_Result}\\
		& \frac{\partial}{\partial \ol{X_{i,j}}} \big[ \bm{R}^{(X)}(z) \big] = - \frac{1}{n} \sum\limits_{s=1}^R \bm{R}^{(X)}(z) \bm{Y}^{(X)} B_s^* e_{j,n} e_{i,d}^\top A_s^* \bm{R}^{(X)}(z) \label{Eq_DerivCalc_olXDeriv_R_Result}\\
		& \frac{\partial}{\partial X_{i,j}} \big[ \tilde{\bm{R}}^{(X)}(z) \big] = - \frac{1}{n} \sum\limits_{r=1}^R \tilde{\bm{R}}^{(X)}(z) (\bm{Y}^{(X)})^* A_r e_{i,d} e_{j,n}^\top B_r \tilde{\bm{R}}^{(X)}(z) \label{Eq_DerivCalc_XDeriv_tR_Result}\\
		& \frac{\partial}{\partial \ol{X_{i,j}}} \big[ \tilde{\bm{R}}^{(X)}(z) \big] = - \frac{1}{n} \sum\limits_{s=1}^R \tilde{\bm{R}}^{(X)}(z) B_s^* e_{j,n} e_{i,d}^\top A_s^* \bm{Y}^{(X)} \tilde{\bm{R}}^{(X)}(z) \label{Eq_DerivCalc_olXDeriv_tR_Result} \ .
	\end{align}
	Lastly, it holds that
	\begin{align}
		& \frac{\partial}{\partial X_{i,j}} \big[ \bm{R}^{(X)}(z) \bm{Y}^{(X)} \big] = -z \sum\limits_{r=1}^R \bm{R}^{(X)}(z) A_r e_{i,d} e_{j,n}^\top B_r \tilde{\bm{R}}^{(X)}(z) \label{Eq_DerivCalc_XDeriv_RY_result}\\
		& \frac{\partial}{\partial \ol{X_{i,j}}} \big[ (\bm{Y}^{(X)})^* \bm{R}^{(X)}(z) \big] = -z \sum\limits_{s=1}^R \tilde{\bm{R}}^{(X)}(z) B_s^* e_{j,n} e_{i,d}^\top A_s^* \bm{R}^{(X)}(z) \label{Eq_DerivCalc_olXDeriv_YR_Result} \ .
	\end{align}
\end{lemma}

\begin{lemma}[Triplet sum bound]\label{Lemma_TraceTriplet}\
	\\
	For any $(d \times n)$-matrices $A,B,C$, where the entries of $C$ satisfy $|C_{i,j}| \leq 1$, the bound
	\begin{align}\label{Eq_TraceTripletBound}
		& \bigg| \sum\limits_{i=1}^d \sum\limits_{j=1}^n A_{i,j} B_{i,j} C_{i,j} \bigg| \leq (d \land n) \, ||A|| \, ||B||
	\end{align}
	holds.
\end{lemma}

\begin{lemma}[Spectral bound for $\bm{R}(z)$ and $\bm{Q}(z)^{-1}$]\label{Lemma_QSpectralBound}\
	\\
	For any $X \in \C^{d \times n}$ and $z \in \C^+$, the bounds
	\begin{align}\label{Eq_R_SpectralBounds}
		& ||\bm{R}^{(X)}(z)|| \leq \frac{1}{\Im(z)} \ \ \text{ and } \ \ ||\tilde{\bm{R}}^{(X)}(z)|| \leq \frac{1}{\Im(z)}
	\end{align}
	hold. Additionally, if~\ref{ItemTempAssumption_NonDegeneracy} is true, then one also has
	\begin{align}\label{Eq_QBounds}
		& ||\bm{Q}^{(X)}(z)^{-1}|| \leq \frac{(|z|+||\bm{S}^{(X)}||)^2}{\tau^2 \Im(z)} \ \ \text{ and } \ \ ||\tilde{\bm{Q}}^{(X)}(z)^{-1}|| \leq \frac{(|z|+||\bm{S}^{(X)}||)^2}{\tau^2 \Im(z)} \ .
	\end{align}
\end{lemma}

\begin{lemma}[Tail bound for $||\bm{Y}||$ in the Gaussian case]\label{Lemma_ZTailBound}\
	\\
	Suppose~\ref{ItemAssumption_sigmaBound} holds and let $\bm{Z}$ denote a random $(d \times n)$-matrix whose entries are independent centered complex-valued Gaussian with $\E[|\bm{Z}_{j,k}|^2]=1$. Note that the entries need not be identically distributed.
	The tail bounds
	\begin{align}\label{Eq_ZTailBound}
		& \forall t \geq 0 : \ \bP\big( ||\bm{Z}|| > 2 \big( \sqrt{d} + \sqrt{n} \big) + t \big) \leq 2\exp\big( -C t^2 \big)
	\end{align}
	and
	\begin{align}\label{Eq_YTailBound}
		& \forall t \geq 0 : \ \bP\big( ||\bm{Y}^{(\bm{Z})}|| > 2 \sigma^2 \big( \sqrt{d} + \sqrt{n} \big) + t \big) \leq 2\exp\Big( -\frac{C t^2}{\sigma^4} \Big)
	\end{align}
	hold for some universal constant $C>0$.
\end{lemma}

\begin{lemma}[Moment bound for $||\bm{S}||$ in the Gaussian case]\label{Lemma_SMomentBound}\
	\\
	Suppose~\ref{ItemAssumption_cBound} and~\ref{ItemAssumption_sigmaBound} hold and let $\bm{Z}$ be as in Lemma~\ref{Lemma_ZTailBound}. The moment bounds
	\begin{align}\label{Eq_Z_MomentBound}
		& \E\big[ ||\bm{Z}||^{2m} \big] \leq 2^{4m} (\sqrt{d}+\sqrt{n})^{2m} K_m
	\end{align}
	and
	\begin{align}\label{Eq_SZ_MomentBound}
		& \E\big[ ||\bm{S}^{(\bm{Z})}||^m \big] \leq \Big(\frac{16 \sigma^4 (\sqrt{d} + \sqrt{n})^2}{n}\Big)^{m} K_m \overset{\text{\ref{ItemAssumption_cBound}}}{\leq} \big(32 \sigma^4 (1+c_*) \big)^{m} K_m
	\end{align}
	hold for all $m \in \N$, where $K_m = 1 + m \sqrt{\frac{\pi}{C}} + \frac{m!}{C^m}$ and $C>0$ is the universal constant from Lemma~\ref{Lemma_ZTailBound}.
\end{lemma}

\begin{proposition}[MP-equation analogue for expectations]\label{Prop_Stein_MP_Analogue}\
	\\
	Suppose~\ref{ItemAssumption_cBound},~\ref{ItemAssumption_sigmaBound} and~\ref{ItemTempAssumption_NonDegeneracy} hold and let $\bm{Z}$ be as in Lemma~\ref{Lemma_ZTailBound}.
	For any $\eta,\kappa>0$ there exists a constant $C=C(c_*,\sigma^2,\tau,\eta,\kappa)>0$ such that the bounds
	\begin{align}\label{Eq_MPSteinProp_A}
		& \forall z \in \bD(\eta,\kappa), \, \forall M \in \C^{d \times d} : \nonumber\\
		& \Big| \E\Big[ \frac{1}{n} \tr\big( M \bm{R}^{(\bm{Z})}(z) \big)\Big] - \frac{-1}{z} \E\Big[ \frac{1}{n} \tr\big( M \bm{Q}^{(\bm{Z})}(z)^{-1} \big) \Big] \Big| \leq C \frac{||M||}{n}
	\end{align}
	and
	\begin{align}\label{Eq_MPSteinProp_B}
		& \forall z \in \bD(\eta,\kappa), \, \forall \tilde{M} \in \C^{n \times n} : \nonumber\\
		& \Big| \E\Big[ \frac{1}{n} \tr\big( \tilde{M} \tilde{\bm{R}}^{(\bm{Z})}(z) \big) \Big] - \frac{-1}{z} \E\Big[ \frac{1}{n} \tr\big( \tilde{M} \tilde{\bm{Q}}^{(\bm{Z})}(z)^{-1} \big) \Big] \Big| \leq C \frac{||\tilde{M}||}{n}
	\end{align}
	hold.
\end{proposition}

\subsection{Proof of Proposition \ref{Prop_Stein_MP_Analogue}}\label{Proof_Prop_Stein_MP_Analogue}
Without loss of generality assume
\begin{align}\label{Eq_MPDirectCalc_Wlog_vareps_kappa}
	& \eta < 1 \ \ \text{ and } \ \ \kappa > 1 \ .
\end{align}
By the simple identity $\Id_d = (\bm{S} - z\Id_d) \bm{R}(z)$ one has
\begin{align}\label{Eq_MPDirectCalc_1}
	& \E\big[ \tr\big( M \bm{Q}(z)^{-1} \big) \big] = \E\big[ \tr\big( M \bm{Q}(z)^{-1} \big( \bm{S} - z\Id_d \big) \bm{R}(z) \big) \big] \nonumber\\
	& = \E\big[ \tr\big( M \bm{Q}(z)^{-1} \bm{S} \bm{R}(z) \big) \big] - z\E\big[ \tr\big( M \bm{Q}(z)^{-1} \bm{R}(z) \big) \big] \ .
\end{align}
The first summand on the right-hand side of (\ref{Eq_MPDirectCalc_1}) will now be calculated with Stein's Lemma. Plugging in the definition of $\bm{S}$ by the cyclic property of the trace yields
\begin{align}\label{Eq_MPDirectCalc_2}
	& \E\big[ \tr\big( M \bm{Q}(z)^{-1} \bm{S} \bm{R}(z) \big) \big] \nonumber\\
	& = \frac{1}{n} \sum\limits_{r',s'=1}^R \E\big[ \tr\big( M \bm{Q}(z)^{-1} A_{r'} \bm{Z} B_{r'} B_{s'}^* \bm{Z}^* A_{s'}^* \bm{R}(z) \big) \big] \nonumber\\
	& = \frac{1}{n} \sum\limits_{r',s'=1}^R \E\big[ \tr\big( \bm{Z} B_{r'} B_{s'}^* \bm{Z}^* A_{s'}^* \bm{R}(z) M \bm{Q}(z)^{-1} A_{r'} \big) \big] \nonumber\\
	& = \frac{1}{n} \sum\limits_{r',s'=1}^R \sum\limits_{i=1}^d \sum\limits_{j=1}^n \E\Big[  \bm{Z}_{i,j} e_{j,n}^\top B_{r'} B_{s'}^* \bm{Z}^* A_{s'}^* \bm{R}(z) M \bm{Q}(z)^{-1} A_{r'} e_{i,d} \Big] \ ,
\end{align}
where $e_{i,d}$ is the $d$-dimensional unit vector with a one at position $i$ and $e_{j,n}$ is defined analogously.
With the notation
\begin{align*}
	& g_{i,j}(\bm{Z}_{i,j},\ol{\bm{Z}_{i,j}}) = G_{i,j}(\bm{Z}_{i,j}) \coloneq \E\Big[ e_{j,n}^\top B_{r'} B_{s'}^* \bm{Z}^* A_{s'}^* \bm{R}(z) M \bm{Q}(z)^{-1} A_{r'} e_{i,d} \ \Big| \ \bm{Z}_{i,j} \Big] \ ,
\end{align*}
the final expectation in (\ref{Eq_MPDirectCalc_2}) is in the correct shape to apply the complex version of Stein's lemma as presented in Lemma~\ref{Lemma_SteinMixed}. One by product rule gets
\begin{align}\label{Eq_MPDirectCalc_3}
	& \E\big[ \tr\big( M \bm{Q}(z)^{-1} \bm{S} \bm{R}(z) \big) \big] \nonumber\\
	& \overset{\text{(\ref{Eq_Stein_Mixed})}}{=} \frac{1}{n} \sum\limits_{r',s'=1}^R \sum\limits_{i=1}^d \sum\limits_{j=1}^n \E[\bm{Z}_{i,j}^2] \E\Big[ \frac{\partial}{\partial \bm{Z}_{i,j}} \big[ e_{j,n}^\top B_{r'} B_{s'}^* \bm{Z}^* A_{s'}^* \bm{R}(z) M \bm{Q}(z)^{-1} A_{r'} e_{i,d} \big] \Big] \nonumber\\
	& \hspace{0.5cm} + \frac{1}{n} \sum\limits_{r',s'=1}^R \sum\limits_{i=1}^d \sum\limits_{j=1}^n \underbrace{\E[|\bm{Z}_{i,j}|^2]}_{=1} \E\Big[ \frac{\partial}{\partial \ol{\bm{Z}_{i,j}}} \big[ e_{j,n}^\top B_{r'} B_{s'}^* \bm{Z}^* A_{s'}^* \bm{R}(z) M \bm{Q}(z)^{-1} A_{r'} e_{i,d} \big] \Big] \nonumber\\
	& = \frac{1}{n} \sum\limits_{r',s'=1}^R \sum\limits_{i=1}^d \sum\limits_{j=1}^n \E[\bm{Z}_{i,j}^2] \E\Big[ e_{j,n}^\top B_{r'} B_{s'}^* \overbrace{\frac{\partial}{\partial \bm{Z}_{i,j}} \big[ \bm{Z}^* \big]}^{\overset{\text{(\ref{Eq_DerivCalc_BaseX})}}{=} 0} A_{s'}^* \bm{R}(z) M \bm{Q}(z)^{-1} A_{r'} e_{i,d} \Big] \nonumber\\
	& \hspace{0.5cm} + \frac{1}{n} \sum\limits_{r',s'=1}^R \sum\limits_{i=1}^d \sum\limits_{j=1}^n \E[\bm{Z}_{i,j}^2] \E\Big[ e_{j,n}^\top B_{r'} B_{s'}^* \bm{Z}^* A_{s'}^* \underbrace{\frac{\partial}{\partial \bm{Z}_{i,j}} \big[ \bm{R}(z) \big]}_{\text{(\ref{Eq_DerivCalc_XDeriv_R_Result})}} M \bm{Q}(z)^{-1} A_{r'} e_{i,d} \Big] \nonumber\\
	& \hspace{0.5cm} + \frac{1}{n} \sum\limits_{r',s'=1}^R \sum\limits_{i=1}^d \sum\limits_{j=1}^n \E[\bm{Z}_{i,j}^2] \E\Big[ e_{j,n}^\top B_{r'} B_{s'}^* \bm{Z}^* A_{s'}^* \bm{R}(z) M \frac{\partial}{\partial \bm{Z}_{i,j}} \big[ \bm{Q}(z)^{-1} \big] A_{r'} e_{i,d} \Big] \nonumber\\
	& \hspace{0.5cm} + \frac{1}{n} \sum\limits_{r',s'=1}^R \sum\limits_{i=1}^d \sum\limits_{j=1}^n \E\Big[ e_{j,n}^\top B_{r'} B_{s'}^* \overbrace{\frac{\partial}{\partial \ol{\bm{Z}_{i,j}}} \big[ \bm{Z}^* \big]}^{\overset{\text{(\ref{Eq_DerivCalc_BaseX})}}{=} e_{j,n} e_{i,d}^\top} A_{s'}^* \bm{R}(z) M \bm{Q}(z)^{-1} A_{r'} e_{i,d} \Big] \nonumber\\
	& \hspace{0.5cm} + \frac{1}{n} \sum\limits_{r',s'=1}^R \sum\limits_{i=1}^d \sum\limits_{j=1}^n \E\Big[ e_{j,n}^\top B_{r'} B_{s'}^* \bm{Z}^* A_{s'}^* \underbrace{\frac{\partial}{\partial \ol{\bm{Z}_{i,j}}} \big[ \bm{R}(z) \big]}_{\text{(\ref{Eq_DerivCalc_olXDeriv_R_Result})}} M \bm{Q}(z)^{-1} A_{r'} e_{i,d} \Big] \nonumber\\
	& \hspace{0.5cm} + \frac{1}{n} \sum\limits_{r',s'=1}^R \sum\limits_{i=1}^d \sum\limits_{j=1}^n \E\Big[ e_{j,n}^\top B_{r'} B_{s'}^* \bm{Z}^* A_{s'}^* \bm{R}(z) M \frac{\partial}{\partial \ol{\bm{Z}_{i,j}}} \big[ \bm{Q}(z)^{-1} \big] A_{r'} e_{i,d} \Big] \ ,
\end{align}
where $\frac{\partial}{\partial \bm{Z}_{i,j}}$ and $\frac{\partial}{\partial \ol{\bm{Z}_{i,j}}}$ denote the Wirtinger derivatives, which ignore occurrences of $\ol{\bm{Z}_{i,j}}$and $\bm{Z}_{i,j}$ respectively.
The derivatives of $\bm{Z}$, $\bm{Z}^*$ and $\bm{R}(z)$ are calculated in Lemma~\ref{Lemma_BasicDerivatives}. For the derivatives of $\bm{Q}(z)$, one observes
\begin{align*}
	& 0 = \frac{\partial}{\partial \bm{Z}_{i,j}} \big[ \bm{Q}(z) \bm{Q}(z)^{-1} \big] = \frac{\partial}{\partial \bm{Z}_{i,j}} \big[ \bm{Q}(z) \big] \bm{Q}(z)^{-1} + \bm{Q}(z) \frac{\partial}{\partial \bm{Z}_{i,j}} \big[ \bm{Q}(z)^{-1} \big] \ ,
\end{align*}
which, when multiplied from the left with $\bm{Q}(z)^{-1}$ and rearranged, gives
\begin{align}\label{Eq_ZDerivQ}
	& \frac{\partial}{\partial \bm{Z}_{i,j}} \big[ \bm{Q}(z)^{-1} \big] = - \bm{Q}(z)^{-1} \frac{\partial}{\partial \bm{Z}_{i,j}} \big[ \bm{Q}(z) \big] \bm{Q}(z)^{-1} \nonumber\\
	& \overset{\text{(\ref{Eq_Def_Q})}}{=} - \sum\limits_{r,s=1}^R \frac{1}{n}\tr\Big( B_{s}^* B_{r} \frac{\partial}{\partial \bm{Z}_{i,j}} \big[ \tilde{\bm{R}}(z) \big] \Big) \bm{Q}(z)^{-1} A_{r} A_{s}^* \bm{Q}(z)^{-1} \nonumber\\
	& \overset{\text{(\ref{Eq_DerivCalc_XDeriv_tR_Result})}}{=} \frac{1}{n^2} \sum\limits_{r,\ul{r},s=1}^R \tr\Big( B_{s}^* B_{r} \tilde{\bm{R}}(z) (\bm{Y})^* A_{\ul{r}} e_{i,d} e_{j,n}^\top B_{\ul{r}} \tilde{\bm{R}}(z) \Big) \bm{Q}(z)^{-1} A_{r} A_{s}^* \bm{Q}(z)^{-1} \nonumber\\
	& = \frac{1}{n^2} \sum\limits_{r,\ul{r},s=1}^R \big( e_{j,n}^\top B_{\ul{r}} \tilde{\bm{R}}(z) B_{s}^* B_{r} \tilde{\bm{R}}(z) (\bm{Y})^* A_{\ul{r}} e_{i,d} \big) \bm{Q}(z)^{-1} A_{r} A_{s}^* \bm{Q}(z)^{-1} \ .
\end{align}
Analogously, one may show
\begin{align}\label{Eq_tZDerivQ}
	& \frac{\partial}{\partial \ol{\bm{Z}_{i,j}}} \big[ \bm{Q}(z)^{-1} \big] 
	= \frac{1}{n^2} \sum\limits_{\ul{s},r,s=1}^R \big( e_{i,d}^\top A_{\ul{s}}^* \bm{Y} \tilde{\bm{R}}(z) B_{s}^* B_{r} \tilde{\bm{R}}(z) B_{\ul{s}}^* e_{j,n} \big) \bm{Q}(z)^{-1} A_{r} A_{s}^* \bm{Q}(z)^{-1} \ .
\end{align}
Inserting these derivatives back into (\ref{Eq_MPDirectCalc_3}) yields
\begin{align}
	& \E\big[ \tr\big( M \bm{Q}(z)^{-1} \bm{S} \bm{R}(z) \big) \big] \nonumber\\
	& = - \frac{1}{n^2} \sum\limits_{\ul{r},r',s'=1}^R \sum\limits_{i=1}^d \sum\limits_{j=1}^n \E[\bm{Z}_{i,j}^2] \E\Big[ e_{j,n}^\top B_{r'} B_{s'}^* \bm{Z}^* A_{s'}^* \bm{R}(z) A_{\ul{r}} e_{i,d} \nonumber\\
	& \hspace{5cm} \times e_{j,n}^\top B_{\ul{r}} \bm{Y}^* \bm{R}(z) M \bm{Q}(z)^{-1} A_{r'} e_{i,d} \Big] \nonumber\\
	& \hspace{0.5cm} + \frac{1}{n^3} \sum\limits_{\ul{r},r',r,s',s=1}^R \sum\limits_{i=1}^d \sum\limits_{j=1}^n \E[\bm{Z}_{i,j}^2] \E\Big[ \big( e_{j,n}^\top B_{\ul{r}} \tilde{\bm{R}}(z) B_{s}^* B_{r} \tilde{\bm{R}}(z) \bm{Y}^* A_{\ul{r}} e_{i,d} \big) \nonumber\\
	& \hspace{4cm} \times e_{j,n}^\top B_{r'} B_{s'}^* \bm{Z}^* A_{s'}^* \bm{R}(z) M \bm{Q}(z)^{-1} A_{r} A_{s}^* \bm{Q}(z)^{-1} A_{r'} e_{i,d} \Big] \nonumber\\
	& \hspace{0.5cm} + \frac{1}{n} \sum\limits_{r',s'=1}^R \sum\limits_{i=1}^d \sum\limits_{j=1}^n \E\Big[ e_{j,n}^\top B_{r'} B_{s'}^*  e_{j,n} e_{i,d}^\top A_{s'}^* \bm{R}(z) M \bm{Q}(z)^{-1} A_{r'} e_{i,d} \Big] \nonumber\\
	& \hspace{0.5cm} - \frac{1}{n^2} \sum\limits_{r',\ul{s},s'=1}^R \sum\limits_{i=1}^d \sum\limits_{j=1}^n \E\Big[ e_{j,n}^\top B_{r'} B_{s'}^* \bm{Z}^* A_{s'}^* \bm{R}(z) \bm{Y} B_{\ul{s}}^* e_{j,n} e_{i,d}^\top A_{\ul{s}}^* \bm{R}(z) M \bm{Q}(z)^{-1} A_{r'} e_{i,d} \Big] \nonumber\\
	& \hspace{0.5cm} + \frac{1}{n^3} \sum\limits_{r',r,\ul{s},s',s=1}^R \sum\limits_{i=1}^d \sum\limits_{j=1}^n \E\Big[ \big( e_{i,d}^\top A_{\ul{s}}^* \bm{Y} \tilde{\bm{R}}(z) B_{s}^* B_{r} \tilde{\bm{R}}(z) B_{\ul{s}}^* e_{j,n} \big) \nonumber\\
	& \hspace{4cm} \times e_{j,n}^\top B_{r'} B_{s'}^* \bm{Z}^* A_{s'}^* \bm{R}(z) M \bm{Q}(z)^{-1} A_{r} A_{s}^* \bm{Q}(z)^{-1} A_{r'} e_{i,d} \Big] \nonumber\\
	& \overset{\text{(\ref{Eq_DefY})}}{=} - \frac{1}{n^2} \sum\limits_{\ul{r},r'=1}^R \sum\limits_{i=1}^d \sum\limits_{j=1}^n \E[\bm{Z}_{i,j}^2] \E\Big[ \big( B_{r'} \bm{Y}^* \bm{R}(z) A_{\ul{r}} \big)^\top_{i,j} \big( B_{\ul{r}} \bm{Y}^* \bm{R}(z) M \bm{Q}(z)^{-1} A_{r'} \big)^\top_{i,j} \Big] \label{Eq_MPDirectCalc_4_1}\\
	& \hspace{0.5cm} + \frac{1}{n^3} \sum\limits_{\ul{r},r',r,s=1}^R \sum\limits_{i=1}^d \sum\limits_{j=1}^n \E[\bm{Z}_{i,j}^2] \E\Big[ \big( B_{\ul{r}} \tilde{\bm{R}}(z) B_{s}^* B_{r} \tilde{\bm{R}}(z) \bm{Y}^* A_{\ul{r}} \big)^\top_{i,j} \nonumber\\
	& \hspace{4cm} \times \big( B_{r'} \bm{Y}^* \bm{R}(z) M \bm{Q}(z)^{-1} A_{r} A_{s}^* \bm{Q}(z)^{-1} A_{r'} \big)^\top_{i,j} \Big]  \label{Eq_MPDirectCalc_4_2}\\
	& \hspace{0.5cm} + \frac{1}{n} \sum\limits_{r',s'=1}^R \tr\big( B_{r'} B_{s'}^* \big) \E\Big[ \tr\big( A_{s'}^* \bm{R}(z) M \bm{Q}(z)^{-1} A_{r'} \big) \Big]  \label{Eq_MPDirectCalc_4_3}\\
	& \hspace{0.5cm} - \frac{1}{n^2} \sum\limits_{r',\ul{s}=1}^R \E\Big[ \tr\big( B_{r'} \bm{Y}^* \bm{R}(z) \bm{Y} B_{\ul{s}}^* \big)  \tr\big( A_{\ul{s}}^* \bm{R}(z) M \bm{Q}(z)^{-1} A_{r'} \big) \Big]  \label{Eq_MPDirectCalc_4_4}\\
	& \hspace{0.5cm} + \frac{1}{n^3} \sum\limits_{r',r,\ul{s},s=1}^R \sum\limits_{i=1}^d \sum\limits_{j=1}^n \E\Big[ \big( A_{\ul{s}}^* \bm{Y} \tilde{\bm{R}}(z) B_{s}^* B_{r} \tilde{\bm{R}}(z) B_{\ul{s}}^* \big)_{i,j} \nonumber\\
	& \hspace{4cm} \times \big( B_{r'} \bm{Y}^* \bm{R}(z) M \bm{Q}(z)^{-1} A_{r} A_{s}^* \bm{Q}(z)^{-1} A_{r'} \big)^\top_{i,j} \Big]  \label{Eq_MPDirectCalc_4_5} \ .
\end{align}
The two summands (\ref{Eq_MPDirectCalc_4_3}) and (\ref{Eq_MPDirectCalc_4_4}) are now of a very similar form. The two identities
\begin{align*}
	& \frac{1}{n} \bm{Y}^* \bm{R}(z) \bm{Y} = \tilde{\bm{S}} \tilde{\bm{R}}(z) \ \ \text{ and } \ \ \tilde{\bm{S}} \tilde{\bm{R}}(z) - \Id_n = z\tilde{\bm{R}}(z)
\end{align*}
may be used to merge them into
\begin{align*}
	& \text{(\ref{Eq_MPDirectCalc_4_3}) $+$ (\ref{Eq_MPDirectCalc_4_4})}\\
	& = - \frac{z}{n} \sum\limits_{r',s'=1}^R \E\big[ \tr\big( B_{s'}^* B_{r'} \tilde{\bm{R}}(z) \big) \tr\big( M \bm{Q}(z)^{-1} A_{r'} A_{s'}^* \bm{R}(z) \big) \big]\\
	& = - z \E\Big[ \tr\big( M \bm{Q}(z)^{-1} \Big(\underbrace{\sum\limits_{r',s'=1}^R \frac{1}{n} \tr\big( B_{s'}^* B_{r'} \tilde{\bm{R}}(z) \big) A_{r'} A_{s'}^*}_{= \bm{Q}(z) - \Id_d}\Big) \bm{R}(z) \big) \Big]\\
	& = z \E\big[ \tr\big( M \bm{Q}(z)^{-1} \bm{R}(z) \big) \big] - z\E\big[\tr\big( M \bm{R}(z) \big)\big] \ .
\end{align*}
This for the decomposition (\ref{Eq_MPDirectCalc_4_1})-(\ref{Eq_MPDirectCalc_4_5}) implies
\begin{align}\label{Eq_MPDirectCalc_5}
	& \E\big[ \tr\big( M \bm{Q}(z)^{-1} \bm{S} \bm{R}(z) \big) \big] \nonumber\\
	& = z \E\big[ \tr\big( M \bm{Q}(z)^{-1} \bm{R}(z) \big) \big] - z\E\big[\tr\big( M \bm{R}(z) \big)\big] \nonumber\\
	& \hspace{0.5cm} + \text{(\ref{Eq_MPDirectCalc_4_1}) $+$ (\ref{Eq_MPDirectCalc_4_2}) $+$ (\ref{Eq_MPDirectCalc_4_5})} \ .
\end{align}
Combining (\ref{Eq_MPDirectCalc_5}) with (\ref{Eq_MPDirectCalc_1}), it is shown that
\begin{align*}
	& \E\big[ \tr\big( M \bm{Q}(z)^{-1} \big) \big] = - z\E\big[\tr\big( M \bm{R}(z) \big)\big] + \text{(\ref{Eq_MPDirectCalc_4_1}) $+$ (\ref{Eq_MPDirectCalc_4_2}) $+$ (\ref{Eq_MPDirectCalc_4_5})} \ ,
\end{align*}
which may be rearranged and multiplied with $\frac{1}{n}$ for
\begin{align}\label{Eq_SteinDecommp_WithoutBounds}
	& \E\Big[ \frac{1}{n} \tr\big( M \bm{R}(z) \big)\Big] = - \frac{1}{z} \E\Big[ \frac{1}{n} \tr\big( M \bm{Q}(z)^{-1} \big) \Big] + \frac{1}{z n} \big(\text{(\ref{Eq_MPDirectCalc_4_1}) $+$ (\ref{Eq_MPDirectCalc_4_2}) $+$ (\ref{Eq_MPDirectCalc_4_5})}\big) \ .
\end{align}
It for (\ref{Eq_MPSteinProp_A}) remains to bound the absolute values of (\ref{Eq_MPDirectCalc_4_1}), (\ref{Eq_MPDirectCalc_4_2}) and (\ref{Eq_MPDirectCalc_4_5}) by using Lemma~\ref{Lemma_TraceTriplet} and the fact that $|\E[\bm{Z}_{i,j}^2]| \leq \E[|\bm{Z}_{i,j}|^2]=1$. With $C_{i,j}=\E[\bm{Z}_{i,j}^2]$ in Lemma~\ref{Lemma_TraceTriplet} one gets
\begin{align*}
	& \frac{|\text{(\ref{Eq_MPDirectCalc_4_1})}|}{|z| n} \leq \frac{|\text{(\ref{Eq_MPDirectCalc_4_1})}|}{\Im(z) n}\\
	& \overset{\text{(\ref{Eq_TraceTripletBound})}}{\leq} \frac{1}{\Im(z) n^3} \sum\limits_{\ul{r},r'=1}^R \E\Big[ \min(d,n) || B_{r'} \bm{Y}^* \bm{R}(z) A_{\ul{r}} || \, || B_{\ul{r}} \bm{Y}^* \bm{R}(z) M \bm{Q}(z)^{-1} A_{r'} || \Big]\\
	& \overset{\substack{\text{(\ref{Eq_SubmultiplicativitySpectralNorm})} \\ \text{(\ref{Eq_sigmaAssumption_NonAsymp})}}}{\leq} \frac{\sigma^4 ||M||}{\Im(z) n^2} \E\big[ \underbrace{||\bm{Y}||^2}_{\overset{\text{(\ref{Eq_DefS})}}{=} n||\bm{S}||} \,  ||\bm{R}(z)||^2 \, ||\bm{Q}(z)^{-1}|| \big]\\
	& \overset{\substack{\text{(\ref{Eq_R_SpectralBounds})} \\ \text{(\ref{Eq_QBounds})}}}{\leq} \frac{\sigma^4 ||M||}{\tau^2 \Im(z)^4 n} \E\big[ ||\bm{S}|| \, \underbrace{\big((1 \lor |z|)+||\bm{S}||\big)^2}_{\leq 2(1 \lor |z|^2) + 2||\bm{S}||^2} \big]\\
	& \leq \frac{2 \sigma^4 ||M||}{\tau^2 \underbrace{\Im(z)^4}_{\geq \eta^4} n} \Big( \underbrace{(1 \lor |z|^2)}_{\overset{\text{(\ref{Eq_MPDirectCalc_Wlog_vareps_kappa})}}{\leq} \kappa^2} \E\big[ ||\bm{S}|| \big] + \E\big[ ||\bm{S}||^3 \big] \Big)\\
	& \overset{\text{(\ref{Eq_SZ_MomentBound})}}{\leq} \frac{2 \sigma^4 ||M||}{\tau^2 \eta^4 n} \Big( \kappa^2 \big(32 \sigma^4 (1+c_*) \big) K_1 + \big(32 \sigma^4 (1+c_*) \big)^{3} K_3 \Big)
\end{align*}
and
\begin{align*}
	& \frac{|\text{(\ref{Eq_MPDirectCalc_4_2})}|}{|z| n} \leq \frac{|\text{(\ref{Eq_MPDirectCalc_4_2})}|}{\Im(z) n}\\
	& \overset{\text{(\ref{Eq_TraceTripletBound})}}{\leq} \frac{1}{\Im(z) n^4} \sum\limits_{\ul{r},r',r,s=1}^R \E\Big[ \min(d,n) || B_{\ul{r}} \tilde{\bm{R}}(z) B_{s}^* B_{r} \tilde{\bm{R}}(z) \bm{Y}^* A_{\ul{r}} || \nonumber\\
	& \hspace{4cm} \times || B_{r'} \bm{Y}^* \bm{R}(z) M \bm{Q}(z)^{-1} A_{r} A_{s}^* \bm{Q}(z)^{-1} A_{r'} || \Big]\\
	& \overset{\substack{\text{(\ref{Eq_SubmultiplicativitySpectralNorm})} \\ \text{(\ref{Eq_sigmaAssumption_NonAsymp})}}}{\leq} \frac{\sigma^8 ||M||}{\Im(z) n^3} \E\big[ ||\bm{Y}||^2 \, ||\tilde{\bm{R}}(z)||^2 \, ||\bm{R}(z)|| \, ||\bm{Q}(z)^{-1}||^2 \big]\\
	& \overset{\substack{\text{(\ref{Eq_R_SpectralBounds})} \\ \text{(\ref{Eq_QBounds})}}}{\leq} \frac{\sigma^8 ||M||}{\tau^4 \Im(z)^6 n^2} \E\big[ ||\bm{S}|| \, \underbrace{\big((1 \lor |z|)+||\bm{S}||\big)^4}_{\leq 8(1 \lor |z|^4) + 8||\bm{S}||^4} \big]\\
	& \leq \frac{8 \sigma^8 ||M||}{\tau^4 \underbrace{\Im(z)^6}_{\geq \eta^6} n^2} \Big( \underbrace{(1 \lor |z|^4)}_{\overset{\text{(\ref{Eq_MPDirectCalc_Wlog_vareps_kappa})}}{\leq} \kappa^4} \E\big[ ||\bm{S}|| \big] + \E\big[ ||\bm{S}||^5 \big] \Big)\\
	& \overset{\text{(\ref{Eq_SZ_MomentBound})}}{\leq} \frac{8 \sigma^8 ||M||}{\tau^4 \eta^6 n^2} \Big( \kappa^4 \big(32 \sigma^4 (1+c_*) \big) K_1 + \big(32 \sigma^4 (1+c_*) \big)^{5} K_5 \Big) \ .
\end{align*}
With $C_{i,j} = 1$ in Lemma~\ref{Lemma_TraceTriplet}, one also sees
\begin{align*}
	& \frac{|\text{(\ref{Eq_MPDirectCalc_4_5})}|}{|z| n} \leq \frac{|\text{(\ref{Eq_MPDirectCalc_4_5})}|}{\Im(z) n}\\
	& \overset{\text{(\ref{Eq_TraceTripletBound})}}{\leq} \frac{1}{\Im(z) n^4} \sum\limits_{r',r,\ul{s},s=1}^R \E\Big[ \min(d,n) || A_{\ul{s}}^* \bm{Y} \tilde{\bm{R}}(z) B_{s}^* B_{r} \tilde{\bm{R}}(z) B_{\ul{s}}^* || \nonumber\\
	& \hspace{4cm} \times || B_{r'} \bm{Y}^* \bm{R}(z) M \bm{Q}(z)^{-1} A_{r} A_{s}^* \bm{Q}(z)^{-1} A_{r'} || \Big]\\
	& \overset{\substack{\text{(\ref{Eq_SubmultiplicativitySpectralNorm})} \\ \text{(\ref{Eq_sigmaAssumption_NonAsymp})}}}{\leq} \frac{\sigma^8 ||M||}{\Im(z) n^3} \E\big[ ||\bm{Y}||^2 \, ||\tilde{\bm{R}}(z)||^2 \, ||\bm{R}(z)|| \, ||\bm{Q}(z)^{-1}||^2 \big]\\
	& \overset{\substack{\text{(\ref{Eq_R_SpectralBounds})} \\ \text{(\ref{Eq_QBounds})}}}{\leq} \frac{\sigma^8 ||M||}{\tau^4 \Im(z)^6 n^2} \E\big[ ||\bm{S}|| \, \underbrace{\big((1 \lor |z|)+||\bm{S}||\big)^4}_{\leq 8(1 \lor |z|^4) + 8||\bm{S}||^4} \big]\\
	& \leq \frac{8 \sigma^8 ||M||}{\tau^4 \underbrace{\Im(z)^6}_{\geq \eta^6} n^2} \Big( \underbrace{(1 \lor |z|^4)}_{\overset{\text{(\ref{Eq_MPDirectCalc_Wlog_vareps_kappa})}}{\leq} \kappa^4} \E\big[ ||\bm{S}|| \big] + \E\big[ ||\bm{S}||^5 \big] \Big)\\
	& \overset{\text{(\ref{Eq_SZ_MomentBound})}}{\leq} \frac{8 \sigma^8 ||M||}{\tau^4 \eta^6 n^2} \Big( \kappa^4 \big(32 \sigma^4 (1+c_*) \big) K_1 + \big(32 \sigma^4 (1+c_*) \big)^{5} K_5 \Big) \ .
\end{align*}
The last three bounds with (\ref{Eq_SteinDecommp_WithoutBounds}) thus prove (\ref{Eq_MPSteinProp_A}).The bound (\ref{Eq_MPSteinProp_B}) may be proved by analogous arguments. This concludes the proof of Proposition~\ref{Prop_Stein_MP_Analogue}. \qed

\section{Proof of Theorem~\ref{Thm_EmpiricalDualSystem}}\label{Section_GaussianCase}
Before starting the proof, two lemmas are introduced.

\begin{lemma}[Concentration of $\bm{R}(z)$]\label{Lemma_R_Concentration}\
	\\
	Suppose $\ref{ItemAssumption_cBound}$ holds.
	For any random $(d \times n)$-matrix $\bm{X}$ with independent entries, the bounds
	\begin{align}\label{Eq_R_Concentration}
		& \forall M \in \C^{d \times d} , \, \forall \delta \in (0,\kappa) , \, \forall t>0 : \nonumber\\
		& \bP\Big( \sup\limits_{z \in \bD(\eta,\kappa)}\Big| \frac{1}{n} \tr(M\bm{R}^{(\bm{X})}(z)) - \E\Big[ \frac{1}{n} \tr(M\bm{R}^{(\bm{X})}(z)) \Big] \Big| > t + 2\frac{c_* \delta}{\eta^2} ||M|| \Big) \nonumber\\
		& \hspace{1cm} \leq \frac{16\kappa^2}{\delta^2} \exp\Big( -\frac{t^2\eta^2 n}{16 R^2 ||M||^2} \Big)
	\end{align}
	and
	\begin{align}\label{Eq_tR_Concentration}
		& \forall \tilde{M} \in \C^{n \times n} , \, \forall \delta \in (0,\kappa) , \, \forall t>0 : \nonumber\\
		& \bP\Big( \sup\limits_{z \in \bD(\eta,\kappa)}\Big| \frac{1}{n} \tr(\tilde{M}\tilde{\bm{R}}^{(\bm{X})}(z)) - \E\Big[ \frac{1}{n} \tr(\tilde{M}\tilde{\bm{R}}^{(\bm{X})}(z)) \Big] \Big| > t + 2\frac{\delta}{\eta^2} ||\tilde{M}|| \Big) \nonumber\\
		& \hspace{1cm} \leq \frac{16\kappa^2}{\delta^2} \exp\Big( -\frac{t^2\eta^2 n}{16 R^2 ||\tilde{M}||^2} \Big)
	\end{align}
	hold.
\end{lemma}

\begin{lemma}[Concentration of $\bm{Q}(z)^{-1}$ in the Gaussian case]\label{Lemma_ConcentrationLargeExpectation}\
	\\
	Suppose~\ref{ItemAssumption_cBound},~\ref{ItemAssumption_sigmaBound} and~\ref{ItemTempAssumption_NonDegeneracy} hold and let $\bm{Z}$ be as in Lemma~\ref{Lemma_ZTailBound}. For any $\eta,\kappa>0$, there exists a constant $K=K(c_*,\sigma^2,\tau,\eta,\kappa)>0$ such that
	\begin{align}\label{Eq_QConcentrationA}
		& \forall M \in \C^{d \times d} , \, \delta \in (0,\kappa) , \, \forall t>0 : \nonumber\\
		& \bP\Big( \sup\limits_{z \in \bD(\eta,\kappa)}\Big| \frac{1}{z n} \tr(M\bm{Q}^{(\bm{Z})}(z)^{-1}) - \E\Big[ \frac{1}{z n} \tr(M\bm{Q}^{(\bm{Z})}(z)^{-1}) \Big] \Big| \nonumber\\
		& \hspace{4cm} > t + 2K n^2 \exp(-C n) ||M|| + 2K ||M|| \delta \Big) \nonumber\\
		& \leq \frac{16\kappa^2}{\delta^2} \exp\big( -C n \big) + \frac{16\kappa^2}{\delta^2} \exp\Big( -\frac{C t^2 n}{4K ||M||^2} \Big) \ ,
	\end{align}
	and
	\begin{align}\label{Eq_QConcentrationB}
		& \forall \tilde{M} \in \C^{n \times n} , \, \delta \in (0,\kappa) , \, \forall t>0 : \nonumber\\
		& \bP\Big( \sup\limits_{z \in \bD(\eta,\kappa)}\Big| \frac{1}{z n} \tr(\tilde{M}\tilde{\bm{Q}}^{(\bm{Z})}(z)^{-1}) - \E\Big[ \frac{1}{z n} \tr(\tilde{M}\tilde{\bm{Q}}^{(\bm{Z})}(z)^{-1}) \Big] \Big| \nonumber\\
		& \hspace{4cm} > t + 2K n^2 \exp(-C n) ||\tilde{M}|| + 2K ||\tilde{M}|| \delta \Big) \nonumber\\
		& \leq \frac{16\kappa^2}{\delta^2} \exp\big( -C n \big) + \frac{16\kappa^2}{\delta^2} \exp\Big( -\frac{C t^2 n}{4 K ||\tilde{M}||^2} \Big)
	\end{align}
	hold, where $C>0$ is the universal constant from Lemma~\ref{Lemma_ZTailBound}.
\end{lemma}

The proof of Theorem~\ref{Thm_EmpiricalDualSystem} commences.

Let $C,K>0$ be the constants from Lemma~\ref{Lemma_ZTailBound} and Lemma~\ref{Lemma_ConcentrationLargeExpectation} respectively. Suppose without loss of generality that $\kappa$ is large enough for
\begin{align}\label{Eq_EmpDualGaussian_kappaWlog}
	& \kappa > \frac{\eta^2}{4c_*\sigma^2} \ \ \text{ and } \ \ \kappa > \frac{1}{4 \sigma^2 K} \ .
\end{align}
Inserting $t \coloneq \frac{1}{2} n^{\frac{\tvarepsilon-1}{2}}$ and $\delta \coloneq \frac{\eta^2}{4c_*\sigma^2} n^{\frac{\tvarepsilon-1}{2}} \overset{\text{(\ref{Eq_EmpDualGaussian_kappaWlog})}}{<} \kappa$ into (\ref{Eq_R_Concentration}) for every $M \in \C^{d \times d}$ with $||M|| \leq \sigma^2$ yields
\begin{align}\label{Eq_R_Concentration_copy}
	& \bP\Big( \sup\limits_{z \in \bD(\eta,\kappa)}\Big| \frac{1}{n} \tr(M\bm{R}^{(\bm{Z})}(z)) - \E\Big[ \frac{1}{n} \tr(M\bm{R}^{(\bm{Z})}(z)) \Big] \Big| > n^{\frac{\tvarepsilon-1}{2}} \Big) \nonumber\\
	& \hspace{1cm} \leq \frac{2^8 c_*^2 \sigma^4 \kappa^2}{\eta^4} n \exp\Big( -\frac{n^{\tvarepsilon}\eta^2}{2^6 R^2 \sigma^4} \Big) \ .
\end{align}
Likewise, inserting $t \coloneq \frac{1}{2} n^{\frac{\tvarepsilon-1}{2}}$ and $\delta \coloneq \frac{1}{4 \sigma^2 K \sqrt{n}} \overset{\text{(\ref{Eq_EmpDualGaussian_kappaWlog})}}{<} \kappa$ into (\ref{Eq_QConcentrationA}) for every $M \in \C^{d \times d}$ with $||M|| \leq \sigma^2$ yields
\begin{align}\label{Eq_QConcentrationA_copy}
	& \bP\Big( \sup\limits_{z \in \bD(\eta,\kappa)}\Big| \frac{1}{z n} \tr(M\bm{Q}^{(\bm{Z})}(z)^{-1}) - \E\Big[ \frac{1}{z n} \tr(M\bm{Q}^{(\bm{Z})}(z)^{-1}) \Big] \Big| \nonumber\\
	& \hspace{5cm} > n^{\frac{\tvarepsilon-1}{2}} + 2K \sigma^2 n^2 \exp(-C n) \Big) \nonumber\\
	& \leq 2^8 \sigma^4 K^2 \kappa^2 n \exp\big( -C n \big) + 2^8 \sigma^4 K^2 \kappa^2 n \exp\Big( -\frac{C n^{\tvarepsilon}}{16 K \sigma^4} \Big) \ .
\end{align}
Combining these bounds with the results of Proposition~\ref{Prop_Stein_MP_Analogue} gives
\begin{align*}
	& \bP\Big( \sup\limits_{z \in \bD(\eta,\kappa)}\Big| \frac{1}{n} \tr\big(M\bm{R}^{(\bm{Z})}(z)\big) - \frac{-1}{z} \frac{1}{n} \tr\big(M\bm{Q}^{(\bm{Z})}(z)^{-1}\big) \Big| \nonumber\\
	& \hspace{3cm} > 2n^{\frac{\tvarepsilon-1}{2}} + \frac{C'\sigma^2}{n} + 2K \sigma^2 n^2 \exp(-C n) \Big) \nonumber\\
	& \overset{\substack{\text{(\ref{Eq_R_Concentration_copy})} \\ \text{(\ref{Eq_QConcentrationA_copy})}}}{\leq} \overbrace{\bP\Big( \sup\limits_{z \in \bD(\eta,\kappa)}\Big| \E\Big[ \frac{1}{n} \tr\big(M\bm{R}^{(\bm{Z})}(z)\big) \Big] - \frac{-1}{z} \E\Big[ \frac{1}{n} \tr\big(M\bm{Q}^{(\bm{Z})}(z)^{-1}\big) \Big] \Big| > \frac{C'\sigma^2}{n} \Big)}^{=0 \text{ by (\ref{Eq_MPSteinProp_A})}} \nonumber\\
	& \hspace{0.5cm} + \frac{2^8 c_*^2 \sigma^4 \kappa^2}{\eta^4} n \exp\Big( -\frac{n^{\tvarepsilon}\eta^2}{2^6 R^2 \sigma^4} \Big) \nonumber\\
	& \hspace{0.5cm} + 2^8 \sigma^4 K^2 \kappa^2 n \exp\big( -C n \big) + 2^8 \sigma^4 K^2 \kappa^2 n \exp\Big( -\frac{C n^{\tvarepsilon}}{16 K \sigma^4} \Big)
\end{align*}
for all $M \in \C^{d \times d}$ with $||M|| \leq \sigma^2$, where $C'$ denotes the constant from Proposition~\ref{Prop_Stein_MP_Analogue}. Choosing
\begin{align*}
	& \mathcal{C} \coloneq \max\Big( C^{-1}, 2K\sigma^2, 2+C'\sigma^2, 3 \cdot \frac{2^8 c_*^2 \sigma^4 \kappa^2}{\eta^4}, \frac{2^6 \sigma^4}{\eta^2}, 3 \cdot 2^8 \sigma^4 K^2 \kappa^2, \frac{16 K \sigma^4}{C} \Big)
\end{align*}
then gives
\begin{align*}
	& \bP\Big( \sup\limits_{z \in \bD(\eta,\kappa)}\Big| \frac{1}{n} \tr\big(M\bm{R}^{(\bm{Z})}(z)\big) - \frac{-1}{z} \frac{1}{n} \tr\big(M\bm{Q}^{(\bm{Z})}(z)^{-1}\big) \Big| \nonumber\\
	& \hspace{5cm} > \mathcal{C}n^{\frac{\tvarepsilon-1}{2}} + \mathcal{C} n^2 \exp(-n/\mathcal{C}) \Big) \nonumber\\
	& \leq \frac{\mathcal{C}}{3} n \exp\Big( -\frac{n^{\tvarepsilon}}{\mathcal{C} R^2} \Big) + \frac{\mathcal{C}}{3} n \exp\big( -n/\mathcal{C} \big) + \frac{\mathcal{C}}{3} n \exp\Big( -\frac{n^{\tvarepsilon}}{\mathcal{C}} \Big)\\
	& \leq \mathcal{C} n \exp\Big( -\frac{n^{\tvarepsilon}}{\mathcal{C} R^2} \Big)
\end{align*}
for every $M \in \C^{d \times d}$ with $||M|| \leq \sigma^2$, thus proving (\ref{Eq_EmpDualGaussian_ResultA_M}). For any $r,s \leq R$, one may then insert $A_rA_s^*$ for $M$, which by construction of $\bm{Q}(z)$ and $\hat{\delta}^{(A/B)}(z)$ proves (\ref{Eq_EmpDualGaussian_ResultA}).
\\[0.5em]
The proof for (\ref{Eq_EmpDualGaussian_ResultB}) and (\ref{Eq_EmpDualGaussian_ResultB_M}) is completely analogous with all occurrences of $\frac{d}{n}$, $A_r$, $M$, $\bm{R}$ and $\bm{Q}$ replaced by $\frac{n}{n}$, $B_r$, $\tilde{M}$, $\tilde{\bm{R}}$ and $\tilde{\bm{Q}}$. The bound $\frac{d}{n} \leq c_*$ is still applicable for $\frac{n}{n} \leq c_*$, since it was assumed in~\ref{ItemAssumption_cBound} that $c_* \geq 1$.
This concludes the proof of Theorem~\ref{Thm_EmpiricalDualSystem}. \qed

\section{Proof of Theorem~\ref{Thm_Universality}}
Before starting the proof, three lemmas are introduced.

\begin{lemma}[Holomorphic Lindeberg principle]\label{Lemma_HolomorphicLindeberg}\
	\\
	Let $U_N$ be an open subset of $\C^N \times \C^N$, which contains
	\begin{align*}
		& D_N \coloneq \big\{ (z_1,\dots,z_N;\ol{z_1},\dots,\ol{z_N}) \ \big| \ z_1,\dots,z_N \in \C \big\} \ .
	\end{align*}
	Let $f_N : U_N \rightarrow \C$ be an entry-wise holomorphic function, which is uniformly bounded on $D_N$.
	Let $X_1,\dots,X_N,Y_1,\dots,Y_N$ be independent $\C$-valued random variables such that the properties
	\begin{align}\label{Eq_Lindeberg_EqualMeans}
		& \forall j \leq N : \ \E[X_j] = 0 = \E[Y_j] \ \ , \ \ \E[X_j^2] = \E[Y_j^2] \ \ , \ \ \E[|X_j|^2] = 1 = \E[|Y_j|^2]
	\end{align}
	and
	\begin{align}\label{Eq_Lindeberg_SixthMomentBound}
		& \exists K_6 > 0 \text{ constant } \ \forall j \leq N : \ \E[|X_j|^6] \leq K_6 \geq \E[|Y_j|^6]
	\end{align}
	hold.
	The bound
	\begin{align}\label{Eq_LindebergEquality}
		& \big|\E\big[ f_N(X_1,\dots,X_N;\ol{X}_1,\dots,\ol{X}_N) - f_N(Y_1,\dots,Y_N;\ol{Y}_1,\dots,\ol{Y}_N) \big]\big| \nonumber\\
		& \leq 3\sqrt{K_6} \sum\limits_{j=1}^N \big( b^{(1)}_j + b^{(2)}_j \big)
	\end{align}
	holds, where $b^{(1)}_j, b^{(2)}_j \in [0,\infty]$ are with the notation
	\begin{align*}
		& Z^{(j,0)} \coloneq \big( X_1,\dots,X_{j-1},0,Y_{j+1},\dots,Y_N \big) \ \ \text{ and } \ \  u_{j} = \big( \underbrace{0,\dots,0}_{\times (j-1)},1,\underbrace{0,\dots,0}_{\times (N-j)} \big)
	\end{align*}
	defined as
	\begin{align*}
		& b^{(1)}_j = \sum\limits_{\substack{a,b \geq 0 \\ a+b=3}} \max\limits_{\tau \in [0,1]} \E\big[ \big|\partial_j^a \partial_{N+j}^b f_N\big(Z^{(j,0)}+\tau u_j X_j;\ol{Z^{(j,0)}+\tau u_j X_j}\big)\big|^2 \big]^{\frac{1}{2}}\\
		& b^{(2)}_j = \sum\limits_{\substack{a,b \geq 0 \\ a+b=3}} \max\limits_{\tau \in [0,1]} \E\big[ \big|\partial_j^a \partial_{N+j}^b f_N\big(Z^{(j,0)}+\tau u_j Y_j;\ol{Z^{(j,0)}+\tau u_j Y_j}\big)\big|^2 \big]^{\frac{1}{2}} \ .
	\end{align*}
\end{lemma}

\begin{lemma}[Derivative bounds]\label{Lemma_DerivativeBounds_Lindeberg}\
	\\
	Suppose~\ref{ItemAssumption_cBound} and~\ref{ItemAssumption_sigmaBound} hold and let $X$ be an arbitrary $(d \times n)$-matrix.
	For any fixed $M \in \C^{d \times d}$ and $z \in \C^+$, define the function
	\begin{align*}
		& f_{d \cdot n}(X,\ol{X}) = \frac{1}{n} \tr\big( M \bm{R}^{(X)}(z) \big) \ .
	\end{align*}
	The third derivatives of $f_{d \cdot n}$ may for any $a,b\geq 0$ with $a+b=3$ be bounded by
	\begin{align}
		& \Big| \frac{\partial^a}{\partial X_{i,j}^a} \frac{\partial^b}{\partial \ol{X_{i,j}}^b} f_{d \cdot n}(X,\ol{X}) \Big| \leq \frac{6 \sigma^6 ||M||}{n^{5/2} \Im(z)^4} \big(||\bm{S}^{(X)}||^{\frac{3}{2}} \lor |z| \, ||\bm{S}^{(X)}||^{\frac{1}{2}} \big) \label{Eq_XXXDerivf1Result} \ .
	\end{align}
\end{lemma}

\begin{lemma}[Sixth-moment bound]\label{Lemma_SixthMoment}\
	\\
	Suppose~\ref{ItemAssumption_cBound} holds and let $\bm{X}$ be a random $(d \times n)$-matrix with independent, centered entries that satisfy
	\begin{align}\label{Eq_SixthMoment_K6Condition}
		& \exists K_6>0 \, \forall i \leq d , \, \forall j \leq n : \ \E\big[ |\bm{X}_{i,j}|^6 \big] \leq K_6 \ ,
	\end{align}
	then there exists a constant $\mathcal{C}=\mathcal{C}(c_*,K_6)>0$ depending only on $c_*$ and $K_6$ and not the distribution of $\bm{X}$ such that
	\begin{align}\label{Eq_SixthMoment_Result}
		& \E\big[ ||\bm{X}||^6 \big] \leq \mathcal{C} n^3 \ .
	\end{align}
\end{lemma}

The proof of Theorem~\ref{Thm_Universality} commences.
\\[0.5em]
Without loss of generality, assume $\kappa > 1$.
The entries of $\bm{X}$ and $\bm{Z}$ are all independent and by construction satisfy
\begin{align}\label{Eq_Lindeberg_EqualMeans_Checking}
	& \forall (i,j) \in \{1,\dots,d\} \times \{1,\dots,n\} : \nonumber\\
	& \E[\bm{X}_{i,j}] = 0 = \E[\bm{Z}_{i,j}] \ \ , \ \ \E[\bm{X}_{i,j}^2] = \E[\bm{Z}_{i,j}^2] \ \ \text{ and } \ \ \E[|\bm{X}_{i,j}|^2] = 1 = \E[|\bm{Z}_{i,j}|^2] \ .
\end{align}
The fact that the entries of $\bm{Z}$ are complex Gaussian and satisfy $1 = \E[|\bm{Z}_{i,j}|^2]$, also guarantees
\begin{align}\label{Eq_Universality_C6_Gaussian}
	& \E[|\bm{Z}_{i,j}|^6] \leq \E_{Z \sim \mathcal{N}(0,1)}[|Z|^6] = 15 \ .
\end{align}
For the constant $K_6 \coloneq 15 \lor C_6$, it then by~\ref{ItemAssumption_boundedSixthMoment} holds that
\begin{align}\label{Eq_Lindeberg_SixthMomentBound_Checking}
	& \forall (i,j) \in \{1,\dots,d\} \times \{1,\dots,n\} : \ \E[|\bm{X}_{i,j}|^6] \leq K_6 \geq \E[|\bm{Z}_{i,j}|^6] \ .
\end{align}
One may thus apply Lemma~\ref{Lemma_HolomorphicLindeberg} to $f_{d \cdot n}$ as defined in Lemma~\ref{Lemma_DerivativeBounds_Lindeberg} with the entries of $\bm{X}$ and $\bm{Z}$ as $(X_1,\dots,X_N)$ and $(Y_1,\dots,Y_N)$ respectively for $N = d \cdot n$. In order to apply (\ref{Eq_LindebergEquality}), the equivalent of $\E\big[ \big|\partial_j^a \partial_{N+j}^b f_N\big(Z^{(j,0)}+\tau u_j X_j;\ol{Z^{(j,0)}+\tau u_j X_j}\big)\big|^2 \big]^{\frac{1}{2}}$ from the formulation of Lemma~\ref{Lemma_HolomorphicLindeberg} must first be analyzed.
\\[0.5em]
For any $i \leq d$ and $j \leq n$, define the mixed $(d \times n)$-matrix
\begin{align*}
	& \Z^{(i,j)} \coloneq \left( \begin{array}{ccccccc}
		\bm{X}_{1,1} & \cdots & \cdots & \bm{X}_{1,j} & \bm{Z}_{1,j+1} & \cdots & \bm{Z}_{1,n}\\
		\vdots & & & \vdots & & & \vdots\\
		\vdots & & & \bm{X}_{i-1,j} & & & \vdots\\
		\vdots & & & 0 & & & \vdots\\
		\vdots & & & \bm{Z}_{i+1,j} & & & \vdots\\
		\vdots & & & \vdots & & & \vdots\\
		\bm{X}_{d,1} & \cdots & \bm{X}_{d,j-1} & \bm{Z}_{d,j} & \cdots & \cdots & \bm{Z}_{d,n}
	\end{array} \right)
\end{align*}
and $u_{(i,j)} \coloneq e_{i,d} e_{j,n}^\top \in \C^{d \times n}$.
By inserting $(\Z^{(i,j)} + \tau u_{(i,j)} \bm{X}_{i,j})$ for $X$ in Lemma~\ref{Lemma_DerivativeBounds_Lindeberg} and simultaneously for $\bm{X}$ in Lemma~\ref{Lemma_SixthMoment}, one gets
\begin{align}\label{Eq_Universality_b1Bound}
	& b^{(1)}_{(i,j)} \nonumber\\
	& \overset{\text{(\ref{Eq_XXXDerivf1Result})}}{\leq} \sum\limits_{\substack{a,b \geq 0 \\ a+b=3}} \max\limits_{\tau \in [0,1]} \E\Big[ \Big|\frac{6 \sigma^6 ||M||}{n^{5/2} \Im(z)^4} \big(||\bm{S}^{(\Z^{(i,j)} + \tau u_{(i,j)} \bm{X}_{i,j})}||^{\frac{3}{2}} \lor |z| \, ||\bm{S}^{(\Z^{(i,j)} + \tau u_{(i,j)} \bm{X}_{i,j})}||^{\frac{1}{2}} \big)\Big|^2 \Big]^{\frac{1}{2}} \nonumber\\
	& \leq \frac{6 \sigma^6 ||M||}{n^{5/2} \Im(z)^4} \max\limits_{\tau \in [0,1]} \E\Big[ ||\bm{S}^{(\Z^{(i,j)} + \tau u_{(i,j)} \bm{X}_{i,j})}||^{3} \lor |z|^2 \, ||\bm{S}^{(\Z^{(i,j)} + \tau u_{(i,j)} \bm{X}_{i,j})}|| \Big]^{\frac{1}{2}} \sum\limits_{\substack{a,b \geq 0 \\ a+b=3}} 1 \nonumber\\
	& \leq \frac{24 \sigma^6 ||M||}{n^{5/2} \Im(z)^4} \max\limits_{\tau \in [0,1]} \E\big[ ||\bm{S}^{(\Z^{(i,j)} + \tau u_{(i,j)} \bm{X}_{i,j})}||^{3} \big]^{\frac{1}{2}} \nonumber\\
	& \hspace{0.5cm} + \frac{24 \sigma^6 |z|^2 ||M||}{n^{5/2} \Im(z)^4} \max\limits_{\tau \in [0,1]} \E\big[ ||\bm{S}^{(\Z^{(i,j)} + \tau u_{(i,j)} \bm{X}_{i,j})}|| \big]^{\frac{1}{2}} \nonumber\\
	& \overset{\text{(\ref{Eq_Def_S_Robust})}}{=} \frac{24 \sigma^6 ||M||}{n^{4} \Im(z)^4} \max\limits_{\tau \in [0,1]} \E\big[ ||\bm{Y}^{(\Z^{(i,j)} + \tau u_{(i,j)} \bm{X}_{i,j})}||^{6} \big]^{\frac{1}{2}} \nonumber\\
	& \hspace{0.5cm} + \frac{24 \sigma^6 |z|^2 ||M||}{n^{3} \Im(z)^4} \max\limits_{\tau \in [0,1]} \E\big[ ||\bm{Y}^{(\Z^{(i,j)} + \tau u_{(i,j)} \bm{X}_{i,j})}||^2 \big]^{\frac{1}{2}} \nonumber\\
	& \overset{\substack{\text{(\ref{Eq_Def_Y_Robust})} \\ \text{\ref{ItemAssumption_sigmaBound}}}}{\leq} \frac{24 \sigma^{12} ||M||}{n^{4} \Im(z)^4} \max\limits_{\tau \in [0,1]} \E\big[ ||\Z^{(i,j)} + \tau u_{(i,j)} \bm{X}_{i,j}||^{6} \big]^{\frac{1}{2}} \nonumber\\
	& \hspace{0.5cm} + \frac{24 \sigma^8 |z|^2 ||M||}{n^{3} \Im(z)^4} \max\limits_{\tau \in [0,1]} \E\big[ ||\Z^{(i,j)} + \tau u_{(i,j)} \bm{X}_{i,j}||^2 \big]^{\frac{1}{2}} \nonumber\\
	& \overset{\text{(\ref{Eq_SixthMoment_Result})}}{\leq} \frac{24 \sigma^{12} ||M||}{n^{4} \Im(z)^4} \max\limits_{\tau \in [0,1]} \big( \mathcal{C}_{\text{\ref{Lemma_SixthMoment}}} n^3 \big)^{\frac{1}{2}} + \frac{24 \sigma^8 |z|^2 ||M||}{n^{3} \Im(z)^4} \max\limits_{\tau \in [0,1]} \big( \mathcal{C}_{\text{\ref{Lemma_SixthMoment}}} n^3 \big)^{\frac{1}{6}} \nonumber\\
	& = \frac{24 \sigma^{12} ||M||}{n^{5/2} \Im(z)^4} \mathcal{C}_{\text{\ref{Lemma_SixthMoment}}}^{\frac{1}{2}} + \frac{24 \sigma^8 |z|^2 ||M||}{n^{5/2} \Im(z)^4} \mathcal{C}_{\text{\ref{Lemma_SixthMoment}}}^{\frac{1}{6}} \leq \frac{24 \sigma^{12} ||M||}{n^{5/2} \eta^4} \mathcal{C}_{\text{\ref{Lemma_SixthMoment}}}^{\frac{1}{2}} + \frac{24 \sigma^8 \kappa^2 ||M||}{n^{5/2} \eta^4} \mathcal{C}_{\text{\ref{Lemma_SixthMoment}}}^{\frac{1}{6}} \nonumber\\
	& \leq \frac{24 \sigma^{8}}{\eta^4} \big( \sigma^4 \mathcal{C}_{\text{\ref{Lemma_SixthMoment}}}^{\frac{1}{2}} + \kappa^2 \mathcal{C}_{\text{\ref{Lemma_SixthMoment}}}^{\frac{1}{6}} \big) \frac{||M||}{n^{5/2}} \ ,
\end{align}
where $\mathcal{C}_{\text{\ref{Lemma_SixthMoment}}} = \mathcal{C}_{\text{\ref{Lemma_SixthMoment}}}(c_*, 15 \lor C_6)$ is the constant from Lemma~\ref{Lemma_SixthMoment}.
One may analogously bound
\begin{align}\label{Eq_Universality_b2Bound}
	& b^{(2)}_{(i,j)} \leq \Big( \frac{24 \sigma^{12}}{\eta^4} \mathcal{C}_{\text{\ref{Lemma_SixthMoment}}}^{\frac{1}{2}} + \frac{24 \sigma^8 \kappa^2}{\eta^4} \mathcal{C}_{\text{\ref{Lemma_SixthMoment}}}^{\frac{1}{6}} \Big) \frac{||M||}{n^{5/2}} \ .
\end{align}
By Lemma~\ref{Lemma_HolomorphicLindeberg}, thus for every $z \in \bD(\eta,\kappa)$ and $M \in \C^{d \times d}$ gives
\begin{align}\label{Eq_Universality_ExpectationApprox}
	& \Big| \E\Big[ \frac{1}{n} \tr\big( M \bm{R}^{(\bm{X})}(z) \big) - \frac{1}{n} \tr\big( M \bm{R}^{(\bm{Z})}(z) \big) \Big] \Big| \nonumber\\
	& = \big| \E\big[ f_{d \times n}(\bm{X}, \ol{\bm{X}}) - f_{d \times n}(\bm{Z}, \ol{\bm{Z}}) \big] \big| \nonumber\\
	& \overset{\text{(\ref{Eq_LindebergEquality})}}{\leq} 3\sqrt{K_6} \sum\limits_{i=1}^d \sum\limits_{j=1}^n \big( b^{(1)}_{(i,j)} + b^{(2)}_{(i,j)} \big) \nonumber\\
	& \overset{\substack{\text{(\ref{Eq_Universality_b1Bound})} \\ \text{(\ref{Eq_Universality_b2Bound})}}}{\leq} 3\sqrt{K_6} \sum\limits_{i=1}^d \sum\limits_{j=1}^n \frac{48 \sigma^{8}}{\eta^4} \big( \sigma^4 \mathcal{C}_{\text{\ref{Lemma_SixthMoment}}}^{\frac{1}{2}} + \kappa^2 \mathcal{C}_{\text{\ref{Lemma_SixthMoment}}}^{\frac{1}{6}} \big) \frac{||M||}{n^{5/2}} \nonumber\\
	& = 3\sqrt{K_6} \frac{48 \sigma^{8}}{\eta^4} \big( \sigma^4 \mathcal{C}_{\text{\ref{Lemma_SixthMoment}}}^{\frac{1}{2}} + \kappa^2 \mathcal{C}_{\text{\ref{Lemma_SixthMoment}}}^{\frac{1}{6}} \big) \underbrace{\frac{d}{n}}_{\leq c_*} \frac{||M||}{\sqrt{n}} \nonumber\\
	& \leq \underbrace{3\sqrt{K_6} \frac{48 \sigma^{8}}{\eta^4} \big( \sigma^4 \mathcal{C}_{\text{\ref{Lemma_SixthMoment}}}^{\frac{1}{2}} + \kappa^2 \mathcal{C}_{\text{\ref{Lemma_SixthMoment}}}^{\frac{1}{6}} \big) c_*}_{\eqcolon  K'} \frac{||M||}{\sqrt{n}} \ .
\end{align}
The concentration result of Lemma~\ref{Lemma_R_Concentration} may be applied for both
\begin{align*}
	& \bm{X}_{\text{Lemma~\ref{Lemma_R_Concentration}}} = \bm{X} \ \ \text{ and } \ \ \bm{X}_{\text{Lemma~\ref{Lemma_R_Concentration}}} = \bm{Z}
\end{align*}
to see
\begin{align*}
	& \forall M \in \C^{d \times d} , \, \forall \delta \in (0,\kappa) , \, \forall t>0 : \nonumber\\
	& \bP\Big( \sup\limits_{z \in \bD(\eta,\kappa)}\Big| \frac{1}{n} \tr(M\bm{R}^{(\bm{X})}(z)) - \frac{1}{n} \tr(M\bm{R}^{(\bm{Z})}(z)) \Big| > 2t + 4\frac{c_* \delta}{\eta^2} ||M|| + K' \frac{||M||}{\sqrt{n}} \Big) \nonumber\\
	& \overset{\text{(\ref{Eq_R_Concentration})}}{\leq} \overbrace{\bP\Big( \sup\limits_{z \in \bD(\eta,\kappa)}\Big| \E\Big[\frac{1}{n} \tr(M\bm{R}^{(\bm{X})}(z))\Big] - \E\Big[\frac{1}{n} \tr(M\bm{R}^{(\bm{Z})}(z))\Big] \Big| > K' \frac{||M||}{\sqrt{n}} \Big)}^{= 0 \text{ by (\ref{Eq_Universality_ExpectationApprox})}} \nonumber\\
	& \hspace{0.5cm} + \frac{2^5\kappa^2}{\delta^2} \exp\Big( -\frac{t^2\eta^2 n}{16 R^2 ||M||^2} \Big) \ .
\end{align*}
Setting $t = n^{\frac{\tvarepsilon-1}{2}} ||M||$ and $\delta = \frac{1}{\sqrt{n}} < \kappa$ with $K'' = 2+4\frac{c_*}{\eta^2}+K'$ then turns the above bound into
\begin{align*}
	& \bP\Big( \sup\limits_{z \in \bD(\eta,\kappa)}\Big| \frac{1}{n} \tr(M\bm{R}^{(\bm{X})}(z)) - \frac{1}{n} \tr(M\bm{R}^{(\bm{Z})}(z)) \Big| > K'' n^{\frac{\tvarepsilon-1}{2}} ||M|| \Big)\\
	& \leq 2^5\kappa^2 n \exp\Big( -\frac{\eta^2 n^{\tvarepsilon}}{16 R^2} \Big)
\end{align*}
for any $M \in \C^{d \times d}$. This proves (\ref{Eq_Universality_Result1}) for
\begin{align*}
	& \mathcal{C} = \max\Big( K'', 2^5 \kappa^2, \frac{16}{\eta^2} \Big) \ .
\end{align*}
The bound (\ref{Eq_Universality_Result2}) may be shown analogously.
This concludes the proof of Theorem~\ref{Thm_Universality}. \qed

\section{Proof of Theorem~\ref{Thm_DetEquiv}}\label{Proof_Thm_DetEquiv}
Before starting the proof, the following lemma is introduced.

\begin{lemma}[Prokhorov for matrix-valued measures]\label{Lemma_Prokhorov}\
	\\
	For a fixed integer $R \in \N$, let $\mu_1,\mu_2,\dots : \cB(\R) \rightarrow \{ M \in \C^{R \times R} \text{ positive semi-definite}\}$ be a sequence of matrix-valued measures on $\R$, with the following two properties.
	\begin{itemize}
		\item[a)] There exists a constant $c>0$ such that $\mu_n(\R) \prec c\Id_R$ holds for all $n \in \N$.
		
		\item[b)] There exists an interval $[0,C]$ and an $N>0$ such that $\supp(\mu_n) \subset [0,C]$ for all $n \geq N$.
	\end{itemize}
	Then there exists a sub-sequence $(\mu_{n_k})_{k \in \N}$ and a matrix valued measure $\mu : \cB(\R) \rightarrow \{ M \in \C^{R \times R} \text{ positive semi-definite}\}$ such that $v^* \mu_{n_k} v \xRightarrow{k \to \infty} v^* \mu v$ holds for every $v \in \C^R$.
\end{lemma}

The proof of Theorem~\ref{Thm_DetEquiv} commences.

\begin{itemize}
	\item[i)] \textit{Uniqueness in (a)}:\\
	By conditions (\ref{Eq_Assumption_NonDegeneracy}) and (\ref{Eq_Assumption_Identifiability}), there exists a $\tau\in(0,1)$ such that condition~\ref{ItemTempAssumption_NonDegeneracy} holds for the model $\big(\bm{X},(A_r)_{r \leq R},(B_r)_{r \leq R}\big)$. Similarly, as $d$ and $n$ are fixed, one may also choose values $c_* \geq \max(\frac{d}{n},1)$ and $\sigma^2 \geq \max\big( \sum\limits_{r=1}^R ||A_r||^2, \sum\limits_{r=1}^R ||B_r||^2 \big)$ such that~\ref{ItemAssumption_cBound} and~\ref{ItemAssumption_sigmaBound} hold. Fix a $z \in \C^+$ and choose $\eta>0$ small enough and $\kappa > 0$ large enough that $z \in \bD(\eta,\kappa)$.
	\\[0.5em]
	Let $(\delta^{(A)}(z),\delta^{(B)}(z))$ and $(\tilde{\delta}^{(A)}(z),\tilde{\delta}^{(B)}(z))$ be two solutions to (\ref{Eq_DualSystem}), which both satisfy (\ref{Eq_Uniqueness_posDefCondition}). Choose $\tilde{\kappa} > 0$ large enough that (\ref{Eq_DualApprox_BoundCondition_delta}) and (\ref{Eq_DualApprox_BoundCondition_tdelta}) hold and choose $\tilde{\tau}>0$ small enough that (\ref{Eq_DualApprox_posDefCondition_delta}) and (\ref{Eq_DualApprox_posDefCondition_tdelta}) hold, the latter of which is possible by the assumptions (\ref{Eq_Uniqueness_posDefCondition}). Theorem~\ref{Thm_DualApprox} is then applicable with $q^{(A)}=0=q^{(B)}$ and yields $\delta^{(A)}(z) = \tilde{\delta}^{(A)}(z)$ as well as $\delta^{(B)}(z) = \tilde{\delta}^{(B)}(z)$, which proves uniqueness in (a).
	
	\item[ii)] \textit{Uniqueness in (b)}:\\
	Suppose there are two tuples of matrix-valued measures $(\rho^{(A)},\rho^{(B)})$ and $(\tilde{\rho}^{(A)},\tilde{\rho}^{(B)})$, which satisfy
	\begin{align}\label{Eq_Uniqueness_tdeltaStil_A}
		& \forall z \in \C^+ : \ \int_\R \frac{1}{\lambda - z} \, d\rho^{(A)}(\lambda) = \delta^{(A)}(z) = \int_\R \frac{1}{\lambda - z} \, d\tilde{\rho}^{(A)}(\lambda)
	\end{align}
	and
	\begin{align}\label{Eq_Uniqueness_tdeltaStil_B}
		& \forall z \in \C^+ : \ \int_\R \frac{1}{\lambda - z} \, d\rho^{(B)}(\lambda) = \delta^{(B)}(z) = \int_\R \frac{1}{\lambda - z} \, d\tilde{\rho}^{(B)}(\lambda) \ .
	\end{align}
	For each $u \in \C^{R}$, the Radon measures $\rho^{(A/B)}_u \coloneq u^* \rho^{(A/B)} u$ and $\tilde{\rho}^{(A/B)}_u \coloneq u^* \tilde{\rho}^{(A/B)} u$ have the Stieltjes transforms
	\begin{align*}
		& \delta^{(A/B)}_u(z) \coloneq \int_\R \frac{1}{\lambda-z} \, d\rho^{(A/B)}_u(\lambda) = u^* \Big( \int_\R \frac{1}{\lambda-z} \, d\rho^{(A/B)}(\lambda) \Big) u = u^* \delta^{(A/B)}(z) u\\
		& = u^* \Big( \int_\R \frac{1}{\lambda-z} \, d\tilde{\rho}^{(A/B)}(\lambda) \Big) u = \int_\R \frac{1}{\lambda-z} \, d\tilde{\rho}^{(A/B)}_u(\lambda) \eqcolon  \tilde{\delta}^{(A/B)}_u(z) \ ,
	\end{align*}
	which by the well-known fact that the Stieltjes transforms uniquely define the underlying Radon measure on $\R$, implies
	\begin{align*}
		& \forall u \in \C^{R} : \ u^* \rho^{(A)} u = u^* \tilde{\rho}^{(A)} u \ \ \text{ and } \ \ u^* \rho^{(B)} u = u^* \tilde{\rho}^{(B)} u \ .
	\end{align*}
	Since $\rho^{(A/B)}(S)$ and $\tilde{\rho}^{(A/B)}(S)$ are positive semi-definite, and thus Hermitian, for every Borel set $S \in \mathcal{B}(\R)$, one by polarization identity has
	\begin{align*}
		& v^* \rho^{(A/B)}(S) w\\
		& = \frac{1}{4} (v+w)^* \rho^{(A/B)}(S) (v+w) - \frac{1}{4} (v-w)^* \rho^{(A/B)}(S) (v-w)\\
		& \hspace{0.5cm} + \frac{\bm{i}}{4} (v-\bm{i}w)^* \rho^{(A/B)}(S) (v-\bm{i}w) - \frac{\bm{i}}{4} (v+\bm{i}w)^* \rho^{(A/B)}(S) (v+\bm{i}w)\\
		& = \frac{1}{4} (v+w)^* \tilde{\rho}^{(A/B)}(S) (v+w) - \frac{1}{4} (v-w)^* \tilde{\rho}^{(A/B)}(S) (v-w)\\
		& \hspace{0.5cm} + \frac{\bm{i}}{4} (v-\bm{i}w)^* \tilde{\rho}^{(A/B)}(S) (v-\bm{i}w) - \frac{\bm{i}}{4} (v+\bm{i}w)^* \tilde{\rho}^{(A/B)}(S) (v+\bm{i}w)\\
		& = v^* \tilde{\rho}^{(A/B)}(S) w
	\end{align*}
	for all $v,w \in \C^R$ and $S \in \cB(\R)$, which proves $\rho^{(A)} = \tilde{\rho}^{(A)}$ and $\rho^{(B)} = \tilde{\rho}^{(B)}$.
	
	\item[iii)] \textit{Construction of the meta model}:\\
	Let $\bm{Z}$ be the Gaussian matrix similar to $\bm{X}$ as constructed in Definition~\ref{Def_SimilarGaussianMatrix}.
	Let $\bm{Z}^{(m)}$ denote the Gaussian $(dm \times nm)$-matrix with independent entries, where the entries of each $(m \times m)$-sub-block have the same distribution of the corresponding entry in $\bm{Z}$, i.e.
	\begin{align*}
		& \forall k \in \{1,\dots,d\} , \, \forall l \in \{1,\dots,n\} , \, \forall k',l' \in \{1,\dots,m\} : \ \bm{Z}^{(m)}_{(k-1)m + k', (l-1)m + l'} \sim \bm{Z}_{k,l} \ .
	\end{align*}
	Further, let $A^{(m)}_1,\dots,A^{(m)}_R$ denote the deterministic $(dm \times dm)$-matrices
	\begin{align}\label{Eq_DetEquiv_metaModel_DefA}
		& A^{(m)}_r \coloneq \left( \begin{array}{ccc}
			(A_r)_{1,1} \Id_m & \cdots & (A_r)_{1,d} \Id_m\\
			\vdots & & \vdots\\
			(A_r)_{d,1} \Id_m & \cdots & (A_r)_{d,d} \Id_m
		\end{array} \right) \ .
	\end{align}
	Analogously, define the deterministic $(nm \times nm)$-matrices $B^{(m)}_1,\dots,B^{(m)}_R$ by
	\begin{align}\label{Eq_DetEquiv_metaModel_DefB}
		& B^{(m)}_r \coloneq \left( \begin{array}{ccc}
			(B_r)_{1,1} \Id_m & \cdots & (B_r)_{1,n} \Id_m\\
			\vdots & & \vdots\\
			(B_r)_{n,1} \Id_m & \cdots & (B_r)_{n,n} \Id_m
		\end{array} \right) \ .
	\end{align}
	With
	\begin{align*}
		& d^{(m)} \coloneq dm \ \ \text{ and } \ \ n^{(m)} \coloneq nm
	\end{align*}
	as well as
	\begin{align*}
		& \hspace{3.5cm} \bm{Y}^{(\bm{Z}^{(m)})} \coloneq \sum\limits_{r=1}^R A^{(m)}_r \bm{Z}^{(m)} B^{(m)}_r ,\\
		& \bm{S}^{(\bm{Z}^{(m)})} \coloneq \frac{1}{n^{(m)}} \bm{Y}^{(\bm{Z}^{(m)})} (\bm{Y}^{(\bm{Z}^{(m)})})^* \ \ \text{ and } \ \ \tilde{\bm{S}}^{(\bm{Z}^{(m)})} \coloneq \frac{1}{n^{(m)}} (\bm{Y}^{(\bm{Z}^{(m)})})^* \bm{Y}^{(\bm{Z}^{(m)})} \ ,
	\end{align*}
	define the empirical matrix-valued measures
	\begin{align*}
		& \hat{\rho}^{(A,m)} , \hat{\rho}^{(B,m)} : \cB(\R) \rightarrow \{ M \in \C^{R \times R} \text{ positive semi-definite} \}
	\end{align*}
	in analogy to (\ref{Eq_Def_hatRhoA}) and (\ref{Eq_Def_hatRhoB}) by
	\begin{align}\label{Eq_Def_hatRhoA_meta}
		& \hat{\rho}^{(A,m)}(S) \coloneq \frac{1}{n^{(m)}} \sum\limits_{\substack{j=1 \\ \lambda_j(\bm{S}^{(\bm{Z}^{(m)})}) \in S}}^{d^{(m)}} \big((u^{(m)}_j)^* A^{(m)}_r (A^{(m)}_s)^* u^{(m)}_j\big)_{r,s \leq R}
	\end{align}
	and
	\begin{align}\label{Eq_Def_hatRhoB_meta}
		& \hat{\rho}^{(B,m)}(S) \coloneq \frac{1}{n^{(m)}} \sum\limits_{\substack{j=1 \\ \lambda_j(\tilde{\bm{S}}^{(\bm{Z}^{(m)})}) \in S}}^{n^{(m)}} \big((\tilde{u}^{(m)}_j)^* (B^{(m)}_s)^* B^{(m)}_r \tilde{u}^{(m)}_j\big)_{r,s \leq R}
	\end{align}
	for any set $S \in \cB(\R)$, where $u^{(m)}_1,\dots,u^{(m)}_{d^{(m)}}$ and $\tilde{u}^{(m)}_1,\dots,\tilde{u}^{(m)}_{n^{(m)}}$ denote the eigenvectors of $\bm{S}^{(\bm{Z}^{(m)})}$ and $\tilde{\bm{S}}^{(\bm{Z}^{(m)})}$ respectively.
	Also, for each $z \in \C^+$, define the empirical matrix-valued Stieltjes transforms
	\begin{align*}
		& \hat{\delta}^{(A,m)}, \hat{\delta}^{(B,m)} : \C^+ \rightarrow \C^{R \times R}
	\end{align*}
	analogously to (\ref{Eq_Def_hatDelta_A}) and (\ref{Eq_Def_hatDelta_B}) by
	\begin{align}\label{Eq_Def_hatDelta_A_meta}
		& \hat{\delta}^{(A,m)}_{r,s}(z) \coloneq \frac{1}{n^{(m)}} \tr\big( A^{(m)}_r (A^{(m)}_s)^* \bm{R}^{(\bm{Z}^{(m)})}(z) \big) = \int_0^\infty \frac{1}{\lambda - z} \, d\hat{\rho}^{(A,m)}_{r,s}(\lambda)
	\end{align}
	and
	\begin{align}\label{Eq_Def_hatDelta_B_meta}
		& \hat{\delta}^{(B,m)}_{r,s}(z) \coloneq \frac{1}{n^{(m)}} \tr\big( (B^{(m)}_s)^* B^{(m)}_r \tilde{\bm{R}}^{(\bm{Z}^{(m)})}(z) \big) = \int_0^\infty \frac{1}{\lambda - z} \, d\hat{\rho}^{(B,m)}_{r,s}(\lambda) \ ,
	\end{align}
	where the resolvents are defined analogously by
	\begin{align*}
		& \bm{R}^{(\bm{Z}^{(m)})}(z) \coloneq \big( \bm{S}^{(\bm{Z}^{(m)})} - z \Id_{d^{(m)}} \big)^{-1} \ \ \text{ and } \ \ \tilde{\bm{R}}^{(\bm{Z}^{(m)})}(z) \coloneq \big( \tilde{\bm{S}}^{(\bm{Z}^{(m)})} - z \Id_{n^{(m)}} \big)^{-1}
	\end{align*}
	
	\item[iv)] \textit{Applicability of prior results to the meta model}:\\
	The meta model $\big(\bm{Z}^{(m)}, (A_r^{(m)})_{r \leq R}, (B^{(m)}_r)_{r \leq R}\big)$
	by construction satisfies the conditions \ref{ItemAssumption_boundedSixthMoment} (with $C_6=15$ by (\ref{Eq_Universality_C6_Gaussian})). As seen in part (i) of this proof, one may choose values $\tau \in (0,1)$ and $c_*,\sigma^2 \geq 1$ such that~\ref{ItemAssumption_cBound},~\ref{ItemAssumption_sigmaBound} and~\ref{ItemTempAssumption_NonDegeneracy} hold for the original model $\big(\bm{X},(A_r)_{r \leq R},(B_r)_{r \leq R}\big)$, and are thus inherited by $\big(\bm{Z}^{(m)}, (A_r^{(m)})_{r \leq R}, (B^{(m)}_r)_{r \leq R}\big)$. Previous results requiring conditions~\ref{ItemAssumption_cBound}-\ref{ItemTempAssumption_NonDegeneracy} are thus applicable to the meta model with constants $\tau \in (0,1)$ and $c_*,\sigma^2 \geq 1$ that do not depend on $m$.
	
	\item[v)] \textit{Application of Theorem~\ref{Thm_EmpiricalDualSystem} to the meta model}:\\
	Theorem~\ref{Thm_EmpiricalDualSystem} for $\hat{\delta}^{(A,m)}(z)$ and $\hat{\delta}^{(B,m)}(z)$ as in (\ref{Eq_Def_hatDelta_A_meta}) and (\ref{Eq_Def_hatDelta_B_meta}) yields
	\begin{align}\label{Eq_EmpDualGaussian_ResultA_copyMeta}
		& \bP\Big( \exists r,s \leq R , \, \exists z \in \bD(\eta,\kappa) : \nonumber\\
		& \hspace{0.5cm} \Big| \hat{\delta}^{(A,m)}_{r,s}(z) - \frac{-1}{z} \overbrace{\frac{1}{n^{(m)}} \tr\Big( A^{(m)}_r(A^{(m)}_s)^* \Big( \Id_{d^{(m)}} + \sum\limits_{r',s'=1}^R \hat{\delta}^{(B,m)}_{r',s'}(z) A^{(m)}_{r'} (A^{(m)}_{s'})^* \Big)^{-1} \Big)}^{\overset{\text{(\ref{Eq_DetEquiv_metaModel_DefA})}}{=} \frac{1}{n} \tr\big( A_rA_s^* \big( \Id_d + \sum\limits_{r',s'=1}^R \hat{\delta}^{(B,m)}_{r',s'}(z) A_{r'} A_{s'}^* \big)^{-1} \big)} \Big| \nonumber\\
		& \hspace{7cm} \geq \mathcal{C}(n^{(m)})^{\frac{\tvarepsilon-1}{2}} + \mathcal{C} (n^{(m)})^2 \exp(-n^{(m)}/\mathcal{C}) \Big) \nonumber\\
		& \leq \mathcal{C} R^2 n^{(m)} \exp\Big( -\frac{(n^{(m)})^{\tvarepsilon}}{\mathcal{C} R^2} \Big)
	\end{align}
	and
	\begin{align}\label{Eq_EmpDualGaussian_ResultB_copyMeta}
		& \bP\Big( \exists r,s \leq R , \, \exists z \in \bD(\eta,\kappa) : \nonumber\\
		& \hspace{0.5cm} \Big| \hat{\delta}^{(B,m)}_{r,s}(z) - \frac{-1}{z} \overbrace{\frac{1}{n^{(m)}} \tr\Big( (B^{(m)}_s)^*B^{(m)}_r \Big( \Id_{n^{(m)}} + \sum\limits_{r',s'=1}^R \hat{\delta}^{(A,m)}_{r',s'}(z) (B^{(m)}_{s'})^* B^{(m)}_{r'} \Big)^{-1} \Big)}^{\overset{\text{(\ref{Eq_DetEquiv_metaModel_DefB})}}{=} \frac{1}{n} \tr\big( B_s^*B_r \big( \Id_n + \sum\limits_{r',s'=1}^R \hat{\delta}^{(A,m)}_{r',s'}(z) B_{s'}^* B_{r'} \big)^{-1} \big)} \Big| \nonumber\\
		& \hspace{7cm} \geq \mathcal{C}(n^{(m)})^{\frac{\tvarepsilon-1}{2}} + \mathcal{C} (n^{(m)})^2 \exp(-n^{(m)}/\mathcal{C}) \Big) \nonumber\\
		& \leq \mathcal{C} R^2 n^{(m)} \exp\Big( -\frac{(n^{(m)})^{\tvarepsilon}}{\mathcal{C} R^2} \Big)
	\end{align}
	for any $\tvarepsilon \in (0,1)$ and $\eta,\kappa>0$. Since $n^{(m)} = nm$, one may then apply Borel-Cantelli over $m$ (for fixed $d$ and $n$) to get
	\begin{align*}
		& 1 = \bP\Big( \forall r,s \leq R, \, \forall z \in \bD(\eta,\kappa) : \nonumber\\
		& \hspace{1cm} \Big| \hat{\delta}^{(A,m)}_{r,s}(z) - \frac{-1}{z n} \tr\Big( A_rA_s^* \Big( \Id_d + \sum\limits_{r',s'=1}^R \hat{\delta}^{(B,m)}_{r',s'}(z) A_{r'} A_{s'}^* \Big)^{-1} \Big) \Big| \xrightarrow{m \to \infty} 0 \nonumber\\
		& \hspace{1cm} \Big| \hat{\delta}^{(B,m)}_{r,s}(z) - \frac{-1}{z n} \tr\Big( B_s^*B_r \Big( \Id_n + \sum\limits_{r',s'=1}^R \hat{\delta}^{(A,m)}_{r',s'}(z) B_{s'}^* B_{r'} \Big)^{-1} \Big) \Big| \xrightarrow{m \to \infty} 0 \Big) \ ,
	\end{align*}
	which by continuity of measures for $\eta \searrow 0$ and $\kappa \nearrow \infty$ gives
	\begin{align}\label{Eq_DetEquiv_metaCovnergenceEvent}
		& 1 = \bP\Big( \forall r,s \leq R, \, \forall z \in \C^+ : \nonumber\\
		& \hspace{1cm} \Big| \hat{\delta}^{(A,m)}_{r,s}(z) - \frac{-1}{z n} \tr\Big( A_rA_s^* \Big( \Id_d + \sum\limits_{r',s'=1}^R \hat{\delta}^{(B,m)}_{r',s'}(z) A_{r'} A_{s'}^* \Big)^{-1} \Big) \Big| \xrightarrow{m \to \infty} 0 \nonumber\\
		& \hspace{1cm} \Big| \hat{\delta}^{(B,m)}_{r,s}(z) - \frac{-1}{z n} \tr\Big( B_s^*B_r \Big( \Id_n + \sum\limits_{r',s'=1}^R \hat{\delta}^{(A,m)}_{r',s'}(z) B_{s'}^* B_{r'} \Big)^{-1} \Big) \Big| \xrightarrow{m \to \infty} 0 \Big) \ .
	\end{align}
	
	\item[vi)] \textit{Almost sure convergence of $(\hat{\rho}^{(A,m)},\hat{\rho}^{(B,m)})_{m \in \N}$ by Prokhorov}:\\
	Lemma~\ref{Lemma_ZTailBound}, when applied to the meta model with $t= \sqrt{\frac{\tvarepsilon}{2} n^{(m)}}$ for some $\tvarepsilon>0$ yields
	\begin{align*}
		& \bP\Big( ||\bm{Y}^{(\bm{Z}^{(m)})}|| > 2\sigma^2 \big(\sqrt{n^{(m)}}+\sqrt{d^{(m)}}\big) + \sqrt{\frac{\tvarepsilon}{2} n^{(m)}} \Big) \leq 2 \exp\Big( -\frac{C\tvarepsilon n^{(m)}}{2\sigma^4} \Big)
	\end{align*}
	for a universal constant $C>0$. Since
	\begin{align*}
		& ||\bm{Y}^{(\bm{Z}^{(m)})}|| > 2\sigma^2 \big(\sqrt{n^{(m)}}+\sqrt{d^{(m)}}\big) + \sqrt{\frac{\tvarepsilon}{2} n^{(m)}}\\
		& \Leftrightarrow \ ||\bm{Y}^{(\bm{Z}^{(m)})}||^2 > \big(2\sigma^2 \big(\sqrt{n^{(m)}}+\sqrt{d^{(m)}}\big) + \sqrt{\frac{\tvarepsilon}{2} n^{(m)}}\big)^2\\
		& \Leftarrow \ ||\bm{Y}^{(\bm{Z}^{(m)})}||^2 > 8\sigma^4 \big(\sqrt{n^{(m)}}+\sqrt{d^{(m)}}\big)^2 + 2\frac{\tvarepsilon}{2}n^{(m)}\\
		& \overset{\text{(\ref{Eq_Def_S_Robust})}}{\Leftrightarrow} \ ||\bm{S}^{(\bm{Z}^{(m)})}|| > 8\sigma^4 \frac{(\sqrt{n^{(m)}}+\sqrt{d^{(m)}})^2}{n^{(m)}} + \tvarepsilon\\
		& \overset{\text{\ref{ItemAssumption_cBound}}}{\Leftarrow} \ ||\bm{S}^{(\bm{Z}^{(m)})}|| > 8\sigma^4 (1+\sqrt{c_*})^2 + \tvarepsilon \ ,
	\end{align*}
	it with $n^{(m)} = nm$ follows that
	\begin{align*}
		& \bP\Big( ||\bm{S}^{(\bm{Z}^{(m)})}|| > 8\sigma^4 (1+\sqrt{c_*})^2 + \tvarepsilon \Big) \leq 2 \exp\Big( -\frac{C \tvarepsilon nm}{2\sigma^4} \Big) \ ,
	\end{align*}
	which by Borel-Cantelli over $m$ (with fixed $d$ and $n$) yields
	\begin{align}\label{Eq_DetEquivProof_TighnessBorelCantelli}
		& 1 = \bP\Big( \limsup\limits_{m \to \infty} ||\bm{S}^{(\bm{Z}^{(m)})}|| \leq 8\sigma^4 (1+\sqrt{c_*})^2 + \tvarepsilon \Big) \ .
	\end{align}
	The bound
	\begin{align}\label{Eq_DetEquivProof_hatRhoA_Bound}
		& ||\hat{\rho}^{(A,m)}(\R)|| \overset{\text{(\ref{Eq_Def_hatRhoA_meta})}}{=} \Big|\Big| \frac{1}{n^{(m)}} \tr\big( A^{(m)}_r(A^{(m)}_s)^* \big)_{r,s \leq R} \Big|\Big| \nonumber\\
		& = \frac{1}{n^{(m)}} \max\limits_{\substack{v \in \C^R \\ ||v||_2 = 1}} \sum\limits_{r,s=1}^R \ol{v_r} v_s \tr\big( A^{(m)}_r(A^{(m)}_s)^* \big) \nonumber\\
		& = \frac{1}{n^{(m)}} \max\limits_{\substack{v \in \C^R \\ ||v||_2 = 1}} \tr\Big( \Big( \sum\limits_{r=1}^R \ol{v_r} A^{(m)}_r \Big) \Big( \sum\limits_{r=1}^R \ol{v_r} A^{(m)}_r \Big)^* \Big) \nonumber\\
		& \overset{\text{(\ref{Eq_TraceBound})}}{\leq} \frac{d^{(m)}}{n^{(m)}} \max\limits_{\substack{v \in \C^R \\ ||v||_2 = 1}} \Big|\Big| \sum\limits_{r=1}^R \ol{v_r} A^{(m)}_r \Big|\Big|^2\overset{\text{C.S.}}{\leq} \frac{d^{(m)}}{n^{(m)}} \max\limits_{\substack{v \in \C^R \\ ||v||_2 = 1}} \Big( \sum\limits_{r=1}^R |v_r|^2 \Big) \Big( \sum\limits_{r=1}^R ||A^{(m)}_r||^2 \Big) \nonumber\\
		& = \underbrace{\frac{d^{(m)}}{n^{(m)}}}_{\leq c_*} \sum\limits_{r=1}^R \underbrace{||A^{(m)}_r||^2}_{= ||A_r||^2} \leq c_* \sigma^2
	\end{align}
	implies $\hat{\rho}^{(A,m)}(\R) \preceq c_* \sigma^2 \Id_R$ and one may analogously show $\hat{\rho}^{(B,m)}(\R) \preceq \sigma^2 \Id_R$. Since the supports of $\hat{\rho}^{(A,m)}$ and $\hat{\rho}^{(B,m)}$ by construction lie in the interval $[0,||\bm{S}^{(\bm{Z}^{(m)})}||]$, one may on the event $\cE_{\text{(\ref{Eq_DetEquivProof_TighnessBorelCantelli})}}$ from (\ref{Eq_DetEquivProof_TighnessBorelCantelli}) apply Lemma~\ref{Lemma_Prokhorov} to any sub-sequences $(\hat{\rho}^{(A,m_j)})_{j \in \N}$ and $(\hat{\rho}^{(B,m_j)})_{j \in \N}$ to get
	\begin{align}\label{Eq_DetEquivProof_SubSubSequenceExistence}
		& \forall \omega \in \cE_{\text{(\ref{Eq_DetEquivProof_TighnessBorelCantelli})}} , \, \forall (m_j)_{j \in \N} \text{ sub-sequence} : \nonumber\\
		& \exists (m_{j_k})_{k \in \N} \text{ sub-sub-sequence} , \, \exists \mu^{(A)}_\omega, \mu^{(B)}_\omega \text{ matrix-valued measures} : \nonumber\\
		& \hat{\rho}^{(A,m_{j_k})}_\omega \xRightarrow{k \to \infty} \mu^{(A)}_\omega \ \ \text{ and } \ \ \hat{\rho}^{(B,m_{j_k})}_\omega \xRightarrow{k \to \infty} \mu^{(B)}_\omega \ ,
	\end{align}
	where the weak convergence is in the sense of Lemma~\ref{Lemma_MatrWeakConv} and the joint convergence follows from a simple diagonal sequence argument.
	
	\item[vii)] \textit{Uniqueness of limits and convergence}:\\
	Taking (c) from Lemma~\ref{Lemma_MatrWeakConv} as the interpretation of weak convergence, shows that (\ref{Eq_DetEquivProof_SubSubSequenceExistence}) implies
	\begin{align}\label{Eq_DetEquivProof_SubSubSequenceExistence2}
		& \forall \omega \in \cE_{\text{(\ref{Eq_DetEquivProof_TighnessBorelCantelli})}} , \, \forall (m_j)_{j \in \N} \text{ sub-sequence} : \nonumber\\
		& \exists (m_{j_k})_{k \in \N} \text{ sub-sub-sequence} , \, \exists \mu^{(A)}_\omega, \mu^{(B)}_\omega \text{ matrix-valued measures} : \nonumber\\
		& \hat{\rho}^{(A,m_{j_k})}_\omega(\R) \xrightarrow{k \to \infty} \mu^{(A)}_\omega(\R) \ \ \text{ and } \ \ \hat{\rho}^{(B,m_{j_k})}_\omega(\R) \xrightarrow{k \to \infty} \mu^{(B)}_\omega(\R) \ , \nonumber\\
		& \forall z \in \C^+ : \ \hat{\delta}^{(A,m_{j_k})}_\omega(z) \xrightarrow{k \to \infty} \delta^{(A)}_\omega(z) \ \ \text{ and } \ \ \hat{\delta}^{(B,m_{j_k})}_\omega(z) \xrightarrow{k \to \infty} \delta^{(B)}_\omega(z) \ ,
	\end{align}
	where
	\begin{align}\label{Eq_DetEquivProof_Def_DeltaOmega}
		& \delta^{(A)}_\omega(z) \coloneq \int_\R \frac{1}{\lambda-z} \, d\mu^{(A)}_\omega(\lambda) \ \ \text{ and } \ \ \delta^{(B)}_\omega(z) \coloneq \int_\R \frac{1}{\lambda-z} \, d\mu^{(B)}_\omega(\lambda) \ .
	\end{align}
	For each such $\omega \in \cE_{\text{(\ref{Eq_DetEquivProof_TighnessBorelCantelli})}}$, the properties
	\begin{align*}
		& \frac{1}{n} \tr\big( A_rA_s^* \big)_{r,s \leq R} \overset{\text{(\ref{Eq_DetEquiv_metaModel_DefA})}}{=} \frac{1}{n^{(m)}} \tr\big( A^{(m)}_r(A^{(m)}_s)^* \big)_{r,s \leq R} \overset{\text{(\ref{Eq_Def_hatRhoA_meta})}}{=} \hat{\rho}^{(A,m_{j_k})}_\omega(\R) \xrightarrow[\text{(\ref{Eq_DetEquivProof_SubSubSequenceExistence2})}]{k \to \infty} \mu^{(A)}_\omega(\R)\\
		& \frac{1}{n} \tr\big( B_s^*B_r \big)_{r,s \leq R} \overset{\text{(\ref{Eq_DetEquiv_metaModel_DefB})}}{=} \frac{1}{n^{(m)}} \tr\big( (B^{(m)}_s)^*B^{(m)}_r \big)_{r,s \leq R} \overset{\text{(\ref{Eq_Def_hatRhoB_meta})}}{=} \hat{\rho}^{(B,m_{j_k})}_\omega(\R) \xrightarrow[\text{(\ref{Eq_DetEquivProof_SubSubSequenceExistence2})}]{k \to \infty} \mu^{(B)}_\omega(\R)
	\end{align*}
	imply $\mu^{(A)}_\omega(\R) = \frac{1}{n} \tr\big( A_rA_s^* \big)_{r,s \leq R} \overset{\text{(\ref{Eq_Assumption_Identifiability})}}{\succ} 0$ and $\mu^{(B)}_\omega(\R) = \frac{1}{n} \tr\big( B_s^*B_r \big)_{r,s \leq R} \overset{\text{(\ref{Eq_Assumption_Identifiability})}}{\succ} 0$, so the property $\Im(\frac{1}{\lambda-z}) = \frac{\Im(z)}{|\lambda-z|^2} > 0$ by (\ref{Eq_DetEquivProof_Def_DeltaOmega}) gives
	\begin{align}\label{Eq_DetEquivProof_PosDef}
		& 0 \prec \Im\big( \delta^{(A)}_\omega(z) \big) \ \ \text{ and } \ \ 0 \prec \Im\big( \delta^{(B)}_\omega(z) \big) \ .
	\end{align}
	For each such $\omega \in \cE_{\text{(\ref{Eq_DetEquivProof_TighnessBorelCantelli})}}$, which is also in the event $\cE_{\text{(\ref{Eq_DetEquiv_metaCovnergenceEvent})}}$ from (\ref{Eq_DetEquiv_metaCovnergenceEvent}), a simple continuity argument with (\ref{Eq_DetEquivProof_PosDef}) shows that $\big( \delta^{(A)}_\omega(z), \delta^{(B)}_\omega(z) \big)$ must be a solution to (\ref{Eq_DualSystem}).
	Part (i) of this proof is thus applicable for each $\omega \in \cE_{\text{(\ref{Eq_DetEquivProof_TighnessBorelCantelli})}} \cap \cE_{\text{(\ref{Eq_DetEquiv_metaCovnergenceEvent})}}$ to see that there is only a single possible value in $\C^{R \times R} \times \C^{R \times R}$ that $(\delta^{(A)}_\omega(z),\delta^{(B)}_\omega(z))$ can take, regardless of the choice of $\omega \in \cE_{\text{(\ref{Eq_DetEquivProof_TighnessBorelCantelli})}} \cap \cE_{\text{(\ref{Eq_DetEquiv_metaCovnergenceEvent})}}$ or sub-sub-sequence $(m_{j_k})_{k \in \N}$.
	As this is true for every $z \in \C^+$, part (ii) of this proof then also yields the same for $(\mu^{(A)}_\omega,\mu^{(B)}_\omega)$, i.e. there is only a single (deterministic) pair of matrix-valued measures $(\rho^{(A)},\rho^{(B)})$ such that $(\mu^{(A)}_\omega,\mu^{(B)}_\omega) = (\rho^{(A)},\rho^{(B)})$ for all $\omega \in \cE_{\text{(\ref{Eq_DetEquivProof_TighnessBorelCantelli})}} \cap \cE_{\text{(\ref{Eq_DetEquiv_metaCovnergenceEvent})}}$ and choices of sub-sub-sequences $(m_{j_k})_{k \in \N}$. The result (\ref{Eq_DetEquivProof_SubSubSequenceExistence}) becomes
	\begin{align}\label{Eq_DetEquivProof_SubSubSequenceExistence3}
		& \forall \omega \in \cE_{\text{(\ref{Eq_DetEquivProof_TighnessBorelCantelli})}} \cap \cE_{\text{(\ref{Eq_DetEquiv_metaCovnergenceEvent})}} , \, \forall (m_j)_{j \in \N} \text{ sub-sequence} : \nonumber\\
		& \exists (m_{j_k})_{k \in \N} \text{ sub-sub-sequence} : \nonumber\\
		& \hat{\rho}^{(A,m_{j_k})}_\omega \xRightarrow{k \to \infty} \rho^{(A)} \ \ \text{ and } \ \ \hat{\rho}^{(B,m_{j_k})}_\omega \xRightarrow{k \to \infty} \rho^{(B)} \ .
	\end{align}
	Standard topological arguments then yield
	\begin{align}\label{Eq_DetEquivProof_SubSubSequenceExistence4}
		& \forall \omega \in \cE_{\text{(\ref{Eq_DetEquivProof_TighnessBorelCantelli})}} \cap \cE_{\text{(\ref{Eq_DetEquiv_metaCovnergenceEvent})}} : \ \hat{\rho}^{(A,m)}_\omega \xRightarrow{m \to \infty} \rho^{(A)} \ \ \text{ and } \ \ \hat{\rho}^{(B,m)}_\omega \xRightarrow{m \to \infty} \rho^{(B)} \ .
	\end{align}
	As both events (\ref{Eq_DetEquiv_metaCovnergenceEvent}) and (\ref{Eq_DetEquivProof_TighnessBorelCantelli}) were shown to have probability one, it follows that
	\begin{align}\label{Eq_DetEquiv_rhoConvergence}
		& 1 = \bP\Big( \hat{\rho}^{(A,m)} \xRightarrow{m \to \infty} \rho^{(A)} \ \ \text{ and } \ \ \hat{\rho}^{(B,m)} \xRightarrow{m \to \infty} \rho^{(B)} \Big)
	\end{align}
	for some deterministic matrix-valued measures $(\rho^{(A)},\rho^{(B)})$ whose matrix-valued Stieltjes transforms
	\begin{align*}
		& \delta^{(A)}(z) = \int_\R \frac{1}{\lambda - z} \, d\rho^{(A)}(\lambda) \ \ \text{ and } \ \ \delta^{(B)}(z) = \int_\R \frac{1}{\lambda - z} \, d\rho^{(B)}(\lambda)
	\end{align*}
	are a solution to (\ref{Eq_DualSystem}). This proves existence in (a) and (b).
	
	\item[viii)] \textit{Properties of $(\rho^{(A)},\rho^{(B)})$}:\\
	By construction of $(\rho^{(A)},\rho^{(B)})$ in part (vii), it is clear that the matrix-valued Stieltjes transforms $(\delta^{(A)}(z),\delta^{(B)}(z))$ are a solution to (\ref{Eq_DualSystem}). A consequence of (\ref{Eq_DetEquiv_rhoConvergence}) and (\ref{Eq_DetEquivProof_TighnessBorelCantelli}) is
	\begin{align}
		& \supp\big( \rho^{(A)} \big) \subset [0,8\sigma^4(1+\sqrt{c_*})^2] \ \ \text{ and } \ \ \supp\big( \rho^{(B)} \big) \subset [0,8\sigma^4(1+\sqrt{c_*})^2] \ ,
	\end{align}
	since the supports of $\hat{\rho}^{(A,m)}$ and $\hat{\rho}^{(B,m)}$ are by construction in $[0,||\bm{S}^{(\bm{Z}^{(m)})}||]$ and $\tvarepsilon>0$ may be chosen arbitrarily small.
	One also observes
	\begin{align*}
		& \hat{\rho}^{(A,m)}(\R) \overset{\text{(\ref{Eq_Def_hatRhoA_meta})}}{=} \frac{1}{n^{(m)}} \tr\big( A^{(m)}_r (A^{(m)}_s)^* \big)_{r,s \leq R} \overset{\text{(\ref{Eq_DetEquiv_metaModel_DefA})}}{=} \underbrace{\frac{m}{n^{(m)}}}_{= \frac{1}{n}} \tr\big( A_r A_s^* \big)_{r,s \leq R}
	\end{align*}
	and analogously 
	\begin{align*}
		& \hat{\rho}^{(B,m)}(\R) = \frac{1}{n} \tr\big( B_s^*B_r \big)_{r,s \leq R} \ .
	\end{align*}
	The convergence (\ref{Eq_DetEquiv_rhoConvergence}) thus yields
	\begin{align}
		& \rho^{(A)}(\R) = \frac{1}{n} \tr\big( A_r A_s^* \big)_{r,s \leq R} \ \ \text{ and } \ \ \rho^{(B)}(\R) = \frac{1}{n} \tr\big( B_s^*B_r \big)_{r,s \leq R}  \ .
	\end{align}
	
	\item[ix)] \textit{Proof of (d)}:\\
	Since every probability measure is uniquely defined by its Stieltjes transform (cf. Theorem B.9 in \cite{BaiSALDRM}), it suffices to show existence of $\ul{\nu}$.
	\\[0.5em]
	Analogously to part (v) of this proof, one may apply (\ref{Eq_EmpDualGaussian_ResultB_M}) with $\tilde{M} = \Id_{nm}$ to the meta model and by Borel-Cantelli see
	\begin{align}\label{Eq_DetEquivProof_NuStilConv0}
		& 1 = \bP\Big( \forall z \in \C^+ : \nonumber\\
		& \hspace{1cm} \Big| \frac{1}{nm} \tr\big(\tilde{\bm{R}}^{(\bm{Z}^{(m)})}\big) - \frac{-1}{z n} \tr\Big( \Big( \Id_n + \sum\limits_{r',s'=1}^R \hat{\delta}^{(A,m)}_{r',s'}(z) B_{s'}^* B_{r'} \Big)^{-1} \Big) \Big| \xrightarrow{m \to \infty} 0 \Big) \ .
	\end{align}
	For each $m \in \N$, let $\hat{\ul{\nu}}^{(m)} : \cB(\R) \rightarrow [0,1]$ denote the random empirical eigenvalue distribution of $\tilde{\bm{S}}^{(\bm{Z}^{(m)})}$, i.e.
	\begin{align*}
		& \forall A \in \cB(\R) : \ \hat{\ul{\nu}}^{(m)}(A) \coloneq \frac{1}{nm} \#\big\{j \leq nm \ \big| \ \lambda_j\big(\tilde{\bm{S}}^{(\bm{Z}^{(m)})}\big) \in A \big\} \ .
	\end{align*}
	For every $\omega \in \cE_{\text{(\ref{Eq_DetEquivProof_TighnessBorelCantelli})}}$, the sequence of probability measures $(\hat{\ul{\nu}}^{(m)}_\omega)_{m \in \N}$ is tight and thus there by Prokhorov one has
	\begin{align}\label{Eq_DetEquivProof_NuProkhorov}
		& \forall \omega \in \cE_{\text{(\ref{Eq_DetEquivProof_TighnessBorelCantelli})}} , \, \exists (m_k)_{k \in \N} \text{ sub-sequence} , \, \exists \ul{\nu}_\omega \text{ prop. measure} : \ \hat{\ul{\nu}}^{(m_k)}_\omega \xRightarrow{k \rightarrow \infty} \ul{\nu}_\omega \ .
	\end{align}
	The fact that $\omega \in \cE_{\text{(\ref{Eq_DetEquivProof_TighnessBorelCantelli})}}$, even yields
	\begin{align}\label{Eq_DetEquivProof_nuSuppOmega}
		& \supp(\ul{\nu}_\omega) \subset [0,8\sigma^4(1+\sqrt{c_*})^2+\tvarepsilon] \ .
	\end{align}
	For every $\omega$ from the intersection $\cE_{\text{(\ref{Eq_DetEquivProof_TighnessBorelCantelli})}} \cap \cE_{\text{(\ref{Eq_DetEquiv_metaCovnergenceEvent})}} \cap \cE_{\text{(\ref{Eq_DetEquivProof_NuStilConv0})}}$ one then has
	\begin{align}\label{Eq_DetEquivProof_NuAlmost}
		& \cs_{\ul{\nu}_\omega}(z) = \frac{-1}{z n} \tr\Big( \Big( \Id_n + \sum\limits_{r',s'=1}^R \delta^{(A)}_{r',s'}(z) B_{s'}^* B_{r'} \Big)^{-1} \Big) \ ,
	\end{align}
	since $\omega \in \cE_{\text{(\ref{Eq_DetEquivProof_TighnessBorelCantelli})}}$ yields
	\begin{align*}
		& \frac{1}{nm_k} \tr\big(\tilde{\bm{R}}^{(\bm{Z}^{(m_k)})}(z)\big) = \cs_{\hat{\ul{\nu}}^{(m_k)}_\omega}(z) \xrightarrow{k \to \infty} \cs_{\ul{\nu}_\omega}(z) \ \text{ by (\ref{Eq_DetEquivProof_NuProkhorov})} \ ,
	\end{align*}
	while $\omega \in \cE_{\text{(\ref{Eq_DetEquivProof_NuStilConv0})}}$ yields
	\begin{align*}
		& \Big| \frac{1}{nm_k} \tr\big(\tilde{\bm{R}}^{(\bm{Z}^{(m_k)})}(z)\big) - \frac{-1}{z n} \tr\Big( \Big( \Id_n + \sum\limits_{r',s'=1}^R \hat{\delta}^{(A,m_k)}_{r',s'}(z) B_{s'}^* B_{r'} \Big)^{-1} \Big) \Big| \xrightarrow{k \to \infty} 0
	\end{align*}
	and $\omega \in \cE_{\text{(\ref{Eq_DetEquivProof_TighnessBorelCantelli})}} \cap \cE_{\text{(\ref{Eq_DetEquiv_metaCovnergenceEvent})}}$ by (\ref{Eq_DetEquivProof_SubSubSequenceExistence4}) yields
	\begin{align*}
		& \Big| \frac{-1}{z n} \tr\Big( \Big( \Id_n + \sum\limits_{r',s'=1}^R \hat{\delta}^{(A,m_k)}_{r',s'}(z) B_{s'}^* B_{r'} \Big)^{-1} \Big)\\
		& \hspace{3cm} - \frac{-1}{z n} \tr\Big( \Big( \Id_n + \sum\limits_{r',s'=1}^R \delta^{(A)}_{r',s'}(z) B_{s'}^* B_{r'} \Big)^{-1} \Big) \Big| \xrightarrow{k \to \infty} 0 \ .
	\end{align*}
	Since (\ref{Eq_DetEquivProof_NuAlmost}) uniquely describes the Stieltjes transform of $\ul{\nu}_\omega$, it is uniquely defined (thus also deterministic) and one may write $\ul{\nu} = \ul{\nu}_\omega$. Existence is clear, since the intersection $\cE_{\text{(\ref{Eq_DetEquivProof_TighnessBorelCantelli})}} \cap \cE_{\text{(\ref{Eq_DetEquiv_metaCovnergenceEvent})}} \cap \cE_{\text{(\ref{Eq_DetEquivProof_NuStilConv0})}}$ of probability-one events can not be empty.
	The calculation
	\begin{align}\label{Eq_DeltaProductCalc}
		& -\frac{1}{z} - \sum\limits_{r,s=1}^R \delta^{(A)}_{r,s}(z) \delta^{(B)}_{r,s}(z) \nonumber\\
		& \overset{\text{(\ref{Eq_DualSystem})}}{=} -\frac{1}{z} + \frac{1}{z}\sum\limits_{r,s=1}^R \delta^{(A)}_{r,s}(z) \frac{1}{n} \tr\bigg( B_s^*B_r \Big( \Id_n + \sum\limits_{r',s'=1}^R \delta^{(A)}_{r',s'}(z) B_{s'}^*B_{r'} \Big)^{-1} \bigg) \nonumber\\
		& = -\frac{1}{z} + \frac{1}{z n} \tr\bigg( \sum\limits_{r,s=1}^R \delta^{(A)}_{r,s}(z) B_s^*B_r \Big( \Id_n + \sum\limits_{r',s'=1}^R \delta^{(A)}_{r',s'}(z) B_{s'}^*B_{r'} \Big)^{-1} \bigg) \nonumber\\
		& = -\frac{1}{z} + \frac{1}{z n} \tr\bigg( \Id_n - \Big( \Id_n + \sum\limits_{r',s'=1}^R \delta^{(A)}_{r',s'}(z) B_{s'}^*B_{r'} \Big)^{-1} \bigg) \nonumber\\
		& = - \frac{1}{z n} \tr\bigg( \Big( \Id_n + \sum\limits_{r',s'=1}^R \delta^{(A)}_{r',s'}(z) B_{s'}^*B_{r'} \Big)^{-1} \bigg)
	\end{align}
	with (\ref{Eq_DetEquivProof_NuAlmost}) shows that $\ul{\nu}=\ul{\nu}_\omega$ indeed satisfies (\ref{Eq_DetEquiv_NuStilProp}). Finally, (\ref{Eq_DetEquivProof_nuSuppOmega}) and the fact that $\tvarepsilon>0$ may be chosen arbitrarily small yields $\supp(\ul{\nu}) \subset [0,8\sigma^4 (1+\sqrt{c_*})^2]$.
\end{itemize}
This concludes the proof of Theorem~\ref{Thm_DetEquiv}. \qed

\section{Proof of Theorem~\ref{Thm_SepCovMP_NonAsymp}}\label{Proof_Thm_SepCovMP_NonAsymp}
Without loss of generality assume
\begin{align}\label{Eq_MainRes_Wlog_delta_eta_kappa}
	& \eta < 1 \ \ \text{ and } \ \ \kappa > 1 \ .
\end{align}
Also introduce the constants
\begin{align}\label{Eq_MainRes_tkappa_ttauChoice}
	& \tilde{\kappa} : = \frac{c_* \sigma^2}{\eta} \ \ \text{ and } \ \ \tilde{\tau} \coloneq \frac{\tau \eta}{2^8\sigma^8(1+\sqrt{c_*})^4 + 4 + 2\kappa^2} \ .
\end{align}
Let $\mathcal{C}_{\text{\ref{Thm_DualApprox}}}$, $\mathcal{C}_{\text{\ref{Thm_EmpiricalDualSystem}}}$ and $\mathcal{C}_{\text{\ref{Thm_Universality}}}$ denote the constants from Theorems~\ref{Thm_DualApprox},~\ref{Thm_EmpiricalDualSystem} and~\ref{Thm_Universality} respectively and note that only $\mathcal{C}_{\text{\ref{Thm_Universality}}}$ depends on $C_6$. Define
\begin{align}\label{Eq_MainRes_DefN}
	& N \coloneq \max\Big( 2, 6\mathcal{C}_{\text{\ref{Thm_EmpiricalDualSystem}}}, \Big( 2R \mathcal{C}_{\text{\ref{Thm_EmpiricalDualSystem}}} \frac{8\kappa^4 c_* \sigma^6 (1 + \frac{c_* \sigma^4}{\eta})^{7}}{\tau^9 \eta \frac{\tau \eta}{2^8\sigma^8(1+\sqrt{c_*})^4 + 4 + 2\kappa^2}} \Big)^{\frac{2}{1-\tvarepsilon}} \Big) \ .
\end{align}
Let $\bm{Z}$ be the Gaussian matrix similar to $\bm{X}$ as constructed in Definition~\ref{Def_SimilarGaussianMatrix}. Before starting the proof, introduce the following lemma.

\begin{lemma}[Spectral bound]\label{Lemma_DetEquivBound}\
	\\
	Suppose~\ref{ItemAssumption_sigmaBound} and~\ref{ItemTempAssumption_NonDegeneracy} hold, then the deterministic equivalents of Theorem~\ref{Thm_DetEquiv} satisfy the bounds
	\begin{align}\label{Eq_DetEquivBoundA}
		& \Big|\Big| \Big( \Id_d + \sum\limits_{r',s'=1}^R \delta^{(B)}_{r',s'}(z) A_{r'} A_{s'}^* \Big)^{-1} \Big|\Big| \leq \frac{(8\sigma^4(1+\sqrt{c_*})^2+|z|)^2}{\tau^2 \Im(z)}
	\end{align}
	and
	\begin{align}\label{Eq_DetEquivBoundB}
		& \Big|\Big| \Big( \Id_n + \sum\limits_{r',s'=1}^R \delta^{(A)}_{r',s'}(z) B_{s'}^* B_{r'} \Big)^{-1} \Big|\Big| \leq \frac{(8\sigma^4(1+\sqrt{c_*})^2+|z|)^2}{\tau^2 \Im(z)}
	\end{align}
	for all $z \in \C^+$.
\end{lemma}

The proof of Theorem~\ref{Thm_SepCovMP_NonAsymp} commences.

\begin{itemize}
	\item[i)] \textit{Properties of the deterministic equivalent $(\delta^{(A)}, \delta^{(B)})$}:\\
	Making use of the properties of $(\delta^{(A)}, \delta^{(B)})$ shown in (c) of Theorem~\ref{Thm_DetEquiv} and of assumption~\ref{ItemTempAssumption_NonDegeneracy}, one observes
	\begin{align*}
		& \lambda_{\min}\big(\Im(\delta^{(A)})\big) \overset{\text{(\ref{Eq_Uniqueness_deltaStil})}}{=} \lambda_{\min}\Big( \int_\R \Im\Big( \frac{1}{\lambda-z} \Big) \, d\rho^{(A)}(\lambda) \Big)\\
		& = \lambda_{\min}\Big( \int_\R \frac{\Im(z)}{|\lambda-z|^2} \, d\rho^{(A)}(\lambda) \Big) \overset{\text{(\ref{Eq_DetEquiv_RhoProp1})}}{\geq} \lambda_{\min}\Big( \int_\R \frac{\Im(z)}{(8\sigma^4(1+\sqrt{c_*})^2 + |z|)^2} \, d\rho^{(A)}(\lambda) \Big)\\
		& \overset{\text{(\ref{Eq_DetEquiv_RhoProp2})}}{=} \frac{\Im(z)}{(8\sigma^4(1+\sqrt{c_*})^2 + |z|)^2} \lambda_{\min}\Big( \frac{1}{n} \tr\big( A_r A_s^* \big)_{r,s \leq R} \Big) \overset{\text{(\ref{Eq_AssumptionIdentifiability})}}{\geq} \frac{\tau \Im(z)}{(8\sigma^4(1+\sqrt{c_*})^2 + |z|)^2} \ ,
	\end{align*}
	which for any $z \in \bD(\eta,\kappa)$ yields
	\begin{align}\label{Eq_MainRes_ImDeltaA_LowerBound}
		& \lambda_{\min}\big(\Im(\delta^{(A)})\big) \geq \frac{\tau \eta}{(8\sigma^4(1+\sqrt{c_*})^2 + \kappa)^2} \geq \frac{\tau \eta}{2^6\sigma^8(1+\sqrt{c_*})^4 + 2\kappa^2}
	\end{align}
	and one may analogously show
	\begin{align}\label{Eq_MainRes_ImDeltaB_LowerBound}
		& \lambda_{\min}\big(\Im(\delta^{(B)})\big) \geq \frac{\tau \eta}{2^6\sigma^8(1+\sqrt{c_*})^4 + 2\kappa^2} \ .
	\end{align}
	One may also bound $\delta^{(A)}$ from above by the calculation
	\begin{align*}
		& ||\delta^{(A)}|| \overset{\text{(\ref{Eq_Uniqueness_deltaStil})}}{=} \Big|\Big| \int_\R \frac{1}{\lambda-z} \, d\rho^{(A)}(\lambda) \Big|\Big| \leq \Big|\Big| \int_\R \overbrace{\frac{1}{|\lambda-z|}}^{\leq \frac{1}{\Im(z)}} \, d\rho^{(A)}(\lambda) \Big|\Big|\\
		& \leq \frac{1}{\Im(z)} ||\rho^{(A)}(\R)|| \overset{\text{(\ref{Eq_DetEquiv_RhoProp2})}}{=} \frac{1}{\Im(z)} \Big|\Big| \frac{1}{n} \tr\big( A_r A_s^* \big)_{r,s \leq R} \Big|\Big| \overset{\text{(\ref{Eq_SimpleBounds_TraceMatrixBounds})}}{\leq} \frac{\sigma^2}{\Im(z)} \underbrace{\frac{d}{n}}_{\overset{\text{\ref{ItemAssumption_cBound}}}{\leq} c_*} \ ,
	\end{align*}
	which for $z \in \bD(\eta,\kappa)$ yields
	\begin{align}\label{Eq_MainRes_DeltaA_UpperBound}
		& ||\delta^{(A)}|| \leq \frac{c_* \sigma^2}{\eta} \ .
	\end{align}
	One may analogously show
	\begin{align}\label{Eq_MainRes_DeltaB_UpperBound}
		& ||\delta^{(B)}|| \leq \frac{\sigma^2}{\eta} \ .
	\end{align}
	
	\item[ii)] \textit{Construction and properties of $(\hat{\delta}^{(A,\bm{Z})}, \hat{\delta}^{(B,\bm{Z})})$}:\\
	As in (\ref{Eq_Def_hatDeltaZ}), define
	\begin{align}\label{Eq_Def_hatDeltaZ_copy}
		& \hat{\delta}^{(A,\bm{Z})}(z)_{r,s} \coloneq \frac{1}{n} \tr\big( A_rA_s^* \bm{R}^{(\bm{Z})}(z) \big) \ \ \text{ and } \ \ \hat{\delta}^{(B,\bm{Z})}(z)_{r,s} \coloneq \frac{1}{n} \tr\big( B_s^*B_r \tilde{\bm{R}}^{(\bm{Z})}(z) \big) \ .
	\end{align}
	Letting $UDU^* = \bm{S}^{(\bm{Z})}$ denote the spectral decomposition of $\bm{S}^{(\bm{Z})}$, one observes
	\begin{align}\label{Eq_R_SpectralDecomp_copy}
		\bm{R}^{(\bm{Z})}(z) & \overset{\text{(\ref{Eq_Def_R_Robust})}}{=} \big( \bm{S}^{(\bm{Z})} - z\Id_d \big)^{-1} = U \big( D - z\Id_d \big)^{-1} U^* \nonumber\\
		& = U \diag\Big( \frac{1}{\lambda_1(\bm{S}^{(\bm{Z})}) - z},\dots,\frac{1}{\lambda_d(\bm{S}^{(\bm{Z})}) - z}\Big) U^* \ ,
	\end{align}
	which yields
	\begin{align*}
		\Im\big( \bm{R}^{(\bm{Z})}(z) \big) & = U \diag\Big( \Im\Big(\frac{1}{\lambda_1(\bm{S}^{(\bm{Z})}) - z}\Big),\dots,\Im\Big(\frac{1}{\lambda_d(\bm{S}^{(\bm{Z})}) - z}\Big)\Big) U^*\\
		& = U \diag\Big( \frac{\Im(z)}{|\lambda_1(\bm{S}^{(\bm{Z})}) - z|^2},\dots,\frac{\Im(z)}{|\lambda_d(\bm{S}^{(\bm{Z})}) - z|^2}\Big) U^* \ ,
	\end{align*}
	thus proving that $\Im\big( \bm{R}^{(\bm{Z})}(z) \big)$ is positive definite with
	\begin{align}\label{Eq_MainRes_ImR_LowerBound}
		& \lambda_{\min}\Big( \Im\big( \bm{R}^{(\bm{Z})}(z) \big) \Big) = \frac{\Im(z)}{\max\limits_{j\leq d} |\lambda_j(\bm{S}^{(\bm{Z})}) - z|^2} \nonumber\\
		& \geq \frac{\Im(z)}{2||\bm{S}^{(\bm{Z})}||^2 + 2|z|^2} \geq \frac{\eta}{2||\bm{S}^{(\bm{Z})}||^2 + 2\kappa^2} \ .
	\end{align}
	Using $||\tilde{\bm{S}}^{(\bm{Z})}|| = ||\bm{S}^{(\bm{Z})}||$, one may analogously show that $\Im\big( \tilde{\bm{R}}^{(\bm{Z})}(z) \big)$ is positive definite with
	\begin{align}\label{Eq_MainRes_ImtR_LowerBound}
		& \lambda_{\min}\Big( \Im\big( \tilde{\bm{R}}^{(\bm{Z})}(z) \big) \Big) \geq \frac{\eta}{2||\bm{S}^{(\bm{Z})}||^2 + 2\kappa^2} \ .
	\end{align}
	Making use of the bound $\tr(A^* M A) \geq \lambda_{\min}(M) \tr(A^*A)$, which holds for any matrix $A$ and positive definite matrix $M$, one may calculate
	\begin{align}\label{Eq_MainRes_ImHatDeltaA_LowerBound}
		& \lambda_{\min}\Big( \Im\big( \hat{\delta}^{(A,\bm{Z})}(z) \big) \Big) \overset{\text{(\ref{Eq_Def_hatDeltaZ_copy})}}{=} \lambda_{\min} \Big( \frac{1}{n}\tr\big( A_r A_s^* \Im\big( \bm{R}^{(\bm{Z})}(z) \big) \big)_{r,s\leq R} \Big) \nonumber\\
		& \overset{\text{(\ref{Eq_NonDegeneracy_TrCalcA})}}{\geq} \tau \lambda_{\min}\Big( \Im\big( \bm{R}^{(\bm{Z})}(z) \big) \Big) \overset{\text{(\ref{Eq_MainRes_ImR_LowerBound})}}{\geq} \frac{\tau \eta}{2||\bm{S}^{(\bm{Z})}||^2 + 2\kappa^2} \ .
	\end{align}
	Analogously, one with (\ref{Eq_MainRes_ImtR_LowerBound}) proves
	\begin{align}\label{Eq_MainRes_ImHatDeltaB_LowerBound}
		& \lambda_{\min}\Big( \Im\big( \hat{\delta}^{(B,\bm{Z})}(z) \big) \Big) \geq \frac{\tau \eta}{2||\bm{S}^{(\bm{Z})}||^2 + 2\kappa^2} \ .
	\end{align}
	One may further bound $\hat{\delta}^{(A,\bm{Z})}(z)$ and $\hat{\delta}^{(B,\bm{Z})}(z)$ from above by
	\begin{align}\label{Eq_MainRes_HatDeltaA_UpperBound}
		& ||\hat{\delta}^{(A,\bm{Z})}(z)|| = \Big|\Big| \frac{1}{n}\tr\big( A_r A_s^* \bm{R}^{(\bm{Z})}(z) \big)_{r,s\leq R} \Big|\Big| \nonumber\\
		& \overset{\text{(\ref{Eq_SimpleBounds_TraceMatrixBounds})}}{\leq} c_* \sigma^2 \big|\big| \bm{R}^{(\bm{Z})}(z) \big|\big| \overset{\text{(\ref{Eq_R_SpectralBounds})}}{\leq} \frac{c_* \sigma^2}{\Im(z)} \leq \frac{c_* \sigma^2}{\eta}
	\end{align}
	and
	\begin{align}\label{Eq_MainRes_HatDeltaB_UpperBound}
		& ||\hat{\delta}^{(B,\bm{Z})}(z)|| = \Big|\Big| \frac{1}{n}\tr\big( B_s^*B_r \tilde{\bm{R}}^{(\bm{Z})}(z) \big)_{r,s\leq R} \Big|\Big| \nonumber\\
		& \overset{\text{(\ref{Eq_SimpleBounds_TraceMatrixBounds})}}{\leq} \sigma^2 \big|\big| \tilde{\bm{R}}^{(\bm{Z})}(z) \big|\big| \overset{\text{(\ref{Eq_R_SpectralBounds})}}{\leq} \frac{\sigma^2}{\Im(z)} \leq \frac{\sigma^2}{\eta} \ .
	\end{align}
	
	\item[iii)] \textit{Showing (\ref{Eq_MainRes_Result}) in the Gaussian case by Theorems~\ref{Thm_DualApprox} and~\ref{Thm_EmpiricalDualSystem}}:\\
	For any $\tvarepsilon \in (0,1)$, the result of Theorem~\ref{Thm_EmpiricalDualSystem} in the notation of (\ref{Eq_DualApprox_DualSystem}) from Theorem~\ref{Thm_DualApprox} yields
	\begin{align}\label{Eq_MainRes_EmpDualSystem}
		& \bP\Big( \forall z \in \bD(\eta,\kappa) : \ \big( \hat{\delta}^{(A,\bm{Z})}, \hat{\delta}^{(B,\bm{Z})} \big) \text{ satisfy (\ref{Eq_DualApprox_DualSystem}) for error terms $\hat{q}^{(A)}$ and $\hat{q}^{(B)}$ with} \nonumber\\
		& \hspace{3cm} ||\hat{q}^{(A)}||, ||\hat{q}^{(B)}|| \leq R \big( \mathcal{C}_{\text{\ref{Thm_EmpiricalDualSystem}}} n^{\frac{\tvarepsilon-1}{2}} + \mathcal{C}_{\text{\ref{Thm_EmpiricalDualSystem}}} n^2 \exp(-n/\mathcal{C}_{\text{\ref{Thm_EmpiricalDualSystem}}}) \big) \Big) \nonumber\\
		& \geq 1 - 2\mathcal{C}_{\text{\ref{Thm_EmpiricalDualSystem}}} R^2 n \exp\Big( -\frac{n^{\tvarepsilon}}{\mathcal{C}_{\text{\ref{Thm_EmpiricalDualSystem}}} R^2} \Big) \ .
	\end{align}
	Before $\omega$-wise application of Theorem~\ref{Thm_DualApprox}, its prerequisites are checked.
	Lemma~\ref{Lemma_ZTailBound} with $t= \sqrt{\frac{n}{2}}$ yields
	\begin{align*}
		& \bP\Big( ||\bm{Y}^{(\bm{Z})}|| > 2\sigma^2 \big(\sqrt{n}+\sqrt{d}\big) + \sqrt{\frac{n}{2}} \Big) \leq 2 \exp\Big( -\frac{Cn}{2\sigma^4} \Big)
	\end{align*}
	for a universal constant $C>0$. Since
	\begin{align*}
		& ||\bm{Y}^{(\bm{Z})}|| > 2\sigma^2 \big(\sqrt{n}+\sqrt{d}\big) + \sqrt{\frac{n}{2}}\\
		& \Leftrightarrow \ ||\bm{Y}^{(\bm{Z})}||^2 > \Big(2\sigma^2 \big(\sqrt{n}+\sqrt{d}\big) + \sqrt{\frac{n}{2}}\Big)^2\\
		& \Leftarrow \ ||\bm{Y}^{(\bm{Z})}||^2 > 8\sigma^4 \big(\sqrt{n}+\sqrt{d}\big)^2 + 2\frac{n}{2}\\
		& \overset{\text{(\ref{Eq_Def_S_Robust})}}{\Leftrightarrow} \ ||\bm{S}^{(\bm{Z})}|| > 8\sigma^4 \frac{(\sqrt{n}+\sqrt{d})^2}{n} + 1\\
		& \overset{\text{\ref{ItemAssumption_cBound}}}{\Leftarrow} \ ||\bm{S}^{(\bm{Z})}|| > 8\sigma^4 (1+\sqrt{c_*})^2 + 1\\
		& \Leftrightarrow \ ||\bm{S}^{(\bm{Z})}||^2 > \big(8\sigma^4 (1+\sqrt{c_*})^2 + 1\big)^2\\
		& \Leftarrow \ ||\bm{S}^{(\bm{Z})}||^2 > 2^7\sigma^8 (1+\sqrt{c_*})^4 + 2
	\end{align*}
	it follows that
	\begin{align}\label{Eq_MainRes_S_SpectralBound}
		& \bP\Big( ||\bm{S}^{(\bm{Z})}||^2 \leq 2^7\sigma^8 (1+\sqrt{c_*})^4 + 2 \Big) \geq 1 - 2 \exp\Big( -\frac{C n}{2\sigma^4} \Big) \ ,
	\end{align}
	The bounds (\ref{Eq_MainRes_DeltaA_UpperBound}), (\ref{Eq_MainRes_DeltaB_UpperBound}), (\ref{Eq_MainRes_HatDeltaA_UpperBound}) and (\ref{Eq_MainRes_HatDeltaB_UpperBound}) show that the prerequisites (\ref{Eq_DualApprox_BoundCondition_delta}) and (\ref{Eq_DualApprox_BoundCondition_tdelta}) with $(\hat{\delta}^{(A,\bm{Z})}, \hat{\delta}^{(B,\bm{Z})})$ hold for the choice
	\begin{align}\label{Eq_MainRes_tkappaChoice_copy}
		& \tilde{\kappa} \overset{\text{(\ref{Eq_MainRes_tkappa_ttauChoice})}}{=} \frac{c_* \sigma^2}{\eta} \overset{c_* \geq 1}{=} \max\Big( \frac{c_* \sigma^2}{\eta}, \frac{\sigma^2}{\eta} \Big) \ .
	\end{align}
	Similarly, the bounds (\ref{Eq_MainRes_ImDeltaA_LowerBound}), (\ref{Eq_MainRes_ImDeltaB_LowerBound}), (\ref{Eq_MainRes_ImHatDeltaA_LowerBound}) and (\ref{Eq_MainRes_ImHatDeltaB_LowerBound}) for each $\omega$ from the event $\cE_{\text{(\ref{Eq_MainRes_S_SpectralBound})}}$ show that the prerequisites (\ref{Eq_DualApprox_posDefCondition_delta}) and (\ref{Eq_DualApprox_posDefCondition_tdelta}) with $(\hat{\delta}^{(A,\bm{Z})}, \hat{\delta}^{(B,\bm{Z})})$ hold for the choice
	\begin{align}\label{Eq_MainRes_ttauChoice_copy}
		& \tilde{\tau} \overset{\text{(\ref{Eq_MainRes_tkappa_ttauChoice})}}{=} \frac{\tau \eta}{2^8\sigma^8(1+\sqrt{c_*})^4 + 4 + 2\kappa^2} \nonumber\\
		& \overset{\omega \in \cE_{\text{(\ref{Eq_MainRes_S_SpectralBound})}}}{\leq} \min\Big( \frac{\tau \eta}{2^6\sigma^8(1+\sqrt{c_*})^4 + 2\kappa^2}, \frac{\tau \eta}{2||\bm{S}^{(\bm{Z}(\omega))}||^2 + 2\kappa^2} \Big) \ .
	\end{align}
	It remains to check that the prerequisites (\ref{Eq_DualApprox_qA_Assumption})-(\ref{Eq_DualApprox_tqSmall}) hold for $\hat{q}^{(A)}$ and $\hat{q}^{(B)}$, for which it by (\ref{Eq_MainRes_EmpDualSystem}) suffices to prove
	\begin{align}
		& R \mathcal{C}_{\text{\ref{Thm_EmpiricalDualSystem}}} n^{\frac{\tvarepsilon-1}{2}} + R \mathcal{C}_{\text{\ref{Thm_EmpiricalDualSystem}}} n^2 \exp(-n/\mathcal{C}_{\text{\ref{Thm_EmpiricalDualSystem}}}) \leq \frac{\tau^2 \tilde{\tau}}{2|z|(1 + \tilde{\kappa} \sigma^2)^2} \label{Eq_DualApprox_qA_Assumption_ToProve}\\
		& R \mathcal{C}_{\text{\ref{Thm_EmpiricalDualSystem}}} n^{\frac{\tvarepsilon-1}{2}} + R \mathcal{C}_{\text{\ref{Thm_EmpiricalDualSystem}}} n^2 \exp(-n/\mathcal{C}_{\text{\ref{Thm_EmpiricalDualSystem}}}) \leq \frac{\tau^7 \tilde{\tau} \Im(z)}{8|z|^4 \sigma^2 (1 + \tilde{\kappa} \sigma^2)^{7}} \label{Eq_DualApprox_qSmall_ToProve}\\
		& R \mathcal{C}_{\text{\ref{Thm_EmpiricalDualSystem}}} n^{\frac{\tvarepsilon-1}{2}} + R \mathcal{C}_{\text{\ref{Thm_EmpiricalDualSystem}}} n^2 \exp(-n/\mathcal{C}_{\text{\ref{Thm_EmpiricalDualSystem}}}) \leq \frac{\tau^9 \tilde{\tau} \Im(z)}{8|z|^3 c_* \sigma^6 (1 + \tilde{\kappa} \sigma^2)^{7}} \label{Eq_DualApprox_tqSmall_ToProve} \ .
	\end{align}
	With the observation
	\begin{align}\label{Eq_MainRes_nLargeEnough}
		& n^{\frac{\tvarepsilon-1}{2}} \geq n^2 \exp(-n/\mathcal{C}_{\text{\ref{Thm_EmpiricalDualSystem}}}) \ \Leftrightarrow \ \frac{\tvarepsilon-1}{2} \log(n) \geq 2\log(n) - \frac{n}{\mathcal{C}_{\text{\ref{Thm_EmpiricalDualSystem}}}} \nonumber\\
		& \Leftrightarrow \ \frac{5-\tvarepsilon}{2} \log(n) \leq \frac{n}{\mathcal{C}_{\text{\ref{Thm_EmpiricalDualSystem}}}} \ \overset{\tvarepsilon \leq 1}{\Leftarrow} \ 3 \log(n) \leq \frac{n}{\mathcal{C}_{\text{\ref{Thm_EmpiricalDualSystem}}}} \ \Leftrightarrow \ 3\mathcal{C}_{\text{\ref{Thm_EmpiricalDualSystem}}} \leq \frac{n}{\log(n)} \nonumber\\
		& \Leftarrow \ (n \geq 2) \text{ and } \Big(3\mathcal{C}_{\text{\ref{Thm_EmpiricalDualSystem}}} \leq \frac{n}{2}\Big) \ \overset{\text{(\ref{Eq_MainRes_DefN})}}{\Leftarrow} \ n \geq N \ ,
	\end{align}
	one for all $n \geq N$ and $z \in \bD(\eta,\kappa)$ checks the prerequisites (\ref{Eq_DualApprox_qA_Assumption_ToProve})-(\ref{Eq_DualApprox_tqSmall_ToProve}) by the calculation
	\begin{align*}
		& R \mathcal{C}_{\text{\ref{Thm_EmpiricalDualSystem}}} n^{\frac{\tvarepsilon-1}{2}} + R \mathcal{C}_{\text{\ref{Thm_EmpiricalDualSystem}}} n^2 \exp(-n/\mathcal{C}_{\text{\ref{Thm_EmpiricalDualSystem}}}) \overset{\text{(\ref{Eq_MainRes_nLargeEnough})}}{\leq} 2R \mathcal{C}_{\text{\ref{Thm_EmpiricalDualSystem}}} n^{\frac{\tvarepsilon-1}{2}} \overset{\tvarepsilon < 1}{\leq} 2R \mathcal{C}_{\text{\ref{Thm_EmpiricalDualSystem}}} N^{\frac{\tvarepsilon-1}{2}}\\
		& \overset{\text{(\ref{Eq_MainRes_DefN})}}{\leq} \frac{\tau^9 \eta \frac{\tau \eta}{2^8\sigma^8(1+\sqrt{c_*})^4 + 4 + 2\kappa^2}}{8\kappa^4 c_* \sigma^6 (1 + \frac{c_* \sigma^2}{\eta} \sigma^2)^{7}} \overset{\text{(\ref{Eq_MainRes_tkappa_ttauChoice})}}{=} \frac{\tau^9 \tilde{\tau} \eta}{8\kappa^4 c_* \sigma^6 (1 + \tilde{\kappa} \sigma^2)^{7}}\\
		& \overset{*}{\leq} \min\Big( \frac{\tau^2 \tilde{\tau}}{2\kappa (1 + \tilde{\kappa} \sigma^2)^{2}}, \frac{\tau^7 \tilde{\tau} \eta}{8\kappa^4 \sigma^2 (1 + \tilde{\kappa} \sigma^2)^{7}}, \frac{\tau^9 \tilde{\tau} \eta}{8\kappa^3 c_* \sigma^6 (1 + \tilde{\kappa} \sigma^2)^{7}} \Big)\\
		& \hspace{-0.45cm} \overset{z \in \bD(\eta,\kappa)}{\leq} \min\Big( \frac{\tau^2 \tilde{\tau}}{2|z| (1 + \tilde{\kappa} \sigma^2)^{2}}, \frac{\tau^7 \tilde{\tau} \Im(z)}{8|z|^4 \sigma^2 (1 + \tilde{\kappa} \sigma^2)^{7}}, \frac{\tau^9 \tilde{\tau} \Im(z)}{8|z|^3 c_* \sigma^6 (1 + \tilde{\kappa} \sigma^2)^{7}} \Big) \ ,
	\end{align*}
	where the bound marked with $*$ used the facts that $\kappa \overset{\text{(\ref{Eq_MainRes_Wlog_delta_eta_kappa})}}{>} 1$, $c_*,\sigma^2 \geq 1$, $\eta \overset{\text{(\ref{Eq_MainRes_Wlog_delta_eta_kappa})}}{<} 1$ and $\tau < 1$.
	One may thus for every $\omega \in \cE_{\text{(\ref{Eq_MainRes_EmpDualSystem})}} \cap \cE_{\text{(\ref{Eq_MainRes_S_SpectralBound})}}$ apply Theorem~\ref{Thm_DualApprox} to get
	\begin{align*}
		& \big|\big| \hat{\delta}^{(A,\bm{Z})}(z;\omega) - \delta^{(A)}(z) \big|\big| \overset{\text{(\ref{Eq_DualApprox_Result_deltaA})}}{\leq} \mathcal{C}_{\text{\ref{Thm_DualApprox}}} \big( ||\hat{q}^{(A)}(\omega)|| + ||\hat{q}^{(B)}(\omega)|| \big) \overset{\substack{\text{(\ref{Eq_MainRes_EmpDualSystem})} \\ \text{(\ref{Eq_MainRes_nLargeEnough})}}}{\leq} 4R\mathcal{C}_{\text{\ref{Thm_DualApprox}}} \mathcal{C}_{\text{\ref{Thm_EmpiricalDualSystem}}} n^{\frac{\tvarepsilon-1}{2}} \nonumber\\
		& \big|\big| \hat{\delta}^{(B,\bm{Z})}(z;\omega) - \delta^{(B)}(z) \big|\big| \overset{\text{(\ref{Eq_DualApprox_Result_deltaB})}}{\leq} \mathcal{C}_{\text{\ref{Thm_DualApprox}}} \big( ||\hat{q}^{(A)}(\omega)|| + ||\hat{q}^{(B)}(\omega)|| \big) \overset{\substack{\text{(\ref{Eq_MainRes_EmpDualSystem})} \\ \text{(\ref{Eq_MainRes_nLargeEnough})}}}{\leq} 4R\mathcal{C}_{\text{\ref{Thm_DualApprox}}} \mathcal{C}_{\text{\ref{Thm_EmpiricalDualSystem}}} n^{\frac{\tvarepsilon-1}{2}}
	\end{align*}
	for any $n \geq N$ and $z \in \bD(\eta,\kappa)$. This proves
	\begin{align}\label{Eq_MainRes_detltaZ_DetEquiv_Approx}
		& \bP\Big( \forall z \in \bD(\eta,\kappa) : \ \big|\big| \hat{\delta}^{(A,\bm{Z})}(z) - \delta^{(A)}(z) \big|\big| \leq 4R\mathcal{C}_{\text{\ref{Thm_DualApprox}}} \mathcal{C}_{\text{\ref{Thm_EmpiricalDualSystem}}} n^{\frac{\tvarepsilon-1}{2}} \nonumber\\
		& \hspace{2.5cm} \text{ and } \big|\big| \hat{\delta}^{(B,\bm{Z})}(z) - \delta^{(B)}(z) \big|\big| \leq 4R\mathcal{C}_{\text{\ref{Thm_DualApprox}}} \mathcal{C}_{\text{\ref{Thm_EmpiricalDualSystem}}} n^{\frac{\tvarepsilon-1}{2}} \Big) \nonumber\\
		& \geq \bP\big( \cE_{\text{(\ref{Eq_MainRes_EmpDualSystem})}} \cap \cE_{\text{(\ref{Eq_MainRes_S_SpectralBound})}} \big) \overset{\substack{\text{(\ref{Eq_MainRes_EmpDualSystem})} \\ \text{(\ref{Eq_MainRes_S_SpectralBound})}}}{\geq} 1 - 2\mathcal{C}_{\text{\ref{Thm_EmpiricalDualSystem}}} R^2 n \exp\Big( -\frac{n^{\tvarepsilon}}{\mathcal{C}_{\text{\ref{Thm_EmpiricalDualSystem}}} R^2} \Big) - 2 \exp\Big( -\frac{C n}{2\sigma^4} \Big)
	\end{align}
	for all $n \geq N$.
	
	\item[iv)] \textit{Proving (\ref{Eq_MainRes_Result_M}) in the Gaussian case by Theorem~\ref{Thm_EmpiricalDualSystem}}:\\
	Combining results (\ref{Eq_EmpDualGaussian_ResultA_M}) and (\ref{Eq_EmpDualGaussian_ResultB_M}) from Theorem~\ref{Thm_EmpiricalDualSystem} with the calculation (\ref{Eq_MainRes_nLargeEnough}) for all $n \geq N$ gives
	\begin{align}\label{Eq_MainRes_MDiff_Gaussian}
		& \bP\Big( \forall z \in \bD(\eta,\kappa) : \nonumber\\
		&\hspace{0.5cm} \Big| \frac{1}{n} \tr\big( M \bm{R}^{(\bm{Z})}(z) \big) - \frac{-1}{z} \frac{1}{n} \tr\Big( M \Big( \overbrace{\Id_d + \sum\limits_{r',s'=1}^R \hat{\delta}^{(B,\bm{Z})}_{r',s'}(z) A_{r'} A_{s'}^*}^{= \bm{Q}^{(\bm{Z})}(z)} \Big)^{-1} \Big) \Big| < 2\mathcal{C}_{\text{\ref{Thm_EmpiricalDualSystem}}} n^{\frac{\tvarepsilon-1}{2}} , \nonumber\\
		&\hspace{0.5cm} \Big| \frac{1}{n} \tr\big( \tilde{M} \tilde{\bm{R}}^{(\bm{Z})}(z) \big) - \frac{-1}{z} \frac{1}{n} \tr\Big( \tilde{M} \Big( \underbrace{\Id_n + \sum\limits_{r',s'=1}^R \hat{\delta}^{(A,\bm{Z})}_{r',s'}(z) B_{s'}^* B_{r'}}_{= \tilde{\bm{Q}}^{(\bm{Z})}(z)} \Big)^{-1} \Big) \Big| < 2\mathcal{C}_{\text{\ref{Thm_EmpiricalDualSystem}}} n^{\frac{\tvarepsilon-1}{2}} \Big) \nonumber\\
		& \geq 1 - 2\mathcal{C}_{\text{\ref{Thm_EmpiricalDualSystem}}} n \exp\Big( -\frac{n^{\tvarepsilon}}{\mathcal{C}_{\text{\ref{Thm_EmpiricalDualSystem}}} R^2} \Big) \ .
	\end{align}
	For every $\omega \in \cE_{\text{(\ref{Eq_MainRes_S_SpectralBound})}}$, one calculates
	\begin{align}\label{Eq_MainRes_QDiffA_Gaussian}
		& \Big|\Big| \Big( \overbrace{\Id_d + \sum\limits_{r',s'=1}^R \hat{\delta}^{(B,\bm{Z}(\omega))}_{r',s'}(z) A_{r'} A_{s'}^*}^{= \bm{Q}^{(\bm{Z}(\omega))}(z)} \Big)^{-1} - \Big( \Id_d + \sum\limits_{r',s'=1}^R \delta^{(B)}_{r',s'}(z) A_{r'} A_{s'}^* \Big)^{-1} \Big|\Big| \nonumber\\
		& \overset{\text{(\ref{Eq_InverseDifferenceIdentity})}}{=} \Big|\Big| \bm{Q}^{(\bm{Z}(\omega))}(z)^{-1} \Big( \sum\limits_{r',s'=1}^R \big(\delta^{(B)}_{r',s'}(z) - \hat{\delta}^{(B,\bm{Z}(\omega))}_{r',s'}(z)\big) A_{r'} A_{s'}^* \Big) \nonumber\\
		& \hspace{1cm} \times \Big( \Id_d + \sum\limits_{r',s'=1}^R \delta^{(B)}_{r',s'}(z) A_{r'} A_{s'}^* \Big)^{-1}\Big|\Big| \nonumber\\
		& \overset{\substack{\text{(\ref{Eq_QBounds})} \\ \text{(\ref{Eq_DetEquivBoundA})}}}{\leq} \frac{\overbrace{(|z|+||\bm{S}^{(\bm{Z}(\omega))}||)^2}^{\leq 2|z|^2 + 2||\bm{S}^{(\bm{Z}(\omega))}||^2}}{\tau^2 \Im(z)} \, \overbrace{\Big|\Big| \sum\limits_{r',s'=1}^R \big(\delta^{(B)}_{r',s'}(z) - \hat{\delta}^{(B,\bm{Z}(\omega))}_{r',s'}(z)\big) A_{r'} A_{s'}^* \Big|\Big|}^{\overset{\text{(\ref{Eq_SimpleBounds_ASum_BSum})}}{\leq} \sigma^2 ||\delta^{(B)}(z) - \hat{\delta}^{(B,\bm{Z}(\omega))}||} \nonumber\\
		& \hspace{1cm} \times \frac{(8\sigma^4(1+\sqrt{c_*})^2+|z|)^2}{\tau^2 \Im(z)} \nonumber\\
		& \overset{\text{(\ref{Eq_MainRes_S_SpectralBound})}}{\leq} \frac{(2|z|^2+2^8\sigma^8 (1+\sqrt{c_*})^4 + 4)^2}{\tau^4 \Im(z)^2} \, \sigma^2 ||\delta^{(B)}(z) - \hat{\delta}^{(B,\bm{Z}(\omega))}|| \nonumber\\
		& \overset{z \in \bD(\eta,\kappa)}{\leq} \sigma^2 \frac{(2\kappa^2+2^8\sigma^8 (1+\sqrt{c_*})^4 + 4)^2}{\tau^4 \eta^2} ||\delta^{(B)}(z) - \hat{\delta}^{(B,\bm{Z}(\omega))}||
	\end{align}
	and it may for every $\omega \in \cE_{\text{(\ref{Eq_MainRes_S_SpectralBound})}}$ analogously be shown that
	\begin{align}\label{Eq_MainRes_QDiffB_Gaussian}
		& \Big|\Big|\Big( \overbrace{\Id_n + \sum\limits_{r',s'=1}^R \hat{\delta}^{(A,\bm{Z}(\omega))}_{r',s'}(z) B_{s'}^*B_{r'}}^{= \tilde{\bm{Q}}^{(\bm{Z}(\omega))}(z)} \Big)^{-1} - \Big( \Id_n + \sum\limits_{r',s'=1}^R \delta^{(A)}_{r',s'}(z) B_{s'}^*B_{r'} \Big)^{-1}\Big|\Big| \nonumber\\
		& \leq \sigma^2 \frac{(2\kappa^2+2^8\sigma^8 (1+\sqrt{c_*})^4 + 4)^2}{\tau^4 \eta^2} ||\delta^{(A)}(z) - \hat{\delta}^{(A,\bm{Z}(\omega))}|| \ ,
	\end{align}
	which by (\ref{Eq_MainRes_detltaZ_DetEquiv_Approx}) implies
	\begin{align}\label{Eq_MainRes_QDiff_Gaussian}
		& \bP\Big( \forall z \in \bD(\eta,\kappa) : \nonumber\\
		& \hspace{0.5cm} \Big| \frac{-1}{z} \frac{1}{n} \tr\big( M \bm{Q}^{(\bm{Z})}(z)^{-1} \big) - \frac{-1}{z} \frac{1}{n} \tr\Big( M \Big( \Id_d + \sum\limits_{r',s'=1}^R \delta^{(B)}_{r',s'}(z) A_{r'} A_{s'}^* \Big)^{-1} \Big) \Big| \nonumber\\
		& \hspace{3cm} \leq \underbrace{\frac{d ||M||}{n |z|}}_{\leq \frac{c_* \sigma^2}{\eta}} \sigma^2 \frac{(2\kappa^2+2^8\sigma^8 (1+\sqrt{c_*})^4 + 4)^2}{\tau^4 \eta^2} 4R\mathcal{C}_{\text{\ref{Thm_DualApprox}}} \mathcal{C}_{\text{\ref{Thm_EmpiricalDualSystem}}} n^{\frac{\tvarepsilon-1}{2}} \nonumber\\
		& \text{ and } \Big| \frac{-1}{z} \frac{1}{n} \tr\big( \tilde{M} \tilde{\bm{Q}}^{(\bm{Z})}(z)^{-1} \big) - \frac{-1}{z} \frac{1}{n} \tr\Big( \tilde{M} \Big( \Id_n + \sum\limits_{r',s'=1}^R \delta^{(A)}_{r',s'}(z) B_{s'}^*B_{r'} \Big)^{-1} \Big) \Big| \nonumber\\
		& \hspace{3cm} \leq \underbrace{\frac{n ||\tilde{M}||}{n |z|}}_{\leq \frac{c_* \sigma^2}{\eta}} \sigma^2 \frac{(2\kappa^2+2^8\sigma^8 (1+\sqrt{c_*})^4 + 4)^2}{\tau^4 \eta^2} 4R\mathcal{C}_{\text{\ref{Thm_DualApprox}}} \mathcal{C}_{\text{\ref{Thm_EmpiricalDualSystem}}} n^{\frac{\tvarepsilon-1}{2}} \Big) \nonumber\\
		& \geq \bP\big( \cE_{\text{(\ref{Eq_MainRes_EmpDualSystem})}} \cap \cE_{\text{(\ref{Eq_MainRes_S_SpectralBound})}} \big) \overset{\substack{\text{(\ref{Eq_MainRes_EmpDualSystem})} \\ \text{(\ref{Eq_MainRes_S_SpectralBound})}}}{\geq} 1 - 2\mathcal{C}_{\text{\ref{Thm_EmpiricalDualSystem}}} R^2 n \exp\Big( -\frac{n^{\tvarepsilon}}{\mathcal{C}_{\text{\ref{Thm_EmpiricalDualSystem}}} R^2} \Big) - 2 \exp\Big( -\frac{C n}{2\sigma^4} \Big) \ .
	\end{align}
	Combining (\ref{Eq_MainRes_MDiff_Gaussian}) with (\ref{Eq_MainRes_QDiff_Gaussian}) yields
	\begin{align}\label{Eq_MainRes_Diff_Gaussian}
		& \bP\Big( \forall z \in \bD(\eta,\kappa) : \nonumber\\
		& \hspace{0.5cm} \Big| \frac{1}{n} \tr\big( M \bm{R}^{(\bm{Z})}(z) \big) - \frac{-1}{z} \frac{1}{n} \tr\Big( M \Big( \Id_d + \sum\limits_{r',s'=1}^R \delta^{(B)}_{r',s'}(z) A_{r'} A_{s'}^* \Big)^{-1} \Big) \Big| \nonumber\\
		& \hspace{2cm} \leq 2\mathcal{C}_{\text{\ref{Thm_EmpiricalDualSystem}}} n^{\frac{\tvarepsilon-1}{2}} + \frac{c_* \sigma^4}{\eta} \frac{(2\kappa^2+2^8\sigma^8 (1+\sqrt{c_*})^4 + 4)^2}{\tau^4 \eta^2} 4\mathcal{C}_{\text{\ref{Thm_DualApprox}}} \mathcal{C}_{\text{\ref{Thm_EmpiricalDualSystem}}} n^{\frac{\tvarepsilon-1}{2}} \nonumber\\
		& \text{ and } \Big| \frac{1}{n} \tr\big( \tilde{M} \tilde{\bm{R}}^{(\bm{Z})}(z) \big) - \frac{-1}{z} \frac{1}{n} \tr\Big( \tilde{M} \Big( \Id_n + \sum\limits_{r',s'=1}^R \delta^{(A)}_{r',s'}(z) B_{s'}^*B_{r'} \Big)^{-1} \Big) \Big| \nonumber\\
		& \hspace{2cm} \leq 2\mathcal{C}_{\text{\ref{Thm_EmpiricalDualSystem}}} n^{\frac{\tvarepsilon-1}{2}} + \frac{c_* \sigma^4}{\eta} \frac{(2\kappa^2+2^8\sigma^8 (1+\sqrt{c_*})^4 + 4)^2}{\tau^4 \eta^2} 4\mathcal{C}_{\text{\ref{Thm_DualApprox}}} \mathcal{C}_{\text{\ref{Thm_EmpiricalDualSystem}}} n^{\frac{\tvarepsilon-1}{2}} \Big) \nonumber\\
		& \geq 1 - 4\mathcal{C}_{\text{\ref{Thm_EmpiricalDualSystem}}} R^2 n \exp\Big( -\frac{n^{\tvarepsilon}}{\mathcal{C}_{\text{\ref{Thm_EmpiricalDualSystem}}} R^2} \Big) - 2 \exp\Big( -\frac{C n}{2\sigma^4} \Big) \ .
	\end{align}
	
	\item[v)] \textit{Application of Theorem~\ref{Thm_Universality}}:\\
	Theorem~\ref{Thm_Universality} by $||M|| \leq \sigma^2 \geq ||\tilde{M}||$ directly yields
	\begin{align}\label{Eq_Universality_Results_M}
		& \bP\Big( \sup\limits_{z \in \bD(\eta,\kappa)}\Big| \frac{1}{n} \tr(M\bm{R}^{(\bm{X})}(z)) - \frac{1}{n} \tr(M\bm{R}^{(\bm{Z})}(z)) \Big| < \mathcal{C}_{\text{\ref{Thm_Universality}}} n^{\frac{\tvarepsilon-1}{2}} \sigma^2 \nonumber\\
		& \hspace{0.5cm} \text{and } \sup\limits_{z \in \bD(\eta,\kappa)}\Big| \frac{1}{n} \tr(\tilde{M}\tilde{\bm{R}}^{(\bm{X})}(z)) - \frac{1}{n} \tr(\tilde{M}\tilde{\bm{R}}^{(\bm{Z})}(z)) \Big| < \mathcal{C}_{\text{\ref{Thm_Universality}}} n^{\frac{\tvarepsilon-1}{2}} \sigma^2 \Big) \nonumber\\
		& \geq 1 - 2\mathcal{C}_{\text{\ref{Thm_Universality}}} n \exp\Big( -\frac{n^{\tvarepsilon}}{\mathcal{C}_{\text{\ref{Thm_Universality}}} R^2} \Big)
	\end{align}
	For any $r,s \leq R$, the matrices $A_rA_s^*$ and $B_s^*B_r$ by~\ref{ItemAssumption_sigmaBound} satisfy
	\begin{align*}
		& ||A_rA_s^*|| \leq \sigma^2 \ \ \text{ and } \ \ ||B_s^*B_r|| \leq \sigma^2 \ .
	\end{align*}
	Applying Theorem~\ref{Thm_Universality} with $M = A_rA_s^*$ and $\tilde{M} = B_s^*B_r$ to each choice of $r,s \leq R$ yields
	\begin{align}\label{Eq_Universality_Result1_copy}
		& \bP\Big( \forall r,s \leq R : \ \sup\limits_{z \in \bD(\eta,\kappa)}\Big| \overbrace{\frac{1}{n} \tr(A_rA_s^*\bm{R}^{(\bm{X})}(z))}^{\overset{\text{(\ref{Eq_Def_hatDelta_A})}}{=} \hat{\delta}^{(A)}_{r,s}(z)} - \overbrace{\frac{1}{n} \tr(A_rA_s^*\bm{R}^{(\bm{Z})}(z))}^{\overset{\text{(\ref{Eq_Def_hatDeltaZ_copy})}}{=} \hat{\delta}^{(A,\bm{Z})}_{r,s}(z)} \Big| < \mathcal{C}_{\text{\ref{Thm_Universality}}} n^{\frac{\tvarepsilon-1}{2}} \sigma^2 \Big) \nonumber\\
		& \geq 1 - R^2 \mathcal{C}_{\text{\ref{Thm_Universality}}} n \exp\Big( -\frac{n^{\tvarepsilon}}{\mathcal{C}_{\text{\ref{Thm_Universality}}} R^2} \Big)
	\end{align}
	and
	\begin{align}\label{Eq_Universality_Result2_copy}
		& \bP\Big( \forall r,s \leq R : \ \sup\limits_{z \in \bD(\eta,\kappa)}\Big| \overbrace{\frac{1}{n} \tr(B_s^*B_r\tilde{\bm{R}}^{(\bm{X})}(z))}^{\overset{\text{(\ref{Eq_Def_hatDelta_B})}}{=} \hat{\delta}^{(B)}_{r,s}(z)} - \overbrace{\frac{1}{n} \tr(B_s^*B_r\tilde{\bm{R}}^{(\bm{Z})}(z))}^{\overset{\text{(\ref{Eq_Def_hatDeltaZ_copy})}}{=} \hat{\delta}^{(B,\bm{Z})}_{r,s}(z)} \Big| < \mathcal{C}_{\text{\ref{Thm_Universality}}} n^{\frac{\tvarepsilon-1}{2}} \sigma^2 \Big) \nonumber\\
		& \geq 1 - R^2 \mathcal{C}_{\text{\ref{Thm_Universality}}} n \exp\Big( -\frac{n^{\tvarepsilon}}{\mathcal{C}_{\text{\ref{Thm_Universality}}} R^2} \Big) \ .
	\end{align}
	The simple bound
	\begin{align*}
		& \big|\big| \hat{\delta}^{(A/B)}(z) - \hat{\delta}^{(A/B,\bm{Z})}(z) \big|\big| \leq R \max\limits_{r,s \leq R} \big| \hat{\delta}^{(A/B)}_{r,s}(z) - \hat{\delta}^{(A/B,\bm{Z})}_{r,s}(z) \big|
	\end{align*}
	then turns the results (\ref{Eq_Universality_Result1_copy}) and (\ref{Eq_Universality_Result2_copy}) into
	\begin{align}\label{Eq_Universality_Result_Implication}
		& \bP\Big( \sup\limits_{z \in \bD(\eta,\kappa)} \big|\big| \hat{\delta}^{(A)}(z) - \hat{\delta}^{(A,\bm{Z})}(z) \big|\big| < R \mathcal{C}_{\text{\ref{Thm_Universality}}} n^{\frac{\tvarepsilon-1}{2}} \sigma^2 \nonumber\\
		& \hspace{0.5cm} \text{ and } \sup\limits_{z \in \bD(\eta,\kappa)} \big|\big| \hat{\delta}^{(B)}(z) - \hat{\delta}^{(B,\bm{Z})}(z) \big|\big| < R \mathcal{C}_{\text{\ref{Thm_Universality}}} n^{\frac{\tvarepsilon-1}{2}} \sigma^2 \Big) \nonumber\\
		& \geq 1 - 2R^2 \mathcal{C}_{\text{\ref{Thm_Universality}}} n \exp\Big( -\frac{n^{\tvarepsilon}}{\mathcal{C}_{\text{\ref{Thm_Universality}}} R^2} \Big) \ .
	\end{align}
	
	\item[vi)] \textit{Choice of $\mathcal{C}$}:\\
	Combining (\ref{Eq_MainRes_detltaZ_DetEquiv_Approx}) with (\ref{Eq_Universality_Result_Implication}) gives
	\begin{align}\label{Eq_MainRes_ResultBeforeC}
		& \bP\Big( \forall z \in \bD(\eta,\kappa) : \ \big|\big| \hat{\delta}^{(A)}(z) - \delta^{(A)}(z) \big|\big| \leq \big( 4R\mathcal{C}_{\text{\ref{Thm_DualApprox}}} \mathcal{C}_{\text{\ref{Thm_EmpiricalDualSystem}}} + R\sigma^2 \mathcal{C}_{\text{\ref{Thm_Universality}}} \big) n^{\frac{\tvarepsilon-1}{2}} \nonumber\\
		& \hspace{2.5cm} \text{ and } \big|\big| \hat{\delta}^{(B)}(z) - \delta^{(B)}(z) \big|\big| \leq \big( 4R\mathcal{C}_{\text{\ref{Thm_DualApprox}}} \mathcal{C}_{\text{\ref{Thm_EmpiricalDualSystem}}} + R\sigma^2 \mathcal{C}_{\text{\ref{Thm_Universality}}} \big) n^{\frac{\tvarepsilon-1}{2}} \Big) \nonumber\\
		& \geq 1 - 2\mathcal{C}_{\text{\ref{Thm_EmpiricalDualSystem}}} R^2 n \exp\Big( -\frac{n^{\tvarepsilon}}{\mathcal{C}_{\text{\ref{Thm_EmpiricalDualSystem}}} R^2} \Big) - 2 \exp\Big( -\frac{C n}{2\sigma^4} \Big) - 2R^2 \mathcal{C}_{\text{\ref{Thm_Universality}}} n \exp\Big( -\frac{n^{\tvarepsilon}}{\mathcal{C}_{\text{\ref{Thm_Universality}}} R^2} \Big)
	\end{align}
	and combining (\ref{Eq_Universality_Results_M}) with (\ref{Eq_MainRes_Diff_Gaussian}) yields
	\begin{align}\label{Eq_MainRes_MDiff}
		& \bP\Big( \forall z \in \bD(\eta,\kappa) : \nonumber\\
		& \hspace{0.5cm} \Big| \frac{1}{n} \tr\big( M \bm{R}^{(\bm{X})}(z) \big) - \frac{-1}{z} \frac{1}{n} \tr\Big( M \Big( \Id_d + \sum\limits_{r',s'=1}^R \delta^{(B)}_{r',s'}(z) A_{r'} A_{s'}^* \Big)^{-1} \Big) \Big| \nonumber\\
		& \hspace{1.5cm} \leq \frac{\sigma^2}{\eta} \mathcal{C}_{\text{\ref{Thm_Universality}}} n^{\frac{\tvarepsilon-1}{2}} + 2\mathcal{C}_{\text{\ref{Thm_EmpiricalDualSystem}}} n^{\frac{\tvarepsilon-1}{2}} \nonumber\\
		& \hspace{2cm} + \frac{c_* \sigma^4}{\eta} \frac{(2\kappa^2+2^8\sigma^8 (1+\sqrt{c_*})^4 + 4)^2}{\tau^4 \eta^2} 4R\mathcal{C}_{\text{\ref{Thm_DualApprox}}} \mathcal{C}_{\text{\ref{Thm_EmpiricalDualSystem}}} n^{\frac{\tvarepsilon-1}{2}} \nonumber\\
		& \text{ and } \Big| \frac{1}{n} \tr\big( \tilde{M} \tilde{\bm{R}}^{(\bm{X})}(z) \big) - \frac{-1}{z} \frac{1}{n} \tr\Big( \tilde{M} \Big( \Id_n + \sum\limits_{r',s'=1}^R \delta^{(B)}_{r',s'}(z) B_{s'}^*B_{r'} \Big)^{-1} \Big) \Big| \nonumber\\
		& \hspace{1.5cm} \leq \frac{\sigma^2}{\eta} \mathcal{C}_{\text{\ref{Thm_Universality}}} n^{\frac{\tvarepsilon-1}{2}} + 2\mathcal{C}_{\text{\ref{Thm_EmpiricalDualSystem}}} n^{\frac{\tvarepsilon-1}{2}} \nonumber\\
		& \hspace{2cm} + \frac{c_* \sigma^4}{\eta} \frac{(2\kappa^2+2^8\sigma^8 (1+\sqrt{c_*})^4 + 4)^2}{\tau^4 \eta^2} 4R\mathcal{C}_{\text{\ref{Thm_DualApprox}}} \mathcal{C}_{\text{\ref{Thm_EmpiricalDualSystem}}} n^{\frac{\tvarepsilon-1}{2}} \Big) \nonumber\\
		& \geq 1 - 4\mathcal{C}_{\text{\ref{Thm_EmpiricalDualSystem}}} R^2 n \exp\Big( -\frac{n^{\tvarepsilon}}{\mathcal{C}_{\text{\ref{Thm_EmpiricalDualSystem}}} R^2} \Big) - 2 \exp\Big( -\frac{C n}{2\sigma^4} \Big) - 2\mathcal{C}_{\text{\ref{Thm_Universality}}} n \exp\Big( -\frac{n^{\tvarepsilon}}{\mathcal{C}_{\text{\ref{Thm_Universality}}} R^2} \Big) \ ,
	\end{align}
	thus proving both
	\begin{align}\label{Eq_MainRes_Result_copy}
		& \bP\Big( \forall z \in \bD(\eta,\kappa) : \ \big|\big| \hat{\delta}^{(A)}(z) - \delta^{(A)}(z) \big|\big| \leq R \mathcal{C} n^{\frac{\tvarepsilon-1}{2}} \nonumber\\
		& \hspace{2.5cm} \text{ and } \big|\big| \hat{\delta}^{(B)}(z) - \delta^{(B)}(z) \big|\big| \leq R \mathcal{C} n^{\frac{\tvarepsilon-1}{2}} \Big) \nonumber\\
		& \geq 1 - \mathcal{C} R^2 n \exp\Big( -\frac{n^{\tvarepsilon}}{\mathcal{C} R^2} \Big)
	\end{align}
	and
	\begin{align}\label{Eq_MainRes_Result_M_copy}
		& \bP\Big( \forall z \in \bD(\eta,\kappa) : \nonumber\\
		& \hspace{0.5cm} \Big| \frac{1}{n} \tr\big( M \bm{R}^{(\bm{X})}(z) \big) - \frac{-1}{z} \frac{1}{n} \tr\Big( M \Big( \Id_d + \sum\limits_{r',s'=1}^R \delta^{(B)}_{r',s'}(z) A_{r'} A_{s'}^* \Big)^{-1} \Big) \Big| \leq R\mathcal{C} n^{\frac{\tvarepsilon-1}{2}} , \nonumber\\
		& \hspace{0.5cm} \Big| \frac{1}{n} \tr\big( \tilde{M} \tilde{\bm{R}}^{(\bm{X})}(z) \big) - \frac{-1}{z} \frac{1}{n} \tr\Big( \tilde{M} \Big( \Id_n + \sum\limits_{r',s'=1}^R \delta^{(A)}_{r',s'}(z) B_{s'}^*B_{r'} \Big)^{-1} \Big) \Big| \leq R\mathcal{C} n^{\frac{\tvarepsilon-1}{2}} \Big) \nonumber\\
		& \geq 1 - \mathcal{C} R^2 n \exp\Big( -\frac{n^{\tvarepsilon}}{\mathcal{C} R^2} \Big) \ .
	\end{align}
	for
	\begin{align*}
		& \mathcal{C} \coloneq \max\Big( 4\mathcal{C}_{\text{\ref{Thm_DualApprox}}} \mathcal{C}_{\text{\ref{Thm_EmpiricalDualSystem}}} + \sigma^2 \mathcal{C}_{\text{\ref{Thm_Universality}}}, 4\mathcal{C}_{\text{\ref{Thm_EmpiricalDualSystem}}} + 2 + 2\mathcal{C}_{\text{\ref{Thm_Universality}}}, \frac{2\sigma^4}{C} , \\
		& \hspace{2cm} \frac{\sigma^2}{\eta} \mathcal{C}_{\text{\ref{Thm_Universality}}} + 2\mathcal{C}_{\text{\ref{Thm_EmpiricalDualSystem}}} + \frac{c_* \sigma^4}{\eta} \frac{(2\kappa^2+2^8\sigma^8 (1+\sqrt{c_*})^4 + 4)^2}{\tau^4 \eta^2} 4\mathcal{C}_{\text{\ref{Thm_DualApprox}}} \mathcal{C}_{\text{\ref{Thm_EmpiricalDualSystem}}} \Big) \ ,
	\end{align*}
	where $C>0$ is the universal constant from Lemma~\ref{Lemma_ZTailBound}. This concludes the proof of Theorem~\ref{Thm_SepCovMP_NonAsymp}. \qed
\end{itemize}

\subsection{Proof of Corollary~\ref{Cor_SepCovMP_Asymp}}\label{Proof_Cor_SepCovMP_Asymp}
By (\ref{Eq_CorAsymp_cInfAssumption}), there exists a $c_* \geq 1$ such that $\frac{d_n}{n} \leq c_*$ for all $n \in \N$. It by construction follows that the models $(\bm{X}_n, (A^{(n)}_r)_{r \leq R}, (B^{(n)}_r)_{r \leq R})$ satisfy assumptions~\ref{ItemAssumption_cBound}-\ref{ItemTempAssumption_NonDegeneracy} with fixed constants $c_*,\sigma^2,C_6,\tau$ for all $n \in \N$.
\\[0.5em]
Theorem~\ref{Thm_SepCovMP_NonAsymp} with $\tvarepsilon$ for any $\eta,\kappa>0$ and all $n \geq N$ gives
\begin{align}\label{Eq_MainRes_Result_Cor_copy}
	& \bP\Big( \sup\limits_{z \in \bD(\eta,\kappa)} ||\hat{\delta}^{(A,n)}(z)-\delta^{(A,n)}(z)|| + ||\hat{\delta}^{(B,n)}(z)-\delta^{(B,n)}(z)|| > 2R\mathcal{C} n^{\frac{1}{4}} \Big) \nonumber\\
	& \leq \mathcal{C} R^2 n \exp\Big( -\frac{\sqrt{n}}{\mathcal{C} R^2} \Big)
\end{align}
and with $\tilde{M}=\Id_n$ also
\begin{align}\label{Eq_MainRes_Result_M_Cor_copy}
	& \bP\Big( \forall z \in \bD(\eta,\kappa) : \nonumber\\
	& \hspace{0.5cm} \Big| \underbrace{\frac{1}{n}\tr\big( (\tilde{\bm{S}}^{(n)} - z\Id_n)^{-1} \big)}_{\overset{\text{(\ref{Eq_Def_hatNu})}}{=}\cs_{\hat{\ul{\nu}}_n}(z)} - \underbrace{\frac{-1}{z n}\tr\Big( \Big( \Id_n + \sum\limits_{r',s'=1}^R \delta^{(A,n)}_{r',s'}(z) (B^{(n)}_{s'})^*B^{(n)}_{r'} \Big)^{-1} \Big)}_{= \cs_{\ul{\nu}_n}(z) \text{ by (\ref{Eq_DetEquiv_NuStilProp}) and (\ref{Eq_DeltaProductCalc})}} \Big| \leq R \mathcal{C} n^{\frac{\tvarepsilon-1}{2}} \Big) \nonumber\\
	& \geq 1 - \mathcal{C} R^2 n \exp\Big( -\frac{n^{\tvarepsilon}}{\mathcal{C} R^2} \Big)
\end{align}
where $N,\mathcal{C} > 0$ do not depend on $n$. Since the above probabilities are both summable over $n$, Borel-Cantelli yields
\begin{align*}
	& 0 = \bP\Big( \bigcap\limits_{m=1}^\infty \bigcup\limits_{n=m}^\infty \Big\{ \sup\limits_{z \in \bD(\eta,\kappa)} ||\hat{\delta}^{(A,n)}(z)-\delta^{(A,n)}(z)|| + ||\hat{\delta}^{(B,n)}(z)-\delta^{(B,n)}(z)||\\
	& \hspace{7cm} + |\cs_{\hat{\ul{\nu}}_n}(z) - \cs_{\ul{\nu}_n}(z)| > 3R\mathcal{C} n^{\frac{1}{4}} \Big\} \Big)
\end{align*}
and by transition to the counter event
\begin{align*}
	1 & = \bP\Big( \bigcup\limits_{m=1}^\infty \bigcap\limits_{n=m}^\infty \Big\{ \sup\limits_{z \in \bD(\eta,\kappa)} ||\hat{\delta}^{(A,n)}(z)-\delta^{(A,n)}(z)|| + ||\hat{\delta}^{(B,n)}(z)-\delta^{(B,n)}(z)||\\
	& \hspace{7cm} + |\cs_{\hat{\ul{\nu}}_n}(z) - \cs_{\ul{\nu}_n}(z)| \leq 3R\mathcal{C} n^{\frac{1}{4}} \Big\} \Big)\\
	& \leq \bP\Big( \sup\limits_{z \in \bD(\eta,\kappa)} ||\hat{\delta}^{(A,n)}(z)-\delta^{(A,n)}(z)|| + ||\hat{\delta}^{(B,n)}(z)-\delta^{(B,n)}(z)||\\
	& \hspace{6cm} + |\cs_{\hat{\ul{\nu}}_n}(z) - \cs_{\ul{\nu}_n}(z)| \xrightarrow{n \to \infty} 0 \Big)\\
	& \leq \bP\Big( \forall z \in \bD(\eta,\kappa) : \ ||\hat{\delta}^{(A,n)}(z)-\delta^{(A,n)}(z)|| + ||\hat{\delta}^{(B,n)}(z)-\delta^{(B,n)}(z)||\\
	& \hspace{6cm} + |\cs_{\hat{\ul{\nu}}_n}(z) - \cs_{\ul{\nu}_n}(z)| \xrightarrow{n \to \infty} 0 \Big) \ .
\end{align*}
Taking $\eta \searrow 0$ and $\kappa \nearrow \infty$ with continuity of measures then gives
\begin{align}\label{Eq_CorAsymp_Convergences}
	& 1 = \bP\Big( \forall z \in \C^+ : \ ||\hat{\delta}^{(A,n)}(z)-\delta^{(A,n)}(z)|| + ||\hat{\delta}^{(B,n)}(z)-\delta^{(B,n)}(z)||\nonumber\\
	& \hspace{6cm} + |\cs_{\hat{\ul{\nu}}_n}(z) - \cs_{\ul{\nu}_n}(z)| \xrightarrow{n \to \infty} 0 \Big)
\end{align}
and assumption (\ref{Eq_CorAsymp_RhoConvergence}) with interpretation (c) of weak convergence from Lemma~\ref{Lemma_MatrWeakConv} yields
\begin{align}\label{Eq_CorAsymp_BorelCantelli}
	& 1 = \bP\Big( \forall z \in \C^+ : \ ||\hat{\delta}^{(A,n)}(z)-\delta^{(A,\infty)}(z)|| + ||\hat{\delta}^{(B,n)}(z)-\delta^{(B,\infty)}(z)|| \xrightarrow{n \to \infty} 0 \Big)
\end{align}
as well as
\begin{align*}
	& ||\rho^{(A,n)}(\R) - \rho^{(A,\infty)}(\R)|| \xrightarrow{n \to \infty} 0 \ \ \text{ and } \ \ ||\rho^{(B,n)}(\R) - \rho^{(B,\infty)}(\R)|| \xrightarrow{n \to \infty} 0 \ .
\end{align*}
Since
\begin{align*}
	& \hat{\rho}^{(A,n)}(\R) \overset{\text{(\ref{Eq_Def_hatRhoA})}}{=} \frac{1}{n} \tr\big( A^{(n)}_r(A^{(n)}_s)^* \big)_{r,s \leq R} \overset{\text{(\ref{Eq_DetEquiv_RhoProp2})}}{=} \rho^{(A,n)}(\R)\\
	& \hat{\rho}^{(B,n)}(\R) \overset{\text{(\ref{Eq_Def_hatRhoB})}}{=} \frac{1}{n} \tr\big( (B^{(n)}_s)^*B^{(n)}_r \big)_{r,s \leq R} \overset{\text{(\ref{Eq_DetEquiv_RhoProp2})}}{=} \rho^{(B,n)}(\R) \ ,
\end{align*}
one may again use interpretation (c) of weak convergence from Lemma~\ref{Lemma_MatrWeakConv}, this time with (\ref{Eq_CorAsymp_BorelCantelli}), to see
\begin{align*}
	& 1 = \bP\Big( \hat{\rho}^{(A,n)} \xRightarrow{n \to \infty} \rho^{(A,\infty)} \ \ \text{ and } \ \ \hat{\rho}^{(B,n)} \xRightarrow{n \to \infty} \rho^{(B,\infty)} \Big) \ ,
\end{align*}
which proves (\ref{Eq_CorAsymp_hatRhoConvergence}).
\\[0.5em]
For (\ref{Eq_CorAsymp_hatNuConvergence}), it by (\ref{Eq_CorAsymp_Convergences}) suffices to show the deterministic convergence $\cs_{\ul{\nu}_n}(z) \xrightarrow{n \to \infty} \cs_{\ul{\nu}_\infty}(z)$ for all $z \in \C^+$. Since any sub-sequence $(\ul{\nu}_{n_j})_{j \in \N}$ will by (d) in Theorem~\ref{Thm_DetEquiv} have the supports of $\ul{\nu}_{n_j}$ contained on the interval $[0,8\sigma^4(1+\sqrt{c}_*)^2]$, there by Prokhorov must exist a sub-sub-sequence $(\ul{\nu}_{n_{j_k}})_{k \in \N}$ and a probability measure $\ul{\nu}_\infty$ such that $\ul{\nu}_{n_{j_k}} \xRightarrow{k \to \infty} \ul{\nu}_\infty$, which implies
\begin{align*}
	& \forall z \in \C^+ : \ \cs_{\ul{\nu}_{n_{j_k}}}(z) \xrightarrow{k \to \infty} \cs_{\ul{\nu}_\infty}(z) \ .
\end{align*}
Recalling the defining property of $\ul{\nu}_{n}$ from (d) in Theorem~\ref{Thm_DetEquiv}, one also has
\begin{align*}
	\forall z \in \C^+ : \ & \cs_{\ul{\nu}_{n_{j_k}}}(z) \overset{\text{(\ref{Eq_DetEquiv_NuStilProp})}}{=} \frac{1}{z} - \sum\limits_{r,s=1}^R \delta^{(A,n_{j_k})}_{r,s}(z) \delta^{(B,n_{j_k})}_{r,s}(z)\\
	& \xrightarrow[\substack{\text{(\ref{Eq_CorAsymp_RhoConvergence}),} \\ \text{(c) in Lemma~\ref{Lemma_MatrWeakConv}}}]{k \to \infty} \frac{1}{z} - \sum\limits_{r,s=1}^R \delta^{(A,\infty)}_{r,s}(z) \delta^{(B,\infty)}_{r,s}(z) \ .
\end{align*}
Combining the last two convergences gives
\begin{align*}
	& \forall z \in \C^+ : \ \cs_{\ul{\nu}_\infty}(z) = \frac{1}{z} - \sum\limits_{r,s=1}^R \delta^{(A,\infty)}_{r,s}(z) \delta^{(B,\infty)}_{r,s}(z) \ ,
\end{align*}
so $\ul{\nu}_\infty$ is uniquely determined and does not depend on the sub-sub-sequence $(n_{j_k})_{k \in \N}$, which by standard topological arguments yields $\ul{\nu}_{n} \xRightarrow{n \to \infty} \ul{\nu}_\infty$ or equivalently
\begin{align*}
	& \forall z \in \C^+ : \ \cs_{\ul{\nu}_n}(z) \xrightarrow{n \to \infty} \cs_{\ul{\nu}_\infty}(z) \ ,
\end{align*}
which by (\ref{Eq_CorAsymp_Convergences}) was seen to be sufficient for (\ref{Eq_CorAsymp_hatNuConvergence}). \qed

\newpage
\section{Proofs of Lemmas}\label{Section_Proofs}

\subsection{Proof of Lemma~\ref{Lemma_MatrWeakConv}}\label{Proof_Lemma_MatrWeakConv}
\begin{itemize}
	\item (a) $\Rightarrow$ (b):\\
	One for every $f \in C_b(\R;\R)$ observes that the entries of the $(R \times R)$-matrix $M_{f,n} \coloneq \int_\R f \, d(\rho_n - \rho)$ are by polarization identity uniquely defined by
	\begin{align}\label{Eq_MatrixWeakConv_b_to_a}
		& e_{r,R}^* M_{f,n} e_{s,R} = \frac{1}{4} \sum\limits_{k=0}^3 \bm{i}^{-k} \underbrace{(e_{r,R} + \bm{i}^k e_{s,R})^* M_{f,n} (e_{r,R} + \bm{i}^{-k} e_{s,R})}_{\xrightarrow{n \to \infty} 0 \text{ by (a)}} \xrightarrow{n \to \infty} 0 \ .
	\end{align}
	As $R$ is fixed, the convergence of all entries of $M_{f,n}$ to zero implies
	\begin{align*}
		& \forall f \in \C_b(\R,\R) : \ \big|\big| \int_\R f \, d(\rho_n - \rho) \big|\big| \xrightarrow{n \to \infty} 0 \ ,
	\end{align*}
	which is easily extended to hold for all $f \in C_b(\R;\C)$.
	
	\item (b) $\Rightarrow$ (c):\\
	For each $z \in \C^+$, the map $\lambda \mapsto \frac{1}{\lambda-z}$ is in $C_b(\R;\C)$, which makes this direction trivial. The property (\ref{Eq_MatrWeakConv_SameMass}) follows by setting $f=1$.
	
	\item (c) $\Rightarrow$ (d):\\
	This direction is also trivial.
	
	\item (d) $\Rightarrow$ (a):\\
	For each $v \in \C^R$, let $m_n \coloneq v^* \rho_n(\R) v$ and $m \coloneq v^* \rho(\R) v$ denote the total mass of the finite Radon measures $(v^* \rho_n v)$ and $(v^* \rho v)$ respectively, then (\ref{Eq_MatrWeakConv_SameMass}) implies $v^* \rho_n(\R) v = m_n \xrightarrow{n \to \infty} m$. If $m=0$, then one immediately has $v^* \rho_n v \xRightarrow{n \to \infty} 0 = v^* \rho v$, which for the case $m=0$ proves (a).\\
	If, on the other hand $m > 0$, then one may by $m_n \xrightarrow{n \to \infty} m$ without loss of generality assume $m_n > 0$ for all $n \in \N$. 
	By (d) for any $z \in \C^+$ observe that the calculation
	\begin{align*}
		& \int_\R \frac{1}{\lambda-z} \, d\Big( \frac{1}{m_n} v^* \rho_n v \Big)(\lambda) = \overbrace{\frac{1}{m_n}}^{\to \frac{1}{m}} \overbrace{v^* \delta_n(z) v}^{\to v^* \delta(z) v}\\
		& \hspace{1cm} \xrightarrow{n \to \infty} \frac{1}{m} v^* \delta(z) v = \int_\R \frac{1}{\lambda-z} \, d\Big( \frac{1}{m} v^* \rho v \Big)(\lambda)
	\end{align*}
	proves point-wise the convergence of Stieltjes transforms of probability measures $\frac{1}{m_n} v^* \rho_n v$ to the Stieltjes transform of the probability measure $\frac{1}{m} v^* \rho v$. It is well-known (cf. for example Theorem B.9 in \cite{BaiSALDRM}) that this implies $\frac{1}{m_n} v^* \rho_n v \xRightarrow{n \rightarrow \infty} \frac{1}{m} v^* \rho v$. Consequently
	\begin{align*}
		& \underbrace{m_n}_{\to m} \underbrace{\frac{1}{m_n} v^* \rho_n v}_{\Rightarrow \frac{1}{m} v^* \rho v} \xRightarrow{n \to \infty} m \frac{1}{m} v^* \rho v = v^* \rho v \ ,
	\end{align*}
	thus proving (a) for the case $m > 0$.
	\qed
\end{itemize}

\subsection{Proof of Lemma~\ref{Lemma_NonDegeneracyNew}}\label{Proof_Lemma_NonDegeneracyNew}
Let $\mathbbm{1}_R$ denote the vector $(1,\dots,1)^\top \in \C^R$.
For every $v \in \C^{d}$ and positive semi-definite $\tilde{C} \in \C^{R \times R}$ observe
\begin{align}\label{Eq_NondegeneracyNew_Calc1}
	& v^* \Big(\sum\limits_{r,s=1}^R \tilde{C}_{r,s} A_rA_s^*\Big) v = \sum\limits_{r,s=1}^R \tilde{C}_{r,s} v^*A_rA_s^*v = \mathbbm{1}_R^\top \big(\tilde{C} \odot M(v)\big) \mathbbm{1}_R \ ,
\end{align}
where $M(v) = (v^*A_rA_s^*v)_{r,s \leq R}$ and $\odot$ denotes the Hadamard product. The calculation
\begin{align*}
	& \forall w \in \C^{R} : \ w^* M(v) w = \sum\limits_{r,s=1}^R \ol{w_r} w_s (v^*A_rA_s^*v) = \Big|\Big| \Big( \sum\limits_{r=1}^R w_r A_r^* \Big) v \Big|\Big|_2^2
\end{align*}
show that $M(v)$ is positive semi-definite.
It is  well known that the Hadamard product preserves positive semi-definiteness, so $\tilde{C} \odot M(v)$ is also positive semi-definite. By (\ref{Eq_NondegeneracyNew_Calc1}), the matrix $\sum\limits_{r,s=1}^R \tilde{C}_{r,s} A_rA_s^*$ is thus positive semi-definite for any positive semi-definite $\tilde{C} \in \C^{R \times R}$. Taking $\tilde{C} = C - \lambda_{\min}(C)\Id_R$ yields
\begin{align*}
	& \lambda_{\min}(C) \sum\limits_{r=1}^R A_rA_r^* = \sum\limits_{r,s=1}^R (\lambda_{\min}(C)\Id_R)_{r,s} A_rA_s^* \preceq \sum\limits_{r,s=1}^R C_{r,s} A_rA_s^* \ , 
\end{align*}
which by assumption~\ref{ItemTempAssumption_NonDegeneracy} proves (\ref{Eq_NonDegeneracyC_A}). The bound (\ref{Eq_NonDegeneracyC_B}) may be shown by repeating the above proof for $A_rA_s^*$ replaced with $B_s^*B_r$.
\\[0.5em]
By the simple bound $\tr\big( AA^* M \big) \geq \lambda_{\min}(M) \tr(AA^*)$, under the assumption~\ref{ItemTempAssumption_NonDegeneracy} calculate
\begin{align*}
	& \min\limits_{\substack{v \in \C^R \\ ||v||_2=1}} v^* \Big( \frac{1}{n}\tr\big( A_r A_s^* M \big)_{r,s\leq R} \Big) v = \min\limits_{\substack{v \in \C^R \\ ||v||_2=1}} \sum\limits_{r,s=1}^R \ol{v_r} v_s \frac{1}{n}\tr\big( A_s^* M A_r \big) \nonumber\\
	& = \min\limits_{\substack{v \in \C^R \\ ||v||_2=1}} \frac{1}{n}\tr\bigg( \Big(\sum\limits_{s=1}^R \ol{v_s} A_s\Big)^* M \Big(\sum\limits_{r=1}^R \ol{v_r} A_r\Big) \bigg) \nonumber\\
	& \geq \lambda_{\min}(M) \min\limits_{\substack{v \in \C^R \\ ||v||_2=1}} \frac{1}{n}\tr\bigg( \Big(\sum\limits_{s=1}^R \ol{v_s} A_s\Big)^* \Big(\sum\limits_{r=1}^R \ol{v_r} A_r\Big) \bigg) \nonumber\\
	& = \lambda_{\min}(M) \min\limits_{\substack{v \in \C^R \\ ||v||_2=1}} v^* \mathbb{G}^{(A)} v \overset{\text{\ref{ItemTempAssumption_NonDegeneracy}}}{\geq} \tau \lambda_{\min}(M)
\end{align*}
to see
\begin{align*}
	& \lambda_{\min}\Big( \frac{1}{n} \tr\big( A_rA_s^* M \big)_{r,s \leq R} \Big) \geq \tau \lambda_{\min}(M)
\end{align*}
and analogously
\begin{align*}
	& \lambda_{\min}\Big( \frac{1}{n} \tr\big( B_s^*B_r \tilde{M} \big)_{r,s \leq R} \Big) \geq \tau \lambda_{\min}(\tilde{M}) \ ,
\end{align*}
thus proving (\ref{Eq_NonDegeneracy_TrCalcA}) and (\ref{Eq_NonDegeneracy_TrCalcB}). \qed

\subsection{Proof of Lemma~\ref{Lemma_SimpleOperatorBounds}}\label{Proof_Lemma_SimpleOperatorBounds}
For any vectors $u,v \in \C^d$, observe
\begin{align*}
	& \Big| u^* \Big( \sum\limits_{r,s=1}^R C_{r,s} A_r A_s^* \Big) v \Big| = \Big| \sum\limits_{r,s=1}^R C_{r,s} u^*A_r A_s^*v \Big|\\
	& = \Big| \sum\limits_{r,s=1}^R \sum\limits_{j=1}^d C_{r,s} \underbrace{u^*A_r e_{j,d}}_{\eqcolon  \ol{\tilde{u}^{(j)}_r}} \underbrace{e_{j,d}^\top A_s^*v}_{\eqcolon  \tilde{v}^{(j)}_s} \Big| = \Big| \sum\limits_{j=1}^d (\tilde{u}^{(j)})^* C \tilde{v}^{(j)} \Big|\\
	& \leq \sum\limits_{j=1}^d ||C|| \, ||\tilde{u}^{(j)}||_2 \, ||\tilde{v}^{(j)}||_2 \overset{\text{C.S.}}{\leq} ||C|| \Big(\sum\limits_{j=1}^d ||\tilde{u}^{(j)}||_2^2 \Big)^{\frac{1}{2}} \Big(\sum\limits_{j=1}^d ||\tilde{v}^{(j)}||_2^2 \Big)^{\frac{1}{2}}\\
	& = ||C|| \Big(\sum\limits_{j=1}^d \sum\limits_{r=1}^R |u^* A_r e_{j,d}|^2 \Big)^{\frac{1}{2}} \Big(\sum\limits_{j=1}^d \sum\limits_{r=1}^R |e_{j,d}^\top A_r^* v|^2 \Big)^{\frac{1}{2}}\\
	& = ||C|| \Big(\sum\limits_{r=1}^R u^* A_r A_r^* u \Big)^{\frac{1}{2}} \Big(\sum\limits_{r=1}^R v^* A_r A_r^* v \Big)^{\frac{1}{2}}\\
	& \overset{\text{\ref{ItemAssumption_sigmaBound}}}{\leq} ||C|| \big( \sigma^2 ||u||_2^2 \big)^{\frac{1}{2}} \big( \sigma^2 ||v||_2^2 \big)^{\frac{1}{2}} \leq \sigma^2 ||C|| \, ||u||_2 \, ||v||_2 \ ,
\end{align*}
which proves the left-hand side of (\ref{Eq_SimpleBounds_ASum_BSum}). The right-hand side may be proved in complete analogy.
To prove (\ref{Eq_SimpleBounds_TraceMatrixBounds}), employ the well-known trace bound
\begin{align}\label{Eq_TraceBound}
	& \forall A \in \C^{N \times N} : \ |\tr(A)| \leq \operatorname{rank}(A) \, ||A|| \ ,
\end{align}
which for example is proved in a stronger form in \cite{mirsky1975trace}, to for any $x,y \in \C^{R}$ calculate
\begin{align*}
	& \Big| x^* \Big(\frac{1}{n} \tr\big( A_r A_s^* M \big)_{r,s \leq R}\Big) y \Big| = \Big| \sum\limits_{r,s=1}^R \ol{x}_r y_s \frac{1}{n} \tr\big( A_s^* M A_r \big) \Big|\\
	& = \Big| \frac{1}{n} \tr\Big( \Big( \sum\limits_{s=1}^R y_s A_s^* \Big) M \Big( \sum\limits_{r=1}^R \ol{x}_r A_r \Big) \Big) \Big|\\
	& \overset{\text{(\ref{Eq_TraceBound})}}{\leq} \frac{d}{n} ||M|| \, \Big|\Big| \sum\limits_{r=1}^R \ol{x_r} A_r \Big|\Big| \, \Big|\Big| \sum\limits_{r=1}^R \ol{y_r} A_r \Big|\Big|\\
	& \leq \frac{d}{n} ||M|| \Big( \sum\limits_{r=1}^R |x_r| \, ||A_r|| \Big) \Big( \sum\limits_{r=1}^R |y_r| \, ||A_r|| \Big)\\
	& \overset{\text{C.S.}}{\leq} \frac{d}{n} ||M|| \, ||x||_2 \, ||y||_2 \sum\limits_{r=1}^R ||A_r||^2 \overset{\text{\ref{ItemAssumption_sigmaBound}}}{\leq} \sigma^2 \frac{d}{n} ||M|| \, ||x||_2 \, ||y||_2
\end{align*}
and analogously
\begin{align*}
	& \Big| x^* \Big(\frac{1}{n} \tr\big( B_s^* B_r \tilde{M} \big)_{r,s \leq R}\Big) y \Big| = \Big| \sum\limits_{r,s=1}^R \ol{x}_r y_s \frac{1}{n} \tr\big( B_r \tilde{M} B_s^* \big) \Big|\\
	& = \Big| \frac{1}{n} \tr\Big( \Big( \sum\limits_{r=1}^R \ol{x_r} B_r \Big) \tilde{M} \Big( \sum\limits_{s=1}^R y_s B_s^* \Big) \Big) \Big|\\
	& \overset{\text{(\ref{Eq_TraceBound})}}{\leq} ||\tilde{M}|| \, \Big|\Big| \sum\limits_{r=1}^R \ol{x_r} B_r \Big|\Big| \, \Big|\Big| \sum\limits_{r=1}^R \ol{y_r} B_r \Big|\Big|\\
	& \leq ||\tilde{M}|| \Big( \sum\limits_{r=1}^R |x_r| \, ||B_r|| \Big) \Big( \sum\limits_{r=1}^R |y_r| \, ||B_r|| \Big)\\
	& \overset{\text{C.S.}}{\leq} ||\tilde{M}|| \, ||x||_2 \, ||y||_2 \sum\limits_{r=1}^R ||B_r||^2 \overset{\text{\ref{ItemAssumption_sigmaBound}}}{\leq} \sigma^2 ||\tilde{M}|| \, ||x||_2 \, ||y||_2 \ .
\end{align*}
\qed

\subsection{Proof of  Lemma~\ref{Lemma_OperatorBound}}\label{Proof_Lemma_OperatorBound}
For any $u,v \in \C^{n}$ observe
\begin{align*}
	& |u^* T(X) v| = \Big|\sum\limits_{j=1}^N u^* C_j X D_j v\Big| \leq \sum\limits_{j=1}^N ||C_j^*u||_2 \, ||X|| \, ||D_j v||_2\\
	& \overset{\text{C.S.}}{\leq} ||X|| \Big( \sum\limits_{j=1}^N ||C_j^*u||_2^2 \Big)^{\frac{1}{2}} \, \Big( \sum\limits_{j=1}^N ||D_jv||_2^2 \Big)^{\frac{1}{2}}\\
	& = ||X|| \Big( u^* \Big( \sum\limits_{j=1}^N C_jC_j^* \Big) u \Big)^{\frac{1}{2}} \, \Big( v^* \Big(\sum\limits_{j=1}^N D_j^*D_j \Big) v \Big)^{\frac{1}{2}}\\
	& \leq ||X|| \, \Big|\Big| \sum\limits_{j=1}^N C_jC_j^* \Big|\Big|^{\frac{1}{2}} \, \Big|\Big| \sum\limits_{j=1}^N D_j^*D_j \Big|\Big|^{\frac{1}{2}} \, ||u||_2 \, ||v||_2 \ .
\end{align*}
\qed

\subsection{Proof of Lemma~\ref{Lemma_ContractionEVs}}\label{Proof_Lemma_ContractionEVs}
By triangle inequality and the contraction property, one has
\begin{align*}
	& ||v||_B - ||w||_B \leq ||v+w||_B = ||T(v)||_B \leq (1-\theta)||v||_B \ ,
\end{align*}
which may be re-arranged to
\begin{align*}
	& \theta ||v||_B \leq ||w||_B \ .
\end{align*}
\qed


\subsection{Proof of Lemma~\ref{Lemma_SteinMixed}}\label{proof_Lemma_SteinMixed}
Define independent random variables $V \sim \mathcal{N}(0,1)$ and $W \sim \mathcal{CN}(0,1)$.
Following Definition~\ref{Def_SimilarGaussianMatrix}, one may choose $a,b \in \C$ such that the first and second moments of $aV + bW$ coincide with those of $Z$, i.e.
\begin{align*}
	& \E[aV + bW] = 0 = \E[Z] \ , \ \E[|aV + bW|^2] = \E[|Z|^2] \ \text{ and } \ \E[(aV + bW)^2] = \E[Z^2] \ .
\end{align*}
Since both $Z$ and $aV + bW$ are complex Gaussian, it follows that $Z \sim aV + bW$ and it suffices to prove the lemma for $aV + bW$ instead of $Z$.
\\[0.5em]
The well-known real valued Stein's Lemma states that standard normal random variables $X \sim \mathcal{N}(0,1)$ satisfy
\begin{align}\label{Eq_Stein_Real}
	& \E\big[ X f(X) \big] = \E\big[ f'(X) \big]
\end{align}
for any differentiable function $f : \R \rightarrow \R$, whenever both expectations exist.
By writing $W = \frac{1}{\sqrt{2}} \big( X + iY \big)$ for two independent standard normal random variables $X$ and $Y$, one may calculate
\begin{align*}
	& \E\big[ W h(W,\ol{W}) \big] = \frac{1}{\sqrt{2}} \E\Big[ (X+iY) h\Big(\frac{X+iY}{\sqrt{2}},\frac{X-iY}{\sqrt{2}}\Big) \Big]\\
	& = \frac{1}{\sqrt{2}} \E_Y\Big[ \E_X\Big[ X h\Big(\frac{X+iY}{\sqrt{2}},\frac{X-iY}{\sqrt{2}}\Big) \Big] \Big] + \frac{i}{\sqrt{2}} \E_X\Big[ \E_Y\Big[ Y h\Big(\frac{X+iY}{\sqrt{2}},\frac{X-iY}{\sqrt{2}}\Big) \Big] \Big]\\
	& \overset{\text{(\ref{Eq_Stein_Real})}}{=} \frac{1}{\sqrt{2}} \E_Y\Big[ \E_X\Big[ \partial_X h\Big(\frac{X+iY}{\sqrt{2}},\frac{X-iY}{\sqrt{2}}\Big) \Big] \Big] + \frac{i}{\sqrt{2}} \E_X\Big[ \E_Y\Big[ \partial_Y h\Big(\frac{X+iY}{\sqrt{2}},\frac{X-iY}{\sqrt{2}}\Big) \Big] \Big]\\
	& = \frac{1}{2} \E\Big[ \partial_1 h\Big(\frac{X+iY}{\sqrt{2}},\frac{X-iY}{\sqrt{2}}\Big) + \partial_2 h\Big(\frac{X+iY}{\sqrt{2}},\frac{X-iY}{\sqrt{2}}\Big) \Big]\\
	& \hspace{0.5cm} + \frac{i^2}{2} \E\Big[ \partial_1 h\Big(\frac{X+iY}{\sqrt{2}},\frac{X-iY}{\sqrt{2}}\Big) - \partial_2 h\Big(\frac{X+iY}{\sqrt{2}},\frac{X-iY}{\sqrt{2}}\Big) \Big]\\
	& = \E\Big[ \partial_2 h\Big(\frac{X+iY}{\sqrt{2}},\frac{X-iY}{\sqrt{2}}\Big) \Big] = \E\big[ \partial_2 h(W,\ol{W}) \big] \ ,
\end{align*}
thus proving
\begin{align}\label{Eq_Stein_Complex}
	& \E\big[ W h(W,\ol{W}) \big] = \E\big[ \partial_2 h(W,\ol{W}) \big] \eqcolon  \E\Big[ \frac{\partial}{\partial \ol{W}} h(W,\ol{W}) \Big]
\end{align}
for any holomorphic function $h : U \rightarrow \C$, whenever $h$ is of the same shape as $g$ and the involved expectations and $\E\big[ \partial_1 h(W,\ol{W}) \big]$ exist. Here, $\frac{\partial}{\partial \ol{W}}$ denotes the Wirtinger derivative, which ignores occurrences of $W$.
One may then directly employ the previous two formulations of Stein's lemma to see
\begin{align*}
	& \E\big[ (aV+bW) g(aV+bW,\ol{aV+bW}) \big]\\
	& = a\E\big[ V g(aV+bW,\ol{a}V+\ol{bW}) \big] + b\E\big[ W g(aV+bW,\ol{a}V+\ol{bW}) \big]\\
	& = a\E_W\big[ \E_V\big[ V g(aV+bW,\ol{a}V+\ol{bW}) \big] \big] + b\E_V\big[ \E_W\big[ W g(aV+bW,\ol{a}V+\ol{bW}) \big] \big]\\
	& \overset{\substack{\text{(\ref{Eq_Stein_Real})} \\ \text{(\ref{Eq_Stein_Complex})}}}{=} a\E_W\big[ \E_V\big[ \partial_V g(aV+bW,\ol{a}V+\ol{bW}) \big] \big] + b\E_V\Big[ \E_W\Big[ \frac{\partial}{\partial \ol{W}} g(aV+bW,\ol{a}V+\ol{bW}) \Big] \Big]\\
	& = a^2\E\big[ \partial_1 g(aV+bW,\ol{a}V+\ol{bW}) \big] + |a|^2\E\big[ \partial_2 g(aV+bW,\ol{a}V+\ol{bW}) \big]\\
	& \hspace{0.5cm} + |b|^2\E\big[ \partial_2 g(aV+bW,\ol{a}V+\ol{bW}) \big]\\
	& = a^2\E\big[ \partial_1 g(aV+bW,\ol{aV+bW}) \big] + (|a|^2 + |b|^2)\E\big[ \partial_2 g(aV+bW,\ol{aV+bW}) \big]\\
	& = \E[(aV+bW)^2] \E\big[ \partial_1 g(aV+bW,\ol{aV+bW}) \big]\\
	& \hspace{0.5cm} + \E[|aV+bW|^2]\E\big[ \partial_2 g(aV+bW,\ol{aV+bW}) \big] \ .
\end{align*}
\qed

\subsection{Proof of Lemma~\ref{Lemma_BasicDerivatives}}\label{Proof_Lemma_BasicDerivatives}
For ease of notation, write $\bm{Y}$, $\bm{S}$ and $\bm{R}(z)$ for $\bm{Y}^{(X)}$, $\bm{S}^{(X)}$ and $\bm{R}^{(X)}(z)$ respectively. 
\begin{itemize}
	\item Derivatives of $\bm{Y}$ and $\bm{Y}^*$:\\
	One by linearity of derivatives and (\ref{Eq_DerivCalc_BaseX}) directly calculates
	\begin{align}\label{Eq_DerivCalc_XDeriv_Y}
		& \frac{\partial}{\partial X_{i,j}} \bm{Y} \overset{\text{(\ref{Eq_DefY})}}{=} \sum\limits_{r=1}^R A_r \frac{\partial}{\partial X_{i,j}} \big[ X \big] B_r = \sum\limits_{r=1}^R A_r e_{i,d} e_{j,n}^\top B_r
	\end{align}
	and
	\begin{align}\label{Eq_DerivCalc_XDeriv_Y*}
		& \frac{\partial}{\partial X_{i,j}} \bm{Y}^* \overset{\text{(\ref{Eq_DefY})}}{=} \sum\limits_{s=1}^R B_s^* \frac{\partial}{\partial X_{i,j}} \big[ X^* \big] A_s^* = 0
	\end{align}
	as well as
	\begin{align}\label{Eq_DerivCalc_olXDeriv_Y}
		& \frac{\partial}{\partial \ol{X_{i,j}}} \bm{Y} \overset{\text{(\ref{Eq_DefY})}}{=} \sum\limits_{r=1}^R A_r \frac{\partial}{\partial \ol{X_{i,j}}} \big[ X \big] B_r = 0
	\end{align}
	and
	\begin{align}\label{Eq_DerivCalc_olXDeriv_Y*}
		& \frac{\partial}{\partial \ol{X_{i,j}}} \bm{Y}^* \overset{\text{(\ref{Eq_DefY})}}{=} \sum\limits_{s=1}^R B_s^* \frac{\partial}{\partial \ol{X_{i,j}}} \big[ X^* \big] A_s^* = \sum\limits_{s=1}^R B_s^* e_{j,n} e_{i,d}^\top A_s^* \ .
	\end{align}
	
	\item Derivatives of $\bm{S}$ and $\tilde{\bm{S}}$:\\
	By product rule, one observes
	\begin{align}
		& \frac{\partial}{\partial X_{i,j}} \bm{S} \overset{\text{(\ref{Eq_DefS})}}{=} \frac{1}{n} \frac{\partial}{\partial X_{i,j}} \big[ \bm{Y} \big] \bm{Y}^* + \frac{1}{n} \bm{Y} \frac{\partial}{\partial X_{i,j}} \big[ \bm{Y}^* \big] \overset{\substack{\text{(\ref{Eq_DerivCalc_XDeriv_Y})} \\ \text{(\ref{Eq_DerivCalc_XDeriv_Y*})}}}{=} \frac{1}{n} \sum\limits_{r=1}^R A_r e_{i,d} e_{j,n}^\top B_r \bm{Y}^* \label{Eq_DerivCalc_XDeriv_S}\\
		& \frac{\partial}{\partial \ol{X_{i,j}}} \bm{S} \overset{\text{(\ref{Eq_DefS})}}{=} \frac{1}{n} \frac{\partial}{\partial \ol{X_{i,j}}} \big[ \bm{Y} \big] \bm{Y}^* + \frac{1}{n} \bm{Y} \frac{\partial}{\partial \ol{X_{i,j}}} \big[ \bm{Y}^* \big] \overset{\substack{\text{(\ref{Eq_DerivCalc_olXDeriv_Y})} \\ \text{(\ref{Eq_DerivCalc_olXDeriv_Y*})}}}{=} \frac{1}{n} \sum\limits_{s=1}^R \bm{Y} B_s^* e_{j,n} e_{i,d}^\top A_s^* \label{Eq_DerivCalc_olXDeriv_S}
	\end{align}
	and
	\begin{align}
		& \frac{\partial}{\partial X_{i,j}} \tilde{\bm{S}} \overset{\text{(\ref{Eq_DefS})}}{=} \frac{1}{n} \frac{\partial}{\partial X_{i,j}} \big[ \bm{Y}^* \big] \bm{Y} + \frac{1}{n} \bm{Y}^* \frac{\partial}{\partial X_{i,j}} \big[ \bm{Y} \big] \overset{\substack{\text{(\ref{Eq_DerivCalc_XDeriv_Y})} \\ \text{(\ref{Eq_DerivCalc_XDeriv_Y*})}}}{=} \frac{1}{n} \sum\limits_{r=1}^R \bm{Y}^* A_r e_{i,d} e_{j,n}^\top B_r \label{Eq_DerivCalc_XDeriv_tS}\\
		& \frac{\partial}{\partial \ol{X_{i,j}}} \tilde{\bm{S}} \overset{\text{(\ref{Eq_DefS})}}{=} \frac{1}{n} \frac{\partial}{\partial \ol{X_{i,j}}} \big[ \bm{Y}^* \big] \bm{Y} + \frac{1}{n} \bm{Y}^* \frac{\partial}{\partial \ol{X_{i,j}}} \big[ \bm{Y} \big] \overset{\substack{\text{(\ref{Eq_DerivCalc_olXDeriv_Y})} \\ \text{(\ref{Eq_DerivCalc_olXDeriv_Y*})}}}{=} \frac{1}{n} \sum\limits_{s=1}^R B_s^* e_{j,n} e_{i,d}^\top A_s^* \bm{Y} \label{Eq_DerivCalc_olXDeriv_tS} \ .
	\end{align}
	
	\item Derivatives of $\bm{R}(z)$ and $\tilde{\bm{R}}(z)$:\\
	Since $\Id_d = \big( \bm{S} -z\Id_d \big) \bm{R}(z)$ is constant, one by product rule has
	\begin{align*}
		0 & = \frac{\partial}{\partial X_{i,j}} \Big[ \big( \bm{S} -z\Id_d \big) \bm{R}(z) \Big]\\
		& = \frac{\partial}{\partial X_{i,j}} \big[ \bm{S} -z\Id_d \big] \bm{R}(z) + \big( \bm{S} -z\Id_d \big) \frac{\partial}{\partial X_{i,j}} \big[ \bm{R}(z) \big] \ ,
	\end{align*}
	which may be multiplied from the left with $\bm{R}(z)$ and rearranged to
	\begin{align}\label{Eq_DerivCalc_XDeriv_R}
		\frac{\partial}{\partial X_{i,j}} \big[ \bm{R}(z) \big] & = -\bm{R}(z) \frac{\partial}{\partial X_{i,j}} \big[ \bm{S} -z\Id_d \big] \bm{R}(z) \nonumber\\
		& \overset{\text{(\ref{Eq_DerivCalc_XDeriv_S})}}{=} - \frac{1}{n} \sum\limits_{r=1}^R \bm{R}(z) A_r e_{i,d} e_{j,n}^\top B_r \bm{Y}^* \bm{R}(z) \ .
	\end{align}
	The equalities (\ref{Eq_DerivCalc_olXDeriv_R_Result})-(\ref{Eq_DerivCalc_olXDeriv_tR_Result}) may be shown analogously by using (\ref{Eq_DerivCalc_olXDeriv_S})-(\ref{Eq_DerivCalc_olXDeriv_tS}). \qed
	
	\item Derivatives of $\bm{R}(z) \bm{Y}$ and $\bm{Y}^* \bm{R}(z)$:\\
	One calculates
	\begin{align}\label{Eq_DerivCalc_XDeriv_RY}
		& \frac{\partial}{\partial X_{i,j}} \big[ \bm{R}(z) \bm{Y} \big] = \frac{\partial}{\partial X_{i,j}} \big[ \bm{R}(z) \big] \bm{Y} + \bm{R}(z) \frac{\partial}{\partial X_{i,j}} \big[ \bm{Y} \big] \nonumber\\
		& \overset{\text{(\ref{Eq_DerivCalc_XDeriv_R_Result})\&(\ref{Eq_DerivCalc_XDeriv_Y})}}{=} - \frac{1}{n} \sum\limits_{r=1}^R \bm{R}(z) A_r e_{i,d} e_{j,n}^\top B_r \bm{Y}^* \bm{R}(z) \bm{Y} + \sum\limits_{r=1}^R \bm{R}(z) A_r e_{i,d} e_{j,n}^\top B_r \nonumber\\
		& \overset{\text{(\ref{Eq_DefS})}}{=} \sum\limits_{r=1}^R \bm{R}(z) A_r e_{i,d} e_{j,n}^\top B_r \big( \Id_n - \tilde{\bm{S}} \tilde{\bm{R}}(z) \big) \overset{\text{(\ref{Eq_Def_RResolvent})}}{=} -z \sum\limits_{r=1}^R \bm{R}(z) A_r e_{i,d} e_{j,n}^\top B_r \tilde{\bm{R}}(z)
	\end{align}
	and
	\begin{align}\label{Eq_DerivCalc_olXDeriv_YR}
		& \frac{\partial}{\partial \ol{X_{i,j}}} \big[ \bm{Y}^* \bm{R}(z) \big] = \frac{\partial}{\partial \ol{X_{i,j}}} \big[ \bm{Y}^* \big] \bm{R}(z) + \bm{Y}^* \frac{\partial}{\partial \ol{X_{i,j}}} \big[ \bm{R}(z) \big] \nonumber\\
		& \overset{\text{(\ref{Eq_DerivCalc_olXDeriv_Y*})\&(\ref{Eq_DerivCalc_olXDeriv_R_Result})}}{=} \sum\limits_{s=1}^R B_s^* e_{j,n} e_{i,d}^\top A_s^* \bm{R}(z) - \frac{1}{n} \sum\limits_{s=1}^R \bm{Y}^* \bm{R}(z) \bm{Y} B_s^* e_{j,n} e_{i,d}^\top A_s^* \bm{R}(z) \nonumber\\
		& \overset{\text{(\ref{Eq_DefS})}}{=} \sum\limits_{s=1}^R \big( \Id_n - \tilde{\bm{S}} \tilde{\bm{R}}(z) \big) B_s^* e_{j,n} e_{i,d}^\top A_s^* \bm{R}(z) \overset{\text{(\ref{Eq_Def_RResolvent})}}{=} -z \sum\limits_{s=1}^R \tilde{\bm{R}}(z) B_s^* e_{j,n} e_{i,d}^\top A_s^* \bm{R}(z) \ .
	\end{align}
	\qed
\end{itemize}

\subsection{Proof of Lemma~\ref{Lemma_TraceTriplet}}\label{Proof_Lemma_TraceTriplet}
Letting $B\circ C = (B_{i,j} C_{i,j})_{i,j}$ denote the entry-wise product of $B$ and $C$ (Hadamard product), then one by the Cauchy-Schwarz inequality sees
\begin{align*}
	& \bigg| \sum\limits_{i=1}^d \sum\limits_{j=1}^n A_{i,j} B_{i,j} C_{i,j} \bigg| = \bigg| \sum\limits_{i=1}^d \sum\limits_{j=1}^n A_{i,j} (B\circ C)_{i,j} \bigg| = \big| \tr\big( A (B\circ C)^\top \big) \big| \overset{\text{CS.}}{\leq} ||A||_F \, ||B\circ C||_F \ .
\end{align*}
By using the bound
\begin{align*}
	& ||B\circ C||_F^2 = \sum\limits_{i=1}^d \sum\limits_{j=1}^n |B_{i,j} C_{i,j}|^2 \overset{|C_{i,j}| \leq 1}{\leq} \sum\limits_{i=1}^d \sum\limits_{j=1}^n |B_{i,j}|^2 = ||B||_F^2
\end{align*}
and $||A||_F \overset{\text{(\ref{Eq_FrobeniusBounds})}}{\leq} \sqrt{\operatorname{rank}(A)} ||A||$, one thus observes
\begin{align*}
	& \bigg| \sum\limits_{i=1}^d \sum\limits_{j=1}^n A_{i,j} B_{i,j} C_{i,j} \bigg| \leq ||A||_F \, ||B||_F \leq \min(d,n) \, ||A|| \, ||B|| \ .
\end{align*}
\qed

\subsection{Proof of Lemma~\ref{Lemma_QSpectralBound}}\label{Proof_Lemma_QSpectralBound}
Letting $UDU^* = \bm{S}^{(X)}$ denote the spectral decomposition of $\bm{S}^{(X)}$, one observes
\begin{align}\label{Eq_R_SpectralDecomp}
	\bm{R}^{(X)}(z) & \overset{\text{(\ref{Eq_Def_RResolvent})}}{=} \big( \bm{S}^{(X)} - z\Id_d \big)^{-1} = U \big( D - z\Id_d \big)^{-1} U^* \nonumber\\
	& = U \diag\Big( \frac{1}{\lambda_1(\bm{S}^{(X)}) - z},\dots,\frac{1}{\lambda_d(\bm{S}^{(X)}) - z}\Big) U^* \ .
\end{align}
Since $\lambda_1(\bm{S}^{(X)}),\dots,\lambda_d(\bm{S}^{(X)})$ are in $[0,||\bm{S}^{(X)}||]$, one has
\begin{align*}
	& ||\bm{R}^{(X)}(z)|| \overset{\text{(\ref{Eq_R_SpectralDecomp})}}{\leq} \frac{1}{\dist(z,[0,||\bm{S}^{(X)}||])} \leq \frac{1}{\Im(z)} \ .
\end{align*}
One may use the fact that $\lambda_1(\tilde{\bm{S}}^{(X)}),\dots,\lambda_n(\tilde{\bm{S}}^{(X)})$ are also in $[0,||\bm{S}^{(X)}||]$ to analogously show $||\tilde{\bm{R}}^{(X)}(z)|| \leq \frac{1}{\Im(z)}$, which proves (\ref{Eq_R_SpectralBounds}).
\\[0.5em]
Since $\lambda_1(\bm{S}^{(X)}),\dots,\lambda_d(\bm{S}^{(X)}) \in [0,||\bm{S}^{(X)}||]$, one by (\ref{Eq_R_SpectralDecomp}) has
\begin{align*}
	\Im\big( \bm{R}^{(X)}(z) \big) & = U \diag\Big( \Im\Big( \frac{1}{\lambda_1(\bm{S}^{(X)}) - z} \Big),\dots,\Im\Big( \frac{1}{\lambda_d(\bm{S}^{(X)}) - z} \Big)\Big) U^*\\
	& = U \diag\Big( \frac{\Im(z)}{|\lambda_1(\bm{S}^{(X)}) - z|^2},\dots,\frac{\Im(z)}{|\lambda_d(\bm{S}^{(X)}) - z|^2}\Big) U^* \ ,
\end{align*}
which implies that $\Im\big( \bm{R}^{(X)}(z) \big)$ is positive definite and
\begin{align}\label{Eq_QBound_lamMinLowerBound_case1}
	& \lambda_{\min}\big( \Im\big( \bm{R}^{(X)}(z) \big) \big) \geq \frac{\Im(z)}{(|z|+||\bm{S}^{(X)}||)^2} \ .
\end{align}
One may then apply Lemma~\ref{Lemma_NonDegeneracyNew} with $M = \Im\big( \bm{R}^{(X)}(z) \big)$ to see
\begin{align*}
	& \lambda_{\min}\big(\Im(\tilde{\bm{Q}}^{(X)}(z))\big) \nonumber\\
	& \overset{\text{(\ref{Eq_Def_tQ})}}{=} \lambda_{\min}\Big( \sum\limits_{r,s=1}^R \frac{1}{n}\tr\big( A_r A_s^* \Im\big( \bm{R}^{(X)}(z) \big) \big) B_s^* B_r \Big) \overset{\text{(\ref{Eq_NonDegeneracyM_A})}}{\geq} \tau^2 \frac{\Im(z)}{(|z|+||\bm{S}^{(X)}||)^2} \ ,
\end{align*}
which implies
\begin{align*}
	& ||\tilde{\bm{Q}}^{(X)}(z)^{-1}|| \leq \frac{(|z|+||\bm{S}^{(X)}||)^2}{\tau^2 \Im(z)} \ .
\end{align*}
This proves the second bound in (\ref{Eq_QBounds}). The first bound in (\ref{Eq_QBounds}) may be shown by analogous methods. \qed

\subsection{Proof of Lemma~\ref{Lemma_ZTailBound}}\label{Proof_Lemma_ZTailBound}
Define independent random $(d \times n)$-matrices $\bm{V}$ and $\bm{W}$, where the entries of $\bm{V}$ are iid standard normal, and the entries of $\bm{W}$ are iid complex standard normal.
Following Definition~\ref{Def_SimilarGaussianMatrix}, one may for all $i \leq d$ and $j \leq n$ choose $a_{i,j},b_{i,j} \in \C$ such that the first and second moments of $a_{i,j}\bm{V}_{i,j} + b_{i,j}\bm{W}_{i,j} \eqcolon  \tilde{\bm{Z}}_{i,j}$ coincide with those of $\bm{Z}_{i,j}$, i.e.
\begin{align*}
	& \E[\tilde{\bm{Z}}_{i,j}] = 0 = \E[\bm{Z}_{i,j}] \ , \ \E[|\tilde{\bm{Z}}_{i,j}|^2] = 1 = \E[|\bm{Z}_{i,j}|^2] \ \text{ and } \ \E[\tilde{\bm{Z}}_{i,j}^2] = \E[\bm{Z}_{i,j}^2] \ .
\end{align*}
Since both $\bm{Z}_{i,j}$ and $\tilde{\bm{Z}}_{i,j}$ are complex Gaussian, it follows that $\bm{Z} \sim \tilde{\bm{Z}}$ and it suffices to prove the lemma for $\tilde{\bm{Z}}$ instead of $\bm{Z}$.
\\[0.5em]
Observe that the entries of $\bm{V}$, $\sqrt{2}\Re(\bm{W})$ and $\sqrt{2}\Im(\bm{W})$ are all independent and have standard normal distribution. View the map
\begin{align*}
	& f(\bm{V},\sqrt{2}\Re(\bm{W}),\sqrt{2}\Im(\bm{W})) \coloneq ||\tilde{\bm{Z}}||
\end{align*}
as a map from $\R^{3dn}$ to $[0,\infty)$ and observe that $f$ is Lipschitz continuous with Lipschitz constant $1$ by the calculation
\begin{align*}
	& \big| ||\tilde{\bm{Z}}|| - ||\ul{\tilde{\bm{Z}}}|| \big| \leq ||\tilde{\bm{Z}} - \ul{\tilde{\bm{Z}}}|| \overset{\text{(\ref{Eq_FrobeniusBounds})}}{\leq} ||\tilde{\bm{Z}} - \ul{\tilde{\bm{Z}}}||_F\\
	& \overset{\text{(\ref{Eq_SimilarDef_Z})}}{=} \bigg( \sum\limits_{i=1}^d \sum\limits_{j=1}^n \big|a_{i,j} (\bm{V}_{i,j} - \ul{\bm{V}}_{i,j}) + b_{i,j}\big( \Re(\bm{W})_{i,j} - \Re(\ul{\bm{W}})_{i,j} \big)\\
	& \hspace{5.5cm} + \bm{i} b_{i,j}\big( \Im(\bm{W})_{i,j} - \Im(\ul{\bm{W}})_{i,j} \big)\big|^2 \bigg)^{\frac{1}{2}}\\
	& \overset{\text{C.S.}}{\leq} \bigg( \sum\limits_{i=1}^d \sum\limits_{j=1}^n \big( \underbrace{|a_{i,j}|^2 + |b_{i,j}|^2}_{=1} \big) \big(|\bm{V}_{i,j} - \ul{\bm{V}}_{i,j}|^2 + |\Re(\bm{W})_{i,j} - \Re(\ul{\bm{W}})_{i,j}|^2\\
	& \hspace{5.5cm} + |\Im(\bm{W})_{i,j} - \Im(\ul{\bm{W}})_{i,j}|^2 \big) \bigg)^{\frac{1}{2}}\\
	& = \bigg( \sum\limits_{i=1}^d \sum\limits_{j=1}^n \big( |\bm{V}_{i,j} - \ul{\bm{V}}_{i,j}|^2 + \overbrace{|\Re(\bm{W})_{i,j} - \Re(\ul{\bm{W}})_{i,j}|^2}^{\leq |\sqrt{2}\Re(\bm{W})_{i,j} - \sqrt{2}\Re(\ul{\bm{W}})_{i,j}|^2}\\
	& \hspace{5.5cm} + \underbrace{|\Im(\bm{W})_{i,j} - \Im(\ul{\bm{W}})_{i,j}|^2}_{\leq |\sqrt{2}\Im(\bm{W})_{i,j} - \sqrt{2}\Im(\ul{\bm{W}})_{i,j}|^2} \bigg)^{\frac{1}{2}} \ .
\end{align*}
One may thus apply Theorem 5.2.2 of \cite{VershyninHDP} to see
\begin{align}\label{Eq_ZProtoTailBound}
	& \forall t > 0 : \ \bP\big( \big| ||\tilde{\bm{Z}}|| - \E\big[ ||\tilde{\bm{Z}}|| \big] \big| > t \big) \leq 2\exp\big( -C t^2 \big)
\end{align}
for some universal constant $C>0$. In analogy to Theorem 7.3.1 of \cite{VershyninHDP}, it will now be shown that
\begin{align}\label{Eq_ZSpectralExpectationBound}
	& \E[||\tilde{\bm{Z}}||] \leq 2\sqrt{d} + 2\sqrt{n} \ .
\end{align}
Begin by realizing $||\Re(\tilde{\bm{Z}})||$ and $||\Im(\tilde{\bm{Z}})||$ as supremum of real-valued Gaussian processes:
\begin{align*}
	& ||\Re(\tilde{\bm{Z}})|| = \max\limits_{\substack{x \in \R^{d}, \, y \in \R^{n} \\ ||x||_2 = 1 = ||y||_2}} x^\top \Re(\tilde{\bm{Z}}) y \ \ \text{ and } \ \ ||\Im(\tilde{\bm{Z}})|| = \max\limits_{\substack{x \in \R^{d}, \, y \in \R^{n} \\ ||x||_2 = 1 = ||y||_2}} x^\top \Im(\tilde{\bm{Z}}) y \ .
\end{align*}
For any $x,\ul{x} \in \R^d$ and $y,\ul{y} \in \R^n$ with $1 = ||x||_2 = ||\ul{x}||_2 = ||y||_2 = ||\ul{y}||_2$, one has
\begin{align*}
	& \E\big[ \big( x^\top \Re(\tilde{\bm{Z}}) y - \ul{x}^\top \Re(\tilde{\bm{Z}}) \ul{y} \big)^2 \big] = \E\bigg[ \bigg( \sum\limits_{i=1}^d \sum\limits_{j=1}^n \Re(\tilde{\bm{Z}}_{i,j}) \big( x_i y_j - \ul{x}_i \ul{y}_j \big) \bigg)^2 \bigg]\\
	& \overset{\text{(\ref{Eq_SimilarDef_Z})}}{=} \sum\limits_{i=1}^d \sum\limits_{j=1}^n \underbrace{\E\big[ \Re(\tilde{\bm{Z}}_{i,j})^2 \big]}_{\leq \E[|\tilde{\bm{Z}}_{i,j}|^2]=1} \big( x_i y_j - \ul{x}_i \ul{y}_j \big)^2 \leq \sum\limits_{i=1}^d \sum\limits_{j=1}^n \big( x_i y_j - \ul{x}_i \ul{y}_j \big)^2\\
	& = ||x y^\top - \ul{x} \ul{y}^\top||_F^2 \leq ||x-\ul{x}||_2^2 + ||y-\ul{y}||_2^2
\end{align*}
and analogously
\begin{align*}
	& \E\big[ \big( x^\top \Im(\tilde{\bm{Z}}) y - \ul{x}^\top \Im(\tilde{\bm{Z}}) \ul{y} \big)^2 \big] \leq ||x-\ul{x}||_2^2 + ||y-\ul{y}||_2^2 \ .
\end{align*}
One may then repeat the remaining steps of the proof of Theorem 7.3.1 in \cite{VershyninHDP} verbatim to get
\begin{align*}
	& \E[||\Re(\tilde{\bm{Z}})||] \leq \sqrt{d} + \sqrt{n} \ \ \text{ and } \ \ \E[||\Im(\tilde{\bm{Z}})||] \leq \sqrt{d} + \sqrt{n} \ .
\end{align*}
The simple observation $\E[||\tilde{\bm{Z}}||] \leq \E[||\Re(\tilde{\bm{Z}})||] + \E[||\Im(\tilde{\bm{Z}})||]$ thus proves (\ref{Eq_ZSpectralExpectationBound}).
Combining this with (\ref{Eq_ZProtoTailBound}) yields
\begin{align*}
	& \forall t \geq 0 : \ \bP\big( ||\tilde{\bm{Z}}|| > 2\big( \sqrt{d} + \sqrt{n} \big) + t \big) \leq 2\exp\big( -C t^2 \big) \ ,
\end{align*}
thus proving (\ref{Eq_ZTailBound}). For (\ref{Eq_YTailBound}), one observes
\begin{align*}
	& ||\bm{Y}^{(\tilde{\bm{Z}})}|| \overset{\text{(\ref{Eq_DefY})}}{\leq} \sum\limits_{r=1}^R ||A_r \tilde{\bm{Z}} B_r|| \overset{\substack{\text{(\ref{Eq_SubmultiplicativitySpectralNorm})} \\ \text{(\ref{Eq_sigmaAssumption_NonAsymp})}}}{\leq} \sigma^2 ||\tilde{\bm{Z}}|| \ ,
\end{align*}
which finally yields
\begin{align*}
	\forall t \geq 0 : \ & \bP\big( ||\bm{Y}^{(\tilde{\bm{Z}})}|| > 2\sigma^2\big( \sqrt{d} + \sqrt{n} \big) + t \big)\\
	& \leq \bP\big( \sigma^2 ||\tilde{\bm{Z}}|| > 2\sigma^2\big( \sqrt{d} + \sqrt{n} \big) + t \big)\\
	& = \bP\Big( ||\tilde{\bm{Z}}|| > 2\big( \sqrt{d} + \sqrt{n} \big) + \frac{t}{\sigma^2} \Big) \overset{\text{(\ref{Eq_ZTailBound})}}{\leq} 2\exp\Big( -\frac{Ct^2}{\sigma^4} \Big) \ .
\end{align*}
\qed

\subsection{Proof of Lemma~\ref{Lemma_SMomentBound}}\label{Proof_Lemma_SMomentBound}
Using the previous lemma and the standard bound \begin{align}\label{Eq_PowerSumBound}
	& \forall p \in \N \, \forall a,b>0 : \ (a+b)^p \leq 2^{p-1} (a^p + b^p) \ ,
\end{align}
observe
\begin{align*}
	& \E\big[ ||\bm{Z}||^{2m} \big] = 2m \int_0^\infty x^{2m-1} \bP\big( ||\bm{Z}|| > x \big)\\
	& \leq \overbrace{2m \int_0^{2\sqrt{d}+2\sqrt{n}} x^{2m-1} \, dx}^{= (2\sqrt{d}+2\sqrt{n})^{2m}}\\
	& \hspace{0.5cm} + 2m \int_0^\infty \underbrace{(2\sqrt{d}+2\sqrt{n} + t)^{2m-1}}_{\overset{\text{(\ref{Eq_PowerSumBound})}}{\leq} 2^{4m-3} (\sqrt{n}+\sqrt{d})^{2m-1} + 2^{2m-2} t^{2m-1}} \underbrace{\bP\big( ||\bm{Z}|| > 2\sqrt{d}+2\sqrt{n} + t \big)}_{\leq 2 \exp(-Ct^2) \text{ by Lemma~\ref{Lemma_ZTailBound}}} \, dt\\
	& \leq (2\sqrt{d}+2\sqrt{n})^{2m} + 2^{4m-2}m (\sqrt{n}+\sqrt{d})^{2m-1} \int_0^\infty \exp(-Ct^2) \, dt\\
	& \hspace{0.5cm} + 2^{2m-1}m \int_0^\infty  t^{2m-1} \exp(-Ct^2) \, dt\\
	& = (2\sqrt{d}+2\sqrt{n})^{2m} + 2^{4m-3}m (\sqrt{n}+\sqrt{d})^{2m-1} \sqrt{\frac{\pi}{C}} + 2^{2m-1}m \frac{(m-1)!}{2C^m}\\
	& \leq 2^{4m} (\sqrt{d}+\sqrt{n})^{2m} \Big( 1 + m \sqrt{\frac{\pi}{C}} + \frac{m!}{C^m} \Big) \ ,
\end{align*}
which proves (\ref{Eq_Z_MomentBound}). To follow (\ref{Eq_SZ_MomentBound}), one calculates
\begin{align*}
	& \E\big[ ||\bm{S}^{(\bm{Z})}||^m \big] \overset{\text{(\ref{Eq_DefS})}}{=} \frac{1}{n^m} \E\big[ ||\bm{Y}^{(\bm{Z})}||^{2m} \big]\\
	& \overset{\substack{\text{(\ref{Eq_DefY})} \\ \text{(\ref{Eq_sigmaAssumption_NonAsymp})}}}{\leq} \Big(\frac{\sigma^4}{n}\Big)^m \E\big[ ||\bm{Z}||^{2m} \big] \overset{\text{(\ref{Eq_Z_MomentBound})}}{\leq} \Big(\frac{16 \sigma^4 (\sqrt{d}+\sqrt{n})^2}{n}\Big)^m K_m \ .
\end{align*}
\qed

\begin{lemma}[Complex valued McDiarmid's inequality]\label{Lemma_McDiarmid}\
	\\
	For any Polish space $E$, let $f : E^N \rightarrow \C$ be a measurable function satisfying the \textit{bounded differences property}:
	\begin{align}\label{Eq_McDiarmid_BoundedDifferences}
		& \forall i \leq N : \ \sup\limits_{x_1,\dots,x_N,x_i' \in E} \big| f\big( x_1,\dots,x_{i-1},x_i',x_{i+1},\dots,x_N \big) - f\big( x_1,\dots,x_N \big) \big| \leq C
	\end{align}
	for some $C>0$, then for any sequence of independent random variables $X_1,\dots,X_N$ with values in $E$ and any $t>0$, it holds that
	\begin{align}\label{Eq_McDiarmid_Bound}
		& \bP\Big( \big| f(X_1,\dots,X_N) - \E\big[ f(X_1,\dots,X_N) \big] \big| > t \Big) \leq 4\exp\Big( -\frac{t^2}{NC^2} \Big) \ .
	\end{align}
\end{lemma}
\begin{proof}\
	\\
	The functions $\Re \circ f$ and $\Im \circ f$ take values in $\R$ and satisfy the same bounded difference property (\ref{Eq_McDiarmid_BoundedDifferences}). Applying the standard McDiarmid inequality (see Theorem 6.2 in \cite{BoucheronConcentrationInequalities}) to these two functions yields
	\begin{align*}
		& \bP\Big( \Re\big(f(X_1,\dots,X_N)\big) - \E\big[ \Re\big(f(X_1,\dots,X_N)\big) \big] > \frac{\eta}{\sqrt{2}} \Big) \leq \exp\Big( -\frac{\eta^2}{NC^2} \Big)\\
		& \bP\Big( \Re\big(f(X_1,\dots,X_N)\big) - \E\big[ \Re\big(f(X_1,\dots,X_N)\big) \big] < -\frac{\eta}{\sqrt{2}} \Big) \leq \exp\Big( -\frac{\eta^2}{NC^2} \Big)
	\end{align*}
	and
	\begin{align*}
		& \bP\Big( \Im\big(f(X_1,\dots,X_N)\big) - \E\big[ \Im\big(f(X_1,\dots,X_N)\big) \big] > \frac{\eta}{\sqrt{2}} \Big) \leq \exp\Big( -\frac{\eta^2}{NC^2} \Big)\\
		& \bP\Big( \Im\big(f(X_1,\dots,X_N)\big) - \E\big[ \Im\big(f(X_1,\dots,X_N)\big) \big] < -\frac{\eta}{\sqrt{2}} \Big) \leq \exp\Big( -\frac{\eta^2}{NC^2} \Big) \ ,
	\end{align*}
	which may be combined to
	\begin{align*}
		& \bP\Big( \big| f(X_1,\dots,X_N) - \E\big[ f(X_1,\dots,X_N) \big] \big| > \eta \Big)\\
		& \leq \bP\Big( \big| \Re\big(f(X_1,\dots,X_N)\big) - \E\big[ \Re\big(f(X_1,\dots,X_N)\big) \big] \big| > \frac{\eta}{\sqrt{2}} \Big)\\
		& \hspace{0.5cm} + \bP\Big( \big| \Im\big(f(X_1,\dots,X_N)\big) - \E\big[ \Im\big(f(X_1,\dots,X_N)\big) \big] \big| > \frac{\eta}{\sqrt{2}} \Big)\\
		& \leq 2\exp\Big( -\frac{\eta^2}{NC^2} \Big) + 2\exp\Big( -\frac{\eta^2}{NC^2} \Big) = 4\exp\Big( -\frac{\eta^2}{NC^2} \Big) \ .
	\end{align*}
\end{proof}

\subsection{Proof of Lemma~\ref{Lemma_R_Concentration}}\label{Proof_Lemma_R_Concentration}
It is first shown that the map $f : X \mapsto \frac{1}{n} \tr(M\bm{R}^{(X)}(z))$ satisfies the bounded differences property
\begin{align}\label{Eq_R_BoundedDifferences}
	& \forall i \leq n : \ \sup\limits_{\substack{X \in \C^{d \times n} \\ \tilde{x} \in \C^d}} \Big| f(X) - f(\tilde{X}^{(i)}) \Big| \leq \frac{4 R ||M||}{n \Im(z)} \ ,
\end{align}
where $X$ is an arbitrary $(d \times n)$-matrix and $\tilde{X}^{(i)}$ is a copy of $X$, where the $i$-th column is exchanged with a vector $\tilde{x} \in \C^d$.
\\[0.5em]
As the rank of the difference
\begin{align*}
	& \bm{Y}^{(X)} - \bm{Y}^{(\tilde{X}^{(i)})} \overset{\text{(\ref{Eq_DefY})}}{=} \sum\limits_{r=1}^R A_r (X - \tilde{X}^{(i)}) B_r \ , 
\end{align*}
may not exceed $R$, one also has
\begin{align*}
	& \operatorname{rank}\big(\bm{S}^{(X)} - \bm{S}^{(\tilde{X}^{(i)})}\big) \overset{\text{(\ref{Eq_DefS})}}{=} \operatorname{rank} \Big( \frac{1}{n} \bm{Y}^{(X)} (\bm{Y}^{(X)})^* - \frac{1}{n} \bm{Y}^{(\tilde{X}^{(i)})} (\bm{Y}^{(\tilde{X}^{(i)})})^* \Big)\\
	& = \operatorname{rank} \Big( \frac{1}{n} \bm{Y}^{(X)} \big(\underbrace{\bm{Y}^{(X)} - \bm{Y}^{(\tilde{X}^{(i)})}}_{\operatorname{rank}(\cdot) \leq R}\big)^* + \frac{1}{n} \big(\underbrace{\bm{Y}^{(X)} - \bm{Y}^{(\tilde{X}^{(i)})}}_{\operatorname{rank}(\cdot) \leq R}\big) (\bm{Y}^{(\tilde{X}^{(i)})})^* \Big) \leq 2R \ .
\end{align*}
It follows that
\begin{align}\label{Eq_R_rankDiffBound}
	& \operatorname{rank}\big( \bm{R}^{(X)}(z) - \bm{R}^{(\tilde{X}^{(i)})}(z) \big) \overset{\text{(\ref{Eq_InverseDifferenceIdentity})}}{=} \operatorname{rank}\big( \bm{R}^{(X)}(z) \big( \bm{S}^{(\tilde{X}^{(i)})} - \bm{S}^{(X)} \big) \bm{R}^{(\tilde{X}^{(i)})}(z) \big) \leq 2R \ .
\end{align}
The bounded difference property then follows from the calculation
\begin{align*}
	& \big| f(X) - f(\tilde{X}^{(i)}) \big| = \frac{1}{n} \Big| \tr\Big(M \big(\bm{R}^{(X)}(z) - \bm{R}^{(\tilde{X}^{(i)})}(z)\big)\Big) \Big| \nonumber\\
	& \overset{\substack{\text{(\ref{Eq_R_rankDiffBound})} \\ \text{(\ref{Eq_TraceBound})}}}{\leq} \frac{2 R}{n} \, ||M|| \big( ||\bm{R}^{(X)}(z)|| + ||\bm{R}^{(\tilde{X}^{(i)})}(z)|| \big) \overset{\text{(\ref{Eq_R_SpectralBounds})}}{\leq} \frac{4 R ||M||}{\Im(z) n} \ ,
\end{align*}
which proves (\ref{Eq_R_BoundedDifferences}). Lemma~\ref{Lemma_McDiarmid} with $N=n$, $E=\C^d$ and $C=\frac{4 R ||M||}{n \Im(z)}$ then yields
\begin{align}\label{Eq_trMRz_Concentration}
	& \forall z \in \C^+ , \, \forall M \in \C^{d \times d} , \, \forall t>0 : \nonumber\\
	& \bP\Big( \Big| \frac{1}{n} \tr(M\bm{R}^{(\bm{X})}(z)) - \E\Big[ \frac{1}{n} \tr(M\bm{R}^{(\bm{X})}(z)) \Big] \Big| > t \Big) \leq 4\exp\Big( -\frac{t^2\Im(z)^2 n}{16 R^2 ||M||^2} \Big) \ .
\end{align}
In order to extend this result to (\ref{Eq_R_Concentration}), one for all $z,z' \in \bD(\eta,\kappa)$ observes the Lipschitz property
\begin{align}\label{Eq_trMRz_Lipschitz}
	& \Big| \frac{1}{n} \tr(M\bm{R}^{(X)}(z)) - \frac{1}{n} \tr(M\bm{R}^{(X)}(z')) \Big| = \frac{1}{n} \big| \tr\big(M\big(\bm{R}^{(X)}(z) - \bm{R}^{(X)}(z')\big)\big) \big| \nonumber\\
	& \overset{\text{(\ref{Eq_TraceBound})}}{\leq} \frac{d}{n} ||M|| \, \big|\big| \bm{R}^{(X)}(z) - \bm{R}^{(X)}(z') \big|\big| \overset{\text{(\ref{Eq_R_SpectralDecomp})}}{=} \frac{d}{n} ||M|| \, \max\limits_{j\leq d} \Big| \frac{1}{\lambda_j(\bm{S}^{(X)}) - z} - \frac{1}{\lambda_j(\bm{S}^{(X)}) - z'} \Big| \nonumber\\
	& \leq \frac{d}{n} ||M|| \frac{|z-z'|}{\Im(z) \Im(z')} \leq \frac{d}{n} ||M|| \frac{|z-z'|}{\eta^2} \overset{\text{\ref{ItemAssumption_cBound}}}{\leq} \frac{c_*}{\eta^2} ||M|| \, |z-z'| \ .
\end{align}
By construction of $\bD(\eta,\kappa)$, there for any $\delta \in (0,\kappa)$ exists a $\delta$-net $J_\delta \subset \bD(\eta,\kappa)$ with the properties
\begin{align}\label{Eq_delta_Net_Properties}
	& \bD(\eta,\kappa) \subset \bigcup\limits_{j \in J_\delta} B^{\C}_\delta(j) \ \ \text{ and } \ \ \#J_\delta \leq \frac{4\kappa^2}{\delta^2} \ ,
\end{align}
where $B^{\C}_\delta(j)$ denotes the (closed) $\delta$-neighborhood around $j$ in $\C$ and $\#J_\delta$ denotes the cardinality of $J_\delta$. Letting $\Pi_\delta : \bD(\eta,\kappa) \rightarrow J_\delta$ denote a (measurable) map with the property $|z-\Pi_\delta(z)| \leq \delta$, one by (\ref{Eq_trMRz_Lipschitz}) has
\begin{align}\label{Eq_trMRz_NetBound}
	& \forall z \in \bD(\eta,\kappa) : \ \Big| \frac{1}{n} \tr(M\bm{R}^{(X)}(z)) - \frac{1}{n} \tr(M\bm{R}^{(X)}(\Pi_\delta(z))) \Big| \leq \frac{c_*}{\eta^2} ||M|| \delta \ .
\end{align}
By combining (\ref{Eq_trMRz_Concentration}) and (\ref{Eq_trMRz_NetBound}), one by triangle inequality gets
\begin{align*}
	& \bP\Big( \sup\limits_{z \in \bD(\eta,\kappa)}\Big| \frac{1}{n} \tr(M\bm{R}^{(\bm{X})}(z)) - \E\Big[ \frac{1}{n} \tr(M\bm{R}^{(\bm{X})}(z)) \Big] \Big| > t + 2\frac{c_* \delta}{\eta^2} ||M|| \Big)\\
	& \overset{\text{(\ref{Eq_trMRz_NetBound})}}{\leq} \bP\Big( \sup\limits_{z \in \bD(\eta,\kappa)}\Big| \frac{1}{n} \tr(M\bm{R}^{(\bm{X})}(\Pi_\delta(z))) - \E\Big[ \frac{1}{n} \tr(M\bm{R}^{(\bm{X})}(\Pi_\delta(z))) \Big] \Big| > t \Big)\\
	& \leq \bP\Big( \max\limits_{j \in J_\delta}\Big| \frac{1}{n} \tr(M\bm{R}^{(\bm{X})}(j)) - \E\Big[ \frac{1}{n} \tr(M\bm{R}^{(\bm{X})}(j)) \Big] \Big| > t \Big)\\
	& \leq \sum\limits_{j \in J_\delta} \bP\Big( \Big| \frac{1}{n} \tr(M\bm{R}^{(\bm{X})}(j)) - \E\Big[ \frac{1}{n} \tr(M\bm{R}^{(\bm{X})}(j)) \Big] \Big| > t \Big)\\
	& \overset{\text{(\ref{Eq_trMRz_Concentration})}}{\leq} \sum\limits_{j \in J_\delta} 4\exp\Big( -\frac{t^2\Im(j)^2 n}{16 R^2 ||M||^2} \Big) \overset{\text{(\ref{Eq_delta_Net_Properties})}}{\leq} \frac{16\kappa^2}{\delta^2} \exp\Big( -\frac{t^2\eta^2 n}{16 R^2 ||M||^2} \Big) \ ,
\end{align*}
thus proving (\ref{Eq_R_Concentration}). The property (\ref{Eq_tR_Concentration}) may be shown with analogous arguments. \qed

\subsection{Proof of Lemma~\ref{Lemma_ConcentrationLargeExpectation}}\label{Proof_Lemma_ConcentrationLargeExpectation}
A proof of (\ref{Eq_QConcentrationA}) is given below. The result (\ref{Eq_QConcentrationB}) may be shown by following the proof verbatim, while replacing every occurrence of $(M,\bm{Q},\bm{R},\bm{S})$ with their $(\tilde{M},\tilde{\bm{Q}},\tilde{\bm{R}},\tilde{\bm{S}})$-counterparts and vice versa. The occurrences of $\frac{d}{n}$ must also be replaced with $\frac{n}{n}$, which does not impede the bound by $c_*$, since it was in~\ref{ItemAssumption_cBound} assumed that $c_* \geq 1$.
\\[0.5em]
Without loss of generality assume
\begin{align}\label{Eq_QConcentration_Wlog_kappa}
	& \eta < 1 \ \ \text{ and } \ \ \kappa > 1 \ .
\end{align}
Define
\begin{align}\label{Eq_QConcentration_Def_kappa}
	& \kappa_* \coloneq 2\big( \sqrt{d} + \sqrt{n} \big) + \sqrt{n}
\end{align}
and let $U_{\kappa_*} \subset \C^{d \times n}$ denote the set of matrices
\begin{align}\label{Eq_QConcentration_DefUkappa}
	& U_{\kappa_*} = \big\{ X \in \C^{d \times n} \ \big| \ ||X|| \leq \kappa_* \big\} \ .
\end{align}
Also, for any $X \in U_{\kappa_*}$, observe
\begin{align}\label{Eq_Q_Concentration_SX_Bound}
	& ||\bm{S}^{(X)}|| \overset{\text{(\ref{Eq_Def_S_Robust})}}{\leq} \frac{1}{n} ||\bm{Y}^{(X)}||^2 \overset{\substack{\text{(\ref{Eq_Def_Y_Robust})} \\ \text{\ref{ItemAssumption_sigmaBound}}}}{\leq} \frac{\sigma^4}{n} ||X||^2 \leq \frac{\kappa_*^2 \sigma^4}{n} \nonumber\\
	& = \sigma^4 \frac{(2\big( \sqrt{d} + \sqrt{n} \big) + \sqrt{n})^2}{n} \overset{\text{\ref{ItemAssumption_cBound}}}{\leq} \sigma^4 (2\sqrt{c_*} + 3 )^2 \eqcolon  K_* = K_*(c_*,\sigma^2) \ .
\end{align}
\begin{itemize}
	\item[i)] \textit{Bounding $\bP(\bm{Z} \notin U_{\kappa_*})$}:\\
	By Lemma~\ref{Lemma_ZTailBound} observe
	\begin{align}\label{Eq_Q_Concentration_PZBound}
		& \bP\big(\bm{Z} \notin U_{\kappa_*}\big) \overset{\text{(\ref{Eq_QConcentration_DefUkappa})}}{=} \bP\big( ||\bm{Z}|| > \kappa_* \big) \nonumber\\
		& \overset{\text{(\ref{Eq_ZTailBound})}}{=} \bP\big( ||\bm{Z}|| > 2\big( \sqrt{d} + \sqrt{n} \big) + \sqrt{n} \big) \overset{\text{(\ref{Eq_ZTailBound})}}{\leq} 2\exp\big( -C n \big) \ .
	\end{align}
	
	\item[ii)] \textit{Lipschitz-properties for $X \in U_{\kappa_*}$}:\\
	For any $z \in \bD(\eta,\kappa)$, $M \in \C^{d \times d}$ and $X,Y \in U_{\kappa_*}$ one has the bound
	\begin{align*}
		& \Big| \frac{1}{z n} \tr(M\bm{Q}^{(X)}(z)^{-1}) - \frac{1}{z n} \tr(M\bm{Q}^{(Y)}(z)^{-1}) \Big|\\
		& \overset{\text{(\ref{Eq_TraceBound})}}{\leq} \frac{1}{|z|} \frac{d}{n} ||M|| \, \big|\big| \bm{Q}^{(X)}(z)^{-1} - \bm{Q}^{(Y)}(z)^{-1} \big|\big|\\
		& \overset{\text{(\ref{Eq_InverseDifferenceIdentity})}}{=} \frac{1}{|z|} \frac{d}{n} ||M|| \, \big|\big| \bm{Q}^{(X)}(z)^{-1} \big( \bm{Q}^{(Y)}(z) - \bm{Q}^{(X)}(z) \big) \bm{Q}^{(Y)}(z)^{-1} \big|\big|\\
		& \overset{\text{(\ref{Eq_QBounds})}}{\leq}  \underbrace{\frac{1}{|z|} \frac{d}{n}}_{\leq \frac{c_*}{\eta}} ||M|| \, \big|\big| \bm{Q}^{(Y)}(z) - \bm{Q}^{(X)}(z) \big|\big| \, \frac{(\overbrace{|z|}^{\leq \kappa} + \overbrace{||\bm{S}^{(X)}||}^{\overset{\text{(\ref{Eq_Q_Concentration_SX_Bound})}}{\leq} K_*})^2 (\overbrace{|z|}^{\leq \kappa} + \overbrace{||\bm{S}^{(Y)}||}^{\overset{\text{(\ref{Eq_Q_Concentration_SX_Bound})}}{\leq} K_*})^2}{\tau^4 \underbrace{\Im(z)^2}_{\geq \eta^2}}\\
		& \overset{\text{(\ref{Eq_Def_Q})}}{\leq} \frac{c_* (\kappa+K_*)^4 ||M||}{\tau^4 \eta^3} \Big|\Big| \sum\limits_{r,s=1}^R \frac{1}{n} \tr\big( B_s^*B_r \big( \tilde{\bm{R}}^{(Y)}(z) - \tilde{\bm{R}}^{(X)}(z) \big) \big) A_rA_s^* \Big|\Big|\\
		& \overset{\substack{\text{(\ref{Eq_TraceBound})} \\ \text{\ref{ItemAssumption_sigmaBound}}}}{\leq} \frac{c_* \sigma^4 (\kappa+K_*)^4 ||M||}{\tau^4 \eta^3} \underbrace{\frac{d}{n}}_{\leq c_*} \big|\big| \tilde{\bm{R}}^{(Y)}(z) - \tilde{\bm{R}}^{(X)}(z) \big|\big|\\
		& \overset{\substack{\text{(\ref{Eq_InverseDifferenceIdentity})} \\ \text{(\ref{Eq_Def_R_Robust})}}}{=} \frac{c_*^2 \sigma^4 (\kappa+K_*)^4 ||M||}{\tau^4 \eta^3} \big|\big| \tilde{\bm{R}}^{(Y)}(z) \big( \tilde{\bm{S}}^{(X)} - \tilde{\bm{S}}^{(Y)} \big) \tilde{\bm{R}}^{(X)}(z) \big|\big|\\
		& \overset{\substack{\text{(\ref{Eq_R_SpectralBounds})} \\ \text{(\ref{Eq_Def_S_Robust})}}}{\leq} \frac{c_*^2 \sigma^4 (\kappa+K_*)^4 ||M||}{\tau^4 \eta^5 n} \big|\big| (\bm{Y}^{(X)})^*\bm{Y}^{(X)} - (\bm{Y}^{(Y)})^*\bm{Y}^{(Y)} \big|\big|\\
		& \leq \frac{c_*^2 \sigma^4 (\kappa+K_*)^4 ||M||}{\tau^4 \eta^5 n} \Big( \big|\big| \bm{Y}^{(X)} \big|\big| \, \big|\big| \bm{Y}^{(X)} - \bm{Y}^{(Y)} \big|\big| + \big|\big| \bm{Y}^{(X)} - \bm{Y}^{(Y)} \big|\big| \, \big|\big| \bm{Y}^{(Y)} \big|\big| \Big)\\
		& \leq \frac{c_*^2 \sigma^4 (\kappa+K_*)^4 ||M||}{\tau^4 \eta^5 \sqrt{n}} \big( \underbrace{\frac{1}{\sqrt{n}}\big|\big| \bm{Y}^{(X)} \big|\big|}_{\overset{\text{(\ref{Eq_Q_Concentration_SX_Bound})}}{\leq} \sqrt{K_*}} + \underbrace{\frac{1}{\sqrt{n}}\big|\big| \bm{Y}^{(Y)} \big|\big|}_{\overset{\text{(\ref{Eq_Q_Concentration_SX_Bound})}}{\leq} \sqrt{K_*}} \big) \big|\big| \bm{Y}^{(X)} - \bm{Y}^{(Y)} \big|\big|\\
		& \overset{\text{(\ref{Eq_Def_Y_Robust})}}{\leq} \frac{2 c_*^2 \sigma^4 \sqrt{K_*} (\kappa+K_*)^4 ||M||}{\tau^4 \eta^5 \sqrt{n}} \Big|\Big| \sum\limits_{r=1}^R A_r \big( X-Y \big) B_r \Big|\Big|\\
		& \overset{\text{\ref{ItemAssumption_sigmaBound}}}{\leq} \frac{2 c_*^2 \sigma^6 \sqrt{K_*} (\kappa+K_*)^4 ||M||}{\tau^4 \eta^5 \sqrt{n}} ||X-Y|| \ ,
	\end{align*}
	which for
	\begin{align}\label{Eq_Q_Concentration_DefL}
		& L = L(M,n) \coloneq \frac{2 c_*^2 \sigma^6 \sqrt{K_*} (\kappa+K_*)^4 ||M||}{\tau^4 \eta^5 \sqrt{n}}
	\end{align}
	proves the Lipschitz-property
	\begin{align}\label{Eq_Q_Concentration_L_Lipschitz_Q}
		& \forall z \in \bD(\eta,\kappa) , \, \forall M \in \C^{d \times d} , \, \forall X,Y \in U_{\kappa_*} : \nonumber\\
		& \Big| \frac{1}{z n} \tr(M\bm{Q}^{(X)}(z)^{-1}) - \frac{1}{z n} \tr(M\bm{Q}^{(Y)}(z)^{-1}) \Big| \leq L \, ||X-Y|| \ .
	\end{align}
	A similar Lipschitz-property in $z$ is now also shown. For any $M \in \C^{d \times d}$, $X \in U_{\kappa_*}$ and $z,z' \in \bD(\eta,\kappa)$ one has the bound
	\begin{align*}
		& \Big| \frac{1}{z n} \tr(M\bm{Q}^{(X)}(z)^{-1}) - \frac{1}{z' n} \tr(M\bm{Q}^{(X)}(z')^{-1}) \Big|\\
		& \leq \Big| \frac{1}{z n} \tr(M\bm{Q}^{(X)}(z)^{-1}) - \frac{1}{z n} \tr(M\bm{Q}^{(X)}(z')^{-1}) \Big| + \Big| \frac{1}{z} - \frac{1}{z'} \Big| \, \Big| \frac{1}{n} \tr(M\bm{Q}^{(X)}(z')^{-1}) \Big|\\
		& \overset{\text{(\ref{Eq_TraceBound})}}{\leq} \frac{1}{|z|} \frac{d}{n} ||M|| \, \big|\big| \bm{Q}^{(X)}(z)^{-1} - \bm{Q}^{(X)}(z')^{-1} \big|\big| + \frac{|z-z'|}{|zz'|} \frac{d}{n} ||M|| \, \big|\big|\bm{Q}^{(X)}(z')^{-1}\big|\big|\\
		& \overset{\text{(\ref{Eq_InverseDifferenceIdentity})}}{=} \frac{1}{|z|} \frac{d}{n} ||M|| \, \big|\big| \bm{Q}^{(X)}(z)^{-1} \big( \bm{Q}^{(X)}(z') - \bm{Q}^{(X)}(z) \big) \bm{Q}^{(X)}(z')^{-1} \big|\big|\\
		& \hspace{0.5cm} + \frac{|z-z'|}{|zz'|} \frac{d}{n} ||M|| \, \big|\big|\bm{Q}^{(X)}(z')^{-1}\big|\big|\\
		& \overset{\text{(\ref{Eq_QBounds})}}{\leq} \underbrace{\frac{1}{|z|} \frac{d}{n}}_{\leq \frac{c_*}{\eta}} ||M|| \, \big|\big| \bm{Q}^{(X)}(z') - \bm{Q}^{(X)}(z) \big|\big| \, \frac{(\overbrace{|z|}^{\leq \kappa} + \overbrace{||\bm{S}^{(X)}||}^{\overset{\text{(\ref{Eq_Q_Concentration_SX_Bound})}}{\leq} K_*})^4}{\tau^4 \underbrace{\Im(z)\Im(z')}_{\geq \eta^2}}\\
		& \hspace{0.5cm} + \underbrace{\frac{|z-z'|}{|zz'|} \frac{d}{n}}_{\leq \frac{c_*}{\eta^2} |z-z'|} ||M|| \, \frac{(\overbrace{|z|}^{\leq \kappa} + \overbrace{||\bm{S}^{(X)}||}^{\overset{\text{(\ref{Eq_Q_Concentration_SX_Bound})}}{\leq} K_*})^2}{\tau^2 \underbrace{\Im(z')}_{\geq \eta}}\\
		& \overset{\text{(\ref{Eq_Def_Q})}}{\leq} \frac{c_* (\kappa+K_*)^4 ||M||}{\tau^4 \eta^3} \Big|\Big| \sum\limits_{r,s=1}^R \frac{1}{n} \tr\big( B_s^*B_r \big( \tilde{\bm{R}}^{(X)}(z') - \tilde{\bm{R}}^{(X)}(z) \big) \big) A_rA_s^* \Big|\Big|\\
		& \hspace{0.5cm} + \frac{c_* (\kappa+K_*)^2 ||M||}{\tau^2 \eta^3} |z-z'|\\
		& \overset{\substack{\text{(\ref{Eq_TraceBound})} \\ \text{\ref{ItemAssumption_sigmaBound}}}}{\leq} \frac{c_* \sigma^4 (\kappa+K_*)^4 ||M||}{\tau^4 \eta^3} \frac{d}{n} \big|\big| \tilde{\bm{R}}^{(X)}(z') - \tilde{\bm{R}}^{(X)}(z) \big|\big| + \frac{c_* (\kappa+K_*)^2 ||M||}{\tau^2 \eta^3} |z-z'|\\
		& \overset{\text{(\ref{Eq_R_SpectralDecomp})}}{=} \frac{c_* \sigma^4 (\kappa+K_*)^4 ||M||}{\tau^4 \eta^3} \overbrace{\frac{d}{n}}^{\leq c_*} \max\limits_{j\leq n} \overbrace{\Big| \frac{1}{\lambda_j(\tilde{\bm{S}}^{(X)}) - z} - \frac{1}{\lambda_j(\tilde{\bm{S}}^{(X)}) - z'} \Big|}^{\leq \frac{|z-z'|}{\Im(z)\Im(z')} \leq \frac{|z-z'|}{\eta^2}}\\
		& \hspace{0.5cm} + \frac{c_* (\kappa+K_*)^2 ||M||}{\tau^2 \eta^3} |z-z'|\\
		& \leq \frac{c_*^2 \sigma^4 (\kappa+K_*)^4 ||M||}{\tau^4 \eta^5} |z-z'| + \frac{c_* (\kappa+K_*)^2 ||M||}{\tau^2 \eta^3} |z-z'|\\
		& \overset{\substack{c_*^2, \sigma^2, \kappa > 1 \\ \tau, \eta < 1}}{\leq} \frac{2 c_*^2 \sigma^4 (\kappa+K_*)^4 ||M||}{\tau^4 \eta^5} |z-z'| \ ,
	\end{align*}
	which for
	\begin{align}\label{Eq_Q_Concentration_DefLz}
		& L' = L'(M) \coloneq \frac{2 c_*^2 \sigma^4 (\kappa+K_*)^4 ||M||}{\tau^4 \eta^5}
	\end{align}
	proves the Lipschitz-property
	\begin{align}\label{Eq_Q_Concentration_Lz_Lipschitz_Q}
		& \forall M \in \C^{d \times d} , \, \forall X \in U_{\kappa_*} , \, \forall z,z' \in \bD(\eta,\kappa) : \nonumber\\
		& \Big| \frac{1}{z n} \tr(M\bm{Q}^{(X)}(z)^{-1}) - \frac{1}{z' n} \tr(M\bm{Q}^{(X)}(z')^{-1}) \Big| \leq L' \, |z-z'| \ .
	\end{align}
	
	\item[iii)] \textit{Concentration of the McShane extension for fixed $z \in \bD(\eta,\kappa)$}:\\
	Fix $z \in \bD(\eta,\kappa)$ and $M \in \C^{d \times d}$ and for ease on notation define
	\begin{align}\label{Eq_Q_Concentration_Def_f}
		& f : \C^{d \times n} \rightarrow \R \ \ ; \ \ f(X) \coloneq \Re\Big(\frac{1}{z n} \tr(M\bm{Q}^{(X)}(z)^{-1})\Big) \ ,
	\end{align}
	which by (\ref{Eq_Q_Concentration_L_Lipschitz_Q}) is Lipschitz continuous with Lipschitz constant $L$. Let $\tilde{f}$ denote the McShane extension of $f$ from $U_{\kappa_*}$ to $\C^{d \times n}$
	\begin{align}\label{Eq_Q_Concentration_DefMcShaneExtension}
		& \tilde{f}(X) = \sup\limits_{Y \in U_{\kappa_*}} \big( f(Y) - L ||X-Y|| \big) \ ,
	\end{align}
	which by \cite{McShaneExtension} satisfies the properties
	\begin{align}
		& \forall X \in U_{\kappa_*} : \  \tilde{f}(X) = f(X) \label{Eq_Q_Concentration_McShaneExtension1}
	\end{align}
	and
	\begin{align}
		& \forall X,Y \in \C^{d \times n} : \ \big| \tilde{f}(X) - \tilde{f}(Y) \big| \leq L \, ||X - Y|| \ . \label{Eq_Q_Concentration_McShaneExtension2}
	\end{align}
	The Gaussian concentration result given in Theorem 5.2.2 of \cite{VershyninHDP}, by a completely analogous argument to the proof of Lemma~\ref{Lemma_ZTailBound}, yields
	\begin{align}\label{Eq_Q_Concentration_McShaneConcnetration}
		& \forall t > 0 : \ \bP\Big( \big| \tilde{f}(\bm{Z}) - \E\big[ \tilde{f}(\bm{Z}) \big] \big| > t \Big) \leq 2\exp\Big( -\frac{Ct^2}{L^2} \Big)
	\end{align}
	for a universal constant $C>0$, which in this case coincides with the universal constant from Lemma~\ref{Lemma_ZTailBound}.
	
	\item[iv)] \textit{Bounding $\E[f(\bm{Z})-\tilde{f}(\bm{Z})]$}:\\
	First, observe
	\begin{align}\label{Eq_Q_Concentration_MeanDiffBound1}
		& \big| \E[f(\bm{Z})-\tilde{f}(\bm{Z})] \big| \leq \E\big[\big|f(\bm{Z})-\tilde{f}(\bm{Z})\big|\big] \nonumber\\
		& = \int_0^\infty \bP\big( \big|f(\bm{Z})-\tilde{f}(\bm{Z})\big| > t \big) \, dt \nonumber\\
		& \overset{\text{(\ref{Eq_Q_Concentration_DefMcShaneExtension})}}{\leq} \int_0^\infty \bP\Big( \bm{Z} \notin U_{\kappa_*} \ \text{ and } \ |f(\bm{Z})| + |\tilde{f}(\bm{Z})| > t \Big) \, dt \ .
	\end{align}
	Using Lemma~\ref{Lemma_QSpectralBound}, one observes
	\begin{align*}
		& |f(\bm{Z})| \overset{\text{(\ref{Eq_TraceBound})}}{\leq} \overbrace{\frac{1}{|z|} \frac{d}{n}}^{\leq \frac{c_*}{\eta}} ||M|| \, ||\bm{Q}^{(\bm{Z})}(z)^{-1}|| \nonumber\\
		& \overset{\text{(\ref{Eq_QBounds})}}{\leq} \frac{c_*}{\eta} ||M|| \frac{(\overbrace{|z|}^{\leq \kappa} + ||\bm{S}^{(\bm{Z})}||)^2}{\tau^2 \eta} \leq \frac{2c_* ||M||}{\tau^2 \eta^2} \big( \kappa^2 + ||\bm{S}^{(\bm{Z})}||^2 \big)\\
		& \overset{\substack{\text{(\ref{Eq_Def_S_Robust})}}}{\leq} \frac{2c_* ||M||}{\tau^2 \eta^2} \big( \kappa^2 + ||\bm{Y}^{(\bm{Z})}||^4/n^2 \big)\\
		& \overset{\substack{\text{(\ref{Eq_Def_Y_Robust})} \\ \text{\ref{ItemAssumption_sigmaBound}}}}{\leq} \frac{2c_* ||M||}{\tau^2 \eta^2} \Big( \kappa^2 + ||\bm{Z}||^4 \frac{\sigma^8}{n^2} \Big)
	\end{align*}
	and
	\begin{align*}
		& |\tilde{f}(\bm{Z})| \overset{\text{(\ref{Eq_Q_Concentration_DefMcShaneExtension})}}{\leq} \sup\limits_{Y \in U_{\kappa_*}} |f(Y)| + L ||\bm{Z}|| \nonumber\\
		& \overset{\text{(\ref{Eq_QBounds})}}{\leq} \frac{c_*}{\eta} ||M|| \frac{(\overbrace{|z|}^{\leq \kappa} + \sup\limits_{Y \in U_{\kappa_*}}||\bm{S}^{(Y)}||)^2}{\tau^2 \eta}  + L ||\bm{Z}|| \nonumber\\
		& \overset{\text{(\ref{Eq_Q_Concentration_SX_Bound})}}{\leq} \frac{2c_* ||M||}{\tau^2 \eta^2} \big( \kappa^2 + K_*^2 \big) + L ||\bm{Z}|| \ .
	\end{align*}
	Inserting these two bounds back into (\ref{Eq_Q_Concentration_MeanDiffBound1}) yields
	\begin{align}
		& \big| \E[f(\bm{Z})-\tilde{f}(\bm{Z})] \big| \nonumber\\
		& \leq \int_0^\infty \bP\Big( \bm{Z} \notin U_{\kappa_*} \ \text{ and } \ \frac{2c_* ||M||}{\tau^2 \eta^2} \big( 2\kappa^2 + K_*^2 + ||\bm{Z}||^4 \frac{\sigma^8}{n^2} \big) + L \underbrace{||\bm{Z}||}_{\overset{\kappa_*>1}{\leq} ||\bm{Z}||^4} > t \Big) \, dt \nonumber\\
		& \leq \int_0^\infty \bP\Big( \bm{Z} \notin U_{\kappa_*} \ \text{ and } \ \Big( \frac{2c_* ||M||}{\tau^2 \eta^2} \frac{\sigma^8}{n^2} + L \Big) ||\bm{Z}||^4 > t - \frac{2c_* ||M||}{\tau^2 \eta^2} (2\kappa^2 + K_*^2) \Big) \, dt \nonumber\\
		& \leq \bP\big( \bm{Z} \notin U_{\kappa_*} \big) \, \frac{2c_* ||M||}{\tau^2 \eta^2} (2\kappa^2 + K_*^2) \label{Eq_Q_Concentration_MeanDiffDecomp1}\\
		& \hspace{0.5cm} + \int_0^\infty \bP\Big( \bm{Z} \notin U_{\kappa_*} \ \text{ and } \ \Big( \frac{2c_* ||M||}{\tau^2 \eta^2} \frac{\sigma^8}{n^2} + L \Big) ||\bm{Z}||^4 > t \Big) \, dt \ . \label{Eq_Q_Concentration_MeanDiffDecomp2}
	\end{align}
	The expression (\ref{Eq_Q_Concentration_MeanDiffDecomp1}) is easily bounded by
	\begin{align}\label{Eq_Q_Concentration_MeanDiffDecomp1_Bound}
		& \text{(\ref{Eq_Q_Concentration_MeanDiffDecomp1})} \overset{\text{(\ref{Eq_Q_Concentration_PZBound})}}{\leq} 2\exp\big( -C n \big) \frac{2c_* ||M||}{\tau^2 \eta^2} (2\kappa^2 + K_*^2) \ .
	\end{align}
	For the expression (\ref{Eq_Q_Concentration_MeanDiffDecomp2}), one calculates
	\begin{align*}
		& \text{(\ref{Eq_Q_Concentration_MeanDiffDecomp2})} = \int_0^\infty \bP\Big( ||\bm{Z}|| > \Big(\frac{t}{\frac{2c_* ||M||}{\tau^2 \eta^2} \frac{\sigma^8}{n^2} + L}\Big)^{\frac{1}{4}} \lor \kappa_* \Big) \, dt\\
		& = \bP\big( ||\bm{Z}||>\kappa_* \big) \Big( \frac{2c_* ||M||}{\tau^2 \eta^2} \frac{\sigma^8}{n^2} + L \Big) \kappa_*^4\\
		& \hspace{0.5cm} + \int_{\big( \frac{2c_* ||M||}{\tau^2 \eta^2} \frac{\sigma^8}{n^2} + L \big) \kappa_*^4}^\infty \bP\Big( ||\bm{Z}|| > \Big(\frac{t}{\frac{2c_* ||M||}{\tau^2 \eta^2} \frac{\sigma^8}{n^2} + L}\Big)^{\frac{1}{4}} \Big) \, dt\\
		& = \overbrace{\bP\big( ||\bm{Z}||>\kappa_* \big)}^{\overset{\text{(\ref{Eq_Q_Concentration_PZBound})}}{\leq} 2\exp(-C n)} \Big( \frac{2c_* ||M||}{\tau^2 \eta^2} \frac{\sigma^8}{n^2} + L \Big) \kappa_*^4\\
		& \hspace{0.5cm} + \int_{0}^\infty \Big( \frac{2c_* ||M||}{\tau^2 \eta^2} \frac{\sigma^8}{n^2} + L \Big) 4(\kappa_*+s)^3 \underbrace{\bP\big( ||\bm{Z}|| > \kappa_* + s \big)}_{\substack{\overset{\substack{\text{(\ref{Eq_QConcentration_Def_kappa})} \\ \text{(\ref{Eq_ZTailBound})}}}{\leq} 2\exp(-C (\sqrt{n}+s)^2) \\ \leq 2\exp(-C (n+s^2))}} \, ds\\
		& \leq \exp(-C n) \Big( \frac{2c_* ||M||}{\tau^2 \eta^2} \frac{\sigma^8}{n^2} + L \Big) \Big( 2 \kappa_*^4 + 8 \int_0^\infty \underbrace{(\kappa_*+s)^3}_{\leq 4\kappa_*^3 + 4s^3} \exp(-C s^2) \, ds \Big)\\
		& \overset{\substack{\text{(\ref{Eq_Q_Concentration_DefL})} \\ \kappa_*>1}}{\leq} \exp(-C n) \Big( \frac{2c_* ||M||}{\tau^2 \eta^2} \frac{\sigma^8}{n^2} + \frac{2 c_*^2 \sigma^6 \sqrt{K_*} (\kappa+K_*)^4 ||M||}{\tau^4 \eta^5 \sqrt{n}} \Big)\\
		& \hspace{1cm} \times \kappa_*^4 \Big( 2 + 2^5 \int_0^\infty \exp(-C s^2) \, ds \Big) + 2^5 \int_0^\infty s^3 \exp(-C s^2) \, ds \Big)\\
		& \leq n^2 \exp(-C n) ||M|| \Big( \frac{2c_* \sigma^8}{\tau^2 \eta^2} + \frac{2 c_*^2 \sigma^6 \sqrt{K_*} (\kappa+K_*)^4}{\tau^4 \eta^5} \Big)\\
		& \hspace{1cm} \times \underbrace{\frac{\kappa_*^4}{n^2}}_{\overset{\substack{\text{(\ref{Eq_QConcentration_Def_kappa})} \\ \text{\ref{ItemAssumption_cBound}}}}{\leq} (2\sqrt{c_*}+3)^4} \Big( 2 + 2^5 \int_0^\infty \exp(-C s^2) \, ds \Big) + 2^5 \int_0^\infty s^3 \exp(-C s^2) \, ds \Big) \ .
	\end{align*}
	Since $K_* = K_*(c_*,\sigma^2)$ is a constant only depending on $c_*$ and $\sigma^2$, this proves the existence of a constant $K' = K'(c_*,\sigma^2,\tau,\eta,\kappa)>0$ such that
	\begin{align*}
		& \text{(\ref{Eq_Q_Concentration_MeanDiffDecomp2})} \leq K' n^2 \exp(-C n) ||M|| \ ,
	\end{align*}
	which together with (\ref{Eq_Q_Concentration_MeanDiffDecomp1_Bound}) may be inserted back into the decomposition (\ref{Eq_Q_Concentration_MeanDiffDecomp1})-(\ref{Eq_Q_Concentration_MeanDiffDecomp2}) to see
	\begin{align}\label{Eq_Q_Concentration_MeanDiffBound2}
		\big| \E[f(\bm{Z})-\tilde{f}(\bm{Z})] \big| & \leq 2\exp\big( -C n \big) \frac{2c_* ||M||}{\tau^2 \eta^2} (2\kappa^2 + K_*^2) + K' n^2 \exp(-C n) ||M|| \nonumber\\
		& \leq K'' n^2 \exp(-C n) ||M||
	\end{align}
	for $K'' \coloneq \frac{4 c_*}{\tau^2 \eta^2} (2\kappa^2 + K_*^2) + K'$.
	
	\item[v)] \textit{Concentration for fixed $z \in \bD(\eta,\kappa)$}:\\
	By combining the results of parts (i), (iii) and (iv) of this proof, one gets
	\begin{align*}
		& \forall M \in \C^{d \times d} , \, \forall z \in \bD(\eta,\kappa) , \, \forall t>0 : \nonumber\\
		& \bP\Big( \Big|\Re\Big( \frac{1}{z n} \tr\big( M \bm{Q}^{(\bm{Z})}(z)^{-1} \big)\Big) - \Re\Big(\E\Big[ \frac{1}{z n} \tr\big( M \bm{Q}^{(\bm{Z})}(z)^{-1} \big) \Big] \Big) \Big| \nonumber\\
		& \hspace{6cm} > t + K'' n^2 \exp(-C n) ||M|| \Big) \nonumber\\
		& \overset{\text{(\ref{Eq_Q_Concentration_McShaneExtension1})}}{\leq} \bP\big( \bm{Z} \notin U_{\kappa_*} \big) + \bP\Big( \bm{Z} \in U_{\kappa_*} \ \text{ and } \ \big| \tilde{f}(\bm{Z}) - \E\big[ f(\bm{Z}) \big] \big| > t + K'' n^2 \exp(-C n) ||M|| \Big) \nonumber\\
		& \overset{\text{(\ref{Eq_Q_Concentration_MeanDiffBound2})}}{\leq} \underbrace{\bP\big( \bm{Z} \notin U_{\kappa_*} \big)}_{\overset{\text{(\ref{Eq_Q_Concentration_PZBound})}}{\leq} 2\exp\big( -C n \big)} + \underbrace{\bP\Big( \bm{Z} \in U_{\kappa_*} \ \text{ and } \ \big| \tilde{f}(\bm{Z}) - \E\big[ \tilde{f}(\bm{Z}) \big] \big| > t \Big)}_{\overset{\text{(\ref{Eq_Q_Concentration_McShaneConcnetration})}}{\leq} 2\exp\big( -\frac{Ct^2}{L^2} \big)} \nonumber\\
		& \leq 2\exp\big( -C n \big) + 2\exp\Big( -\frac{Ct^2}{L^2} \Big) \nonumber\\
		& \overset{\text{(\ref{Eq_Q_Concentration_DefL})}}{=} 2\exp\big( -C n \big) + 2\exp\Big( -\frac{C \tau^8 \eta^{10} t^2 n}{4 c_*^4 \sigma^{12} K_* (\kappa+K_*)^8 ||M||^2} \Big) \ ,
	\end{align*}
	which proves
	\begin{align}\label{Eq_Q_ConcentrationA_singleZ_real}
		& \forall M \in \C^{d \times d} , \, \forall z \in \bD(\eta,\kappa) , \, \forall t>0 : \nonumber\\
		& \bP\Big( \Big| \Re\Big(\frac{1}{z n} \tr\big( M \bm{Q}^{(\bm{Z})}(z)^{-1} \big)\Big) - \Re\Big(\E\Big[ \frac{1}{z n} \tr\big( M \bm{Q}^{(\bm{Z})}(z)^{-1} \big) \Big]\Big) \Big| \nonumber\\
		& \hspace{6.5cm} > t + K n^2 \exp(-C n) ||M|| \Big) \nonumber\\
		& \leq 2\exp\big( -C n \big) + 2\exp\Big( -\frac{C t^2 n}{K ||M||^2} \Big)
	\end{align}
	for $K \coloneq K'' \lor \frac{4 c_*^4 \sigma^{12} K_* (\kappa+K_*)^8}{\tau^8 \eta^{10}}$.
	parts (iii), (iv) and the above calculation hold analogously for $f(X) = \Im\big( \frac{1}{z n} \tr(M \bm{Q}^{(X)}(z)^{-1}) \big)$. Combining the real and imaginary versions of (\ref{Eq_Q_ConcentrationA_singleZ_real}) gives
	\begin{align}\label{Eq_Q_ConcentrationA_singleZ}
		& \forall M \in \C^{d \times d} , \, \forall z \in \bD(\eta,\kappa) , \, \forall t>0 : \nonumber\\
		& \bP\Big( \Big| \frac{1}{z n} \tr\big( M \bm{Q}^{(\bm{Z})}(z)^{-1} \big) - \E\Big[ \frac{1}{z n} \tr\big( M \bm{Q}^{(\bm{Z})}(z)^{-1} \big) \Big] \Big| \nonumber\\
		& \hspace{4cm} > 2t + 2K n^2 \exp(-C n) ||M|| \Big) \nonumber\\
		& \leq 4\exp\big( -C n \big) + 4\exp\Big( -\frac{C t^2 n}{K ||M||^2} \Big) \ .
	\end{align}
	
	\item[vi)] \textit{Concentration uniformly over $z \in \bD(\eta,\kappa)$}:\\
	By construction of $\bD(\eta,\kappa)$, there for any $\delta \in (0,\kappa)$ exists a $\delta$-net $J_\delta \subset \bD(\eta,\kappa)$ with the properties
	\begin{align}\label{Eq_delta_Net_Properties_copy}
		& \bD(\eta,\kappa) \subset \bigcup\limits_{j \in J_\delta} B^{\C}_\delta(j) \ \ \text{ and } \ \ \#J_\delta \leq \frac{4\kappa^2}{\delta^2} \ ,
	\end{align}
	where $B^{\C}_\delta(j)$ denotes the (closed) $\delta$-neighborhood around $j$ in $\C$ and $\#J_\delta$ denotes the cardinality of $J_\delta$. Letting $\Pi_\delta : \bD(\eta,\kappa) \rightarrow J_\delta$ denote a (measurable) map with the property $|z-\Pi_\delta(z)| \leq \delta$, one by (\ref{Eq_Q_Concentration_Lz_Lipschitz_Q}) and (\ref{Eq_Q_Concentration_DefLz}) has
	\begin{align}\label{Eq_trMRz_NetBound_copy}
		& \forall z \in \bD(\eta,\kappa) : \nonumber\\
		& \Big| \frac{1}{z n} \tr\big( M \bm{Q}^{(\bm{Z})}(z)^{-1} \big) - \frac{1}{\Pi_\delta(z) n} \tr\big( M \bm{Q}^{(\bm{Z})}(\Pi_\delta(z)))^{-1} \big) \Big| \nonumber\\
		& \leq \Big(\underbrace{\frac{2 c_*^2 \sigma^4 (\kappa+K_*)^4}{\tau^4 \eta^5}}_{\leq K' \leq K'' \leq K}\Big) ||M|| \delta \leq K ||M|| \delta \ .
	\end{align}
	By combining (\ref{Eq_Q_ConcentrationA_singleZ}) and (\ref{Eq_trMRz_NetBound_copy}), one by triangle inequality gets
	\begin{align*}
		& \bP\Big( \sup\limits_{z \in \bD(\eta,\kappa)}\Big| \frac{1}{z n} \tr(M\bm{Q}^{(\bm{Z})}(z)^{-1}) - \E\Big[ \frac{1}{z n} \tr(M\bm{Q}^{(\bm{Z})}(z)^{-1}) \Big] \Big| \nonumber\\
		& \hspace{4cm} > t + 2K n^2 \exp(-C n) ||M|| + 2K ||M|| \delta \Big)\\
		& \overset{\text{(\ref{Eq_trMRz_NetBound_copy})}}{\leq} \bP\Big( \sup\limits_{z \in \bD(\eta,\kappa)}\Big| \frac{1}{\Pi_\delta(z) n} \tr(M\bm{Q}^{(\bm{Z})}(\Pi_\delta(z))^{-1}) - \E\Big[ \frac{1}{\Pi_\delta(z) n} \tr(M\bm{Q}^{(\bm{Z})}(\Pi_\delta(z))^{-1}) \Big] \Big| \nonumber\\
		& \hspace{8cm} > t + 2K n^2 \exp(-C n) ||M|| \Big)\\
		& \leq \bP\Big( \max\limits_{j \in J_\delta}\Big| \frac{1}{j n} \tr(M\bm{Q}^{(\bm{Z})}(j)^{-1}) - \E\Big[ \frac{1}{j n} \tr(M\bm{Q}^{(\bm{Z})}(j)^{-1}) \Big] \Big| \nonumber\\
		& \hspace{5cm} > t + 2K n^2 \exp(-C n) ||M|| \Big)\\
		& \leq \sum\limits_{j \in J_\delta} \bP\Big( \Big| \frac{1}{j n} \tr(M\bm{Q}^{(\bm{Z})}(j)^{-1}) - \E\Big[ \frac{1}{j n} \tr(M\bm{Q}^{(\bm{Z})}(j)^{-1}) \Big] \Big| \nonumber\\
		& \hspace{5cm} > t + 2K n^2 \exp(-C n) ||M|| \Big)\\
		& \overset{\text{(\ref{Eq_Q_ConcentrationA_singleZ})}}{\leq} \sum\limits_{j \in J_\delta} \Big( 4\exp\big( -C n \big) + 4\exp\Big( -\frac{C t^2 n}{4 K ||M||^2} \Big) \Big)\\
		& \overset{\text{(\ref{Eq_delta_Net_Properties_copy})}}{\leq} \frac{16\kappa^2}{\delta^2} \exp\big( -C n \big) + \frac{16\kappa^2}{\delta^2} \exp\Big( -\frac{C t^2 n}{4 K ||M||^2} \Big) \ ,
	\end{align*}
	thus proving (\ref{Eq_QConcentrationA}). \qed
\end{itemize}

\subsection{Proof of Lemma~\ref{Lemma_HolomorphicLindeberg}}\label{Proof_Lemma_HolomorphicLindeberg}
By telescope sum, one has
\begin{align}\label{Eq_Lindeberg_TelescopeSum}
	& \E\big[ f_N\big(X_1,\dots,X_N;\ol{X_1},\dots,\ol{X_N}\big) - f_N\big(Y_1,\dots,Y_N;\ol{Y_1},\dots,\ol{Y_N}\big) \big] \nonumber\\
	& = \sum\limits_{j=1}^N \E\big[ f_N\big(\overbrace{X_1,\dots,X_j,Y_{j+1},\dots,Y_N}^{\eqcolon  Z^{(j)}};\overbrace{\ol{X_1},\dots,\ol{X_j},\ol{Y_{j+1}},\dots,\ol{Y_N}}^{\eqcolon  \ol{Z^{(j)}}}\big) \nonumber\\
	& \hspace{2cm} - f_N\big(\underbrace{X_1,\dots,X_{j-1},Y_{j},\dots,Y_N}_{\eqcolon Z^{(j-1)}};\underbrace{\ol{X_1},\dots,\ol{X_{j-1}},\ol{Y_{j}},\dots,\ol{Y_N}}_{\eqcolon  \ol{Z^{(j-1)}}} \big] \ .
\end{align}
With the notation
\begin{align*}
	& Z^{(j,0)} \coloneq \big( X_1,\dots,X_{j-1},0,Y_{j+1},\dots,Y_N \big)
\end{align*}
and
\begin{align*}
	& u_{j} = \big( \underbrace{0,\dots,0}_{\times (j-1)},1,\underbrace{0,\dots,0}_{\times (N-j)} \big) \ ,
\end{align*}
one by third order Taylor series with integral remainder has
\begin{align}\label{Eq_Lindeberg_XTaylor}
	& f_N(Z^{(j)};\ol{Z^{(j)}}) - f_N(Z^{(j,0)};\ol{Z^{(j,0)}}) \nonumber\\
	& = \partial_\tau f_N\big(Z^{(j,0)}+\tau u_j X_j;\ol{Z^{(j,0)}+\tau u_j X_j}\big) \big|_{\tau=0} \nonumber\\
	& \hspace{0.5cm} + \frac{1}{2} \partial_\tau^2 f_N\big(Z^{(j,0)}+\tau u_j X_j;\ol{Z^{(j,0)}+\tau u_j X_j}\big) \big|_{\tau=0} \nonumber\\
	& \hspace{0.5cm} + \frac{1}{2} \int_0^1 (1-\tau)^2 \partial_\tau^3 f_N\big(Z^{(j,0)}+\tau u_j X_j;\ol{Z^{(j,0)}+\tau u_j X_j}\big) \, d\tau
\end{align}
and analogously
\begin{align}\label{Eq_Lindeberg_YTaylor}
	& f_N(Z^{(j-1)};\ol{Z^{(j-1)}}) - f_N(Z^{(j,0)};\ol{Z^{(j,0)}}) \nonumber\\
	& = \partial_\tau f_N\big(Z^{(j,0)}+\tau u_j Y_j;\ol{Z^{(j,0)}+\tau u_j Y_j}\big) \big|_{\tau=0} \nonumber\\
	& \hspace{0.5cm} + \frac{1}{2} \partial_\tau^2 f_N\big(Z^{(j,0)}+\tau u_j Y_j;\ol{Z^{(j,0)}+\tau u_j Y_j}\big) \big|_{\tau=0} \nonumber\\
	& \hspace{0.5cm} + \frac{1}{2} \int_0^1 (1-\tau)^2 \partial_\tau^3 f_N\big(Z^{(j,0)}+\tau u_j Y_j;\ol{Z^{(j,0)}+\tau u_j Y_j}\big) \, d\tau \ .
\end{align}
subtracting (\ref{Eq_Lindeberg_YTaylor}) from (\ref{Eq_Lindeberg_XTaylor}) and taking expectations yields
\begin{align}
	& \big| \E\big[ f_N(X;\ol{X}) - f_N(Y;\ol{Y}) \big] \big|\\
	& \overset{\text{(\ref{Eq_Lindeberg_TelescopeSum})}}{=} \bigg| \sum\limits_{j=1}^N \E\big[ f_N(Z^{(j)};\ol{Z^{(j)}}) - f_N(Z^{(j-1)};\ol{Z^{(j-1)}}) \big] \bigg| \nonumber\\
	& \leq \bigg| \sum\limits_{j=1}^N \E\Big[ \partial_\tau f_N\big(Z^{(j,0)}+\tau u_j X_j;\ol{Z^{(j,0)}+\tau u_j X_j}\big) \big|_{\tau=0} \nonumber\\
	& \hspace{1.5cm} - \partial_\tau f_N\big(Z^{(j,0)}+\tau u_j Y_j;\ol{Z^{(j,0)}+\tau u_j Y_j}\big) \big|_{\tau=0} \Big] \bigg| \label{Eq_Lindeberg_MeanSummand1}\\
	& \hspace{0.5cm} + \frac{1}{2} \bigg| \sum\limits_{j=1}^N \E\Big[ \partial_\tau^2 f_N\big(Z^{(j,0)}+\tau u_j X_j;\ol{Z^{(j,0)}+\tau u_j X_j}\big) \big|_{\tau=0} \nonumber\\
	& \hspace{2.5cm} - \partial_\tau^2 f_N\big(Z^{(j,0)}+\tau u_j Y_j;\ol{Z^{(j,0)}+\tau u_j Y_j}\big) \big|_{\tau=0} \Big] \bigg| \label{Eq_Lindeberg_MeanSummand2}\\
	& \hspace{0.5cm} + \frac{1}{2} \sum\limits_{j=1}^N \max\limits_{\tau \in [0,1]} \Big| \E\Big[ \partial_\tau^3 f_N\big(Z^{(j,0)}+\tau u_j X_j;\ol{Z^{(j,0)}+\tau u_j X_j}\big) \nonumber\\
	& \hspace{3.5cm} - \partial_\tau^3 f_N\big(Z^{(j,0)}+\tau u_j Y_j;\ol{Z^{(j,0)}+\tau u_j Y_j}\big) \Big] \Big| \ . \label{Eq_Lindeberg_MeanSummand3}
\end{align}
The above summands are calculated as follows.
\begin{itemize}
	\item \textit{Calculation of (\ref{Eq_Lindeberg_MeanSummand1})}:\\
	Let $\partial_j f_N$ denote the holomorphic derivative of $f_N$ with regard to the $j$-th component. One by chain rule observes
	\begin{align}\label{Eq_Lindeberg_RealDeriv1}
		& \partial_\tau f_N\big(z_1,\dots,\tau z_j,\dots,z_N;w_1,\dots,\tau w_j,\dots,w_N\big) \big|_{\tau=0} \nonumber\\
		& = \partial_\tau f_N\big(z_1,\dots,\tau z_j,\dots,z_N;w_1,\dots,0,\dots,w_N\big) \big|_{\tau=0} \nonumber\\
		& \hspace{0.5cm} + \partial_\tau f_N\big(z_1,\dots,0,\dots,z_N;w_1,\dots,\tau w_j,\dots,w_N\big) \big|_{\tau=0} \nonumber\\
		& = z_j \partial_{j} f_N\big(z_1,\dots,0,\dots,z_N;w_1,\dots,0,\dots,w_N\big) \nonumber\\
		& \hspace{0.5cm} + w_j \partial_{N+j} f_N\big(z_1,\dots,0,\dots,z_N;w_1,\dots,0,\dots,w_N\big) \ ,
	\end{align}
	which for the first derivatives yields
	\begin{align}\label{Eq_Lindeberg_DirstDerivX}
		& \partial_\tau f_N\big(Z^{(j,0)}+\tau u_j X_j;\ol{Z^{(j,0)}+\tau u_j X_j}\big) \big|_{\tau=0} \nonumber\\
		& \overset{\text{(\ref{Eq_Lindeberg_RealDeriv1})}}{=} X_j \partial_{j} f_N\big(Z^{(j,0)};\ol{Z^{(j,0)}}\big) + \ol{X_j} \partial_{N+j} f_N\big(Z^{(j,0)};\ol{Z^{(j,0)}}\big)
	\end{align}
	and analogously
	\begin{align}\label{Eq_Lindeberg_DirstDerivY}
		& \partial_\tau f_N\big(Z^{(j,0)}+\tau u_j Y_j;\ol{Z^{(j,0)}+\tau u_j Y_j}\big) \big|_{\tau=0} \nonumber\\
		& = Y_j \partial_{j} f_N\big(Z^{(j,0)};\ol{Z^{(j,0)}}\big) + \ol{Y_j} \partial_{N+j} f_N\big(Z^{(j,0)};\ol{Z^{(j,0)}}\big) \ .
	\end{align}
	By independence of $X_1,\dots,X_N,Y_1,\dots,Y_N$, one has
	\begin{align}
		& \E\Big[ \partial_\tau f_N\big(Z^{(j,0)}+\tau u_j X_j;\ol{Z^{(j,0)}+\tau u_j X_j}\big) \big|_{\tau=0} \nonumber\\
		& \hspace{1cm} - \partial_\tau f_N\big(Z^{(j,0)}+\tau u_j Y_j;\ol{Z^{(j,0)}+\tau u_j Y_j}\big) \big|_{\tau=0} \Big] \nonumber\\
		& \hspace{-0.5cm} \overset{\text{(\ref{Eq_Lindeberg_DirstDerivX})\&(\ref{Eq_Lindeberg_DirstDerivY})}}{=} \E\big[ (X_j-Y_j) \big] \, \E\big[ \partial_{j} f_N\big(Z^{(j,0)};\ol{Z^{(j,0)}}\big) \big] \nonumber\\
		& \hspace{2cm} + \ol{\E\big[ (X_j-Y_j) \big]} \, \E\big[ \partial_{N+j} f_N\big(Z^{(j,0)};\ol{Z^{(j,0)}}\big) \big] \nonumber\\
		& \overset{\text{(\ref{Eq_Lindeberg_EqualMeans})}}{=} 0 \ ,
	\end{align}
	which proves $\text{(\ref{Eq_Lindeberg_MeanSummand1})} = 0$.
	
	\item \textit{Calculation of (\ref{Eq_Lindeberg_MeanSummand2})}:\\
	As in (\ref{Eq_Lindeberg_RealDeriv1}), one calculates
	\begin{align*}
		& \partial_\tau^2 f_N\big(z_1,\dots,\tau z_j,\dots,z_N;w_1,\dots,\tau w_j,\dots,w_N\big) \big|_{\tau=0} \nonumber\\
		& = \partial_\tau^2 f_N\big(z_1,\dots,\tau z_j,\dots,z_N;w_1,\dots,0,\dots,w_N\big) \big|_{\tau=0}\\
		& \hspace{0.5cm} + 2\partial_\tau \partial_{\tau'} f_N\big(z_1,\dots,\tau z_j,\dots,z_N;w_1,\dots,\tau'w_j,\dots,w_N\big) \big|_{\tau=0=\tau'}\\
		& \hspace{0.5cm} + \partial_\tau^2 f_N\big(z_1,\dots,0,\dots,z_N;w_1,\dots,\tau w_j,\dots,w_N\big) \big|_{\tau=0}\\
		& = z_j^2 \partial_j^2 f_N\big(z_1,\dots,0,\dots,z_N;w_1,\dots,0,\dots,w_N\big)\\
		& \hspace{0.5cm} + 2 z_jw_j \partial_j \partial_{N+j} f_N\big(z_1,\dots,0,\dots,z_N;w_1,\dots,0,\dots,w_N\big)\\
		& \hspace{0.5cm} + w_j^2 \partial_{N+j}^2 f_N\big(z_1,\dots,0,\dots,z_N;w_1,\dots,0,\dots,w_N\big) \ ,
	\end{align*}
	which for the second derivatives yields
	\begin{align*}
		& \frac{1}{2} \partial_\tau^2 f_N\big(Z^{(j,0)}+\tau u_j X_j;\ol{Z^{(j,0)}+\tau u_j X_j}\big) \big|_{\tau=0}\\
		& = \frac{1}{2} X_j^2 \partial_j^2 f_N\big(Z^{(j,0)};\ol{Z^{(j,0)}}\big)\\
		& \hspace{0.5cm} + |X_j|^2 \partial_j \partial_{N+j} f_N\big(Z^{(j,0)};\ol{Z^{(j,0)}}\big)\\
		& \hspace{0.5cm} + \frac{1}{2} \ol{X_j}^2 \partial_{N+j}^2 f_N\big(Z^{(j,0)};\ol{Z^{(j,0)}}\big) \ .
	\end{align*}
	Analogously, one has
	\begin{align*}
		& \frac{1}{2} \partial_\tau^2 f_N\big(Z^{(j,0)}+\tau u_j Y_j;\ol{Z^{(j,0)}+\tau u_j Y_j}\big) \big|_{\tau=0}\\
		& = \frac{1}{2} Y_j^2 \partial_j^2 f_N\big(Z^{(j,0)};\ol{Z^{(j,0)}}\big)\\
		& \hspace{0.5cm} + |Y_j|^2 \partial_j \partial_{N+j} f_N\big(Z^{(j,0)};\ol{Z^{(j,0)}}\big)\\
		& \hspace{0.5cm} + \frac{1}{2} \ol{Y_j}^2 \partial_{N+j}^2 f_N\big(Z^{(j,0)};\ol{Z^{(j,0)}}\big) \ .
	\end{align*}
	The summand (\ref{Eq_Lindeberg_MeanSummand2}) may thus be bounded by
	\begin{align}\label{Eq_Lineberg_M0ProtoBound}
		& \frac{1}{2} \bigg| \sum\limits_{j=1}^N \E\Big[ \partial_\tau^2 f_N\big(Z^{(j,0)}+\tau u_j X_j;\ol{Z^{(j,0)}+\tau u_j X_j}\big) \big|_{\tau=0} \nonumber\\
		& \hspace{1.5cm} - \partial_\tau^2 f_N\big(Z^{(j,0)}+\tau u_j Y_j;\ol{Z^{(j,0)}+\tau u_j Y_j}\big) \big|_{\tau=0} \Big] \bigg| \nonumber\\
		& \leq \frac{1}{2} \bigg| \sum\limits_{j=1}^N \E\big[ X_j^2 - Y_j^2 \big] \E\big[ \partial_j^2 f_N\big(Z^{(j,0)};\ol{Z^{(j,0)}}\big) \big] \bigg| \nonumber\\
		& \hspace{0.5cm} + \bigg| \sum\limits_{j=1}^N \E\big[ |X_j|^2 - |Y_j|^2 \big] \E\big[ \partial_j \partial_{N+j} f_N\big(Z^{(j,0)};\ol{Z^{(j,0)}}\big) \big] \bigg| \nonumber\\
		& \hspace{0.5cm} + \frac{1}{2} \bigg| \sum\limits_{j=1}^N \E\big[ \ol{X_j}^2 - \ol{Y_j}^2 \big] \E\big[ \partial_{N+j}^2 f_N\big(Z^{(j,0)};\ol{Z^{(j,0)}}\big) \big] \bigg| \nonumber\\
		& \overset{\text{(\ref{Eq_Lindeberg_EqualMeans})}}{=} 0 \ .
	\end{align}

	\item \textit{Calculation of (\ref{Eq_Lindeberg_MeanSummand3})}:\\
	For the third derivative, one calculates
	\begin{align*}
		& \partial_\tau^3 f_N\big(z_1,\dots,\tau z_j,\dots,z_N;w_1,\dots,\tau w_j,\dots,w_N\big) \big|_{\tau=0} \nonumber\\
		& = \partial_\tau^3 f_N\big(z_1,\dots,\tau z_j,\dots,z_N;w_1,\dots,0,\dots,w_N\big) \big|_{\tau=0}\\
		& \hspace{0.5cm} + 3\partial_\tau^2 \partial_{\tau'} f_N\big(z_1,\dots,\tau z_j,\dots,z_N;w_1,\dots,\tau'w_j,\dots,w_N\big) \big|_{\tau=0=\tau'}\\
		& \hspace{0.5cm} + 3\partial_\tau \partial_{\tau'}^2 f_N\big(z_1,\dots,\tau z_j,\dots,z_N;w_1,\dots,\tau'w_j,\dots,w_N\big) \big|_{\tau=0=\tau'}\\
		& \hspace{0.5cm} + \partial_\tau^3 f_N\big(z_1,\dots,0,\dots,z_N;w_1,\dots,\tau w_j,\dots,w_N\big) \big|_{\tau=0}\\
		& = z_j^3 \partial_j^3 f_N\big(z_1,\dots,0,\dots,z_N;w_1,\dots,0,\dots,w_N\big)\\
		& \hspace{0.5cm} + 3z_j^2w_j \partial_j^2 \partial_{N+j} f_N\big(z_1,\dots,0,\dots,z_N;w_1,\dots,0,\dots,w_N\big)\\
		& \hspace{0.5cm} + 3z_jw_j^2 \partial_j \partial_{N+j}^2 f_N\big(z_1,\dots,0,\dots,z_N;w_1,\dots,0,\dots,w_N\big)\\
		& \hspace{0.5cm} + w_j^3 \partial_{N+j}^3 f_N\big(z_1,\dots,0,\dots,z_N;w_1,\dots,0,\dots,w_N\big) \ .
	\end{align*}
	Thus, the terms in the integrals in (\ref{Eq_Lindeberg_XTaylor}) and (\ref{Eq_Lindeberg_YTaylor}) become
	\begin{align*}
		& \partial_\tau^3 f_N\big(Z^{(j,0)}+\tau u_j X_j;\ol{Z^{(j,0)}+\tau u_j X_j}\big)\\
		& = X_j^3 \partial_j^3 f_N\big(Z^{(j,0)}+\tau u_j X_j;\ol{Z^{(j,0)}+\tau u_j X_j}\big)\\
		& \hspace{0.5cm} + 3 X_j |X_j|^2 \partial_j^2 \partial_{N+j} f_N\big(Z^{(j,0)}+\tau u_j X_j;\ol{Z^{(j,0)}+\tau u_j X_j}\big)\\
		& \hspace{0.5cm} + 3 \ol{X_j} |X_j|^2 \partial_j \partial_{N+j}^2 f_N\big(Z^{(j,0)}+\tau u_j X_j;\ol{Z^{(j,0)}+\tau u_j X_j}\big)\\
		& \hspace{0.5cm} + \ol{X_j}^3 \partial_{N+j}^3 f_N\big(Z^{(j,0)}+\tau u_j X_j;\ol{Z^{(j,0)}+\tau u_j X_j}\big)
	\end{align*}
	and
	\begin{align*}
		& \partial_\tau^3 f_N\big(Z^{(j,0)}+\tau u_j Y_j;\ol{Z^{(j,0)}+\tau u_j Y_j}\big)\\
		& = Y_j^3 \partial_j^3 f_N\big(Z^{(j,0)}+\tau u_j Y_j;\ol{Z^{(j,0)}+\tau u_j Y_j}\big)\\
		& \hspace{0.5cm} + 3 Y_j |Y_j|^2 \partial_j^2 \partial_{N+j} f_N\big(Z^{(j,0)}+\tau u_j Y_j;\ol{Z^{(j,0)}+\tau u_j Y_j}\big)\\
		& \hspace{0.5cm} + 3 \ol{Y_j} |Y_j|^2 \partial_j \partial_{N+j}^2 f_N\big(Z^{(j,0)}+\tau u_j Y_j;\ol{Z^{(j,0)}+\tau u_j Y_j}\big)\\
		& \hspace{0.5cm} + \ol{Y_j}^3 \partial_{N+j}^3 f_N\big(Z^{(j,0)}+\tau u_j Y_j;\ol{Z^{(j,0)}+\tau u_j Y_j}\big) \ ,
	\end{align*}
	so for each $j \leq N$ one may bound
	\begin{align*}
		& \frac{1}{2} \max\limits_{\tau \in [0,1]} \Big| \E\Big[ \partial_\tau^3 f_N\big(Z^{(j,0)}+\tau u_j X_j;\ol{Z^{(j,0)}+\tau u_j X_j}\big) \nonumber\\
		& \hspace{2cm} - \partial_\tau^3 f_N\big(Z^{(j,0)}+\tau u_j Y_j;\ol{Z^{(j,0)}+\tau u_j Y_j}\big) \Big] \Big|\\
		& \leq \frac{3}{2} \sum\limits_{\substack{a,b \geq 0 \\ a+b=3}} \max\limits_{\tau \in [0,1]} \E\big[ \big( |X_j|^3 + |Y_j|^3 \big) \big|\partial_j^a \partial_{N+j}^b f_N\big(Z^{(j,0)}+\tau u_j X_j;\ol{Z^{(j,0)}+\tau u_j X_j}\big)\big| \big]\\
		& \hspace{0.5cm} + \frac{3}{2} \sum\limits_{\substack{a,b \geq 0 \\ a+b=3}} \max\limits_{\tau \in [0,1]} \E\big[ \big( |X_j|^3 + |Y_j|^3 \big) \big|\partial_j^a \partial_{N+j}^b f_N\big(Z^{(j,0)}+\tau u_j Y_j;\ol{Z^{(j,0)}+\tau u_j Y_j}\big)\big| \big]\\
		& \leq \overbrace{\frac{3}{2} \E\big[ 2|X_j|^6 + 2|Y_j|^6 \big]^{\frac{1}{2}}}^{\overset{\text{(\ref{Eq_Lindeberg_SixthMomentBound})}}{\leq} 3\sqrt{K_6}} \sum\limits_{\substack{a,b \geq 0 \\ a+b=3}} \max\limits_{\tau \in [0,1]} \E\big[ \big|\partial_j^a \partial_{N+j}^b f_N\big(Z^{(j,0)}+\tau u_j X_j;\ol{Z^{(j,0)}+\tau u_j X_j}\big)\big|^2 \big]^{\frac{1}{2}}\\
		& \hspace{0.2cm} + \underbrace{\frac{3}{2} \E\big[ 2|X_j|^6 + 2|Y_j|^6 \big]^{\frac{1}{2}}}_{\overset{\text{(\ref{Eq_Lindeberg_SixthMomentBound})}}{\leq} 3\sqrt{K_6}} \sum\limits_{\substack{a,b \geq 0 \\ a+b=3}} \max\limits_{\tau \in [0,1]} \E\big[ \big|\partial_j^a \partial_{N+j}^b f_N\big(Z^{(j,0)}+\tau u_j Y_j;\ol{Z^{(j,0)}+\tau u_j Y_j}\big)\big|^2 \big]^{\frac{1}{2}} \ ,
	\end{align*}
	which proves that the summand (\ref{Eq_Lindeberg_MeanSummand3}) may be bounded by $3\sqrt{K_6} \sum\limits_{j=1}^N (b^{(1)}_j + b^{(2)}_j)$.
\end{itemize}
Inserting these bounds back into the expansion (\ref{Eq_Lindeberg_MeanSummand1})-(\ref{Eq_Lindeberg_MeanSummand3}) yields
\begin{align*}
	& \big| \E\big[ f_N(X;\ol{X}) - f_N(Y;\ol{Y}) \big] \big|\leq 3\sqrt{K_6} \sum\limits_{j=1}^N \big( b^{(1)}_j + b^{(2)}_j \big) \ .
\end{align*}
which proves (\ref{Eq_LindebergEquality}). \qed

\subsection{Proof of Lemma~\ref{Lemma_DerivativeBounds_Lindeberg}}\label{Proof_Lemma_DerivativeBounds_Lindeberg}
For ease of notation, write $\bm{Y}$, $\bm{S}$ and $\bm{R}(z)$ for $\bm{Y}^{(X)}$, $\bm{S}^{(X)}$ and $\bm{R}^{(X)}(z)$ respectively.
\begin{itemize}
	\item First derivatives:\\
	For the holomorphic derivative of $f_{d \cdot n}$ with respect to the component $(i,j)$, one thus has
	\begin{align}\label{Eq_fN_R_FirstDeriv}
		\frac{\partial}{\partial X_{i,j}} f_{d \cdot n}(X,\ol{X}) & = \frac{\partial}{\partial X_{i,j}} \frac{1}{n} \tr\big( M \bm{R}(z) \big) = \frac{1}{n} \tr\Big( M \frac{\partial}{\partial X_{i,j}} \bm{R}(z) \Big) \nonumber\\
		& \hspace{-0.1cm} \overset{\text{(\ref{Eq_DerivCalc_XDeriv_R_Result})}}{=} -\frac{1}{n^2} \tr\Big( M \sum\limits_{r=1}^R \bm{R}(z) A_r e_{i,d} e_{j,n}^\top B_r \bm{Y}^* \bm{R}(z) \Big) \nonumber\\
		& = -\frac{1}{n^2} \sum\limits_{r=1}^R e_{j,n}^\top B_r \bm{Y}^* \bm{R}(z) M \bm{R}(z) A_r e_{i,d}
	\end{align}
	and analogously
	\begin{align}\label{Eq_fN_R_FirstDerivOl}
		\frac{\partial}{\partial \ol{X_{i,j}}} f_{d \cdot n}(X,\ol{X}) & = \frac{\partial}{\partial \ol{X_{i,j}}} \frac{1}{n} \tr\big( M \bm{R}(z) \big) = \frac{1}{n} \tr\Big( M \frac{\partial}{\partial \ol{X_{i,j}}} \bm{R}(z) \Big) \nonumber\\
		& \hspace{-0.15cm} \overset{\text{(\ref{Eq_DerivCalc_olXDeriv_R_Result})}}{=} -\frac{1}{n^2} \tr\Big( M \sum\limits_{s=1}^R \bm{R}(z) \bm{Y} B_s^* e_{j,n} e_{i,d}^\top A_s^* \bm{R}(z) \Big) \nonumber\\
		& = -\frac{1}{n^2} \sum\limits_{s=1}^R e_{i,d}^\top A_s^* \bm{R}(z) M \bm{R}(z) \bm{Y} B_s^* e_{j,n} \ .
	\end{align}
	
	\item Second derivatives:\\
	The relevant second partial derivative for Lemma~\ref{Lemma_HolomorphicLindeberg} may be calculated directly by
	\begin{align}\label{Eq_XXDerivf1Result}
		& \frac{\partial^2}{\partial X_{i,j}^2} f_{d \cdot n}(X,\ol{X}) \overset{\text{(\ref{Eq_fN_R_FirstDeriv})}}{=} -\frac{1}{n^2} \sum\limits_{r=1}^R \frac{\partial}{\partial X_{i,j}} \big[e_{j,n}^\top B_r \bm{Y}^* \bm{R}(z) M \bm{R}(z) A_r e_{i,d}\big] \nonumber\\
		& \overset{\text{(\ref{Eq_DerivCalc_BaseX})}}{=} -\frac{1}{n^2} \sum\limits_{r=1}^R e_{j,n}^\top B_r \bm{Y}^* \frac{\partial}{\partial X_{i,j}} \big[\bm{R}(z)\big] M \bm{R}(z) A_r e_{i,d} \nonumber\\
		& \hspace{0.5cm} - \frac{1}{n^2} \sum\limits_{r=1}^R e_{j,n}^\top B_r \bm{Y}^* \bm{R}(z) M \frac{\partial}{\partial X_{i,j}} \big[\bm{R}(z)\big] A_r e_{i,d} \nonumber\\
		& \overset{\text{(\ref{Eq_DerivCalc_XDeriv_R_Result})}}{=} \frac{1}{n^3} \sum\limits_{r,r'=1}^R e_{j,n}^\top B_r \bm{Y}^* \bm{R}(z) A_{r'} e_{i,d} \, e_{j,n}^\top B_{r'} \bm{Y}^* \bm{R}(z) M \bm{R}(z) A_r e_{i,d} \nonumber\\
		& \hspace{0.5cm} + \frac{1}{n^3} \sum\limits_{r,r'=1}^R e_{j,n}^\top B_r \bm{Y}^* \bm{R}(z) M \bm{R}(z) A_{r'} e_{i,d} \, e_{j,n}^\top B_{r'} \bm{Y}^* \bm{R}(z) A_r e_{i,d} \nonumber\\
		& = \frac{2}{n^3} \sum\limits_{r,r'=1}^R e_{j,n}^\top B_r \bm{Y}^* \bm{R}(z) A_{r'} e_{i,d} \, e_{j,n}^\top B_{r'} \bm{Y}^* \bm{R}(z) M \bm{R}(z) A_r e_{i,d} \ .
	\end{align}
	One analogously proves
	\begin{align}\label{Eq_olXolXDerivf1Result}
		& \frac{\partial^2}{\partial \ol{X_{i,j}}^2} f_{d \cdot n}(X,\ol{X}) \nonumber\\
		& = \frac{1}{n^3} \sum\limits_{s,s'=1}^R e_{i,d}^\top A_s^* \bm{R}(z) \bm{Y} B_{s'}^* e_{j,n} \, e_{i,d}^\top A_{s'}^* \bm{R}(z) M \bm{R}(z) \bm{Y} B_s^* e_{j,n} \nonumber\\
		& \hspace{0.5cm} + \frac{1}{n^3} \sum\limits_{s,s'=1}^R e_{i,d}^\top A_s^* \bm{R}(z) M \bm{R}(z) \bm{Y} B_{s'}^* e_{j,n} \, e_{i,d}^\top A_{s'}^* \bm{R}(z) \bm{Y} B_s^* e_{j,n} \ .
	\end{align}
	
	\item Third derivatives:\\
	Again, the third derivatives may be directly calculated by product rule. One sees
	{\small
		\begin{align}\label{Eq_MixedThirdDerivCalc1}
			& \frac{\partial^3}{\partial X_{i,j}^3} f_{d \cdot n}(X,\ol{X}) \nonumber\\
			& \overset{\text{(\ref{Eq_XXDerivf1Result})}}{=} \frac{2}{n^3} \sum\limits_{r,r'=1}^R \frac{\partial}{\partial X_{i,j}} \big[ e_{j,n}^\top B_r \bm{Y}^* \bm{R}(z) A_{r'} e_{i,d} \, e_{j,n}^\top B_{r'} \bm{Y}^* \bm{R}(z) M \bm{R}(z) A_r e_{i,d} \big] \nonumber\\
			& \overset{\text{(\ref{Eq_DerivCalc_BaseX})}}{=} \frac{2}{n^3} \sum\limits_{r,r'=1}^R e_{j,n}^\top B_r \bm{Y}^* \frac{\partial}{\partial X_{i,j}} \big[ \bm{R}(z) \big] A_{r'} e_{i,d} \, e_{j,n}^\top B_{r'} \bm{Y}^* \bm{R}(z) M \bm{R}(z) A_r e_{i,d} \nonumber\\
			& \hspace{0.5cm} + \frac{2}{n^3} \sum\limits_{r,r'=1}^R e_{j,n}^\top B_r \bm{Y}^* \bm{R}(z) A_{r'} e_{i,d} \, e_{j,n}^\top B_{r'} \bm{Y}^* \frac{\partial}{\partial X_{i,j}} \big[ \bm{R}(z) \big] M \bm{R}(z) A_r e_{i,d} \nonumber\\
			& \hspace{0.5cm} + \frac{2}{n^3} \sum\limits_{r,r'=1}^R e_{j,n}^\top B_r \bm{Y}^* \bm{R}(z) A_{r'} e_{i,d} \, e_{j,n}^\top B_{r'} \bm{Y}^* \bm{R}(z) M \frac{\partial}{\partial X_{i,j}} \big[ \bm{R}(z) \big] A_r e_{i,d} \nonumber\\
			& \overset{\text{(\ref{Eq_DerivCalc_XDeriv_R_Result})}}{=} - \frac{2}{n^4} \sum\limits_{r,r',r''=1}^R e_{j,n}^\top B_r \bm{Y}^* \bm{R}(z) A_{r''} e_{i,d} e_{j,n}^\top B_{r''} \bm{Y}^* \bm{R}(z) A_{r'} e_{i,d} \, e_{j,n}^\top B_{r'} \bm{Y}^* \bm{R}(z) M \bm{R}(z) A_r e_{i,d} \nonumber\\
			& \hspace{0.5cm} - \frac{2}{n^4} \sum\limits_{r,r',r''=1}^R e_{j,n}^\top B_r \bm{Y}^* \bm{R}(z) A_{r'} e_{i,d} \, e_{j,n}^\top B_{r'} \bm{Y}^* \bm{R}(z) A_{r''} e_{i,d} e_{j,n}^\top B_{r''} \bm{Y}^* \bm{R}(z) M \bm{R}(z) A_r e_{i,d} \nonumber\\
			& \hspace{0.5cm} - \frac{2}{n^4} \sum\limits_{r,r',r''=1}^R e_{j,n}^\top B_r \bm{Y}^* \bm{R}(z) A_{r'} e_{i,d} \, e_{j,n}^\top B_{r'} \bm{Y}^* \bm{R}(z) M \bm{R}(z) A_{r''} e_{i,d} e_{j,n}^\top B_{r''} \bm{Y}^* \bm{R}(z) A_r e_{i,d} \nonumber\\
			& = - \frac{6}{n^4} \sum\limits_{r,r',r''=1}^R e_{j,n}^\top B_r \bm{Y}^* \bm{R}(z) A_{r''} e_{i,d} e_{j,n}^\top B_{r''} \bm{Y}^* \bm{R}(z) A_{r'} e_{i,d} \, e_{j,n}^\top B_{r'} \bm{Y}^* \bm{R}(z) M \bm{R}(z) A_r e_{i,d}
		\end{align}
	}
	and may then bound
	\begin{align}\label{Eq_ThridDerivBound1}
		& \Big| \frac{\partial^3}{\partial X_{i,j}^3} f_{d \cdot n}(X,\ol{X}) \Big| \nonumber\\
		& \overset{\text{(\ref{Eq_MixedThirdDerivCalc1})}}{=} \frac{6}{n^4} \Big| \sum\limits_{r,r',r''=1}^R e_{j,n}^\top B_r \bm{Y}^* \bm{R}(z) A_{r''} e_{i,d} e_{j,n}^\top B_{r''} \bm{Y}^* \bm{R}(z) A_{r'} e_{i,d} \nonumber\\
		& \hspace{6cm} \times e_{j,n}^\top B_{r'} \bm{Y}^* \bm{R}(z) M \bm{R}(z) A_r e_{i,d} \Big| \nonumber\\
		& \leq \frac{6}{n^4}  \sum\limits_{r,r',r''=1}^R \big|\big|B_r \bm{Y}^* \bm{R}(z) A_{r''}\big|\big| \, \big|\big|B_{r''} \bm{Y}^* \bm{R}(z) A_{r'} \big|\big| \, \big|\big| B_{r'} \bm{Y}^* \bm{R}(z) M \bm{R}(z) A_r\big|\big|  \nonumber\\
		& \overset{\text{(\ref{Eq_sigmaAssumption_NonAsymp})}}{\leq} \frac{6 \sigma^6 ||M||}{n^4} ||\bm{Y}||^{3} \, ||\bm{R}(z)||^4 \overset{\substack{\text{(\ref{Eq_R_SpectralBounds})} \\ \text{(\ref{Eq_DefS})}}}{\leq} \frac{6 \sigma^6 ||M||}{n^{5/2} \Im(z)^4} ||\bm{S}^{(X)}||^{\frac{3}{2}} \ ,
	\end{align}
	thus proving (\ref{Eq_XXXDerivf1Result}) for $a=3$ and $b=0$.
	Similarly, one calculates
	\begin{align}\label{Eq_MixedThirdDerivCalc2}
		& \frac{\partial}{\partial \ol{X_{i,j}}} \frac{\partial^2}{\partial X_{i,j}^2} f_{d \cdot n}(X,\ol{X}) \nonumber\\
		& \overset{\text{(\ref{Eq_XXDerivf1Result})}}{=} \frac{2}{n^3} \sum\limits_{r,r'=1}^R \frac{\partial}{\partial \ol{X_{i,j}}} \big[ e_{j,n}^\top B_r \bm{Y}^* \bm{R}(z) A_{r'} e_{i,d} \, e_{j,n}^\top B_{r'} \bm{Y}^* \bm{R}(z) M \bm{R}(z) A_r e_{i,d} \big] \nonumber\\
		& = \frac{2}{n^3} \sum\limits_{r,r'=1}^R e_{j,n}^\top B_r \frac{\partial}{\partial \ol{X_{i,j}}} \big[ \bm{Y}^* \bm{R}(z) \big] A_{r'} e_{i,d} \, e_{j,n}^\top B_{r'} \bm{Y}^* \bm{R}(z) M \bm{R}(z) A_r e_{i,d} \nonumber\\
		& \hspace{0.5cm} + \frac{2}{n^3} \sum\limits_{r,r'=1}^R e_{j,n}^\top B_r \bm{Y}^* \bm{R}(z) A_{r'} e_{i,d} \, e_{j,n}^\top B_{r'} \frac{\partial}{\partial \ol{X_{i,j}}} \big[ \bm{Y}^* \bm{R}(z) \big] M \bm{R}(z) A_r e_{i,d} \nonumber\\
		& \hspace{0.5cm} + \frac{2}{n^3} \sum\limits_{r,r'=1}^R e_{j,n}^\top B_r \bm{Y}^* \bm{R}(z) A_{r'} e_{i,d} \, e_{j,n}^\top B_{r'} \bm{Y}^* \bm{R}(z) M \frac{\partial}{\partial \ol{X_{i,j}}} \big[ \bm{R}(z) \big] A_r e_{i,d} \nonumber\\
		& \overset{\substack{\text{(\ref{Eq_DerivCalc_olXDeriv_YR_Result})} \\ \text{(\ref{Eq_DerivCalc_olXDeriv_R_Result})}}}{=} - \frac{2z}{n^3} \sum\limits_{r,r',s=1}^R e_{j,n}^\top B_r \tilde{\bm{R}}(z) B_s^* e_{j,n} e_{i,d}^\top A_s^* \bm{R}(z) A_{r'} e_{i,d} \, e_{j,n}^\top B_{r'} \bm{Y}^* \bm{R}(z) M \bm{R}(z) A_r e_{i,d} \nonumber\\
		& \hspace{0.5cm} - \frac{2z}{n^3} \sum\limits_{r,r',s=1}^R e_{j,n}^\top B_r \bm{Y}^* \bm{R}(z) A_{r'} e_{i,d} \, e_{j,n}^\top B_{r'} \tilde{\bm{R}}(z) B_s^* e_{j,n} e_{i,d}^\top A_s^* \bm{R}(z) M \bm{R}(z) A_r e_{i,d} \nonumber\\
		& \hspace{0.5cm} - \frac{2}{n^4} \sum\limits_{r,r',s=1}^R e_{j,n}^\top B_r \bm{Y}^* \bm{R}(z) A_{r'} e_{i,d} \, e_{j,n}^\top B_{r'} \bm{Y}^* \bm{R}(z) M \bm{R}(z) \bm{Y} B_s^* e_{j,n} e_{i,d}^\top A_s^* \bm{R}(z) A_r e_{i,d} \ .
	\end{align}
	The final summand on the right-hand side of (\ref{Eq_MixedThirdDerivCalc2}) may be bounded analogously to (\ref{Eq_ThridDerivBound1}), while the first summand may in turn be bounded by
	\begin{align}\label{Eq_ThridDerivBound2}
		& \frac{2|z|}{n^3} \Big| \sum\limits_{r,r',s=1}^R e_{j,n}^\top B_r \tilde{\bm{R}}(z) B_s^* e_{j,n} e_{i,d}^\top A_s^* \bm{R}(z) A_{r'} e_{i,d} \, e_{j,n}^\top B_{r'} \bm{Y}^* \bm{R}(z) M \bm{R}(z) A_r e_{i,d} \Big| \nonumber\\
		& \leq \frac{2|z|}{n^3} \sum\limits_{r,r',s=1}^R \big|\big| B_r \tilde{\bm{R}}(z) B_s^* \big|\big| \, \big|\big| A_s^* \bm{R}(z) A_{r'} \big|\big| \, \big|\big| B_{r'} \bm{Y}^* \bm{R}(z) M \bm{R}(z) A_r \big|\big| \nonumber\\
		& \overset{\text{(\ref{Eq_sigmaAssumption_NonAsymp})}}{\leq} \frac{2 |z| \sigma^6 ||M||}{n^3} ||\bm{Y}|| \, ||\bm{R}(z)||^3 \, ||\tilde{\bm{R}}(z)|| \overset{\substack{\text{(\ref{Eq_R_SpectralBounds})} \\ \text{(\ref{Eq_DefS})}}}{\leq} \frac{2 |z| \sigma^6 ||M||}{n^{5/2} \Im(z)^4} ||\bm{S}^{(X)}||^{\frac{1}{2}} \ .
	\end{align}
	As the second summand on the right-hand side of (\ref{Eq_MixedThirdDerivCalc2}) may be bounded using the same arguments, this proves (\ref{Eq_XXXDerivf1Result}) for $a=2$ and $b=1$.
	By analogous methods one calculates
	\begin{align}\label{Eq_MixedThirdDerivCalc3}
		& \frac{\partial}{\partial \ol{X_{i,j}}} \frac{\partial^2}{\partial \ol{X_{i,j}}^2} f_{d \cdot n}(X,\ol{X}) \nonumber\\
		& = - \frac{2}{n^4} \sum\limits_{s,s',s''=1}^R e_{i,d}^\top A_s^* \bm{R}(z) \bm{Y} B_{s''}^* e_{j,n} e_{i,d}^\top A_{s''}^* \bm{R}(z) \bm{Y} B_{s'}^* e_{j,n} \, e_{i,d}^\top A_{s'}^* \bm{R}(z) M \bm{R}(z) \bm{Y} B_s^* e_{j,n} \nonumber\\
		& \hspace{0.5cm} - \frac{2}{n^4} \sum\limits_{s,s',s''=1}^R e_{i,d}^\top A_s^* \bm{R}(z) \bm{Y} B_{s'}^* e_{j,n} \, e_{i,d}^\top A_{s'}^* \bm{R}(z) \bm{Y} B_{s''}^* e_{j,n} e_{i,d}^\top A_{s''}^* \bm{R}(z) M \bm{R}(z) \bm{Y} B_s^* e_{j,n} \nonumber\\
		& \hspace{0.5cm} - \frac{2}{n^4} \sum\limits_{s,s',s''=1}^R e_{i,d}^\top A_s^* \bm{R}(z) \bm{Y} B_{s'}^* e_{j,n} \, e_{i,d}^\top A_{s'}^* \bm{R}(z) M \bm{R}(z) \bm{Y} B_{s''}^* e_{j,n} e_{i,d}^\top A_{s''}^* \bm{R}(z) \bm{Y} B_s^* e_{j,n}
	\end{align}
	as well as
	\begin{align}\label{Eq_MixedThirdDerivCalc4}
		& \frac{\partial}{\partial X_{i,j}} \frac{\partial^2}{\partial \ol{X_{i,j}}^2} f_{d \cdot n}(X,\ol{X}) \nonumber\\
		& = - \frac{2z}{n^3} \sum\limits_{r,s,s'=1}^R e_{i,d}^\top A_s^* \bm{R}(z) A_r e_{i,d} e_{j,n}^\top B_r \tilde{\bm{R}}(z) B_{s'}^* e_{j,n} \, e_{i,d}^\top A_{s'}^* \bm{R}(z) M \bm{R}(z) \bm{Y} B_s^* e_{j,n} \nonumber\\
		& \hspace{0.5cm} - \frac{2}{n^4} \sum\limits_{r,s,s'=1}^R e_{i,d}^\top A_s^* \bm{R}(z) \bm{Y} B_{s'}^* e_{j,n} \, e_{i,d}^\top A_{s'}^* \bm{R}(z) A_r e_{i,d} e_{j,n}^\top B_r \bm{Y}^* \bm{R}(z) M \bm{R}(z) \bm{Y} B_s^* e_{j,n} \nonumber\\
		& \hspace{0.5cm} - \frac{2z}{n^3} \sum\limits_{r,s,s'=1}^R e_{i,d}^\top A_s^* \bm{R}(z) \bm{Y} B_{s'}^* e_{j,n} \, e_{i,d}^\top A_{s'}^* \bm{R}(z) M \bm{R}(z) A_r e_{i,d} e_{j,n}^\top B_r \tilde{\bm{R}}(z) B_s^* e_{j,n} \ .
	\end{align}
	The calculations (\ref{Eq_MixedThirdDerivCalc2})-(\ref{Eq_MixedThirdDerivCalc4}) by the same method as in (\ref{Eq_ThridDerivBound1}) and (\ref{Eq_ThridDerivBound2}) yield (\ref{Eq_XXXDerivf1Result}) for $(a,b) \in \{(0,3), (1,2)\}$. \qed
\end{itemize}

In the belief that bounds of the type given in Lemma \ref{Lemma_SixthMoment} must have been established in the literature, the author asked ChatGPT to search for such references. While no direct references were found, ChatGPT supplied the basic idea behind the following proof of Lemma \ref{Lemma_SixthMoment} which, while requiring significant correction in part (iii) of the proof, was otherwise correct.

\subsection{Proof of Lemma~\ref{Lemma_SixthMoment}}\label{Proof_Lemma_SixthMoment}
\begin{itemize}
	\item[i)] \textit{Decomposition}:\\
	Decompose $\bm{X}$ into
	\begin{align}\label{Eq_SixMoment_X_Decomp1}
		& \bm{X} = \ul{\bm{X}} + \ul{\bm{Y}}
	\end{align}
	for random $(d \times n)$-matrices $\ul{\bm{X}}$ and $\ul{\bm{Y}}$ defined by
	\begin{align}\label{Eq_SixMoment_Def_ulX_ulY}
		& \ul{\bm{X}}_{i,j} \coloneq \mathbbm{1}_{|\bm{X}_{i,j}| \leq \sqrt{n}} \, \bm{X}_{i,j} \ \ \text{ and } \ \ \ul{\bm{Y}}_{i,j} \coloneq \mathbbm{1}_{|\bm{X}_{i,j}| > \sqrt{n}} \, \bm{X}_{i,j} \ .
	\end{align}
	
	\item[ii)] \textit{Bounding $\E\big[ ||\ul{\bm{Y}}||^6 \big]$}:\\
	Using Minkowski's inequality, the mean $\E\big[ ||\ul{\bm{Y}}||^6 \big]^{\frac{1}{3}}$ may be bounded by
	\begin{align}\label{Eq_SixMoment_ulY_Decomp}
		& \E\big[ ||\ul{\bm{Y}}||^6 \big]^{\frac{1}{3}} \leq \E\big[ ||\ul{\bm{Y}}||_F^6 \big]^{\frac{1}{3}} = \E\Big[ \Big( \sum\limits_{i=1}^d \sum\limits_{j=1}^n |\ul{\bm{Y}}_{i,j}|^2 \Big)^3 \Big]^{\frac{1}{3}} \nonumber\\
		& = \E\Big[ \Big( \sum\limits_{i=1}^d \sum\limits_{j=1}^n \big(|\ul{\bm{Y}}_{i,j}|^2 - \E[|\ul{\bm{Y}}_{i,j}|^2] + \E[|\ul{\bm{Y}}_{i,j}|^2] \big) \Big)^3 \Big]^{\frac{1}{3}} \nonumber\\
		& \leq \E\Big[ \Big| \sum\limits_{i=1}^d \sum\limits_{j=1}^n \big(|\ul{\bm{Y}}_{i,j}|^2 - \E[|\ul{\bm{Y}}_{i,j}|^2] \big) \Big|^3 \Big]^{\frac{1}{3}} + \underbrace{\E\Big[ \Big( \sum\limits_{i=1}^d \sum\limits_{j=1}^n \E[|\ul{\bm{Y}}_{i,j}|^2] \Big)^3 \Big]^{\frac{1}{3}}}_{= \sum\limits_{i=1}^d \sum\limits_{j=1}^n \E[|\ul{\bm{Y}}_{i,j}|^2]} \ .
	\end{align}
	For the first summand, Rosenthal's inequality (see Theorem 9.1 in \cite{Gut_Probability}) yields
	\begin{align*}
		& \E\Big[ \Big| \sum\limits_{i=1}^d \sum\limits_{j=1}^n \big(|\ul{\bm{Y}}_{i,j}|^2 - \E[|\ul{\bm{Y}}_{i,j}|^2] \big) \Big|^3 \Big]\\
		& \leq D_3 \max\bigg( \sum\limits_{i=1}^d \sum\limits_{j=1}^n \E\big[ \big| |\ul{\bm{Y}}_{i,j}|^2 - \E[|\ul{\bm{Y}}_{i,j}|^2] \big|^3 \big] , \Big( \sum\limits_{i=1}^d \sum\limits_{j=1}^n \E\big[ \big| |\ul{\bm{Y}}_{i,j}|^2 - \E[|\ul{\bm{Y}}_{i,j}|^2] \big|^2 \big] \Big)^{\frac{3}{2}} \bigg)
	\end{align*}
	for some universal constant $D_3>0$. The bounds
	\begin{align*}
		& \sum\limits_{i=1}^d \sum\limits_{j=1}^n \E\big[ \big| |\ul{\bm{Y}}_{i,j}|^2 - \E[|\ul{\bm{Y}}_{i,j}|^2] \big|^3 \big] \leq \sum\limits_{i=1}^d \sum\limits_{j=1}^n \big( 4\E\big[ |\ul{\bm{Y}}_{i,j}|^6 \big] + 4 \overbrace{\E[|\ul{\bm{Y}}_{i,j}|^2]^3}^{\overset{\substack{\text{Jensen'} \\ \text{ineq.}}}{\leq} \E[|\ul{\bm{Y}}_{i,j}|^6]} \big)\\
		& \overset{\text{(\ref{Eq_SixMoment_Def_ulX_ulY})}}{\leq} 8 \sum\limits_{i=1}^d \sum\limits_{j=1}^n \E[|\bm{X}_{i,j}|^6] \overset{\text{(\ref{Eq_SixthMoment_K6Condition})}}{\leq} 8 dn K_6
	\end{align*}
	and
	\begin{align*}
		& \sum\limits_{i=1}^d \sum\limits_{j=1}^n \E\big[ \big| |\ul{\bm{Y}}_{i,j}|^2 - \E[|\ul{\bm{Y}}_{i,j}|^2] \big|^2 \big] \leq \sum\limits_{i=1}^d \sum\limits_{j=1}^n \big( 2\E\big[ |\ul{\bm{Y}}_{i,j}|^4 \big] + 2 \overbrace{\E[|\ul{\bm{Y}}_{i,j}|^2]^2}^{\overset{\substack{\text{Jensen's} \\ \text{ineq.}}}{\leq} \E[|\ul{\bm{Y}}_{i,j}|^4]} \big)\\
		& \overset{\text{(\ref{Eq_SixMoment_Def_ulX_ulY})}}{\leq} 4 \sum\limits_{i=1}^d \sum\limits_{j=1}^n \E[|\bm{X}_{i,j}|^4] \overset{\substack{\text{Jensen's} \\ \text{ineq.}}}{\leq} 4 \sum\limits_{i=1}^d \sum\limits_{j=1}^n \E[|\bm{X}_{i,j}|^6]^{\frac{4}{6}} \overset{\text{(\ref{Eq_SixthMoment_K6Condition})}}{\leq} 4 dn K_6^{\frac{2}{3}}
	\end{align*}
	thus give
	\begin{align}\label{Eq_SixthMoment_Rosenthal2}
		& \E\Big[ \Big| \sum\limits_{i=1}^d \sum\limits_{j=1}^n \big(|\ul{\bm{Y}}_{i,j}|^2 - \E[|\ul{\bm{Y}}_{i,j}|^2] \big) \Big|^3 \Big] \leq D_3 \max\Big( 8 dn K_6 , \big( 4 dn K_6^{\frac{2}{3}} \big)^{\frac{3}{2}} \Big) \nonumber\\
		& \leq 8K_6 D_3 \max\big( dn, (dn)^{\frac{3}{2}} \big) = 8K_6 D_3 (dn)^{\frac{3}{2}} \overset{\text{\ref{ItemAssumption_cBound}}}{\leq} 8K_6 D_3 c_*^{\frac{3}{2}} n^3 \ .
	\end{align}
	For the second summand in (\ref{Eq_SixMoment_ulY_Decomp}), observe
	\begin{align}\label{Eq_SixthMoment_Calc1}
		& \sum\limits_{i=1}^d \sum\limits_{j=1}^n \E[|\ul{\bm{Y}}_{i,j}|^2] \overset{\text{(\ref{Eq_SixMoment_Def_ulX_ulY})}}{=} \sum\limits_{i=1}^d \sum\limits_{j=1}^n \E[\mathbbm{1}_{|\bm{X}_{i,j}| > \sqrt{n}} \, |\bm{X}_{i,j}|^2] \nonumber\\
		& \overset{\substack{\text{Hölder} \\ \text{ineq.}}}{\leq} \sum\limits_{i=1}^d \sum\limits_{j=1}^n \bP\big( |\bm{X}_{i,j}| > \sqrt{n} \big)^{\frac{2}{3}} \E[|\bm{X}_{i,j}|^6]^{\frac{1}{3}} \nonumber\\
		& \overset{\substack{\text{Markov} \\ \text{ineq.}}}{\leq} \sum\limits_{i=1}^d \sum\limits_{j=1}^n \Big( \frac{K_6}{n^3} \Big)^{\frac{2}{3}} K_6^{\frac{1}{3}} = K_6 \frac{dn}{n^2} = K_6 \frac{d}{n} \overset{\text{\ref{ItemAssumption_cBound}}}{\leq} K_6 c_* \ ,
	\end{align}
	which with (\ref{Eq_SixthMoment_Rosenthal2}) may be inserted into (\ref{Eq_SixMoment_ulY_Decomp}) for
	\begin{align*}
		& \E\big[ ||\ul{\bm{Y}}||^6 \big]^{\frac{1}{3}} \leq \big(8K_6 D_3 c_*^{\frac{3}{2}} n^3\big)^{\frac{1}{3}} + K_6 c_* \ ,
	\end{align*}
	which by the standard bound $(a+b)^3 \leq 4a^3 + 4b^3$ gives
	\begin{align}\label{Eq_SixMoment_ulY_SpectralBound}
		& \E\big[ ||\ul{\bm{Y}}||^6 \big] \leq 2^5 K_6 D_3 c_*^{\frac{3}{2}} n^3 + 4K_6^3 c_*^3 \ .
	\end{align}
	
	\item[iii)] \textit{Bounding $\E\big[ ||\ul{\bm{X}}|| \big]$}:\\
	Latala's bound from \cite{Latala_SpectralBound} yields
	\begin{align}\label{Eq_SixMoment_Latala0}
		& \E\big[ ||\ul{\bm{X}}-\E[\ul{\bm{X}}]|| \big] \nonumber\\
		& \leq C \bigg( \max\limits_{i \leq d} \sqrt{\sum\limits_{j=1}^n \E\big[ |\ul{\bm{X}}_{i,j}-\E[\ul{\bm{X}}_{i,j}]|^2 \big]} + \max\limits_{j \leq n} \sqrt{\sum\limits_{i=1}^d \E\big[ |\ul{\bm{X}}_{i,j}-\E[\ul{\bm{X}}_{i,j}]|^2 \big]} \nonumber\\
		& \hspace{1cm} + \Big(\sum\limits_{i=1}^d \sum\limits_{j=1}^n \E\big[ |\ul{\bm{X}}_{i,j}-\E[\ul{\bm{X}}_{i,j}]|^4 \big]\Big)^{\frac{1}{4}} \bigg) \nonumber\\
		& \overset{\text{(\ref{Eq_SixthMoment_K6Condition})}}{\leq} C \bigg( \max\limits_{i \leq d} \sqrt{\sum\limits_{j=1}^n 2 K_6^{\frac{1}{3}}} + \max\limits_{j \leq n} \sqrt{\sum\limits_{i=1}^d 2 K_6^{\frac{1}{3}}} + \Big(\sum\limits_{i=1}^d \sum\limits_{j=1}^n 8 K_6^{\frac{2}{3}}\Big)^{\frac{1}{4}} \bigg) \nonumber\\
		& \leq 2 C K_6^{\frac{1}{6}} \underbrace{\big( \sqrt{n} + \sqrt{d} + (dn)^{\frac{1}{4}} \big)}_{\leq (n^{\frac{1}{4}} + d^{\frac{1}{4}})^2 \leq 2\sqrt{n}+2\sqrt{d}} \overset{\text{\ref{ItemAssumption_cBound}}}{\leq} 4 C K_6^{\frac{1}{6}} (1+\sqrt{c_*}) \sqrt{n}
	\end{align}
	for some universal constant $C>0$. Additionally, one has
	\begin{align*}
		& ||\E[\ul{\bm{X}}]||^6 \overset{\text{(\ref{Eq_SixMoment_X_Decomp1})}}{=} ||\overbrace{\E[\bm{X}]}^{=0} - \E[\ul{\bm{Y}}]||^6 = ||\E[\ul{\bm{Y}}]||^6 \nonumber\\
		& \overset{\substack{\text{Jensen's} \\ \text{ineq.}}}{\leq} \E\big[ ||\ul{\bm{Y}}||^6 \big] \overset{\text{(\ref{Eq_SixMoment_ulY_SpectralBound})}}{\leq} 2^5 K_6 D_3 c_*^{\frac{3}{2}} n^3 + 4K_6^3 c_*^3 \ ,
	\end{align*}
	which may be combined with (\ref{Eq_SixMoment_Latala0}) for
	\begin{align}\label{Eq_SixMoment_Latala}
		& \E\big[ ||\ul{\bm{X}}|| \big] \leq 4 C K_6^{\frac{1}{6}} (1+\sqrt{c_*}) \sqrt{n} + \big( 2^5 K_6 D_3 c_*^{\frac{3}{2}} n^3 + 4K_6^3 c_*^3 \big)^{\frac{1}{6}} \nonumber\\
		& \leq \big( \underbrace{4 C K_6^{\frac{1}{6}} (1+\sqrt{c_*}) + \big( 2^5 K_6 D_3 c_*^{\frac{3}{2}} + 4K_6^3 c_*^3 \big)^{\frac{1}{6}}}_{\eqcolon  K} \big) \sqrt{n} \ .
	\end{align}
	
	\item[iv)] \textit{Bounding $\E\big[ ||\ul{\bm{X}}||^6 \big]$}:\\
	Coming to the bound of $\E\big[ ||\ul{\bm{X}}||^6 \big]$, Theorem 1 of \cite{Meckes_SpectralBound} with $p=2=q$ and $D=2\sqrt{n}$ yields
	\begin{align}\label{Eq_SixMoment_Meckes1}
		& \forall t>0 : \ \bP\Big( \big| \, ||\ul{\bm{X}}|| - m \big| \geq t \Big) \leq 4\exp\Big( -\frac{t^2}{16 n} \Big) \ ,
	\end{align}
	where $m$ denotes the median of $||\ul{\bm{X}}||$. It is a trivial consequence of Markov's inequality that $m \leq 2\E[||\ul{\bm{X}}||] \overset{\text{(\ref{Eq_SixMoment_Latala})}}{\leq} 2K \sqrt{n}$, so the above bound yields
	\begin{align}\label{Eq_SixMoment_Meckes2}
		& \forall t>0 : \ \bP\Big( ||\ul{\bm{X}}|| \geq 2K \sqrt{n} + t \Big) \leq 4\exp\Big( -\frac{t^2}{16 n} \Big)
	\end{align}
	and thus also
	\begin{align}\label{Eq_SixMoment_Meckes3}
		& \E\big[ ||\ul{\bm{X}}||^6 \big] = 6 \int_0^\infty t^5 \bP\big( ||\ul{\bm{X}}|| > t \big) \, dt \nonumber\\
		& = 6 \int_0^{2K\sqrt{n}} t^5 \bP\big( ||\ul{\bm{X}}|| > t \big) \, dt + 6 \int_0^\infty (2K\sqrt{n}+t)^5 \bP\big( ||\ul{\bm{X}}|| > 2K\sqrt{n}+t \big) \, dt \nonumber\\
		& \overset{\text{(\ref{Eq_SixMoment_Meckes2})}}{\leq} \underbrace{6 \int_0^{2K\sqrt{n}} t^5 \, dt}_{= (2K\sqrt{n})^6} +  24 \underbrace{\int_0^\infty (2K\sqrt{n}+t)^5 \exp\Big( -\frac{t^2}{16 n} \Big) \, dt}_{\substack{= 64n^3 (64+80K^2 + 10K^4 \\ \hspace{1.5cm} +\sqrt{\pi}(60K+20K^3+4K^5) ) \hspace{-1cm}}} = K' n^3
	\end{align}
	for some constant $K'=K'(c_3)>0$.
	
	\item[v)] \textit{Bringing the bounds together for (\ref{Eq_SixthMoment_Result})}:\\
	Combining the results of (i), (ii) and (iv) yields
	\begin{align*}
		& \E\big[ ||\bm{X}||^6 \big] \overset{\text{(\ref{Eq_SixMoment_X_Decomp1})}}{=} \E\big[ ||\ul{\bm{X}} + \ul{\bm{Y}}||^6 \big] \leq 2^5 \E\big[ ||\ul{\bm{X}}||^6 \big] + 2^5 \E\big[ ||\ul{\bm{Y}}||^6 \big]\\
		& \overset{\substack{\text{(\ref{Eq_SixMoment_Meckes3})} \\ \text{(\ref{Eq_SixMoment_ulY_SpectralBound})}}}{\leq} 2^5 \cdot K' n^3 + 2^5 \cdot \big( 2^5 K_6 D_3 c_*^{\frac{3}{2}} n^3 + 4K_6^3 c_*^3 \big) \leq K n^3 \ ,
	\end{align*}
	which proves (\ref{Eq_SixthMoment_Result}) for $\mathcal{C} = 2^5 K' + 2^{10} K_6 D_3 c_*^{\frac{3}{2}} + 2^7 K_6^3 c_*^3$.
	\qed
\end{itemize}

\subsection{Proof of Lemma~\ref{Lemma_Prokhorov}}\label{Proof_Lemma_Prokhorov}
Observe the finite set of vectors
\begin{align*}
	& V \coloneq \{e_{r,R} + e_{s,R} \mid r,s \leq R\} \cup \{e_{r,R} - e_{s,R} \mid r,s \leq R\}\\
	& \hspace{1cm} \cup \{e_{r,R} + \bm{i} e_{s,R} \mid r,s \leq R\} \cup \{e_{r,R} - \bm{i} e_{s,R} \mid r,s \leq R\} \ ,
\end{align*}
where $e_{r,R} \coloneq (\mathbbm{1}_{j=r})_{j \leq R} \in \C^R$.
The polarization identity
\begin{align}\label{Eq_Prokhorov_Polariaztion}
	& v^* M w = \frac{1}{4} (v+w)^* M (v+w) - \frac{1}{4} (v-w)^* M (v-w) \nonumber\\
	& \hspace{1.5cm} + \frac{\bm{i}}{4} (v-\bm{i}w)^* M (v-\bm{i}w) - \frac{\bm{i}}{4} (v+\bm{i}w)^* M (v+\bm{i}w)
\end{align}
shows that any positive semi-definite matrix $M \in \C^{R \times R}$ is uniquely determined by the values $(v^* M v)_{v \in V}$.
\\[0.5em]
For each $v \in V$, the sequence of Radon measures $v^* \mu_n v$ on $\R$ satisfies $v^* \mu_n(\R) v \overset{\text{(a)}}{\leq} c ||v||_2^2$ and $\supp(v^* \mu_n v) \subset [0,C]$ for all $n \geq N$. Prokhorov's theorem thus guarantees that $(v^* \mu_n v)_{n \in \N}$ is relatively compact with regards to the topology of weak convergence. There thus exists a sub-sequence $(n_k)_{k \in \N}$ such that
\begin{align}\label{Eq_Prokhorov_VConvergence}
	& \forall v \in V : \ v^* \mu_{n_k} v \xRightarrow{k \to \infty} \nu_v
\end{align}
for some family of Radon measures $(\nu_v)_{v \in V}$. Define the map
\begin{align*}
	& \mu : \cB(\R) \rightarrow \C^{R \times R}
\end{align*}
by
\begin{align}\label{Eq_Prokhorov_Def_mu}
	& \forall r,s \leq R : \ \mu_{r,s} = e_{r,R}^* \mu e_{s,R} \coloneq \frac{1}{4} \nu_{(e_{r,R}+e_{s,R})} - \frac{1}{4} \nu_{(e_{r,R}-e_{s,R})} \nonumber\\
	& \hspace{5cm} + \frac{\bm{i}}{4} \nu_{(e_{r,R}-\bm{i}e_{s,R})} - \frac{\bm{i}}{4} \nu_{(e_{r,R}+\bm{i}e_{s,R})} \ .
\end{align}
Each entry $\mu_{r,s}$, is thus by construction a complex measure, which is finite by (a) and (\ref{Eq_Prokhorov_VConvergence}).
It remains to show that $\mu$ is a matrix-valued measure and that
\begin{align}\label{Eq_Prokhorov_Conv}
	& \forall w \in \C^R : \ w^* \mu_{n_k} w \xRightarrow{k \to \infty} w^* \mu w \ .
\end{align}
holds.
\\[0.5em]
One may use (\ref{Eq_Prokhorov_Polariaztion}) and (\ref{Eq_Prokhorov_VConvergence}) to for every $w \in \C^R$ and $f \in \C_b(\R;\R)$ see
\begin{align}\label{Eq_Prokhorov_fConvergence}
	& \int_\R f \, d\big(w^* \mu_{n_k} w\big) = \sum\limits_{r,s=1}^R \ol{w_r} w_s \, e_{r,R}^* \Big(\int_\R f \, d\mu_{n_k}\Big) e_{s,R} \nonumber\\
	& \overset{\text{(\ref{Eq_Prokhorov_Polariaztion})}}{=} \frac{1}{4}\sum\limits_{r,s=1}^R \ol{w_r} w_s \bigg( (e_{r,R}+e_{s,R})^* \Big(\int_\R f \, d\mu_{n_k}\Big) (e_{r,R}+e_{s,R}) \nonumber\\
	& \hspace{3cm} - (e_{r,R}-e_{s,R})^* \Big(\int_\R f \, d\mu_{n_k}\Big) (e_{r,R}-e_{s,R}) \nonumber\\
	& \hspace{3cm} + \bm{i}(e_{r,R}-\bm{i}e_{s,R})^* \Big(\int_\R f \, d\mu_{n_k}\Big) (e_{r,R}-\bm{i}e_{s,R}) \nonumber\\
	& \hspace{3cm} - \bm{i}(e_{r,R}+\bm{i}e_{s,R})^* \Big(\int_\R f \, d\mu_{n_k}\Big) (e_{r,R}+\bm{i}e_{s,R}) \bigg) \nonumber\\
	& = \frac{1}{4}\sum\limits_{r,s=1}^R \ol{w_r} w_s \bigg( \int_\R f \, d\big((e_{r,R}+e_{s,R})^* \mu_{n_k} (e_{r,R}+e_{s,R})\big) \nonumber\\
	& \hspace{3cm} - \int_\R f \, d\big((e_{r,R}-e_{s,R})^* \mu_{n_k} (e_{r,R}-e_{s,R})\big) \nonumber\\
	& \hspace{3cm} + \bm{i}\int_\R f \, d\big((e_{r,R}-\bm{i}e_{s,R})^* \mu_{n_k} (e_{r,R}-\bm{i}e_{s,R})\big) \nonumber\\
	& \hspace{3cm} - \bm{i}\int_\R f \, d\big((e_{r,R}+\bm{i}e_{s,R})^* \mu_{n_k} (e_{r,R}+\bm{i}e_{s,R})\big) \bigg) \nonumber\\
	& \underset{\text{(\ref{Eq_Prokhorov_VConvergence})}}{\xrightarrow{k \to \infty}} \frac{1}{4}\sum\limits_{r,s=1}^R \ol{w_r} w_s \bigg( \int_\R f \, d\nu_{(e_{r,R}+e_{s,R})} - \int_\R f \, d\nu_{(e_{r,R}-e_{s,R})} \nonumber\\
	& \hspace{3cm} + \bm{i}\int_\R f \, d\nu_{(e_{r,R}-\bm{i}e_{s,R})} - \bm{i}\int_\R f \, d\nu_{(e_{r,R}+\bm{i}e_{s,R})} \bigg) \nonumber\\
	& \overset{\text{(\ref{Eq_Prokhorov_Def_mu})}}{=} \sum\limits_{r,s=1}^R \ol{w_r} w_s \int_\R f \, d\mu_{r,s} = \int_\R f \, d\big(w^* \mu w\big) \ .
\end{align}
Consequently, for every non-negative $f \in \C_b(\R;[0,\infty))$, one has
\begin{align}\label{Eq_Prokhorov_ReisMarkov_Condition}
	& \int_\R f \, d\big(w^* \mu w\big) \geq 0 \ ,
\end{align}
where until now $w^* \mu w$ was only known to be a (finite) complex-valued measure.
The Riesz-Markov-Kakutani representation theorem with (\ref{Eq_Prokhorov_ReisMarkov_Condition}) proves that $w^* \mu w$ is a Radon measure on $\R$.
The convergence (\ref{Eq_Prokhorov_fConvergence}) for $f=1$ by (a) also yields $w^* \mu(\R) w \leq c ||w||_2^2$ and the non-negativity of $w^* \mu(S) w$ for each $S \in \cB(\R)$ and $w \in \C^R$ thus proves that $\mu$ is a matrix-valued measure. The convergence (\ref{Eq_Prokhorov_Conv}) follows directly from (\ref{Eq_Prokhorov_fConvergence}) by definition of weak convergence of matrix-valued measures (see Lemma~\ref{Lemma_MatrWeakConv}). \qed

\subsection{Proof of Lemma~\ref{Lemma_DetEquivBound}}\label{Proof_Lemma_DetEquivBound}
Calculate
\begin{align*}
	& \lambda_{\min}\Big(\Im\Big( \Id_d + \sum\limits_{r',s'=1}^R \delta^{(B)}_{r',s'}(z) A_{r'} A_{s'}^* \Big)\Big) = \lambda_{\min}\Big(\sum\limits_{r',s'=1}^R \Im\big(\delta^{(B)}_{r',s'}(z)\big) A_{r'} A_{s'}^*\Big)\\
	& \overset{\text{(\ref{Eq_NonDegeneracyC_A})}}{\geq} \tau \lambda_{\min}\big( \Im\big(\delta^{(B)}(z)\big) \big) \overset{\text{(\ref{Eq_Uniqueness_deltaStil})}}{=} \tau \lambda_{\min}\Big( \int_\R \Im\Big(\frac{1}{\lambda-z}\Big) \, d\rho^{(B)}(\lambda) \Big)\\
	& = \tau \lambda_{\min}\Big( \int_\R \frac{\Im(z)}{|\lambda-z|^2} \, d\rho^{(B)}(\lambda) \Big) \overset{\text{(\ref{Eq_DetEquiv_RhoProp1})}}{\geq} \tau \lambda_{\min}\Big( \int_\R \frac{\Im(z)}{(8\sigma^4(1+\sqrt{c_*})^2+|z|)^2} \, d\rho^{(B)}(\lambda) \Big)\\
	& \overset{\text{(\ref{Eq_DetEquiv_RhoProp2})}}{=} \frac{\tau \Im(z)}{(8\sigma^4(1+\sqrt{c_*})^2+|z|)^2} \lambda_{\min}\Big( \frac{1}{n} \tr\big( B_s^*B_r \big)_{r,s \leq R} \Big) \overset{\text{\ref{ItemTempAssumption_NonDegeneracy}}}{\geq} \frac{\tau^2 \Im(z)}{(8\sigma^4(1+\sqrt{c_*})^2+|z|)^2} \ ,
\end{align*}
which by the simple bound $||A^{-1}|| \leq \frac{1}{\lambda_{\min}(\Im(A))}$ proves (\ref{Eq_DetEquivBoundA}). The bound (\ref{Eq_DetEquivBoundB}) may be shown in complete analogy. \qed

\newpage
\markright{List of symbols}
\markleft{List of symbols}
\section*{List of symbols}\label{ListOfSymbols_SepCov}
\begin{itemize}
	
	\item[] $A_r$ \tabto{1.7cm} deterministic $(d \times d)$-matrices which are part of the model\\
	\tabto{1.7cm} (see Subsection~\ref{Subsection_Model} or~\ref{ItemAssumption_sigmaBound})
	
	\item[] $B_r$ \tabto{1.7cm} deterministic $(n \times n)$-matrices which are part of the model\\
	\tabto{1.7cm} (see Subsection~\ref{Subsection_Model} or~\ref{ItemAssumption_sigmaBound})
	
	\item[] $C$ \tabto{1.7cm} placeholder for a locally defined constant (see e.g. Lemma~\ref{Lemma_ZTailBound})
	
	\item[] $\mathcal{C}$ \tabto{1.7cm} placeholder for a locally defined constant (see e.g. Theorems~\ref{Thm_SepCovMP_NonAsymp} or~\ref{Thm_DualApprox})
	
	\item[] $C_6$ \tabto{1.7cm} a constant (see~\ref{ItemAssumption_boundedSixthMoment})
	
	\item[] $\C^+$ \tabto{1.7cm} the set of complex numbers with positive imaginary part (see (\ref{Eq_hatDeltaIntro}))
	
	\item[] $c_*$ \tabto{1.7cm} a constant (see~\ref{ItemAssumption_cBound} or (c) in Theorem~\ref{Thm_DetEquiv})
	
	\item[] $\bD(\eta,\kappa)$ \tabto{1.7cm} a spectral domain (see Definition~\ref{Def_SpectralDomain})
	
	\item[] $d$ \tabto{1.7cm} number of rows in the data-matrix $\bm{Y}$\\
	\tabto{1.7cm} (see Subsection~\ref{Subsection_Model} or~\ref{ItemAssumption_cBound})
	
	\item[] $\diag$ \tabto{1.7cm} an operator for turning a sequence of numbers into the diagonal matrix\\
	\tabto{1.7cm} with said numbers on its diagonal
	
	\item[] $\dist$ \tabto{1.7cm} the distance between subsets of $\C$
	
	\item[] $\delta$ \tabto{1.7cm} symbol used to denote small constants in various settings
	
	\item[] $\delta^{(A/B)}(z)$ \tabto{1.7cm} deterministic equivalent (matrix-valued) Stieltjes transforms (see Theorem~\ref{Thm_DetEquiv})
	
	\item[] $\hat{\delta}^{(A/B)}(z)$ \tabto{1.7cm} empirical matrix-valued Stieltjes transforms (see (\ref{Eq_Def_hatDelta_A}) and (\ref{Eq_Def_hatDelta_B}))
	
	\item[] $e_{\bullet,d}$ \tabto{1.7cm} standard unit vectors in $\C^{d}$ (see (\ref{Eq_Def_UnitVec}))
	
	\item[] $\varepsilon$ \tabto{1.7cm} placeholder for a (small) constant (see e.g. Theorem~\ref{Thm_SepCovMP_NonAsymp})
	
	\item[] $\eta$ \tabto{1.7cm} a parameter for the spectral domain $\bD(\eta,\kappa)$ (see Definition~\ref{Def_SpectralDomain})
	
	\item[] $\Id_d$ \tabto{1.7cm} the $(d \times d)$ identity matrix
	
	\item[] $\Im(z)$ \tabto{1.7cm} the imaginary part of a complex number $z \in \C$
	
	\item[] $\Im(A)$ \tabto{1.7cm} the imaginary part of a matrix (see (\ref{Eq_DefImaginaryMatrix}))
	
	\item[] $\bm{i}$ \tabto{1.7cm} the imaginary unit
	
	\item[] $K$ \tabto{1.7cm} placeholder for a locally defined constant (see e.g. proof of Lemma~\ref{Lemma_ConcentrationLargeExpectation})
	
	\item[] $K_6$ \tabto{1.7cm} placeholder for a locally defined constant (see e.g. Lemma~\ref{Lemma_SixthMoment})
	
	\item[] $\kappa$ \tabto{1.7cm} a parameter for the spectral domain $\bD(\eta,\kappa)$ (see Definition~\ref{Def_SpectralDomain})
	
	\item[] $\lambda$ \tabto{1.7cm} an integration variable (see e.g. (\ref{Eq_Def_hatDelta_A}) or (\ref{Eq_DetEquiv_NuStilProp}))
	
	\item[] $\lambda_j(M)$ \tabto{1.7cm} an eigenvalue of a Hermitian matrix $M$ (see above (\ref{Eq_EigenvalueOrdering}))
	
	\item[] $M$ \tabto{1.7cm} placeholder for a deterministic $(d \times d)$ test-matrix\\
	\tabto{1.7cm} (see e.g. Theorems~\ref{Thm_SepCovMP_NonAsymp} or~\ref{Thm_EmpiricalDualSystem})
	
	\item[] $\tilde{M}$ \tabto{1.7cm} placeholder for a deterministic $(n \times n)$ test-matrix\\
	\tabto{1.7cm} (see e.g. Theorems~\ref{Thm_SepCovMP_NonAsymp} or~\ref{Thm_EmpiricalDualSystem})
	
	\item[] $n$ \tabto{1.7cm} number of columns in the data-matrix $\bm{Y}$\\
	\tabto{1.7cm} (see Subsection~\ref{Subsection_Model} or~\ref{ItemAssumption_cBound})
	
	\item[] $\ul{\nu}$ \tabto{1.7cm} deterministic equivalent probability measure (see Theorem~\ref{Thm_DetEquiv})
	
	\item[] $\bm{Q}(z)$ \tabto{1.7cm} a specially defined $(d \times d)$ random variable (see (\ref{Eq_Def_Q}))
	
	\item[] $\tilde{\bm{Q}}(z)$ \tabto{1.7cm} a specially defined $(n \times n)$ random variable (see (\ref{Eq_Def_tQ}))
	
	\item[] $R$ \tabto{1.7cm} number of summands to the separable covariance mixture model\\
	\tabto{1.7cm} (see Subsection~\ref{Subsection_Model})
	
	\item[] $\bm{R}(z)$ \tabto{1.7cm} a $(d \times d)$ resolvent matrix (see (\ref{Eq_Def_RResolvent}) and (\ref{Eq_Def_R_Robust}))
	
	\item[] $\tilde{\bm{R}}(z)$ \tabto{1.7cm} an $(n \times n)$ resolvent matrix (see (\ref{Eq_Def_RResolvent}) and (\ref{Eq_Def_R_Robust}))
	
	\item[] $\Re(z)$ \tabto{1.7cm} the real part of a complex number $z \in \C$
	
	\item[] $\Re(A)$ \tabto{1.7cm} the real part of a matrix (see (\ref{Eq_DefImaginaryMatrix}))
	
	\item[] $\rho^{(A/B)}$ \tabto{1.7cm} deterministic equivalent (matrix-valued) measures (see Theorem~\ref{Thm_DetEquiv})
	
	\item[] $\hat{\rho}^{(A/B)}$ \tabto{1.7cm} empirical matrix-valued measures (see (\ref{Eq_Def_hatRhoA}) and (\ref{Eq_Def_hatRhoB}))
	
	\item[] $\tilde{\bm{S}}$ \tabto{1.7cm} an $(n \times n)$ sample covariance matrix analogue (see (\ref{Eq_DefS}) and (\ref{Eq_Def_S_Robust}))
	
	\item[] $\bm{S}$ \tabto{1.7cm} a $(d \times d)$ sample covariance matrix analogue (see (\ref{Eq_DefS}) and (\ref{Eq_Def_S_Robust}))
	
	\item[] $\tilde{\bm{S}}$ \tabto{1.7cm} an $(n \times n)$ sample covariance matrix analogue (see (\ref{Eq_DefS}) and (\ref{Eq_Def_S_Robust}))
	
	\item[] $\cs_{\mu}(z)$ \tabto{1.7cm} the (scalar-valued) Stieltjes transform of a probability measure $\mu$\\
	\tabto{1.7cm} (see e.g. (d) in Theorem~\ref{Thm_DetEquiv})
	
	\item[] $\supp$ \tabto{1.7cm} the support of a measure
	
	\item[] $\sigma^2$ \tabto{1.7cm} a constant (see~\ref{ItemAssumption_sigmaBound} or (c) in Theorem~\ref{Thm_DetEquiv})
	
	\item[] $\tr$ \tabto{1.7cm} the trace of a square matrix
	
	\item[] $\tau$ \tabto{1.7cm} a constant (see~\ref{ItemTempAssumption_NonDegeneracy})
	
	\item[] $\bm{X}$ \tabto{1.7cm} a random $(d \times n)$-matrix with independent, centered, variance-one entries\\
	\tabto{1.7cm} (see Subsection~\ref{Subsection_Model} or~\ref{ItemAssumption_boundedSixthMoment})
	
	\item[] $\bm{Y}$ \tabto{1.7cm} a $(d \times n)$-data-matrix assumed to be of the separable covariance mixture form\\
	\tabto{1.7cm} (see Subsection~\ref{Subsection_Model})
	
	\item[] $\prec$ \tabto{1.7cm} a symbol denoting that a matrix dominates the other in the sense that the\\
	\tabto{1.7cm} difference in positive definite (see above (\ref{Eq_Def_UnitVec}))
	
	\item[] $\preceq$ \tabto{1.7cm} a symbol denoting that a matrix dominates the other in the sense that the\\
	\tabto{1.7cm} difference in positive semi-definite (see above (\ref{Eq_Def_UnitVec}))
	
	\item[] $\|\cdot\|$ \tabto{1.7cm} spectral norm of a matrix
	
	\item[] $\|\cdot\|_F$ \tabto{1.7cm} the Frobenius norm
	
	\item[] $\xRightarrow{n \to \infty}$ \tabto{1.7cm} weak convergence of measures
	
\end{itemize}

\section*{Funding}\label{Funding}
The author acknowledges the support of the Research Unit 5381 (DFG) RO 3766/8-1.

\newpage
\bibliographystyle{apalike}
\phantomsection
\addcontentsline{toc}{section}{References}
\bibliography{literature}
	
\end{document}